\documentclass[11pt]{amsart}

\usepackage{amsmath,amsfonts,amsthm,amssymb,amscd,mathabx}
\usepackage{epsfig}
\usepackage{upref}
\usepackage{bm}
\usepackage{calc}
\usepackage{caption}
\usepackage{subcaption}
\usepackage{hyperref}
\usepackage[usenames,dvipsnames]{xcolor}
\usepackage[normalem]{ulem}
\usepackage{wasysym}
\usepackage{mathrsfs}
\usepackage{enumitem}
\usepackage{setspace}
\usepackage{array}
\usepackage{cancel}
\usepackage{tikz-cd}
\usepackage{colonequals}
\usepackage[titletoc]{appendix}
\usepackage{overpic}
\usepackage{float} 

\newtheorem{theorem}{Theorem}[section]
\newtheorem*{theorem*}{Theorem}	

\newtheorem{corollary}[theorem]{Corollary}
\newtheorem{lemma}[theorem]{Lemma}
\newtheorem{proposition}[theorem]{Proposition}

\newtheorem{assumpt}[theorem]{Assumption}

\theoremstyle{definition}

\newtheorem{defx}[theorem]{Definition}
\newenvironment{definition}
{\pushQED{\qed}\defx}
{\popQED\enddefx}

\newtheorem{examplex}[theorem]{Example}
\newenvironment{example}
{\pushQED{\qed}\examplex}
{\popQED\endexamplex}

\newtheorem{remarkx}[theorem]{Remark}
\newenvironment{remark}
{\pushQED{\qed}\remarkx}
{\popQED\endremarkx}

\numberwithin{equation}{section}

\DeclareMathOperator{\form}{Form}

\DeclareMathOperator{\im}{Im}
\DeclareMathOperator{\re}{Re}

\DeclareMathOperator{\End}{End}
\DeclareMathOperator{\Ind}{Ind}
\DeclareMathOperator{\Id}{Id}
\DeclareMathOperator{\Hess}{Hess}

\DeclareMathOperator{\Hom}{Hom}
\DeclareMathOperator{\Sym}{Sym}

\DeclareMathOperator{\tr}{tr}

\DeclareMathOperator{\Diff}{Diff}
\DeclareMathOperator{\Ric}{Ric}
\DeclareMathOperator{\rank}{rank}

\DeclareMathOperator{\Map}{Map}

\DeclareMathOperator{\Vol}{Vol}
\DeclareMathOperator{\Ch}{CF}
\DeclareMathOperator{\Lie}{Lie}
\DeclareMathOperator{\Mult}{Mult}
\DeclareMathOperator{\Spinc}{Spin^c}
\DeclareMathOperator{\Spincr}{Spin^c_R}

\DeclareMathOperator{\coker}{coker}

\DeclareMathOperator{\map}{Map}

\DeclareMathOperator{\grad}{grad}
\DeclareMathOperator{\Crit}{Crit}
\DeclareMathOperator{\Cob}{Cob}
\DeclareMathOperator{\SCob}{SCob}
\DeclareMathOperator{\Cylin}{Cylin}
\DeclareMathOperator{\Comp}{Comp}
\DeclareMathOperator{\spn}{span}
\DeclareMathOperator{\bd}{\mathbf{d}}

\DeclareMathOperator{\HFK}{HFK}

\DeclareMathOperator{\gr}{\mathbf{gr}}
\DeclareMathOperator{\grt}{\mathbf{gr^{(2)}}}
\DeclareMathOperator{\SW}{\underline{SW}}
\DeclareMathOperator{\Rm}{\mathit{Rm}}

  \DeclareMathOperator{\HM}{\it HM}
\DeclareMathOperator{\KHM}{\it KHM}

\DeclareMathOperator{\SLL}{\it SL}

\newcommand{\Sph}{\mathbb{S}}
\newcommand{\R}{\mathbb{R}}
\newcommand{\C}{\mathbb{C}}

\newcommand{\N}{\mathbb{N}}
\newcommand{\Z}{\mathbb{Z}}

\newcommand{\HH}{\mathbb{H}}
\newcommand{\T}{\mathbb{T}}
\newcommand{\A}{\mathbb{A}}
\newcommand{\BD}{\mathbb{D}}
\newcommand{\BF}{\mathbb{F}}
\newcommand{\x}{\mathbb{X}}
\newcommand{\y}{\mathbb{Y}}
\newcommand{\BW}{\mathbb{W}}

\newcommand{\Step}{\textit{Step }}

\newcommand{\supp}{\text{supp}}
\newcommand{\half}{\frac{1}{2}}

\newcommand{\cA}{\check{A}}

\newcommand{\cPhi}{\check{\Phi}}
\newcommand{\embed}{\hookrightarrow}
\newcommand{\pt}{\partial_t }
\newcommand{\ps}{\partial_s}
\newcommand{\pr}{\partial_r }
\newcommand{\px}{\partial_x }
\newcommand{\dt}{\frac{d}{dt}}
\newcommand{\ds}{\frac{d}{ds}}
\newcommand{\Pt}{\frac{\partial}{\partial t} }
\newcommand{\Ps}{\frac{\partial}{\partial s} }

\newcommand{\CA}{\mathcal{A}}
\newcommand{\CB}{\mathcal{B}}
\newcommand{\SC}{\mathcal{C}}

\newcommand{\D}{\mathcal{D}}
\newcommand{\E}{\mathcal{E}}
\newcommand{\CF}{\mathcal{F}}
\newcommand{\CG}{\mathcal{G}}

\newcommand{\J}{\mathcal{J}}
\newcommand{\K}{\mathcal{K}}
\newcommand{\M}{\mathcal{M}}
\newcommand{\W}{\mathcal{W}}

\newcommand{\SO}{\mathcal{O}}
\newcommand{\Pa}{\mathcal{P}}
\newcommand{\NR}{\mathcal{R}}

\newcommand{\CT}{\mathcal{T}}
\newcommand{\CQ}{\mathcal{Q}}

\newcommand{\SH}{\mathcal{H}}
\newcommand{\X}{\mathcal{X}}

\newcommand{\CL}{\mathcal{L}}
\newcommand{\CX}{\mathcal{X}}
\newcommand{\CY}{\mathcal{Y}}

\newcommand{\V}{\mathcal{V}}
\newcommand{\CZ}{\mathcal{Z}}

\newcommand{\fa}{\mathfrak{a}}
\newcommand{\fb}{\mathfrak{b}}
\newcommand{\fc}{\mathfrak{c}}
\newcommand{\fd}{\mathfrak{d}}
\newcommand{\F}{\mathfrak{F}}

\newcommand{\q}{\mathfrak{q}}
\newcommand{\p}{\mathfrak{p}}

\newcommand{\su}{\mathfrak{su}}

\newcommand{\s}{\mathfrak{s}}
\newcommand{\bs}{\widehat{\mathfrak{s}}}
\newcommand{\FC}{\mathfrak{C}}

\newcommand{\SA}{\mathscr{A}}

\newcommand{\SB}{\mathscr{B}}
\newcommand{\SL}{\mathscr{L}}
\newcommand{\SR}{\mathscr{R}}
\newcommand{\SU}{\mathscr{U}}

\newcommand{\bpartial}{\bar{\partial}}

\newcommand{\hy}{\widehat{Y}}
\newcommand{\hx}{\widehat{X}}
\newcommand{\hz}{\widehat{Z}}
\newcommand{\hr}{\widehat{R}}

\newcommand{\cB}{\check{B}}
\newcommand{\cM}{\widecheck{\mathcal{M}}}

\newcommand{\cgamma}{\check{\gamma}}
\newcommand{\hq}{\widehat{\mathfrak{q}}}

\newcommand{\bomega}{\overline{\omega}}
\newcommand{\tup}{\tilde{\Upsilon}}
\newcommand{\upd}{\Upsilon^\dagger}
\newcommand{\updd}{\Upsilon^\ddagger}

\newcommand{\wQ}{\widetilde{Q}}

\newcommand{\vb}{\delta b}
\newcommand{\va}{\delta a}
\newcommand{\vphi}{\delta \phi}
\newcommand{\vpsi}{\delta \psi}
\newcommand{\vn}{\vec{n}}

\newcommand{\spinc}{$spin^c\ $}

\newcommand{\dg}{\textbf{d}_\gamma}

\newcommand{\bn}{\bm{n}}

\newcommand{\ind}{\text{Ind}}

\newcommand{\hatx}{\widehat{X}}
\newcommand{\hatD}{\widehat{D}_{\kappa_*}}

\newcommand{\EHess}{\widehat{\Hess}}
\newcommand{\CSd}{\rlap{$-$}\mathcal{L}}

\newcommand{\Mod}{\text{-}\mathrm{Mod}}

\title{Monopoles And Landau-Ginzburg Models II:\\ Floer Homology}
\author{Donghao Wang}

\date{\today}

\address{Department of Mathematics, Massachusetts Institute of Technology, Cambridge, MA 02139, USA}
\email{donghaow@mit.edu}

\begin{document}
	
	\setcounter{section}{1}
		
	\begin{abstract} This is the second paper in this series. Following the setup of Meng-Taubes, we define the monopole Floer homology for any pair $(Y,\omega)$, where $Y$ is a compact oriented 3-manifold with toroidal boundary and $\omega$ is a suitable closed 2-form viewed as a decoration. This construction fits into a (3+1)-topological quantum field theory and generalizes the work of Kronheimer-Mrowka for closed oriented 3-manifolds. By a theorem of Meng-Taubes and Turaev, the Euler characteristic of this Floer homology recovers the Milnor-Turaev torsion invariant of the 3-manifold.  
	\end{abstract}
	
	\maketitle
	\tableofcontents

\part{Introduction}

\subsection{Motivations}
The monopole Floer homology of a closed oriented 3-manifold as introduced by Kronheimer-Mrowka \cite{Bible} has greatly influenced the study of 3-manifold topology since its inception. In this paper, we generalize their construction to the case of a compact oriented 3-manifold $Y$ with toroidal boundary, which has the potential to recover the knot Floer homology \cite{KFH,KFH1}, including both the hat-version $\widehat{\HFK}_*$ and the minus-version $\HFK^{-}_*$ as special cases. This generalization, however, depends on an additional choice of a closed 2-form $\omega$ on the 3-manifold $Y$ and a geometric datum $\fd$ on the boundary $\partial Y$, and therefore should be viewed as an invariant of the pair $(Y,\omega)$ relative to this boundary datum $\fd$. As we shall follow closely the setup of Meng-Taubes \cite{MT96}, the Euler characteristic of this monopole Floer homology group recovers the Milnor-Turaev torsion invariant of $Y$ \cite{MT96, T98}, which is independent of the 2-form $\omega$ or the boundary datum $\fd$. 

\smallskip

The first paper of this series \cite{Wang202} focused on the symplectic geometry that underpins this Floer homology: we introduced an infinite dimensional gauged Landau-Ginzburg model for any Riemann surface $(\Sigma, g_{\Sigma})$
\begin{equation}\label{E0.1}
(M(\Sigma), W_\lambda,\CG(\Sigma))
\end{equation}
whose gauged Witten equations on the complex plane $\C$ recover the Seiberg-Witten equations on the product manifold $\C\times\Sigma$. Explicitly, $M(\Sigma)$ is an infinite dimensional K\"{a}hler manifold associated to the surface $\Sigma$, and $W_\lambda: M(\Sigma)\to \C$ is a holomorphic Morse function, commonly referred to as a superpotential. The Seiberg-Witten equations on the 3-manifold $[0,+\infty)_s\times\Sigma$ forms a downward gradient flowline of the real function $\re W_\lambda$ on $M(\Sigma)$. The monopole Floer homology group of $(Y,\omega)$ is then interpreted as the Lagrangian Floer homology of an infinite dimensional Lagrangian submanifold in $M(\Sigma)$, which is associated to the 3-manifold $Y$ using the Seiberg-Witten equations by a work of Nguyen \cite{NguyenI}, with thimbles of the superpotential $W_\lambda: M(\Sigma)\to \C$. 

The monopole Floer homology of $(Y,\omega)$, which we denote temporarily by $\HM_*(Y,\omega)$,  fits into the broader program of constructing the Fukaya-Seidel category of \eqref{E0.1} and developing an infinite dimensional Picard-Lefschetz theory, as proposed first by Donaldson-Thomas \cite{DT96} and then by Haydys \cite{Haydys15} and Gaiotto-Moore-Witten \cite{GMW15,GMW17}. A more detailed plan for this program has appeared in \cite{Wang204} in an attempt to lift this monopole Floer homology into an $A_\infty$-module over the Fukaya-Seidel category of \eqref{E0.1}, providing more refined topological invariants for the 3-manifold $Y$. 

 The present paper focuses on the delicate analysis that implement this infinite dimensional construction for the underlying chain complexes of the $A_\infty$-category/module. The construction of the $A_\infty$-structure following \cite{Wang204} will be explored in a companion paper in the future.

\smallskip

Another motivation of this work is to define Floer-type invariants for knots and links inside a closed 3-manifold using the Seiberg-Witten equations. It has been long believed \cite{M16} that the knot Floer homology of a knot $K\subset Z$, as introduced by Ozsv\'{a}th-Szab\'{o} \cite{KFH} and independently Rasmussen \cite{KFH1} using Heegaard splittings, encodes some information about the Seiberg-Witten moduli space on $\R_t$ times the knot complement $Z\setminus N(K)$. By choosing the closed 2-form $\omega$ and the boundary datum $\fd$ properly, this speculation can be made concrete by applying our construction to the complement $Y=Z\setminus N(K)$ or $Z\setminus N(K\cup m)$,  where $m$ is a meridian of the knot $K\subset Z$:
\begin{align*}
\HM_*(Y,\omega)&\xrightarrow[\cong]{\hspace{1em}?\hspace{1em}}\HFK^{-}_*(S^3, K) &&\text{ if } Y=S^3\setminus N(K), \\
\HM_*(Y,\omega)&\xrightarrow[\cong]{\hspace{1em}?\hspace{1em}} \widehat{\HFK}_*(S^3, K) \text{ or } \KHM_*(S^3,K)&&\text{ if } Y=S^3\setminus N(K\cup m).
\end{align*}
While the first isomorphism is mostly conjectural, we shall confirm the second isomorphism in the third paper \cite{Wang203} for the hat version of the knot Floer homology.  

\smallskip

Some constructions of knot Floer homology already exist in the Seiberg-Witten theory. Motivated by the sutured manifold technique developed by Juh\'{a}sz \cite{J06,J08},  the monopole knot Floer homology $\KHM_*$, which is the analogue of $\widehat{\HFK}_*$, was introduced by Kronheimer-Mrowka \cite{KS} as the sutured monopole Floer homology of the knot complement with two meridional sutures. By further exploring this idea, Li  \cite{L19} constructed the analogue of $\HFK^{-}_*$ in the Seiberg-Witten theory using a direct system of sutures on the knot complement. However, the fact that these invariants fit into a suitable topological quantum field theory was not verified until the foundational work of Juh\'{a}sz \cite{Ju16} and Li \cite{L18}, due to some technical difficulty involved in these constructions.   
\smallskip

As our approach is much more geometric, the functoriality of our Floer homology groups can be easily verified for a special class of cobordisms called \textit{strict}. Following the setup of \cite{MT96}, we complete $(Y,\partial Y)$ into a 3-manifold with cylindrical ends using the boundary datum $\fd$ associated to $\partial Y$. The monopole Floer homology group of $(Y,\omega)$ is then defined as an infinite dimensional Morse homology for the perturbed Chern-Simons-Dirac functional on this completion, while the cobordism maps are constructed by counting Seiberg-Witten monopoles on a completion of the strict cobordism, which is reminiscent of the original construction of Kronheimer-Mrowka \cite{Bible} for the case of closed oriented 3-manifolds. However, this Floer homology group may depend on the choice of $\omega$ and the boundary datum $\fd$. We give a short discussion on this in Section \ref{Sec1.5} \& \ref{Sec1.6} below.

\subsection{Outline of Construction}\label{Subsec1.2} To state our results, let $Y$ be a compact oriented 3-manifold whose boundary $\partial Y\cong \Sigma\colonequals \coprod_{1\leq i\leq n} \T^2_i$ is a union of 2-tori. Throughout this paper, we assume that $Y$ is connected and its boundary $\partial Y$ is non-empty. A boundary datum $\fd=(g_\Sigma, \lambda, \mu)$ on $\Sigma$ is a triple consisting of 
\begin{itemize}
	\item a flat metric $g_\Sigma$ of $\Sigma$;
	\item an imaginary-valued harmonic 1-form $\lambda\in \Omega_h^1(\Sigma; i\R)$;
	\item an imaginary-valued harmonic 2-form $\mu\in \Omega_h^2(\Sigma; i\R)$,
\end{itemize}
which satisfies the conditions \ref{P7}\ref{P3} and \ref{P4} in Section \ref{Sec2}. This means in particular that $\lambda|_{\T_i^2}\neq 0$ for all $1\leq i\leq n$. This datum $\fd$ specifies the gauged Landau-Ginzburg model \eqref{E0.1} associated to $\Sigma$. The 1-form $\lambda$ is used to perturb the superpotential $W_{\lambda}$ to make it Morse. Choose a closed 2-form $\omega\in \Omega^2(Y; i\R)$ such that 
\[
\omega=\mu+ds\wedge\lambda
\]
in a collar neighborhood $(-1,0]_s\times\Sigma\subset Y$. This pair $(Y,\omega)$ along with other perturbation datum to be used in the construction is denoted by a thickened letter $\y$ in this paper; see Section \ref{Sec2} for the precise definition. We focus on the \spinc structure $\bs_{std}$ on $\Sigma$ such that
\[
c_1(\bs_{std})[\T^2_i]=0
\]
on each component $\T^2_i\subset \Sigma$.  For any relative \spinc structure $\bs$ on $Y$, which is a \spinc structure $\s$ along with an identification of $\s$ with $\bs_{std}$ on the boundary $\Sigma$, we follow the setup of Meng-Taubes \cite{MT96} and associate a finitely generated module over a base ring $\NR$:
\begin{equation}
\HM_*(\y,\bs),
\end{equation}
called \textit{the monopole Floer homology group} of $(\y,\bs)$. This group is constructed as an infinite dimensional Morse homology of the perturbed Chern-Simons-Dirac functional $\CL_\omega$ on the complete Riemannian manifold:
\[
\hy\colonequals Y\ \bigcup_\Sigma\ [0,+\infty)_s\times\Sigma, 
\]
where the cylindrical end is equipped with the product metric $d^2s+g_\Sigma$, and the closed 2-form $\omega$ is extended constantly on the end, i.e.,
\[
\omega=\mu+ds\wedge\lambda \text{ on } [0,+\infty)_s\times\Sigma.
\]

The Chern-Simons-Dirac functional $\CL_\omega$ is perturbed by this extension of $\omega$. To think of it another way, the boundary datum $\fd=(g_\Sigma,\lambda,\mu)$ on $\Sigma$ specifies the geometry of $(\hy,\omega)$ along this cylindrical end; each Seiberg-Witten configuration on $\hy$ is required to approximate a distinguished critical point of the superpotential $W_{\lambda}: M(\Sigma)\to \C$ as $s\to\infty$ and hence is always irreducible. The critical points of $\CL_\omega$ are solutions to the perturbed 3-dimensional Seiberg-Witten equations on $\hy$, while the Floer differential is defined by counting solutions on $\R_t\times \hy$. The set of isomorphism classes of relative \spinc structures on $Y$:
\[
\Spincr(Y)
\]
is a torsor over $H^2(Y,\partial Y; \Z)$. The desired invariant of $\y$ is obtained by forming the direct sum
\begin{equation}
\HM_*(\y)\colonequals\bigoplus_{\bs\in \Spincr(Y)} \HM_*(\y,\bs).
\end{equation}

This group carries the same formal structure of the monopole Floer homology for closed 3-manifolds: it is bigraded and admits an additional homology grading by homotopy classes of oriented relative 2-plane fields on $Y$ (i.e., oriented 2-plane fields that take a standard form near $\Sigma$). If $\bs$ and $\bs'$ give rise to the same \spinc structure, then their grading sets are the same; see Section \ref{Sec28} for more details. Moreover, a homology orientation of $Y$ determines a canonical mod 2 grading of $\HM_*(\y)$. One may take the base ring $\NR$ to be $\Z$, if $\mu=0$, and if the closed 2-form $\omega$ is monotone in the sense of Definition \ref{D27.3}, or otherwise a Novikov ring over $\Z$. 

\subsection{The Euler Characteristic and Finiteness} By \cite{MT96}, for any closed oriented 3-manifold $Z$ with $b_1(Z)\geq 2$, the Euler characteristic of the reduced monopole Floer homology $\HM_*^{red} (Z)$ as defined in \cite{Bible} (with any generic non-exact perturbation) recovers the Milnor torsion invariant of $Z$. Since we have followed the same setup of Meng-Taubes \cite{MT96}, the Euler characteristic of $\HM_*(\y)$, which is the same as the signed count of critical points of $\CL_\omega$ on $\hy$, recovers the Seiberg-Witten invariant $\SW(Y,\partial Y)$ defined in their paper. In particular, it is independent of the 2-form $\omega$ or the boundary datum $\fd=(g_\Sigma,\lambda,\mu)$ and recovers the Milnor-Turaev torsion of $(Y,\partial Y)$.

\begin{theorem}[{\cite[Theorem 1.1]{MT96}}]\label{1T1} For any compact oriented 3-manifold $(Y,\partial Y)$ with toroidal boundary, the Floer homology $\HM_*(\y)$ categorifies the Milnor torsion invariant of $(Y,\partial Y)$; in particular, $\chi(\HM_*(\y,\bs))$ is non-zero only for finitely many relative \spinc structures $\bs\in \Spincr(Y)$ if $b_1(Y)\geq 2$. 
\end{theorem}

\begin{remark} Turaev \cite{T98} later refined this result by showing that $\chi(\HM_*(\y))$ as a map
	\[
	\Spincr(Y)\to \Z
	\]
	agrees with the Milnor--Turaev invariant of $(Y,\partial Y)$ up to an overall sign ambiguity. The version proved in \cite{MT96} is slightly weaker: relative \spinc structures with the same $c_1(\bs)\in H^2(Y,\partial Y;\Z)$ are not distinguished. Readers are referred to their original papers for the precise statements.
\end{remark}

In light of Theorem \ref{1T1}, it is natural to ask whether $\HM_*(\y)$ enjoys a similar finiteness property when $b_1(Y)\geq 2$. Due to the presence of the 2-form $\omega$, this turns out to be a subtler question than the case of closed 3-manifolds, and the answer is unknown in general. We record a result along this line:

\begin{theorem}\label{1T4} Suppose the harmonic 2-form $\mu\in \Omega^2_h(\Sigma;i \R)$ in the boundary datum $\fd=(\Sigma,\mu,\lambda)$ is nowhere vanishing on $\Sigma$, i.e., $\mu|_{\T^2_i}\neq 0$, $1\leq i\leq n$. Then $\HM_*(\y,\bs)\neq 0$ only for finitely many relative \spinc structures, and so the monopole Floer homology group $\HM_*(\y)$ is finitely generated over the Novikov Ring $\NR$. 
\end{theorem}

We remark that under the assumptions of Theorem \ref{1T4}, the boundary $\Sigma$ must be disconnected, since $\int_{\Sigma} \mu=\int_Y d\omega=0$; thus $b_1(Y)\geq |\Sigma|\geq 2$.

\subsection{The TQFT Property and Invariance} To state the $(3+1)$ TQFT property enjoyed by $\HM_*$, we introduce the strict cobordism category $\Cob_s=\Cob_s(\Sigma, \fd)$ associated to a surface $\Sigma$ with a fixed boundary datum $\fd=(g_{\Sigma},\lambda,\mu)$. Each object of $\Cob_s$ is a pair $\y=(Y,\omega)$ consisting of a compact oriented 3-manifold with  $\partial Y\cong \Sigma$ and a closed 2-form $\omega$ compatible with $\fd$. Ignoring many technical constraints, a morphism of $\Cob_s$ from $\y_1$ to $\y_2$, written as 
\[
\x=(X,W,\omega_X): \y_1\to \y_2,\ \y_i=(Y_i, \omega_i),\ i=1,2,
\]
is a 4-manifold with corners
\[
(X,W): (Y_1,\Sigma)\to (Y_2,\Sigma)
\]
such that the cobordism $W=[-1,1]_t\times \Sigma: \partial Y_1\to \partial Y_2$ between the two boundaries is a standard product (so is the name \textit{strict}). Moreover, $X$ is equipped with a closed 2-form $\omega_X$ compatible with the ones on $Y_1$ and $Y_2$. This TQFT property can be stated easily, if $\fd$ satisfies the conditions of Theorem \ref{1T4}, and if one ignores the issue of orientations and works with the $\BF_2$-coefficients.

\begin{theorem} \label{1Tideal} Suppose the harmonic 2-form $\mu\in \Omega^2_h(\Sigma;i \R)$ in the boundary datum $\fd=(\Sigma,\mu,\lambda)$ is nowhere vanishing on $\Sigma$, and we work with the Novikov ring $\NR_2$ with $\BF_2$-coefficients. Then the monopole Floer homology $\HM_*$ extends to a covariant functor:
	\[
	\HM: \Cob_s(\Sigma,\fd)\to \NR_2\text{-}\mathrm{Mod}
	\]
	from the strict cobordism category $\Cob_s(\Sigma,\fd)$ to the category of finitely generated $\NR_2$-modules. 
\end{theorem}

In general, the Floer homology group $\HM_*(\y)$ might be infinitely generated over the base ring, and a completion of $\HM_*(\y)$ (as considered in \cite[Definition 23.1.3]{Bible}) is necessary in order to state this functoriality. In this paper, we work instead with the strict \spinc cobordism category $\SCob_s=\SCob_s(\Sigma, \fd)$, where each object $(\y,\bs)$ is now equipped with a relative \spinc structure, and whose morphism sets are same as those of $\Cob_s$. In order to deal with the orientation issue, one has to consider an enlargement $\SCob_{s,b}$ of $\SCob_s$, which includes a base-point of the configuration space for each object $(\y,\bs)\in \SCob_s$; see Definition \ref{D24.8}. 

\begin{theorem}\label{1T2} Let $\NR$ be the Novikov ring defined over $\Z$. Then the monopole Floer homology $\HM_*$ extends to a covariant functor:
	\[
	\HM: \SCob_{s,b}(\Sigma,\fd)\to \NR\text{-}\mathrm{Mod}
	\]
with a modified composition law: if $\x_{13}: \y_1\to \y_3$ is the composition of $\x_{12}:\y_1\to \y_2$ and $\x_{23}:\y_2\to\y_3$, then for any relative \spinc structures $\bs_i\in \Spincr(Y_i),i=1,3$, the map
\[
\HM\big(\x_{13}:(\y_1,\bs_1)\to (\y_3,\bs_3)\big): \HM_*(\y_1,\bs_1)\to \HM_*(\y_3,\bs_3)
\]
induced  by $\x_{13}$ is equal to the sum 
\[
\bigoplus_{\bs\in \Spincr(Y_2)} \HM\big(\x_{23}:(\y_2,\bs_2)\to (\y_3,\bs_3)\big)\circ \HM\big(\x_{12}:(\y_1,\bs_1)\to (\y_2,\bs_2)\big).
\]
which may involve infinitely many non-zeros terms; nevertheless, this sum converges in the topology of $\NR$. 
\end{theorem}

As a consequence of Theorem \ref{1T2}, the monopole Floer homology group $\HM_*(\y,\bs)$ is invariant under the following operations (see Corollary \ref{C24.10} and \ref{C24.11} for more details):
	\begin{itemize}
\item changing of the base point, the interior metric of $Y$ and the tame perturbation of the Chern-Simons-Dirac functional $\CL_{\omega}$;
\item an isotopy to the identification map $\partial Y\cong \Sigma$;
\item replacing $\omega$ by $\omega+d_{\hy}b$ where $b\in \Omega^1(\hy; i\R)$ is a compactly supported 1-form.
	\end{itemize}

\subsection{A dichotomy I: when the harmonic 2-form $\mu$ is nowhere vanishing}\label{Sec1.5} In order to obtain a topological invariant, one has to understand how the group $\HM_*(\y,\bs)$ depends on the closed 2-form $\omega$ and the boundary datum $\fd=(g_\Sigma,\lambda,\mu)$. At this point, a striking dichotomy emerges from the non-vanishing property of the harmonic 2-form $\mu$, which can be seen already from the Finiteness Theorem \ref{1T4} and Theorem \ref{1Tideal}. In this case, one can verify that $\HM_*(\y)$ depends at most on the cohomology class of $\omega$ and $*_\Sigma\lambda$. The next theorem will be proved in the third paper \cite{Wang203} in this series. 

\begin{theorem}[{\cite[Theorem 5.2]{Wang203}}]\label{1T8} Suppose that the harmonic 2-form $\mu\in \Omega^2_h(\Sigma; i\R)$ is nowhere vanishing on $\Sigma$, and its paring with each component $\T^2_j$ is sufficiently small for all $1\leq j\leq n$. Then the isomorphism class of the monopole Floer homology  $\HM_*(\y)$ depends only on the 3-manifold $Y$, the cohomology class $[\omega]\in H^2(Y;i\R)$ and $[*_\Sigma\lambda]\in H^1(\Sigma; i\R)$.
\end{theorem}

 It is not clear to the author, however, how the group $\HM_*(\y)$ will be affected, as one changes the class $[\omega]$ and $[*_\Sigma\lambda]$. In the case of closed 3-manifolds  \cite[Section 31]{Bible}, using a closed 2-form as a non-exact perturbation \cite[Section 29]{Bible}, or using balanced perturbations yet with a non-trivial local coefficient system induced by this form, will produce the same reduced monopole Floer homology, so the dependence is understood via this local system. The similar structural result for $\HM_*(\y)$, however, is unknown and will require further exploration.

\medskip

Nevertheless, in the third paper \cite{Wang203}, we shall establish a series of topological properties of $\HM_*(\y)$, including a Thurston norm \& fiberness detection result, which are independent of the choice of $[\omega]$ and $[*_\Sigma\lambda]$. To certain extent, $\HM_*(\y)$ behaves like the sutured Floer homology of $Y$ with an empty set of sutures on the boundary $\partial Y$, despite the presence of the 2-form $\omega$. 

\medskip

It is somewhat disappointing as the cobordism map in Theorem \ref{1Tideal} is only constructed for strict cobordisms. This failure is also remedied in \cite[Corollary 4.2]{Wang203} where generalized cobordism maps are constructed. If $W:\partial Y_1\to \partial Y_2$ is the restriction of $\x:\y_1\to \y_2$ between the two boundaries and is also equipped with a closed 2-form
	\[
\mathbb{W}=(W,\omega_W): (\partial Y_1,g_{\Sigma_1}, \lambda_1,\mu_1)\to (\partial Y_2,g_{\Sigma_2}, \lambda_2,\mu_2),
	\]
 then there is a map 
	\[
	\HM_*(\y_1)\otimes \HM_*(\mathbb{W})\to 	\HM_*(\y_2)
	\]
	satisfying a modified composition law. When $\mathbb{W}$ is the product cobordism $[-1,1]_t\times (\Sigma, g_\Sigma,\lambda, \mu)$, it recovers the monopole Floer functor in Theorem \ref{1Tideal} by inserting the canonical generator of $\HM_*(\mathbb{W})\cong \NR$. It is hoped that a contact structure on $W$ determines a canonical element \cite{KM97,OS05,HKM09} of $\HM_*(\mathbb{W})$ (which is still a missing ingredient in our story) and so a map $\HM_*(\y_1)\to \HM_*(\y_2)$; this is then analogous to the sutured cobordism category considered in \cite{Ju16}. 

\subsection{A dichotomy II: when the harmonic 2-form $\mu$ is somewhere zero}\label{Sec1.6} The results stated in the previous section all rely crucially on the property that $\mu$ is nowhere vanishing on $\Sigma$; very little is known in the opposite case. This dichotomy can be explained conceptually as follows: the 2-form $\mu$ determines a symplectic quotient of the infinite dimensional K\"{a}hler manifold $M(\Sigma)$ by the gauge action and therefore may change the Fukaya-Seidel category of \eqref{E0.1}. In the previous case, this Fukaya-Seidel category is expected to be trivial, cf.\cite[Section 2.2]{Wang202}, while in the opposite case it may have non-trivial product maps. Thus interesting wall-crossing phenomena will occur as we vary the boundary datum $\fd=(g_\Sigma,\lambda,\mu)$ with $\mu$ somewhere zero.

\medskip

In fact, the $\mu=0$ case may produce a more canonical invariant for the 3-manifold $(Y,\partial Y)$. The author admits that this paper did \textbf{not} provide the most practical construction of the $\mu=0$ version of the Floer homology $\HM_*(\y)$. For instance, when $\Sigma$ is connected, the technical conditions (cf. Section \ref{Sec2}) on the boundary datum $\fd=(g_\Sigma, \lambda, \mu)$ says that $\mu=0$ and 
\begin{itemize}
	\item $[*_\Sigma \lambda] \in \im (H^1(Y;i\R)\to H^1(\Sigma;i\R))$;
	\item  $[\lambda]\neq 0\in H^1(\Sigma; i\R)$ is not proportional over $\R$ to any integral class.
\end{itemize}
The second condition forces the flat metric $g_\Sigma$ to be irrational, which is not practical for any computation. To see why this is needed for our construction, recall that the set of critical values of the superpotential $W_\lambda: M(\Sigma)\to \C$ forms a lattice in $\C$, which is the image of 
\begin{equation}\label{E0.3}
H^1(\Sigma; \Z)\to \C,\ x\mapsto 4\pi^2\bigg\langle x\cup\bigg( [\frac{*_\Sigma\lambda}{2\pi i}]-i[\frac{\lambda}{2\pi i}]\bigg), [\Sigma]\bigg\rangle.
\end{equation}

The gauge transformations coming from $Y$ corresponds to the subgroup $\im H^1(Y;\Z)\subset  H^1(\Sigma;\Z)$, which translate this lattice in the imaginary direction. A relative \spinc structure on $Y$ is then identified with a lattice point up to this vertical translation; see Figure \ref{Pic1} below. The Seiberg-Witten equations on the 3-manifold $[0,+\infty)_s\times \Sigma$ (the cylindrical end of $\hy$) is modeled on the downward gradient equation
\begin{equation}\label{E0.5}
p: [0,+\infty)_s\to M(\Sigma),\ \ps p(s)+\nabla \re (W_\lambda)(p(s))=0,
\end{equation} 
along which $\im (W_\lambda)(p(s))$ is constant. The irrationality of $g_\Sigma$ is to ensure that the critical values in \eqref{E0.3} all have distinct imaginary parts. This implies that any finite energy flowlines of $-\nabla \re W_\lambda$ on $\R_s$ must be constant, and hence the moduli space on the 3-manifold $\hy$ is compact. 

\medskip

The most geometric setup one would hope is to identify $\Sigma$ isometrically  as $(\R/2\pi\Z)_{\theta_1}\times (\R/2\pi\Z)_{\theta_2}$ by fixing a meridian and a longitude on $\Sigma=\partial Y$ with $[\frac{\lambda}{2\pi i}]=\frac{d\theta_1}{2\pi}$ dual to the meridian and $[\frac{*_\Sigma\lambda}{2\pi i}]=\frac{d\theta_2}{2\pi}$ dual to the longitude. The lattice \eqref{E0.3} then becomes 
\[
4\pi^2(\Z\oplus i\Z). 
\]

A serious non-compactness issue, however, will occur in this case for the Seiberg-Witten moduli space on the 3-manifold $\hy$. For $Y=S^1\times D^2$ and the relative \spinc structure $\bs_k$ on $Y$ with $\langle c_1(\bs_k), [D^2]\rangle =2k+1, k\geq 0$, this moduli space agrees with the vortex moduli space on $D^2$:
\begin{equation}\label{E0.4}
\Sym^k D^2\cong \C^k,
\end{equation}
thinking of $D^2$ as a surface with a cylindrical end. To get rid of the non-compactness issue, one has to modify the Seiberg-Witten equations on $[0,+\infty)_s\times \Sigma$ so that it is modeled on the perturbed equation:
\begin{equation}\label{E0.6}
p: [0,+\infty)_s\to M(\Sigma),\ \ps p(s)+\nabla \re (e^{-i\eta}W_\lambda)(p(s))=0,
\end{equation}
where $W_\lambda$ is multiplied by some $e^{-i\eta}\in S^1$ with $0<|\eta|\ll 1$. The projection of such a flowline under $W_{\lambda}$ follows the direction of $-e^{i\eta}$; see Figure \ref{Pic1} below. The subtlety arises from the fact that this $\eta$ must change for different relative \spinc structures in order to upgrade the Floer homology $\HM_*(\y)$ into an $A_\infty$-module over the Fukaya-Seidel category of \eqref{E0.1}. To verify that this $A_\infty$-module is independent of the choice of $\eta's$ up to canonical quasi-isomorphisms requires additional technical inputs, which are not addressed in the present paper. In fact, this Floer homology will depend on the sign of $\eta$: for $Y=S^1\times D^2$, $\HM_*(\y,\bs_k)$ has rank $1$ in homology grading $0$ if $\eta>0$ and $2k$ if $\eta<0$, corresponding to the homology and relative homology of the symmetric product \eqref{E0.4} respectively. 

\begin{figure}[H]
	\centering
	\begin{overpic}[scale=.8]{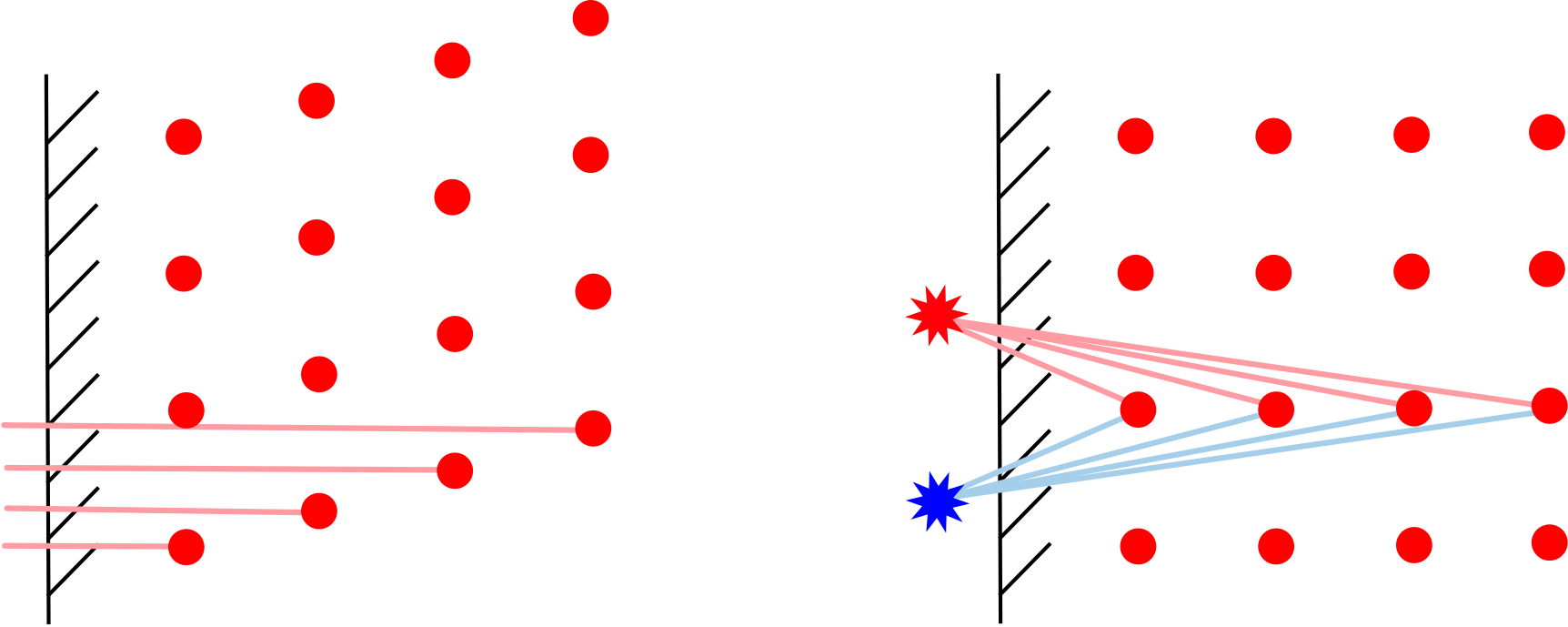}
		\put(48,19){$\eta<0$}
		\put(48,7){$\eta>0$}
	\end{overpic}	
	\caption{An irrational lattice (left) versus the standard lattice (right) $\subset \C$}
	\label{Pic1}
\end{figure}

In the finite dimensional case, this wall-crossing phenomenon (as one changes the sign of $\eta$'s) is understood via Seidel's generating theorem and has been reformulated in \cite{Wang204}; see \cite[Theorem 1.9 (3)]{Wang204}. The generalization to the Seiberg-Witten theory will appear in a companion paper in the future.

\begin{remark}By the time of preparing this paper, the framework of \cite{Wang204} was not available, so we restrict ourself here to the non-practical case (with $\eta=0$ and irrational $g_\Sigma$) for the proof.  Nevertheless, the analytic foundation laid out in this paper will apply also to the general case (with nonzero $\eta$'s) with minor changes following \cite{Wang204}. 
\end{remark}

\subsection{Relations with embedded contact homology} The irrational metric $g_\Sigma$ treated in this paper looks awkward at first sight, but a similar setup already appeared in the context of embedded contact homology. Following an earlier work of Hutchings-Sullivan \cite{HS05} on periodic Floer homology, Colin-Ghiggini-Honda \cite{CGH10} introduced a version of embedded contact homology for a contact 3-manifold $(Y,\alpha)$ with torus boundary when the boundary $\partial Y$ is invariant under the Reeb flow of the contact 1-form $\alpha$. This homology is more canonically defined in the case that $\partial Y$ is foliated by this Reeb flow with irrational slope so that $J$-holomorphic curves in the symplectization $\R_t\times Y$ will not intersect the boundary $\R_t\times \partial Y$ by a blocking lemma; cf. \cite[Section 5.2]{CGH10}. When $\partial Y$ is foliated by a circle family of closed Reeb orbits of $\alpha$, this family is then treated as a Morse-Bott singularity and gives arise to a pair of elliptic and hyperbolic Reeb orbits after perturbing $\alpha$. 

We expect that this version of embedded contact homology is isomorphic to the monopole Floer homology introduced in this paper at least in the irrational case and if $g_\Sigma$ is chosen suitably according to $\alpha$. In any case, it is a challenging task to understand the dependence of this Floer homology on the irrational slope of the Reeb flow of $\alpha$ on $\partial Y$ or on the irrational metric $g_\Sigma$.

\subsection{Relations with knot Floer homology: some speculation} The simplest examples of $(Y,\partial Y)$ arise from the knot complement for a knot $K\subset S^3$. In this case, there exists a unique \spinc structure $\s$ on $(Y,\partial Y)$, and $\Spincr(Y)$ is a torsor over 
\begin{equation}\label{E0.2}
 H^1(\partial Y; \Z)/ \im H^1( Y;\Z)\cong \Z. 
\end{equation}
A choice of $[\frac{*_\Sigma\lambda}{2\pi i}]\in \im (H^1(Y;i\R)\to H^1(\Sigma;i\R))$ specifies an isomorphism of \eqref{E0.2}. The Floer homology $\HM_*(\y)$ then carries a bi-grading of $\Z\oplus \Z$. The first grading arises from relative \spinc structures with
\[
\HM_*(\y,\bs+n)=\{0\}
\] 
when $n\gg 1$ under \eqref{E0.2}, while the second arises from the homology grading by oriented relative 2-plane fields. The group $\HM_*(\y)$ shares the same Euler characteristic and formal structures as the knot Floer homology $\HFK^-_*(S^3, K)$. But one ingredient is missing here: $\HFK^-_*(S^3, K)$ is an $\BF_2[U]$-module with $\deg U=(-1,-1)$.

\smallskip

If $\HM_*(\y)$ can be lifted into an $A_\infty$-module over the Fukaya-Seidel category $\CA$ of \eqref{E0.1}, then by passing to homology,  $\HM_*(\y)$ is a module over the algebra $H(\CA)$, which may supplement the $U$-action in our story. Using the geometric setup from Section \ref{Sec1.6}, this may provide a gauge theoretic construction of $\HFK^-_*(S^3, K)$ as an $\BF_2[U]$-module. It is hoped that the original $A_\infty$-module (the chain-level invariant) may provide refined information about the knot $K\subset S^3$. 

\smallskip
 
On the other hand, pick a meridian $m$ of $K\subset S^3$, and consider the link complement $Y_K\colonequals S^3\setminus N(K\cup m)$. By gluing the two boundary components of $Y_K$ (using a suitable orientation reversing diffeomorphism), we obtain a closed 3-manifold $\tilde{Y}_K$. We will establish an internal gluing theorem in the third paper \cite{Wang203} to identify $\HM_*(\y_K)$ with the monopole Floer homology of the closure $\tilde{Y}_K$, which is isomorphic to $\KHM_*(S^3, K)$ by \cite{KS}. This gluing result falls into Case I of the dichotomy and therefore can be easily understood. Interested readers are referred to \cite[Section 2]{Wang202} for more heuristics on this gluing formula.

\subsection{Organizations}  To define the monopole Floer homology $\HM_*(\y)$ and implement the construction sketched in Section \ref{Subsec1.2}, we address five analytic problems in this paper, as summarized below. We follow closely the plotline of  \cite{Bible}.

\smallskip

\textbf{Compactness.} To obtain the correct compactification of moduli spaces on $\R_t\times\hy$, we have to study the equations over the planar end of $\R_t\times \hy$:
\begin{equation}\label{1E4}
\HH^+\times \Sigma\colonequals \R_t\times [0,\infty)_s\times\Sigma,
\end{equation}
where $\HH^+=\R_t\times [0,+\infty)_s$ is furnished with the Euclidean metric. At this point, we make essential use of results from the first paper \cite{Wang202}. Our constraints on the boundary datum $\fd=(g_\Sigma, \lambda,\mu)$ are intended to make the following properties hold:
\begin{itemize}
\item finite energy solutions are trivial on $\C\times \Sigma$, namely, they have to be $\C$-translation invariant up to gauge \cite[Theorem 1.2 or 8.1]{Wang202}. 
\item  finite energy solutions on $\R_s\times \Sigma$ are trivial, namely, they have to be $\R_s$-translation invariant up to gauge. This result is due to Taubes; see \cite[Proposition 4.4 \& 4.7]{Taubes01} or  \cite[Proposition 10.1 \& 10.3]{Wang202} for a version that we exploit.
\end{itemize}

In Part \ref{Part2}, we first set up the strict cobordism category $\Cob_s$ and derive an energy identity for the Seiberg-Witten equations on $\hy$. Combining with the results in \cite{Wang202}, we establish the compactness theorem in Section \ref{Sec11}. Part \ref{Part2} is the counterpart of \cite[Section 4, 5, 16]{Bible}.

\smallskip

\textbf{Perturbations.} To achieve the regularity of moduli spaces, we apply a further perturbation to the Chern-Simons-Dirac functional $\CL_{\omega}$ confined in the compact region 
\[
Y=\{s\leq 0\}\subset \hy.
\]
Thus the monopole equations always take a standard form on the planar end $\HH^+\times \Sigma$. The cylinder functions that we use here are slightly different from those in \cite[Section 11]{Bible} since global gauge fixing conditions never give rise to such perturbations, in the sense of Definition \ref{D14.1}. Inspired by holonomy perturbations from instanton Floer homology, we will look at embeddings of $S^1\times D^2$ into $Y$ instead. The construction is carried out in details in Part \ref{Part4}, as the counterpart of \cite[Section 10, 11]{Bible}.

\smallskip

\textbf{Linear Analysis.} This part is more or less standard. Since $\hy$ is non-compact, the extended Hessian of $\CL_{\omega}$ on $\hy$ is a self-adjoint operator with essential spectrum, a major distinction of this theory from the case of closed 3-manifolds. Fortunately, this essential spectrum is disjoint from the origin, so it makes sense to speak of the spectrum flow when perturbed by compact operators. We will follow the setup of \cite{RS95} and summarize relevant results in Part \ref{Part5}, as the counterpart of \cite[Section 17]{Bible}.

\smallskip

\textbf{Unique Continuation.} As our perturbation space is not large enough, we need a stronger unique continuation property to achieve transversality. The non-linear version is stated as follows: if two solutions $\gamma_1,\gamma_2$ to the perturbed monopole equations on $\R_t\times \hy$ are gauge equivalent on the slice
\[
\{0\}\times Y \text{ with } Y=\{s\leq 0\}\subset\hy,
\]
then they are gauge equivalent on the whose space. The proof relies on the Carleman estimates from \cite{K95}. Part \ref{Part6} is the counterpart of \cite[Section 7, 12, 15]{Bible}. The proof of transversality is accomplished in Section \ref{Sec21}. 

\smallskip

\textbf{Orientations.} To work with a Novikov ring $\NR$ with integral coefficients, we have to orient moduli spaces coherently. For closed 3-manifolds, this is done by first looking at reducible configurations in the blown-up space. See \cite[Section 20]{Bible} for details. In our case, we have to adopt a different approach as configurations are never reducible and the action of the gauge group is free. 

Our situation here is similar to that of \cite{KM97}, where a Riemannian 4-manifold with a conic end is considered, so one may follow the argument of \cite[Appendix]{KM97} to orient moduli spaces. The key ingredients are relative determinant line bundles or \textbf{relative orientations} that compare two Fredholm operators. We adopt a more direct approach to this notion without referring to either $K$-theory or the proof of the index theorem \cite{AS68}. This combinatoric construction is based on a simple proof of the excision principle due to Mrowka and is carried out in Appendix \ref{AppD}.

Part \ref{Part7} is the counterpart of \cite[Section 20, 22, 28]{Bible}. The canonical grading of $\HM_*(\y)$ by homotopy classes of oriented relative 2-plane fields is introduced in Section \ref{Sec28}. We will first define the monopole Floer homology of $\y$ using $\BF_2$-coefficients in Section \ref{Sec27}, and then address the orientation issue in Section \ref{Sec29}.

\smallskip

Most results and proofs in the present paper are intended to generalize the ones in \cite{Bible}. Readers are assumed to have a reasonable understanding of the monopole Floer homology of closed 3-manifolds, at least in the case that $c_1(\s)$ is non-torsion. 

\begin{remark} The following results, however, are not covered in this paper:
	\begin{itemize}
\item the exponential decay of solutions in the time-direction, cf. \cite[Section 13]{Bible};
\item the gluing theorem, cf. \cite[Section 18, 19]{Bible};
	\end{itemize}
since they will follow immediately from \cite{Bible}, once we set up the rest of the theory correctly. 
\end{remark}

\textbf{Acknowledgments.} The author would like to thank his advisor Tom Mrowka for introducing him to this subject, for suggesting the present problem, and for his patient help and constant encouragement throughout this project. The author would also like to thank Micheal Hutchings, Chris Gerig, Siqi He, Lante Ma, Jianfeng Lin, Matt Stoffregen, Guangbo Xu, Yuan Yao and Boyu Zhang for helpful conversations. This work is partially supported by NSF through his thesis advisor's award DMS-1808794.

\part{Three-Manifolds with Toroidal Boundary}\label{Part2}

In this part, we introduce the strict cobordism category $\Cob_s$ of compact oriented 3-manifolds with toroidal boundary and study the Seiberg-Witten equations on the completion of these manifolds. Throughout this paper, we use $(\Sigma,g_\Sigma)$ to denote a disjoint union of 2-tori equipped with a flat metric $g_\Sigma$. Although most results in the first paper \cite{Wang202} do not require the flatness of $g_\Sigma$, it is assumed here in order to apply a theorem of Taubes; see Theorem \ref{T2.6} below.

For any compact oriented 3-manifold $(Y,\partial Y)$ with toroidal boundary $\partial Y\cong \Sigma$, we attach a cylindrical end to obtain a complete Riemannian 3-manifold 
\[
\hy\colonequals Y\ \bigcup_\Sigma\ [0,\infty)_s\times\Sigma.
\] 
Any strict cobordism between such manifolds $Y_1$ and $Y_2$
\[
(X, [-1,1]_t\times \Sigma): (Y_1, \partial Y_1)\to (Y_2, \partial Y_2),
\]
is then associated with a complete Riemannian manifold $\CX$ with a planar end:
\begin{align*}
\CX&\colonequals (-\infty,-1]_t\times \hy_1\ \bigcup\ \hx\ \bigcup\  [1,+\infty)_t \times \hy_2 \text{ with } \\
\hx&\colonequals X\ \bigcup\ [-1,1]_t\times [0,\infty)_s\times\Sigma. 
\end{align*}
The manifold $\hx$ is obtained from $X$ by completing only in the spatial direction, which gives a cobordism from $Y_1$ from $Y_2$. 

\smallskip

\autoref{Part2} is devoted to the proof of the Compactness Theorem \ref{T11.1} for the Seiberg-Witten moduli spaces on $\R_t\times \hy$ and $\CX$, which is essential to the construction of Floer homology and cobordism maps. The proof in turn relies on three important properties for the perturbed Seiberg-Witten equations:
\begin{enumerate}[label=(K\arabic*)]
	\item\label{K1} there is a uniform upper bound on the analytic energy;
	\item\label{K2} finite energy solutions are trivial on $\C_z\times \Sigma$; in other words, they are gauge equivalent to the unique $\C_z$-translation invariant solution on $\C_z\times\Sigma$; see Theorem \ref{T2.4} below. 
	\item\label{K3}  finite energy solutions on $\R_s\times \Sigma$ are trivial; in other words, they are gauge equivalent to the unique $\R_s$-translation invariant solution on $\R_s\times\Sigma$. This result is due to Taubes and requires $g_\Sigma$ to be flat; see Theorem \ref{T2.6} below.
\end{enumerate} 

For these properties to hold, a suitable closed 2-form is needed to perturb the Seiberg-Witten equations. In Section \ref{Sec5}, we summarize a list of results from the first paper \cite{Wang202} and spell out the requirement for this 2-form, which will then ensure \ref{K2} and \ref{K3}. 

\smallskip

 The rest of Part \ref{Part2} is organized as follows. In Section \ref{Sec2}, we define the strict cobordism category and set up the configuration spaces on $\hy$ and $\hx$ respectively. In Section \ref{Sec10}, we prove that the quotient configuration space in our case is still Hausdorff and remains a Hilbert manifold after Sobolev completions. 
 
 \smallskip
 
The closed 2-form $\omega$ on $Y$ is chosen such that the perturbed Seiberg-Witten equations on the cylindrical end $[0,+\infty)_s\times\Sigma$ is modeled on the equation \eqref{E0.5}, and therefore the energy equation can be derived easily. This is carried out in detail in Section \ref{Sec12} and gives the first property \ref{K1}. We remark that in order to derive the similar energy equation for the perturbed version \eqref{E0.6} in Section \ref{Sec1.6}, one must follow the computation in \cite[Section 2]{Wang204}, which will complicate the story significantly. 

\smallskip

Finally, the Compactness Theorem \ref{T11.1} is stated and proved in Section \ref{Sec11}. The analogue of Theorem \ref{T11.1} in the case of Landau-Ginzburg models is \cite[Proposition 3.5]{Wang204}, whose proof has followed the same argument and is technically easier. 


\section{Results from the First Paper}\label{Sec5}

In this section, we summarize a list of results from the first paper \cite{Wang202}, which are concerned with the geometry of the planar end of $\CX$. They are used in an essential way in the proof of the Compactness Theorem \ref{T11.1} in Section \ref{Sec11}. We focus on the case that $\Sigma$ is connected for their statements.

\subsection{Review} Let $X$ be any smooth oriented 4-manifold equipped with a metric. A \spinc structure $\s=(S_X,\rho_4)$ on $X$ is a pair consisting of a Hermitian vector bundle  $S_X=S^+\oplus S^-$ with $\rank S^\pm =2$, called the spin bundle, and a smooth bundle map $\rho_4: T^*X\to \End(S_X)$ that defines the Clifford multiplication. A Seiberg-Witten configuration $\gamma=(A,\Phi)$ on $(X,\s)$ then consists of a smooth \spinc connection $A$ and a smooth section $\Phi$ of $S^+$. The space of all such configurations, which we denote by $\SC(X,\s)$, is acted on by the gauge group $\CG(X)=\map (X,S^1)$ by the formula
\[
u\in\CG(X): \SC(X,\s)\to \SC(X,\s),\ (A,\Phi)\mapsto (A-u^{-1}du, u\Phi).
\]

The Seiberg-Witten equations on the \spinc manifold $(X,\s)$ are invariant under this gauge action and can be perturbed a closed 2-form $\omega\in \Omega^2(X;i\R)$. Let $A^t$ denote the induced connection of $A$ on $\Lambda^{0,1} S^+$ and $\omega^+$ the self-dual component of $\omega$. Then the perturbed equations take the form
\begin{equation}\label{SWEQ}
\left\{
\begin{array}{r}
\half \rho_4(F_{A^t}^+)-(\Phi\Phi^*)_0-\rho_4(\omega^+)=0,\\
D_A^+\Phi=0,
\end{array}
\right.
\end{equation}
where $D_A^+: \Gamma(S^+)\to \Gamma(S^-)$ is the Dirac operator, and $(\Phi\Phi^*)_0=\Phi\Phi^*-\half |\Phi|^2\otimes\Id_{S^+}$ denotes the traceless part of the endomorphism $\Phi\Phi^*:S^+\to S^+$, 

\smallskip

Let $\Sigma=(\T^2, g_\Sigma)$ be a 2-torus equipped with a \textbf{flat} metric. In the case that $X=\C_z\times \Sigma$ a K\"{a}hler manifold, furnished with the product metric and oriented by the complex structure, the equations (\ref{SWEQ}) can be understood more explicitly as follows. 

\smallskip

Let $dvol_\C$ (resp. $dvol_\Sigma$) denote the volume form of $\C$ (resp.  $\Sigma$). The symplectic form on $X$ is given by the sum $\omega_{sym}\colonequals dvol_\C+dvol_\Sigma$. The spin bundle $S^+$ is acted on by $\rho_4(\omega_{sym})$ and splits as $L^+\oplus L^-$ such that
\[
\rho_4(\omega_{sym})=\begin{pmatrix}
-2i & 0\\
0 & 2i
\end{pmatrix},
\]
i.e., $L^\pm$ is the $\mp 2i$ eigen-subbundle of $\rho_4(\omega_{sym})$. Write $\Phi=(\Phi_+,\Phi_-)$ with $\Phi_\pm\in \Gamma(X,L^\pm)$.

We focus on the \spinc structure on $\C_z\times \Sigma$ with 
\[
c_1(S^+)[\Sigma]=0,
\]
so both $L^+$ and $L^-$ are topologically trivial.

Let $z=t+is$ be the coordinate function on $\C_z$. Then the Clifford multiplication $\rho_4: T^*X\to \Hom(S,S)$ can be constructed by setting: 
\begin{eqnarray*}
	\rho_4(dt)=\begin{pmatrix}
		0& -\Id\\
		\Id & 0\\
	\end{pmatrix},\ 
	\rho_4(ds)=\begin{pmatrix}
		0& \sigma_1\\
		\sigma_1 & 0\\
	\end{pmatrix} :\ S^+\oplus S^-\to S^+\oplus S^-,
\end{eqnarray*}
where $\sigma_1=\begin{pmatrix}
i & 0\\
0 & -i\\
\end{pmatrix}: S^+=L^+\oplus L^-\to L^+\oplus L^-$ is the first Pauli matrix. If one identifies $L^+\cong \C$ and $L^-\cong \Lambda^{0,1}\Sigma$, then for any $x\in \Sigma$ and $w\in T_x\Sigma$, $\rho_4(w)$ is determined by the relation
\[
\rho_3(w)\colonequals\rho_4(dt)^{-1}\cdot\rho_4(w)= \begin{pmatrix}
0 & -\iota(\sqrt{2}w^{0,1})\ \cdot \\
\sqrt{2}w^{0,1}\otimes \cdot  & 0\\
\end{pmatrix}: S^+\to S^+,
\]

\medskip

\begin{remark} The Clifford multiplications in dimension $3$ and $4$ are denoted respectively by $\rho_3$ and $\rho_4$ in this paper. Identify $\C_z$ as $\R_t\times \R_s$, then they are related by 
	\[
	\rho_3(w)=\rho_4(dt)^{-1}\cdot\rho_4(w): S^+\to S^+,
	\]
	for any $w\in T^*(\R_s\times \Sigma)$. In particular, $\rho_3(ds)=\sigma_1$. 
\end{remark}
The symplectic form $\omega_{sym}$ is parallel, and so is the splitting $S^+=L^+\oplus L^-$. As a result, any \spinc connection $A$ then takes a diagonal form
\[
\nabla_A=\begin{pmatrix}
\nabla_{A_+} & 0 \\
0 & \nabla_{A_-}\\
\end{pmatrix}.
\]

We also regard $L^+=\C$ and $L^-=\Lambda^{0,1}\Sigma$ as bundles over $\Sigma$. Let $\cB_*=(d,\nabla^{LC})$ denote the reference connection on $\C\oplus  \Lambda^{0,1}\Sigma\to\Sigma$. A reference connection $A_*$ on $S^+\to X$ is then obtained by setting
\begin{equation}\label{E2.2}
\nabla_{A_*}=dt\otimes \Pt+ds\otimes \Ps+ \nabla_{\cB_*}.
\end{equation}

One can easily check that $A_*$ is a \spinc connection. Any \spinc connection on $X$ then takes the form $A=A_*+a$ for an imaginary valued 1-form $a\in \Omega^1(X;i\R)$. Their curvature tensors are related by
\[
F_A=F_{A_*}+d_Xa\otimes \Id_S,\text{ so }F_{A^t}=F_{A^t_*}+2d_Xa. 
\]

\subsection{Point-Like Solutions} Now we recall the list of results from \cite{Wang202} which we shall use in the proof of Theorem \ref{T11.1}. Let $X=\C_z\times\Sigma$. For our applications, the 2-form perturbation that appears in the Seiberg-Witten equations \eqref{SWEQ} will take the special form
\[
\omega\colonequals \mu+ds\wedge \lambda
\]
where 
\begin{itemize}
	\item $\lambda\in \Omega_h^1(\Sigma; i\R)$ is an imaginary-valued harmonic 1-form on $\Sigma$, and 
	\item $\mu\in \Omega_h^2(\Sigma; i\R)$ is an imaginary-valued harmonic 2-form.
\end{itemize}

Since the metric $g_\Sigma$ is flat, the 2-form $\omega$ is parallel on $X=\C_z\times\Sigma$. The pair $(\lambda, \mu)$ is subject to further conditions in order to deduce the properties \ref{K2} and \ref{K3} from the first paper \cite{Wang202}.

\begin{assumpt}\label{A1.2} The pair $(\lambda,\mu)\in \Omega_h^1(\Sigma; i\R)\times  \Omega_h^2(\Sigma; i\R)$ is said to be admissible if $\lambda\neq 0$, and if one of the following two conditions holds:
	\begin{enumerate}[label=$(\mathrm{W}\arabic*)$]
\item\label{VV1}$\mu\neq 0$;
\item $\lambda$ is not the real multiple of any integral class in $H^1(\Sigma; i\Z)\subset \Omega_h^1(\Sigma; i\R)$. 
	\end{enumerate}

The pair $(\lambda,\mu)$ is always assumed to be admissible in this paper.
\end{assumpt}

For the rest of this section, we state the main results of \cite{Wang202} and explain why Assumption \ref{A1.2} is crucial. First, recall the definition of the local energy functional associated to any configuration $\gamma=(A,\Phi)$ on $X$.

 \begin{definition}[{\cite[Definition 8.3]{Wang202}}]\label{D1.3} For any region $\Omega\subset \C_z$ and any configuration $\gamma=(A,\Phi)$ on $\C_z\times\Sigma$, define \textit{the local energy functional} of $\gamma$ over $\Omega$ as 
	\begin{equation*}
	\E_{an}(A,\Phi; \Omega)\colonequals\int_{\Omega}\int_{\Sigma} \frac{1}{4}|F_{A^t}|^2+|\nabla_A\Phi|^2+|(\Phi\Phi^*)_0+\rho_4(\omega^+)|^2. \qedhere
	\end{equation*}
\end{definition}

A solution $\gamma$ to the Seiberg-Witten equations \eqref{SWEQ} is called \textbf{point-like} if its global energy $\E_{an}(\gamma; \C_z)$ is finite. There is a constant solution $\gamma_*=(A_*,\Phi_*)$ on the 4-manifold $\C_z\times\Sigma$ with $\E_{an}(\gamma_*; \C_z)=0$. The \spinc connection of $\gamma_*$ is given by the formula \eqref{E2.2}, while the spinor $\Phi_*$ takes the form
\[
(r_+,\sqrt{2}\lambda^{0,1} r_-)\in \Gamma(X,\C\oplus\Lambda^{0,1}\Sigma),
\]
where $r_\pm$ are real numbers subject to the relations: 
\[
r_+r_-=1 \text{ and } \frac{i}{2}(|r_+|^2-|r_-|^2|\lambda|^2)=-*_\Sigma\mu. 
\]
See \cite[Section 7.2]{Wang202}. In particular, $\Phi_*$ is parallel with respect to $A_*$. 

\smallskip

One important consequence of Assumption \ref{A1.2} is that $\gamma_*$ will be the only point-like solution on $X$ up to gauge. For practical reasons, we give a more general statement. Let $I_n=[n-2,n+2]_t\subset \R_t$. Choose a compact domain $\Omega_0\subset I_0\times [0,\infty)_s$ with a smooth boundary such that 
\begin{equation}\label{E2.3}
I_0\times [1,3]_s\subset \Omega_0\subset I_0\times [0,4]_s.
\end{equation}
For any  $n\in\Z$ and $S\in \R_s$, define $\Omega_{n,S}\subset \C$ to be the translated domain
\begin{equation}\label{E2.4}
\Omega_{n,S}\colonequals \{(t,s):(t-n,s-S)\in \Omega_0\}\subset  I_n\times [0,\infty)_s.
\end{equation}

\begin{theorem}[{\cite[Proposition 8.3]{Wang202}}]\label{T2.4} If $\lambda\neq 0$, then there exists a constant $\epsilon_*=\epsilon_*(g_\Sigma,\lambda,\mu)>0$ with the following significance.  If a solution $\gamma=(A,\Phi)$ to \eqref{SWEQ} on $X=\C_z\times\Sigma$ satisfies the estimate
	\[
	\E_{an}(A, \Phi; \Omega_{n,S})<\epsilon_*
	\]
for all $|n|+|S|\gg 1$, then $\gamma$ is gauge equivalent to the constant configuration $\gamma_*=(A_*,\Phi_*)$. In particular, a point-like solution on $X$ is always trivial. 
\end{theorem}

We are also interested in solutions on the planar end $\HH^+\times\Sigma$ with $\HH^+\colonequals\R_t\times [0,+\infty)_s\subset \C_z$. The next theorem says that if a solution $\gamma$ on $\HH^+\times\Sigma$ is close to $\gamma_*$ everywhere, then $\gamma$ converges to $\gamma_*$ exponentially in the spatial direction: 

\begin{theorem}[{\cite[Theorem 9.1]{Wang202}}]\label{T2.5} If $\lambda\neq 0$, then there exists constants $\epsilon,\zeta>0$ depending only the datum $(g_\Sigma, \lambda\neq 0, \mu)$ with the following significance. Suppose that a configuration $\gamma=(A,\Phi)$ solves the Seiberg-Witten equations \eqref{SWEQ}  on $\HH^+\times\Sigma$, and $\E_{an}(\gamma; \Omega_{n,S})<\epsilon$ for all $n\in \Z$ and $S\geq 0$. Then 
	\[
	\E_{an}(\gamma;\Omega_{n,S})<e^{-\zeta S}. 
	\] 
\end{theorem}

This theorem will be used in Section \ref{Sec11} to control the $L^2_k$-norm of  $\gamma-\gamma_*$ on the region $\Omega_{n,S}$, $k\geq 2$ if $\gamma$ lies in the Coulomb slice of $\gamma_*$; see Theorem \ref{11.5}.

\subsection{Solutions on $\R_s\times \Sigma$} Now we study the dimensional reduction of \eqref{SWEQ}, the Seiberg-Witten equations defined on the 3-manifold $\R_s\times \Sigma$:
\begin{equation}\label{3DDSWEQ}
\left\{
\begin{array}{r}
\half \rho_3(F_{B^t})-(\Psi\Psi^*)_0-\rho_3(\omega)=0,\\
D_B\Psi=0.
\end{array}
\right.
\end{equation}
with $\omega=\mu+ds\wedge \lambda$. The pair $\cgamma=(B,\Psi)$ is a configuration on the 3-manifold. To go back to the 4-dimensional case, one takes
\[
A=dt\otimes \Pt+B,\ \Phi(t)=\Psi \text{ on } \R_t\times \R_s\times \Sigma,
\]
so $\E_{an}(A,\Phi; [0,1]_t\times \R_s)$ comes down to the energy of $(B,\Psi)$:
\[
\E_{an}(B,\Psi; \R_s)=\int_{\R_s\times \Sigma} \frac{1}{4}|F_{B^t}|^2+|\nabla_B\Psi|^2+|(\Psi\Psi^*)_0+\rho_3(\omega)|^2.
\]

The trivial solution $\cgamma_*=(B_*,\Psi_*)$ of \eqref{3DDSWEQ} can be then written as 
\begin{equation}\label{E2.6}
B_*=ds\otimes \Ps+\begin{pmatrix}
d & 0\\
0 & \nabla^{LC}
\end{pmatrix}, \Psi_*=(r_+,\sqrt{2}\lambda^{0,1} r_-),
\end{equation}
which has $\E_{an}(\cgamma_*; \R_s)=0$. In fact, this is the only solution with finite energy if Assumption \ref{A1.2} holds.
\begin{theorem}[{\cite[Proposition 4.4 \& 4.7]{Taubes01}}]\label{T2.6} If $g_\Sigma$ is flat and Assumption \ref{A1.2} holds, then any solution $\cgamma$ of $(\ref{3DDSWEQ})$ with $\E_{an}(\cgamma; \R_s)<\infty$ is gauge equivalent to the unqiue $\R_s$-translation solution $\cgamma_*$. 
\end{theorem}
\begin{remark} Theorem \ref{T2.6} is due to Taubes and is the main reason to assume the flatness of $g_\Sigma$. A short discussion on its proof can be found in \cite[Section 10]{Wang202}. In fact, Theorem \ref{T2.4} \& \ref{T2.5} also hold for any non-flat metric $g_\Sigma$ of $\Sigma$ with a slightly different expression of $\E_{an}$; see \cite{Wang202}.
\end{remark}

\section{The Strict Cobordism Category}\label{Sec2}

Let $\Sigma=\coprod_{i=1}^n\T^2_i$ be a union of 2-tori. A boundary datum $\fd=(g_\Sigma,\lambda,\mu)$ on $\Sigma$ consists of 

\begin{itemize}
	\item a flat metric $g_\Sigma$ of $\Sigma$,
	\item an imaginary-valued harmonic 1-form $\lambda\in \Omega_h^1(\Sigma; i\R)$,
	\item an imaginary-valued harmonic 2-form $\mu\in \Omega_h^2(\Sigma; i\R)$,
\end{itemize}
such that 

\begin{enumerate}[label=(P0)]
\item\label{P7} on each component $\T^2_i$, Assumption \ref{A1.2} holds for the restriction $(g_\Sigma,\lambda,\mu)\big|_{\T^2_i}$. In particular, $\lambda|_{\T^2_i}\neq 0$ for all $i$. 

\end{enumerate}
\smallskip

To any boundary datum $\fd$ on $\Sigma$ is associated with a strict cobordism category $\Cob_s(\Sigma,\fd)$. It is called strict, because objects and morphisms are subject to certain constraints. Roughly speaking, each object of $\Cob_s$ is a 3-manifold $Y$ with $\partial Y\cong \Sigma$, together with a choice of a cylindrical metric $g_Y$ and a closed 2-form $\omega$ compatible with $\fd$. A morphism of $\Cob_s$ is a 4-manifold with corners, equipped again with a suitable closed 2-form:
\[
(X,W): (Y_1,\partial Y_1)\to (Y_2,\partial Y_2), \partial Y_1\cong \partial Y_2\cong \Sigma.
\]
 The restriction of this cobordism between boundaries has to be a product, so $W=[-1,1]_t\times \Sigma$. Some of these constraints on objects and morphisms might be dropped in the future by looking at the Seiberg-Witten moduli spaces on 4-manifolds with more complicated geometry. But for now we restrict attention to this smaller category $\Cob_s$ for the sake of simplicity. Subsection \ref{Subsec3.1} and \ref{Subsec3.2} below are devoted to the precise definition of $\Cob_s$. Once this is done, the configuration spaces on the completed manifolds $\hy$ and $\hx$ will be introduced in Subsection \ref{Subsec9.3}.


\subsection{Objects} \label{Subsec3.1} For any fixed boundary datum $\fd=(g_\Sigma,\lambda,\mu)$ on $\Sigma$, an object of the strict cobordism category $\Cob_s(\Sigma,\fd)$ is a quintuple $\y=(Y, \psi, g_Y, \omega, \{\q\})$ satisfying the following properties:
\begin{enumerate}[label=(P\arabic*)]
\item $Y$ is a compact oriented 3-manifold with boundary, and $\psi : \partial Y\to \Sigma$ is an orientation preserving diffeomorphism. The identification map $\psi$ will be dropped from our notations when it is clear from the context.

\item $g_Y$ is a metric on $Y$ cylindrical near $\Sigma$, i.e., $g_Y$ is the product metric 
\[
ds^2+\psi^*g_\Sigma
\]
within a collar neighborhood $(-2,0]_s\times \partial Y$ of $\partial Y$.  We form a complete Riemannian $3$-manifold $\hy$ by attaching cylindrical ends along $\Sigma$:
\[
\hy=Y\cup_\psi [-1,\infty)_s\times \Sigma,
\]  
whose metric is denoted also by $g_Y$. 

\item\label{P2} $\omega\in \Omega^2(Y;i\R)$ is an imaginary valued \textbf{closed} 2-form on $Y$ such that within the collar neighborhood $[-1,0]_s\times \partial Y$, $\omega$ restricts to an $s$-independent form 
\[
 \mu+ds\wedge \lambda,
\]
 so $\omega$ extends naturally to a closed 2-form on  $\hy$, denoted also  by $\omega$. 
\item\label{P3}  The 1-form $*_\Sigma\lambda$ lies in the image
\[
\im (H^1(Y;i\R)\to H^1(\Sigma; i\R)). 
\]
\item\label{P4} The 2-form $\mu$ lies in the image
\[
\im (H^2(Y; i\R)\to H^2(\Sigma; i\R)). 
\]

\item\label{P5} Choose a cut-off function $\chi_1: [0,\infty)_s\to \R$ such that 
\[
\chi_1(s)\equiv 1 \text{ if } s\geq -1;\ \chi_1(s)\equiv 0 \text{ if } s\leq -3/2.
\]
Set $\omega_\lambda=\chi_1(s)ds\wedge\lambda$ and 
\begin{equation}\label{9.2}
\bomega\colonequals \omega-\omega_\lambda=\omega-\chi_1(s)ds\wedge\lambda.
\end{equation}
Then $\bomega\equiv \mu$ on $[-1,0]_s\times \partial Y$. Any two such forms $\bomega,\bomega'$ are said to be \textbf{relatively cohomologous}, if $\omega=\omega'+d_Yb$ for a compactly supported 1-form $b\in \Omega^1_c(Y;i\R)$. For each $[\mu]\in H^2(\Sigma; i\R)$, the space of relative cohomology classes is denoted by 
\[
H^2_{dR}(Y,\partial Y; [\mu])
\]
which is a torsor over $H^2_{dR}(Y,\partial Y; [0])\cong H^2(Y, \partial Y; i\R)$. The class of $\bomega$ is denoted by $[\omega]_{cpt}$. There is a natural map $j^*$:
\[
\begin{array}{ccccc}
H^2_{dR}(Y,\partial Y; [\mu])&\xrightarrow{j^*} &H^2(Y;i\R)&\xrightarrow{i^*}& H^2(\Sigma; i\R)\\
{[\omega]}_{cpt}&\mapsto &[\omega]&\mapsto& [\mu]
\end{array}
\]
sending $[\omega]_{cpt}$ to the cohomology class $[\omega]$ of $\omega$. Moreover, $i^*([\omega])=[\mu]$ for the inclusion map $i:\Sigma\to Y$. We refer to $[\omega]\in H^2(Y;i\R)$ as \textbf{the period class}, which is independent of $\lambda$ and the cut-off function $\chi_1$. The closed 2-form $\omega$ in \ref{P2} can be recovered from ($\lambda, \mu, [\omega]_{cpt}$) up to a relatively exact 2-form. 

\item\label{P8} $\{\q\}$ is a collection of admissible perturbations (in the sense of Definition \ref{D19.3}) of the Chern-Simons-Dirac functional $\CL_\omega$ for each relative \spinc structures $\bs$.
\end{enumerate}

\begin{remark} The closed 2-form $\omega$ is used to perturb the Chern-Simons-Dirac functional on $\hy$, see Definition \ref{D9.4} below. \ref{P7} will allow us to apply Theorem \ref{T2.4}$-$\ref{T2.6} in Section \ref{Sec11}, so the Seiberg-Witten moduli spaces will have the right compactness property. We will address the issue of perturbations in Part \ref{Part4}, so at present readers may ignore the last property \ref{P8}. 
\end{remark}

 The property \ref{P3} requires some further explanation: it is used to find a closed 1-form on $\hy$ that equals $*_3(\chi_1(s)ds\wedge\lambda)$ on the cylindrical end. It will play an essential role in the energy equation in Section \ref{Sec12}, cf. Theorem \ref{Energy10.2} below. 
\begin{lemma} \label{9.1} For any object $\y\in\Cob_s$, there exists a smooth 2-form $\omega_h$ on $\hy$ such that $*_3\omega_h$ is closed and $
	\omega_h=ds\wedge\lambda$ on $[-1,\infty)_s\times \Sigma$. In particular, $\omega_h-\omega_\lambda\in L^2(\hy)$.
\end{lemma}

\begin{remark}\label{R9.3}
In Part \ref{Part7}, we will define for each object $\y$ and each relative \spinc structure $\bs$ on $(Y,\partial Y)$ (see Subsection \ref{Subsec9.3} below) a finitely generated module
\[
\HM_*(\y,\bs)
\]
over a Novikov ring $\NR$. The isomorphism class of $\HM_*(\y,\bs)$  is independent of 
\begin{itemize}
\item the cylindrical metric $g_Y$;
\item isotopy of the diffeomorphism $\psi:\partial Y\to \Sigma$;
\item the choice of admissible perturbations in \ref{P8}.
\end{itemize}

Moreover, this module is not altered if we replace $\omega$ by $\omega+d_Yb$ for a compactly supported 1-form $b\in \Omega^1_c(Y;i\R)$. We refer to $\HM_*(\y,\bs)$ as the monopole Floer homology of $(\y,\bs)$, which depends at most on 
\begin{itemize}
\item the 3-manifold $(Y,\partial Y)$,
\item the isotopy class of $\psi:\partial Y\to \Sigma$,
\item the boundary datum $\fd=(g_\Sigma,\lambda,\mu)$ and
\item the relative cohomology class $[\omega]_{cpt}\in H^2_{dR}(Y,\partial Y; [\mu])$ as defined in \ref{P5}.
\end{itemize}

However, the definition of $[\omega]_{cpt}$ relies on the cut-off function $\chi_1$. This ambiguity is removed by the following fact: the group $\HM_*(\y,\bs)$ is not affected if one replaces $\omega$ by 
\[
\omega+\sum_{i=1}^n \chi_i^\circ(s)ds\wedge\lambda_i
\]
where $\lambda_i=\lambda|_{\T^2_i}\in \Omega^1(\T^2_i;i\R)$ and $\chi_i^\circ(s)$ is any compactly supported function on $[0,\infty)_s\times\T^2_i$. Thus only a suitable quotient class of $[\omega]_{cpt}\in H^2(Y,\partial Y; [\mu])$ matters, and this class is independent of $\chi_1$. See Corollary \ref{C24.11} for more details. 
\end{remark}

\subsection{Morphisms}\label{Subsec3.2} Having described objects in the strict cobordism category $\Cob_s(\Sigma,\fd)$, we now turn to the morphism spaces in this subsection. Since each object $\y$ is coupled with a closed 2-form $\omega$, so is any morphism in this category. Given two objects $\y_i=(Y_i, \psi_i, g_i, \omega_i,\q_i), i=1,2$ in $\Cob_s$, a morphism 
\[
\x: \y_1\to \y_2
\]
 is a quadruple $\x=(X,\psi_X, W,[\omega_X]_{cpt})$ satisfying the following properties.
 \begin{enumerate}[label=(Q\arabic*)]
 \item \label{Q1}$X$ is a manifold with corners, i.e. $X$ is a space stratified by manifolds 
 \[
 X\supset X_{-1}\supset X_{-2}\supset X_{-3}=\emptyset.
 \]
 The co-dimensional 1 stratum $X_{-1}$ consists of three parts
 \[
X_{-1}= (-Y_1)\cup (Y_2)\cup W_X.
 \]
$W_X$ is an oriented 3-manifold with boundary $\partial W_X=\partial Y_1\cap\partial Y_2$, $\partial Y_i=Y_i\cap W_X$ and $X_{-2}=\partial Y_1\cup \partial Y_2$.  
 
 \item\label{Q2} $W=[-1,1]_t\times \Sigma$ is the product cobordism from $\Sigma$ to itself. 
 
 \item\label{Q3} $\psi_X: W_X\to W$ is an orientation preserving diffeomorphism compatible with $\psi_1$ and $\psi_2$. To be more precise, we require that 
 \begin{align*}
 \psi_X |_{\partial Y_1}&=\psi_1: \partial Y_1\to \{-1\}\times \Sigma, \\
\psi_X |_{\partial Y_2}&=\psi_2: \partial Y_2\to \{1\}\times \Sigma,
 \end{align*}
 and these relations hold also in a collar neighborhood of $\partial W_X$.  When it is clear from the context, we shall not distinguish $W_X$ from $W$, and $\psi_X$ will be dropped from our notations. Such a pair $(X, \psi_X)$ is called \textbf{a strict cobordism} from $(Y_1,\psi_1)$ to $(Y_2,\psi_2)$.
 
 	\item \label{Q4} The closed 2-form $\omega_i$ on $Y_i$ contains a bit more information than the period class $[\omega_i]\in H^2(Y_i;i\R)$. Note first that
 	\[
 	\omega_1|_{\Sigma}=\omega_2|_{\Sigma}=\mu\in \Omega^2_h(\Sigma; i\R),
 	\] 
so the triple $(\omega_1, \mu, \omega_2)$ determines a class $[\alpha]$ in $H^2((-Y_1)\cup W\cup Y_2;i\R)$. $[\alpha]$ is required to lie in the image 
 	\[
 	\im \bigg(m^*_0: H^2(X;i\R)\to H^2((-Y_1)\cup W\cup Y_2);i\R)\bigg),
 	\]
 	where $m_0: (-Y_1)\cup W\cup Y_2 \embed X$ is the inclusion map, and let $[\omega_X]$ be a lift of $[\alpha]$. As a result, $[\omega_X]$ generates all  cohomology classes in the right diagram below:
 \begin{equation*}
 \begin{tikzcd}
 & H^2(Y_1)\arrow[r, "k_1^*"] &  H^2(\Sigma) & & {[\omega_1]}\arrow[r,mapsto, "k_1^*"] & {[\mu_1]}\arrow[dd,equal]\\
 H^2(X)\arrow[ru,"m_1^*"] \arrow[r, "m_b^*"]\arrow[rd, "m_2^*"]& H^2(W)\arrow[ru, "j_1^*"] \arrow[rd,"j_2^*"] & &  {[\omega_X]}\arrow[ru,"m_1^*",mapsto] \arrow[r, "m_b^*",mapsto]\arrow[rd, "m_2^*",mapsto]& m_b^*[\omega_X]\arrow[ru, "j_1^*",mapsto] \arrow[rd,"j_2^*",mapsto]\\
 &  H^2(Y_2)\arrow[r,"k_2^*"] & H^2(\Sigma) & &  {[\omega_2]}\arrow[r,"k_2^*",mapsto] & {[\mu_2]}.
 \end{tikzcd}
 \end{equation*}
 \item\label{Q6} There exists a closed 2-form $\bomega_X\in \Omega^2(X;i\R)$ on $X$ with the following properties:
 \begin{itemize}
 	\item $\bomega_X$ realizes the class  $[\omega_X]\in H^2(X;i\R)$;
\item $\bomega_X=\bomega_i $ (see \ref{P5}) within a collar neighborhood of $Y_i\subset X_{-1}$ for $i=1,2$;
\item\label{Q7} $
\bomega_X=\mu$, within a collar neighborhood of $W\subset X_{-1}$.
 \end{itemize}
The existence of such a form $\bomega_X$ is guaranteed by the cohomological condition in  \ref{Q4}. Finally, set $\omega_\lambda=\chi_1(s)ds\wedge \lambda$ and 
\[
\omega_X\colonequals\bomega_X+\omega_\lambda=\bomega_X+\chi_1(s)ds\wedge\lambda \text{ on } X.
\] 

\item\label{Q8} For any two closed forms $\omega_X$  and $\omega_X'$ satisfying the condition in \ref{Q7}, they are said to be equivalent if $\omega_X'-\omega_X=da$ for a compactly supported smooth 1-form $a\in \Omega^1(X;i\R)$. Denote by $[\omega_X]_{cpt}$ the equivalence classes of $\omega_X$.

 \end{enumerate}

\begin{example}\label{E9.2} The product cobordism $\x=[-1,1]_t\times \y: \y\to\y$. In this case, $X=[-1,1]_t\times Y$ and $\psi_X=\Id_{[-1,1]_t}\times\psi$ is the product map. We obtain $\omega_X$ as the pull-back of the 2-form $\omega$ from $Y$. 
\end{example}

\begin{example}\label{Ex9.5} Take $\y_1, \y_2\in\Cob_s$ with $Y_1= Y_2=Y$ and $\psi_1 $ isotopic to $\psi_2$. Suppose in addition that $\omega_2-\omega_1=d_Yb$ for a compactly supported 1-from $b\in \Omega^1(Y;i\R)$, then one may construct a cobordism $\x: \y_1\to \y_2$ as follows. Let $X=[-1,1]_t\times Y$ and $\psi_X$ be an isotopy from $\psi_1$ to $\psi_2$. Set $\omega_X= d_X(\chi(t)b)+\omega_1$, where $\chi(t)$ is any cut-off function such that
\[
\chi(t)\equiv 0 \text{ if } t\leq -1/2;\ \chi(t)\equiv 1 \text{ if } t\geq 1/2.\qedhere
\]
\end{example}

Similar to the definition of $\hy$, for each strict cobordism $X: Y_1\to Y_2$, we obtain a cobordism between $\hy_1$ and $\hy_2$  by attaching cylindrical ends to $X$:
\[
\hx\colonequals X\cup_{\psi_X} [-1,1]_t\times [-1,\infty)_s\times \Sigma: \hy_1\to \hy_2. 
\]

A planar metric $g_X$ on $X$ is a metric compatible with the corner structure. We insist that the metric $g_W$ of $W=[-1,1]_t\times \Sigma$ is the product one
\[
v^2(t)dt^2+g_\Sigma
\] 
for some function $v:[-1,1]_t\to [0,\infty)$ such that $v(t)\equiv 1$ for $1-|t|\ll 1$. This flexibility allows us to compose strict cobordisms equipped with planar metrics, but we shall always take $v(t)\equiv 1$ in the sequel for the sake of simplicity. We require the planar metric $g_X$ be the product metric 
\[
d^2t+d^2s+g_\Sigma
\]
in a neighborhood $(-\epsilon,0]_t\times (-1,0]_s\times X_{-2}$ of the co-dimension $2$ stratum $X_{-2}=(-\Sigma)\cup \Sigma$. Also, $g_X$ needs to be cylindrical near the co-dimensional 1 stratum $X_{-1}$:
\begin{align*}
g_X|_{ [-1, -1+\epsilon)\times Y_1}&=d^2t+ g_1,&
g_X|_{(1-\epsilon, 1]\times Y_2 }&=d^2t+ g_2,\\ g_X|_{[-1,1]_t\times (-1,0]_s\times \Sigma}&=d^2s+g_{W}=d^2t+d^2s+g_\Sigma.  
\end{align*}

Such a metric extends to a cylindrical metric on $\hx$ compatible with that of $(-\hy_1)\cup \hy_2$. When it is clear from the context, we also use $g_X$ to denote this extended metric on $\hx$.

Although a planar metric $g_X$ of $X$ is \textbf{not} encoded in a morphism $\x: \y_1\to\y_2$, it is needed to define the cobordism map induced by $\x$ (and so the
functor $\HM_*$ in Theorem \ref{1T2}). This cobordism map will be independent of the choice of $g_X$. 

Similar to Lemma \ref{9.1}, one can find a co-closed 2-form $\omega_{X,h}$ on $\hx$ extending $\omega_\lambda=\chi_1(s)ds\wedge\lambda$. In this case, we also insist a Dirichlet boundary condition for $*_4\omega_{X,h}$. This property is crucial for the energy equations in Section \ref{Sec12}, cf. Theorem \ref{Energy10.1}.

\begin{lemma} \label{9.3} For any morphism $\x\in \Cob_s$, there exists a co-closed 2-form $\omega_{X,h}$ on $\hx$ such that $\omega_{X,h}=ds\wedge\lambda$ on $[-1,1]_t\times [-1,\infty)_s\times\Sigma$ and 
	\begin{equation}\label{9.4}
	*_4 \omega_{X,h}\big|_{(-\hy_1)\cup\hy_2}=0.
	\end{equation}
 In particular, $\omega_{X,h}-\omega_\lambda\in L^2(\hx)$. 
\end{lemma}
\begin{proof}[Proof of Lemma \ref{9.3}]  It suffices to verify that the class 
	\[
	[dt\wedge *_\Sigma\lambda]\in H^2(W,\partial W; i\R)
	\]
	lies in the image $ \im\big(H^2(X, Y_1\cup Y_2; i\R)\to H^2(W,\partial W; i\R)\big)$ with $W=[-1,1]_t\times \Sigma$. By the property \ref{P3}, take $z$ to be a lift of $[*_\Sigma\lambda]$ in $H^1(Y_1;i\R)$. In the diagram below, all cohomology groups take value in $i\R$:
	\[
	\begin{tikzcd}[column sep=2.3cm]
H^1(Y_1)\arrow[r,"{z\mapsto (z,0)}"]\arrow[d] &H^1(Y_1) \oplus H^1(Y_2) \arrow[r,"\delta"]\arrow[d]&H^2(X, Y_1\cup Y_2)\arrow[d] \\
 H^1(\Sigma) \arrow[r,"{[*_\Sigma\lambda]\mapsto ([*_\Sigma\lambda],0)}"]& H^1(\{-1\}\times \Sigma)\oplus H^1(\{1\}\times \Sigma) \arrow[r,"\delta"]& H^2(W,\partial W)
 	\end{tikzcd}
	\]
	
	Alternatively, one may construct the form $\omega_{X,h}$ by hands using the co-closed form $\omega_{1,h}\in \Omega^1(\hy_1;i\R)$ in Lemma \ref{9.1}. Take a smooth function $h: [-1,1]_t\to \R$ such that $h$ is supported on $[-1,-1+\epsilon]$, $h(t)\equiv 1$ for $t\in [-1,-1+\epsilon/2]$, and 
	\[
	\half \int_{-1}^1 h(t)dt=1.
	\]
	 Then there is another function $f:[-1,1]_t\to \R$ with $f(1)=f(-1)=0$ and $f'(t)=h(t)-1$. Finally, set  $*_4\omega_{X,h}=h(t)dt\wedge *_3\omega_{1,h}-d_X(f(t)\chi_1(s)*_\Sigma\lambda)$. 
\end{proof}

\subsection{Relative \spinc Structures and Configuration Spaces}\label{Subsec9.3} Let $\s_{std}=(S_{std},\rho_{std,3})$ be the standard \spinc structure on $\R_s\times\Sigma$ as described in Section \ref{Sec5} with
\[
S_{std}=\C\oplus \Lambda^{0,1} \Sigma.
\]

For any object $\y=(Y, \psi, g_Y, \omega, \q)\in \Cob_s$, a relative \spinc structure $\bs=(\s, \varphi)$ is a pair consisting of a \spinc structure $\s=(S,\rho_3)$ on $Y$ and an isomorphism of \spinc structures
\[
\varphi: (S,\rho_3)|_{\partial Y}\to \psi^*\s_{std}|_{\partial Y}
\]
near the boundary compatible with $\psi$. The set of isomorphism classes of relative \spinc structures on $Y$, denoted by
\[
\Spincr(Y),
\]
is a torsor over $H^2(Y,\partial Y; \Z)$. There is a natural forgetful map from $\Spincr(Y)$ to the set of isomorphism classes of \spinc structures:
\[
\Spincr(Y) \to \Spinc(Y),\ \bs=(\s,\varphi)\mapsto \s,
\]
whose fiber is acted on freely and transitively by  $H^1(\Sigma, \Z)/\im (H^1(Y,\Z))$ reflecting the change of boundary trivializations. Any $\bs\in \Spincr(Y)$ extends to a relative \spinc structure on $\hy$ (interpreted suitably), denoted also by $\bs$.

\medskip

Let $(B_*,\Psi_*)$ be the translation invariant configuration on $\R_s\times\Sigma$ such that the restriction
\begin{equation}\label{reference}
(B_*,\Psi_*)|_{\R_s\times\T^2_i}
\end{equation}
on each component $\T^2_i\subset \Sigma$ is defined by the formula \eqref{E2.6} for all $1\leq i\leq n$. Take $(B_0,\Psi_0)$ to be a smooth configuration on $\hy$ which agrees with $(B_*,\Phi_*)$ on the cylindrical end $[0,\infty)_s\times \Sigma$. Recall from \ref{P2} that the closed 2-form $\omega\in \Omega^2(Y;i\R)$ defined on $Y$ extends to a closed 2-form on the completion $\hy$ by setting 
\[
\omega|_{[-1,\infty)\times \Sigma}=\mu+ds\wedge \lambda,
\]
and $[\omega]\in H^2(Y;i\R)$ is the period class of $\omega$. Now consider the configuration space for any $k> \half$:
\begin{align*}
\SC_k(\hy,\bs)=\{(B,\Psi): (b,\psi)=(B,\Psi)-(B_0,\Psi_0)\in L^2_k (\hy; iT^*\hy\oplus S)
\}.
\end{align*}
\begin{remark} Since $\hy$ is non-compact, the condition that $(b,\psi)\in L^2_k$ says that the section $(b,\psi)$ decays mildly along the cylindrical end of $\hy$. It turns out that this decay is always exponential for solutions to the Seiberg-Witten equations, cf. Theorem \ref{11.5} below. 
\end{remark}

\begin{definition}\label{D9.4} The perturbed Chern-Simons-Dirac functional on $\SC_k(\hy, \bs)$ is defined as 
	\begin{equation}
\CL_\omega (B,\Psi)=-\frac{1}{8}\int_{\hy} (B^t-B_0^t)\wedge (F_{B^t}+F_{B_0^t})+\half \int_{\hy}\langle D_B\Psi, \Psi\rangle+\half \int_{\hy}(B^t-B_0^t)\wedge \omega. \qedhere
	\end{equation}
\end{definition}
\begin{remark} $\CL_\omega$ is the analogue of the gauged action functional $\CA_H$ in the context of gauged Witten equations, see \cite[Definition 4.1]{Wang202}. 
\end{remark}

The configuration space $\SC_k(\hy,\bs)$ is acted on freely by the gauge group
\begin{align*}
\CG_{k+1}(\hy)=\{u: \hy\to S^1\subset \C: u-1\in L^2_{k+1} (\hy, \C)\},
\end{align*}
via the formula:
\[
u(B,\Psi)=(B-u^{-1}du, u\Psi). 
\]

The Lie algebra of $\CG_{k+1}$ is $
\Lie (\CG_{k+1})= L^2_{k+1}(\hy;i\R).
$ The exponential map $f\mapsto e^f$ is surjective onto the identity component $\CG_{k+1}^e$of $\CG_{k+1}$; they fit to a short exact sequence:
\[
0\to \CG_{k+1}^e\to \CG_{k+1}\to\pi_0(\CG_{k+1})\cong H^1(Y,\Sigma;\Z)\to 0.
\]

The Chern-Simons-Dirac functional $\CL_\omega$ is not fully gauge-invariant in general: 
\begin{lemma}[{Compare \cite[Section 29]{Bible}}]\label{L9.4} For any $\gamma=(B,\Psi)\in \SC_k(\hy,\bs)$ and $u\in \CG_{k+1}(\hy)$, we have 
\[
\CL_\omega(u\cdot \gamma)-\CL_\omega(\gamma)=(2\pi^2[u]\cup c_1(S)-2\pi i[u]\cup [\omega])[Y,\partial Y],
\]
where $[u]=[\frac{u^{-1}du}{2\pi i}]\in H^1(Y,\partial Y;\Z)$ is the relative cohomology class determined by $u$ and $[\omega]$ is the period class of $\omega$. 
\end{lemma}

The tangent space at each $\gamma\in \SC_k(\hy,\bs)$ is naturally identified with $L^2_k (\hy, iT^*\hy\oplus S)$. We compute the gradient of $\CL_\omega$ with respect to the $L^2$ inner product:
\begin{equation}\label{F9.6}
\grad \CL_\omega(B,\Psi)= (\half *_3 F_{B^t}+\rho_3^{-1}(\Psi\Psi^*)_0-*_3\omega, D_B\Psi).
\end{equation}

A configuration $\gamma\in \SC_{k}(\hy,\bs)$ is a critical point of $\CL_\omega$ if and only if it solves the perturbed Seiberg-Witten equations on $\hy$.
\begin{definition} For any object $\y=(Y, \psi, g_Y, \omega, \q)\in \Cob_s$, the Seiberg-Witten map defined on $\SC_k(\hy,\bs)$ is given by (ignoring the perturbation $\q$ for a moment)
	\[
	\F_\omega(B,\Psi)=(\half \rho_3(F_{B^t}-2\omega)-(\Psi\Psi^*)_0, D_B\Psi).
	\]
	and the equation 
	\begin{equation}\label{3DSWEQ}
\F_\omega(B,\Psi)=0
	\end{equation}
is called the 3-dimensional Seiberg-Witten equations. 
\end{definition}

\begin{remark} By Theorem \ref{T2.6}, the reference configuration $(B_*,\Psi_*)$ defined in (\ref{reference}) is the unique $\R_s$-translation invariant solution of (\ref{3DSWEQ}) on $\R_s\times \Sigma$ up to gauge. 
\end{remark}

The downward gradient flowline equation of $\CL_\omega$
\[
\dt (B(t), \Psi(t))=-\grad \CL_\omega(B(t),\Psi(t))
\]
can be cast into the 4-dimensional Seiberg-Witten equations:
\begin{equation}\label{4DSWEQ}
\left\{
\begin{array}{r}
\half \rho_4(F_{A^t}^+-2\omega^+_X)-(\Phi\Phi^*)_0=0,\\
D_A^+\Phi=0,
\end{array}
\right.
\end{equation}
on $\R_t\times \hy$ with $A=\dt +B(t), \Phi=\Psi(t)$ and $\omega_X=\pi^*\omega$ where $\pi:\R_t\times \hy\to\hy$ is the projection map. This corresponds to the product cobordism $[-1,1]_t\times \hy$ in Example \ref{E9.2}.

\medskip

In general, let $(A_*,\Phi_*)$ be the $\C$-translation-invariant solution on $\C_z\times\Sigma$ with 
\begin{equation}\label{E9.7}
A_*=dt\otimes\Pt+B_*, \Phi_*(t)=\Psi_*, z=t+is.
\end{equation}

Let $\x=(X,\psi_X, W,[\omega_X]_{cpt}): \y_1\to \y_2$ be a morphism in $\Cob_s$ and suppose $\hx: \hy_1\to \hy_2$ extends to a \textit{relative} \spinc cobordism:
\begin{equation}\label{E9.8}
(\hx,\bs_X): (\hy_1, \bs_1)\to (\hy_2,\bs_2).
\end{equation}

\begin{remark}\label{R9.2} For a \textit{relative} \spinc cobordism, we insist that identification maps 
	\[
	(\hx,\bs_X)|_{\hy_i}\cong  (\hy_i, \bs_i), i=1,2
	\]
	are implicitly baked in the definition. 
\end{remark}
Let $(A_0,\Phi_0)$ be a reference configuration on $\hx$ whose restriction on $[-1,1]_t\times [0,\infty)_s\times \Sigma$ agrees with $(A_*,\Phi_*)$. For each $k\geq 1$, define 
\begin{align*}
\SC_k(\hx,\bs_X)=\{(A,\Phi): (a,\phi)=(A,\Phi)-(A_0,\Phi_0)\in L^2_k (\hx, iT^*\hx\oplus S^+)
\}.
\end{align*}

In this case, we take $\omega_X\in \Omega^2(\hx;i\R)$ to be the closed 2-form constructed in \ref{Q6} and extended constantly over the cylindrical end $[-1,1]_t\times [0,\infty)_s\times \Sigma$; so for some $\epsilon>0$, 
\begin{itemize}
\item $\omega_X=\omega_1$ on $\hy_1\times [-1,-1+\epsilon)_t$;
\item $\omega_X=\omega_2$ on $\hy_2\times (1-\epsilon, 1]_t$;
\item $\omega_X=\mu+ ds\wedge\lambda$ on $[-1,1]_t\times [0,\infty)_s\times \Sigma$.
\end{itemize}

Then the left hand side of (\ref{4DSWEQ}) defines a smooth map:
\begin{equation}\label{E9.9}
\F_X: \SC_k(\hx,\bs_X)\to L^2_{k-1}(\hx, i\su(S^+)\oplus S^-)
\end{equation}
called the Seiberg-Witten map on $\hx$. For $0\leq j\leq k$, let $\V_j$ be the trivial vector bundle with fiber $L^2_{j}(i\su(S^+)\oplus S^-)$ over $\SC_k(\hx, \bs)$:
\[
\V_j\colonequals L^2_{j}(i\su(S^+)\oplus S^-)\times \SC_k(\hx, \bs).
\]
The Seiberg-Witten map $\F_X$ defines a smooth section of $\V_{k-1}\to \SC_k(\hx,\bs_X)$.

\subsection{Strict \spinc cobordisms} Now let us introduce the strict \spinc cobordism category $\SCob_s$, which plays the central role in Theorem \ref{1T2}:
\begin{itemize}
	\item each object of $\SCob_s$ is a pair $(\y,\bs)$ where $\y$ is an object of $\Cob_s$, and $\bs\in \Spincr(Y)$ is a relative \spinc structure on $Y$;
	\item for any objects $(\y_1,\bs_1)$ and $(\y_2,\bs_2)$,
	\[
	\Hom_{\SCob_s} ((\y_1,\bs_1),(\y_2,\bs_2))=\Hom_{\Cob_s}(\y_1,\y_2). 
	\] 
\end{itemize}

\subsection{Homotopy Classes of Paths}\label{Subsec2.4} To define the monopole Floer homology $\HM_*(\y,\bs)$ for each object $(\y,\bs)\in \SCob_s$, we will look at the moduli spaces of the Seiberg-Witten equations \eqref{4DSWEQ} on $\R_t\times (\hy,\bs)$ and define a Floer chain complex:
\[
 \Ch_*(\y, \bs);
\]
The underlying idea is an infinite dimensional Morse theory in the quotient configuration space: 
\[
\CB_k(Y,\bs)\colonequals \SC_k(Y,\bs)/\CG_{k+1}(Y). 
\]
For any $\fa,\fb\in \SC_k(Y,\bs)$, the relative homotopy classes of paths $\pi_1(\CB_k(Y,\bs); [\fa],[\fb])$ is a torsor over 
\[
\pi_1(\CB_k(Y,\bs); [\fb])\cong \pi_0(\CG_{k+1})\cong H^1(Y,\partial Y;\Z). 
\]
Moreover, for any $[\gamma]\in \pi_1(\CB_k(Y,\bs); [\fa],[\fb])$, the relative loop space $\Omega_{[\gamma]}(\CB_k(Y,\bs); [\fa],[\fb])$ in the class $[\gamma]$ is simply connected, since 
\[
\pi_2(\CB_k(Y,\bs); [\fb])\cong \pi_1(\CG_{k+1})=\{0\}.
\]

There are three additional ways to think of a path $\cgamma:[-1,1]\to \CB_{k}(Y,\bs)$ with $\cgamma(-1)=\fa$ and $\cgamma(1)=\fb$, and we shall use them interchangeably:
\begin{enumerate}
\item a path $\cgamma_1: [-1,1]\to \SC_{k}(Y,\bs)$ that connects $\fa$ and $u\cdot \fb$ for some $u\in \CG_{k+1}(\hy)$; 
\item  a configuration $\gamma$ on the 4-manifold $I\times (\hy,\bs)$ with $I=[-1,1]_t$ such that $\gamma|_{\{-1\}\times \hy}=\fa$ and 
$\gamma|_{\{1\}\times \hy}=u\cdot \fb$ for some $u\in \CG_{k+1}(\hy)$;
\item a configuration $\gamma'$ for a relative \spinc cobordism
\[
(\hx=I\times\hy,\bs_X): (\hy,\bs)\to (\hy,\bs)
\]
such that $\gamma|_{\{-1\}\times \hy}=\fa$ and $\gamma|_{\{1\}\times \hy}=\fb$. Indeed, all such relative \spinc structures on $I\times \hy$ form a torsor over 
\begin{align*}
H^2(I\times Y,\partial (I\times Y);\Z)&\cong H^1(Y,\partial Y;\Z)\otimes H^1(I,\partial I; \Z)\cong H^1(Y,\partial Y;\Z). 
\end{align*}
\end{enumerate}

The last standpoint makes it easier to think about a general morphism $\x:\y_1\to \y_2$. To make $\HM_*$ into a functor from $\SCob_s$ to $\NR\text{-}\mathrm{Mod}$ as in Theorem \ref{1T2}, we attach cylindrical ends to $\hx$ and obtain a complete Riemannian manifold $\CX$:
\[
\CX\colonequals\bigg( (-\infty, -1]_t\times  \hy_1\bigg)\cup \hx\cup\bigg( 
[1,\infty)_t\times \hy_2\bigg).
\] 
 The closed 2-form $\omega_X$ extend over $\CX$ by setting
\begin{equation}\label{E9.11}
\omega_X=\omega_1 \text{ on }  (-\infty, -1]_t\times Y_1;\ \omega_X=\omega_2 \text{ on }  [1, \infty)_t\times Y_2. 
\end{equation}

The goal of this paper is to analyze the Seiberg-Witten equations \eqref{4DSWEQ} on $\CX$ and construct a chain map:
\begin{equation}\label{E9.12}
\Ch_*(\x): \Ch_*(\y_1,\bs_1)\to \Ch_*(\y_2,\bs_2)
\end{equation}
that is independent of the choice of 
\begin{itemize}
\item the planar metric $g_X$ compatible with $(g_{Y_1}, g_{Y_2}, g_\Sigma)$;
\item the closed 2-form $\omega_X\in \omega^2(X;i\R)$ in the class $[\omega_X]_{cpt}$;
\item any auxiliary perturbation of \eqref{4DSWEQ} defined in Subsection \ref{Subsec22.1};
\end{itemize}
up to chain homotopy. To do so, we have to take into account of all isomorphism classes of relative \spinc cobordisms:
\[
\Spincr(X; \bs_1,\bs_2)\colonequals \{ \text{all possible }\eqref{E9.8}: (\y_1,\bs_1)\to (\y_2,\bs_2)   \} \text{ modulo isomorphisms}
\]
which is a torsor over $H^2(X,\partial X;\Z)$. Indeed, any two relative \spinc cobordisms $\bs_{X,1},\bs_{X,2}$ that cover the 4-manifold $X$ with corners are related by a complex line bundle $L_{12}\to X$:
\[
\bs_{X,2}=\bs_{X,1}\otimes L_{12},
\]
and a trivialization $L_{12}\cong \C$ is specified along $\partial X$. Some of elements of $\Spincr(X; \bs_1,\bs_2)$ may arise from different underlying \spinc structures, but they all contribute to the chain map \eqref{E9.12} and will not be separated from each other. For any $\fa_i\in \SC_k(\hy_i,\bs_i),\ i=1,2$, an element of $\Spincr(X; \bs_1,\bs_2)$ can be viewed a homotopy class of $\x$-paths that connect $\fa_1$ and $\fa_2$.

\section{The Quotient Configuration Space and Slices}\label{Sec10}

 Configurations in $\SC_k(\hy, \bs)$ and $\SC_k(\hx, \bs_X)$ are required to converge to a fixed limit in the spatial direction, so by definition, they are never reducible, i.e. $\Psi$ or $\Phi\not\equiv 0$. This prevents us from  finding a global slice of the gauge action as in \cite[Section 9.6]{Bible} over the non-compact manifold $\hy$ or $\hx$. Nevertheless, local slices always exists. In this section, we prove that:
 
 \begin{proposition}\label{P10.1} For either $(M,\bs_M)=(\hy,\bs)$ or $(\hx, \bs_X)$, the quotient space 
 	 \[
 	\CB_k(M,\bs_M)\colonequals \SC_k(M,\bs_M)/\CG_{k+1}(M)
 	\]
 	is a Hilbert manifold when $2(k+1)>\dim M$ and $k\in \Z$. 
 \end{proposition}

It is clear from the formula
\[
(uv-1)=(u-1)(v-1)+(u-1)+(v-1),\ \forall u, v\in \CG_{k+1}(M)
\]
that $\CG_{k+1}(M)$ is a Hilbert Lie group when $2(k+1)>\dim M$. Following the book \cite[Section 9]{Bible}, we base the argument on a general principle:
\begin{lemma}[\cite{P68},\cite{Bible} Lemma 9.3.2]\label{L10.2} Suppose a Hilbert Lie group $G$ acts smoothly and freely on a Hilbert manifold $C$, and the quotient space  $C/G$ is Hausdorff. Suppose that at each $c\in C$, the differential 
\[
d_c: T_eG\to T_c G
\]
has closed range, then $C/G$ is also a Hilbert manifold. 
\end{lemma}

It remains to verify the condition of Lemma \ref{L10.2}.

\begin{lemma}\label{L10.3}For either $(M,\bs_M)=(\hy,\bs)$ or $(\hx, \bs_X)$, the quotient configuration space $\CB_k(M,\bs_M)$ is Hausdorff. 
\end{lemma}
\begin{proof} Suppose we have a sequence of configurations $\gamma_n=(A_n,\Phi_n)\in \SC_k(M,s)$ and a sequence of gauge transformations $u_n\in \CG_{k+1}(M)$ such that 
	\[
	\gamma_n\to \gamma \text{ and } u_n\cdot \gamma_n\to \gamma'
	\]
	for some $\gamma=(A,\Phi)$ and $\gamma'=(A',\Phi')$. We wish to show that $u\cdot \gamma=\gamma'$ for some $u\in \CG_{k+1}(M)$. We prove that $v_n\colonequals 1-u_n$ has uniformly bounded $L^2_{k+1}$ norm, so there is a weakly converging subsequence among $\{v_n\}$. Let $v$ be the weak limit and define $u\colonequals 1-v$. 
	
	We begin with the $L^2$-norm of $v_n$. Since $\|v_n\|_\infty\leq 2$, $|v_n|_2^2$ contributes to a bounded integral over any compact region of $M$. It suffices to estimate $|v_n|_2^2$ over the cylindrical end of $M$. Note that 
	\begin{align*}
\|v_n\Phi\|_2&=\|(1-u_n)\Phi\|_2\leq \|\Phi-\Phi'\|_2+\|\Phi'-u_n\Phi_n\|_2+\|u_n(\Phi_n-\Phi)\|_2,
	\end{align*}
	which is uniformly bounded. As $s\to\infty$, $\Phi$ approximates the standard spinor and is non-vanishing everywhere. It follows that $\|v_n\|_2\leq C$ for some uniform $C>0$. 

To deal with derivatives of $v_n$, let $w_n=u_n^{-1}du_n$. Then $\|w_n\|_{L^2_k}\leq \|u_n\cdot \gamma_n-\gamma_n\|_{L^2_k}\leq \|\gamma-\gamma'\|_{L^2_k}+1$ when $n\gg 1$. The estimate for $\|\nabla^l v_n\|_{L^2}\ (1\leq l\leq k+1)$ now follows from the relation 
\[
\nabla v_n=\nabla u_n=w_n-v_n\cdot w_n
\]
and an induction argument. If we already know $2k> \dim M$, then $L^2_k$ is a Banach algebra itself; otherwise, the first a few steps in the induction requires special treatments. For instance, if $\dim M=3$ and $k=1$, then we have to bound
\[
\|\nabla v_n\|_{p} \text{ for } 2\leq p\leq 6 \text{ and }\|\nabla^2 v_n\|_2.
\]
If $\dim M=4$ and $k=2$,  then we have to bound
\[
\|\nabla v_n\|_{p} \text{ for } 2\leq p<\infty,\ \|\nabla^2 v_n\|_p \text{ for } 2\leq p\leq 4 \text{ and } \|\nabla^3 v_n\|_2.
\]
For the Sobolev embedding theorem on cylinders, see \cite[Section 13.2]{Bible}.
\end{proof}

Let $\CT_k$ be the tangent space of $\SC_k(M,\bs_M)$. For each configuration $\gamma=(A,\Phi)\in \SC_k(M,\bs_M)$, let $\dg$ be the map obtained by linearizing the action of $\CG_{k+1}(M)$, extended to lower Sobolev regularities $(0\leq j\leq k)$:
\begin{align*}
\dg: L^2_{j+1}(M;i\R)&\to L^2_j(M, iT^*M\oplus S^+)=\CT_{j,\gamma}\\
f&\mapsto (-df, f\Phi).
\end{align*}
 
Let $\J_{j,\gamma}\subset \CT_{j,\gamma}$ be the image of $\dg$ and $\K_{j,\gamma}$ be the $L^2$-orthogonal complement of $\J_{j,\gamma}$: 
\begin{align*}
\K_{j,\gamma}&\colonequals \{ v\in \CT_{j,\gamma}: \langle v, \dg(f)\rangle_{L^2(M)}=0,\forall f\in L^2_{j+1}(M;i\R)\}\\
&=\{v=(\va,\vphi)\in L^2_j(M, iT^*M\oplus S^+): \dg^*(v)=0, \langle a, \vn\rangle =0 \text{ at }\partial M\}
\end{align*}
where $\vn$ is the outward normal vector at $\partial M$ and 
\begin{align*}
\dg^*: L^2_{j}(M, iT^*M\oplus S^+)&\to L^2_{j-1}(M;i\R)\\
(\va,\vphi)&\mapsto -d^*\va+i\re\langle i\Phi, \vphi\rangle. 
\end{align*}
is the formal adjoint of $\dg$. 

\begin{lemma}[cf. {\cite[Proposition 9.3.4]{Bible}}]\label{L10.4} As $\gamma$ varies over $\SC_k(M,\bs_M)$, $\J_{j,\gamma}$ and $\K_{j,\gamma}$ form complementary closed sub-bundles of $\CT_j$, and we have a smooth decomposition 
	\[
	\CT_j|_{\SC_k(M,\bs)}=\J_j\oplus \K_j, 0\leq j\leq k.
	\]
	In particular, $T\SC_k(M,\bs)=\CT_{k}=\J_k\oplus \K_k$. 
\end{lemma}
Proposition \ref{P10.1} now follows from Lemma \ref{L10.3} and \ref{L10.4}.  
\begin{proof}[Proof of Lemma \ref{L10.4}] For any $v=(\va,\vphi)\in \CT_{j,\gamma}$, we need to find the unique element $f\in L^2_{j+1}(M;i\R)$ such that $v-\dg(f)\in \K_{j,\gamma}$. Such an element solves the Neumann boundary value problem:
	\begin{equation}\label{F10.1}
\left\{\begin{array}{rl}
\Delta_M f+|\Phi|^2f &= -\dg^*(v)\\
\langle df,\vn\rangle&=\langle \delta a, \vn\rangle \text{ at } \partial M. 
\end{array}
\right.
	\end{equation}
	
	The left hand side of $\eqref{F10.1}$ forms a Fredholm operator ($1\leq j\leq k$):
	\begin{equation}\label{E10.2}
	(\Delta_M+|\Phi|^2, \frac{\partial }{\partial \vn}\bigg|_{\partial M} ): L^2_{j+1}(M;i\R)\to L^2_{j-1}(M;i\R)\times L^2_{j+1/2}(\partial M;i\R)
	\end{equation}
	which is in fact invertible. If $M$ is compact, this follows from \cite[Proposition 7.5]{PDEI}. In general, one may start with the special case when 
	\[
	(M, \Phi)=(\R_s\times\Sigma, \Psi_*) \text{ or } ([-1,1]\times \R_s\times \Sigma, \Phi_*)
	\]
	using Fourier transformation on the real line $\R_s$ and the positivity of $|\Psi_*|^2$. To show \eqref{E10.2} is Fredholm, apply the parametrix patching argument. To compute the index of \ref{E10.2}, note that the restriction map
	\[
\frac{\partial }{\partial \vn}\bigg|_{\partial M}  : L^2_{j+1}(M.i\R)\to L^2_{j+1/2}(\partial M;i\R)
	\] 
	is surjective, and the operator
	\[
	\Delta_M+|\Phi|^2: \{f\in L^2_2(M;i\R): \langle df, \vn\rangle=0\}\to L^2(M;i\R)
	\]
	is positive and self-adjoint. This proves that the operator \eqref{E10.2} is invertible.
	
	 Alternatively, one may follow the proof of \cite[Proposition 7.5]{PDEI}. Details are left as exercises.
\end{proof}
We record the next proposition for convenience:
\begin{proposition}\label{P10.5} Over the configuration space $\SC_k(\hy,\bs)$, the gradient \eqref{F9.6} of the Chern-Simons-Dirac functional $\CL_{\omega}$ defines a smooth section of $\K_{k-1}\to \SC_k(\hy,\bs)$ when $k\geq 1$.
\end{proposition}

\section{Energy Equations}\label{Sec12}

This section is devoted to the energy equations of the Seiberg-Witten equations \eqref{4DSWEQ} on $\hx$, which will play an important role in the proof of the Compactness Theorem \ref{T11.1} in Section \ref{Sec11}. In particular, it gives property \ref{K1}. The main results of this section are Theorem \ref{Energy10.1} and Proposition \ref{Energy10.2}.

\subsection{The 4-Dimensional Case}Following the book \cite[Section 4]{Bible}, we prove an energy equation associated to the perturbed Seiberg-Witten equations (\ref{4DSWEQ}):
\begin{theorem}[cf. {\cite[P.593]{Bible}}]\label{Energy10.1}For any morphism $\x: \y_1\to\y_2 $ in the strict cobordism category $\Cob_s$, choose a planar metric $g_X$ on $X$ and consider a relative \spinc cobordism $(\hx,\bs_X): (\hy_1, \bs_1)\to (\hy_2,\bs_2)$. Then for any configuration $\gamma=(A,\Phi)\in \SC(\hx, \bs_X)$, the $L^2$-norm of the Seiberg-Witten map $\F_X(A,\Phi)$ can be expressed as 
	\[
\int_{\hatx}| \F_X(A,\Phi)|^2=\E_{an}(A,\Phi)-\E_{top}(A,\Phi),
	\]
where 
	\begin{align}
	\E_{an}(A,\Phi)&\colonequals \int_{\hx} \frac{1}{4}|F_{A^t}|^2+|\nabla_A\Phi|^2+|(\Phi\Phi^*)_0+\rho_4(\omega^+_X)|^2+\frac{s}{4}|\Phi|^2-\langle F_{A^t},\bomega_X\rangle\label{an} \\
	&\qquad -\int_{\hx}\langle F_{A^t},\omega_\lambda-\omega_{X,h}\rangle-\int_{\hx}F_{A^t_0}\wedge *_4\omega_{X,h}, \nonumber\\
	\E_{top}(A,\Phi)&\colonequals 2\CL_{\omega_1}(B_1,\Psi_1)-2\CL_{\omega_2}(B_2,\Psi_2)+\frac{1}{4}\int_{\hx} F_{A^t_0}\wedge F_{A^t_0}-\int_{\hx}F_{A^t_0}\wedge \omega_X,\label{top}
	\end{align}
and $(B_i,\Psi_i)=(A,\Phi)|_{\hy_i}$ are restrictions of $\gamma$ at $\hy_i$ for $i=1,2$. Here, $\omega_X=\bomega_X+\omega_\lambda$ is the closed 2-form constructed in \ref{Q6} with $\omega_\lambda=\chi_1(s)ds\wedge\lambda$. The co-closed 2-form $\omega_{X,h}$ is subject to the Neumann boundary condition and the constraint that $\omega_\lambda-\omega_{X,h}\in L^2(\hx)$. Its existence is guaranteed by Lemma \ref{9.3}. 
\end{theorem}

\begin{remark}\label{rmk-10.2} Let us explain why (\ref{an}) is a useful expression. Errors terms in the second line of $(\ref{an})$ are bounded below by 
	\[
	-\frac{1}{16} \|F_{A^t}\|^2_{L^2(\hx)}-C(A_0, \omega_X,g_X)
	\]
	for some constant $C(A_0, \omega_X,g_X)>0$.
	
	 The first line of (\ref{an}) is consistent with the \textit{local energy functional} $\E_{an}(A,\Phi; \Omega)$ in Definition \ref{D1.3}. Indeed, over the cylindrical end  $I\times [0,\infty)_s\times \Sigma$, (\ref{an}) becomes (with $I=[-1,1]_t)$:
	\begin{align}\label{E10.3}
\int_{I\times [0,\infty)_s}\int_{\Sigma}  \frac{1}{4}|F_{A^t}|^2+|\nabla_A\Phi|^2+|(\Phi\Phi^*)_0+\rho_4(\omega^+)|^2-\langle F_{A^t}^\Sigma,\mu\rangle
	\end{align}
	where $\omega=\mu+ds\wedge\lambda$. The last term in \eqref{E10.3} 
	\[
	-\int_{\Sigma} \langle F_{A^t}^\Sigma,\mu\rangle 
	\]
	is always zero. Indeed, if we write $a=A-A_0\in L^2(\hx, iT^*\hx)$, then $
	F_{A^t}^\Sigma=2d_\Sigma a $	
	is an exact form on the  surface $\Sigma$. Since $\mu$ is harmonic on $\Sigma$, their inner product is always zero. Hence, \eqref{E10.3} has a definite sign. The integral in \eqref{an} over the compact region $X=\{s\leq 0\}\subset \hx$ can be treated in the usual way. We summarize this remark into a lemma.
	 \end{remark}
	
	\begin{lemma}\label{L11.3} Under the assumption of Theorem \ref{Energy10.1}, there exists a constant $C_2(A_0, \omega_X,g_X)$ independent of $(A,\Phi)$ such that 
		\[
		\E_{an}(A,\Phi)+C_2>\int_{\hx} \frac{1}{8}|F_{A^t}|^2+|\nabla_A\Phi|^2+|(\Phi\Phi^*)_0+\rho_4(\omega^+_X)|^2+\frac{s}{4}|\Phi|^2.
		\]
	\end{lemma} 
	\begin{proof} Note that 
		\[
		|\int_{\hx} \langle F_{A^t}, \bomega_X\rangle| =	|\int_X \langle F_{A^t}, \bomega_X\rangle| \leq \frac{1}{16} \|F_{A^t}\|^2_{L^2(\hx)}+C_3(A_0,\omega_X, g_X).\qedhere 
		\]
	\end{proof}

\begin{proof}[Proof of Theorem \ref{Energy10.1}] Let $\gamma_0=(A_0,\Phi_0)$ be the reference configuration in $\SC(\hx, \bs_X)$. For convenience, take its restrictions at the boundary
	\[
	(B_{i0},\Phi_{i0})=\gamma_0|_{\hy_i}\in \SC(\hy_i)
	\]
	as reference configurations in the definition of $\CL_{\omega_i}$ for $i=1,2$. It suffices to prove the theorem when the section
	\[
	(a,\phi)=(A,\Phi)-(A_0,\Phi_0)\in \SC_c^\infty(\hx, iT^*\hx\oplus S),
	\]
	is compactly support, and the rest will follow by continuity. Let $X_S=\{s\leq S\}\subset \hx$ be the truncated manifold and $Y_{i, S}=\hy_i\cap X_S$. The boundary of $X_S$ consist of three parts:
	\[
	-Y_{1,S}, Y_{2,S} \text{ and } \{S\}\times W=[-1,1]_t\times \{S\}\times \Sigma.
	\]
	
	Since $(a,\phi)$ is compactly supported, we may discard any boundary integrals over $\{S\}\times W\subset \partial X_S$ when $S\gg 1$. By the Lichnerowicz-Weizenb\"{o}ck formula \cite[(4.15)]{Bible}, we have 
	\begin{align}
	\int_{X_S} |D_A^+\Phi|^2=\int_{X_S}|\nabla_A\Phi|^2&+\half \langle \rho_4(F_{A^t}^+)\Phi, \Phi\rangle+\frac{s}{4}|\Phi|^2\label{10.1}\\
	&-\int_{Y_{1,S}} \langle D_{B_1}\Phi_1,\Phi_1\rangle+\int_{Y_{2,S}} \langle D_{B_2}\Phi_2,\Phi_2\rangle. \nonumber
	\end{align} 
	Now consider the first equation of (\ref{4DSWEQ}): 
\begin{align}
\int_{X_S} |\half \rho_4(F_{A^t}^+-2\omega^+_X)-(\Phi\Phi^*)_0|^2&=\int_{X_S} \frac{1}{4}|F_{A^t}|^2-\half \langle \rho_4(F_{A^t}^+)\Phi, \Phi\rangle+|(\Phi\Phi^*)_0+\rho_4(\omega^+_X)|^2\nonumber\\
&-\frac{1}{4}\int_{X_S} F_{A^t}\wedge F_{A^t}-2\int_{X_S}\langle F_{A^t}, \omega_X^+\rangle.\label{10.2}
\end{align}
 
Only the second line requires some further work. Note that 
\begin{align*}
-\frac{1}{4}\int_{X_S} F_{A^t}\wedge F_{A^t}&= -\frac{1}{4}\int_{X_S} F_{A^t_0}\wedge F_{A^t_0}-\half\int_{\partial X_S} a\wedge (F_{A^t}+F_{A_0^t}).
\end{align*}
 Finally, using the relation $\omega_X=\bomega_X+\omega_\lambda$, we compute
\begin{align*}
2\int_{X_S}\langle F_{A^t}, \omega_X^+\rangle&= \int_{X_S}\langle F_{A^t}, \omega_X+*_4\omega_X\rangle\\
&=  \int_{X_S}\langle F_{A^t}, \bomega_X\rangle +\langle F_{A^t}, \omega_\lambda\rangle+\langle F_{A^t_0}, *_4\omega_X\rangle+\langle 2da, *_4\omega_X\rangle\\
&=J_1+J_2+J_3+J_4. 
\end{align*}
$J_1$ and $J_3$ already show up in (\ref{an}) and \eqref{top}. As for $J_2$ and $J_4$, note that  
\begin{align*}
J_4&=-2\int_{\partial X_S} a\wedge \omega_X=\int_{Y_{1,S}} (B_1^t-B^t_{10})\wedge \omega_1+ \int_{Y_{2,S}} (B_2^t-B^t_{20})\wedge \omega_2,\\
J_2&= \int_{X_S} \langle F_{A^t}, \omega_\lambda\rangle=- \int_{X_S}  F_{A^t}\wedge *_4\omega_{X,h}+ \int_{X_S} \langle F_{A^t},\omega_\lambda-\omega_{X,h}\rangle.
\end{align*}

Since $*_4\omega_{X,h}$ is closed, the first term in $J_2$ is a pairing in cohomology: 
\[
[\frac{i}{2\pi} F_{A^t}]\cup [\frac{i}{2\pi}*_4\omega_{X,h}]\in H^4(X,\partial X)\xleftarrow{\cup} H^2(X,Z)\otimes H^2(X,Y_1\cup Y_2),
\]
so one may replace $A$ by $A_0$. Now the energy identity follows by adding (\ref{10.1}) and (\ref{10.2}) together. 
\end{proof}

\subsection{The 3-Manifold Case}  Let $I=[t_1,t_2]_t$. In the special case when $\x=I\times \y: \y\to\y$ is the product morphism, Theorem \ref{Energy10.1} takes a simpler form. 

The 4-manifold $\hx=I\times \hy$ is furnished with the product metric. Let $\omega_X=\pi^*\omega$ be the pull-back of $\omega$ where $\pi: \hx\to \hy$ is the projection map. Any \spinc connection $A$ on $\hx$ can be written as
\begin{equation}\label{10.3}
A=\dt+ B(t)+c(t)dt\otimes\Id_S. 
\end{equation}
where $B(t)$ is a path of \spinc connections on $(\hy,\bs)$ and $c(t)\in L^2_k(\hy;i\R)$. Any configuration $\gamma\in (A,\Phi)\in \SC_k(\hx,\bs_X)$ gives rise to a path $\cgamma(t)=(B(t),\Psi(t))$ in $\SC_{k-1/2}(\hy, \bs)$ by setting
\[
\Psi(t)=\Phi|_{\{t\}\times \hy}. 
\]

 Moreover, $\gamma$ solves the Seiberg-Witten equations (\ref{4DSWEQ}) on $\hx$ if and only if the path $(\cgamma(t),c(t))$ forms a  downward gradient flowline of $\CL_\omega$:
\[
\dt\gamma(t)=-\grad \CL_{\omega}(\gamma(t))-\bd_{\gamma(t)} c(t). 
\]
Let $A_0=\frac{d}{dt}+B_0$ be the reference connection on $(\hx, \bs_X)=I\times (\hy, \bs)$. The curvature form $F_{A_0^t}$ does not involve any $dt$-component, so $F_{A_0^t}\wedge F_{A^t_0}\equiv 0$. 

\begin{proposition}\label{Energy10.2} For any configuration $\gamma=(A,\Phi)$ on $(\hx, \bs_X)=I\times (\hy, \bs)$, the $L^2$-norm of the Seiberg-Witten map $\F_X(A,\Phi)$ can be expressed as 
	\[
	\int_{\hatx}| \F_X(A,\Phi)|^2=\E_{an}(A,\Phi)-\E_{top}(A,\Phi)
	\]
	where $	\E_{top}(A,\Phi)\colonequals 2\CL_{\omega}(\cgamma(t_1))-2\CL_{\omega}(\cgamma(t_2))$ and 
	\begin{align}\label{E10.8}
	\E_{an}(A,\Phi)&\colonequals\int_I\|\dt \cgamma(t)+d_{\cgamma(t)}c(t)\|^2_{L^2(\hy)}+\|\grad \CL_{\omega}(\cgamma(t))\|^2_{L^2(\hy)}\\
	&=\int_{I\times \hy}  \frac{1}{4}|F_{A^t}|^2+|\nabla_A\Phi|^2+|(\Phi\Phi^*)_0+\rho_4(\omega^+)|^2+\frac{s}{4}|\Phi|^2-\langle F_{A^t}, \omega\rangle.\nonumber
	\end{align}
	The last term can be written as 
	\begin{equation*}
	\int_{I\times\hy}\langle F_{A^t}, \omega\rangle= \int_{I\times\hy}\langle F_{A^t}, \bomega\rangle+ \int_{I\times\hy}\langle F_{A^t},  \omega_\lambda-\omega_h\rangle-|I|\int_{\hy} F_{B_0^t}\wedge *_3\omega_h,
	\end{equation*}
	where $\omega=\bomega+\omega_\lambda$ and $\omega_\lambda=\chi_1(s)ds\wedge\lambda$. The co-closed 2-form $\omega_h$ is constructed by Lemma \ref{9.1} such that $\omega_\lambda-\omega_h\in L^2(\hy)$. In particular, for any $(B,\Psi)\in \SC_1(\hy,\bs)$, 
	\begin{align*}
\|\grad \CL_{\omega}(B,\Psi)\|_{L^2(\hy)}^2&=\int_{\hy}  \frac{1}{4}|F_{B^t}|^2+|\nabla_B\Psi|^2+|(\Psi\Psi^*)_0+\rho_3(\omega)|^2+\frac{s}{4}|\Psi|^2-\langle F_{B^t}, \omega\rangle.
	\end{align*}
\end{proposition}

\section{Compactness}\label{Sec11}

\subsection{Statements} With all machinery developed so far, we are ready to state and prove the compactness theorem for the (unperturbed) Seiberg-Witten equations on $\R_t\times \hy$. The result easily generalizes to a complete Riemannian manifold $\X$ induced from a morphism $\x:\y_1\to\y_2$ in $\Cob_s$. Nevertheless, we will focus on the first case for the sake of simplicity. The analogous results for perturbed equations will be addressed in Section \ref{Sec15}, after we set up tame perturbations in the next part. Now let 
\[
\gamma_0\colonequals (A_0,\Phi_0) \text{ with } A=\dt+B_0,\ \Phi(t)=\Psi_0, 
\]
be the reference configuration on $\R_t\times \hy$, then it agrees with the standard configuration $(A_*,\Phi_*)$ over the planar end $\R_t\times [0,\infty)_s\times\Sigma$. For any $k\geq 2$, define 
\[
\SC_{k,loc}(\R_t\times (\hy,\bs))=\{(A,\Phi): (A,\Phi)\big|_{I\times \hy}\in \SC_k(I\times (\hy,\bs)) ,\forall \text{ finite interval } I\subset\R_t \}
\]
and $\CG_{k+1,loc}(\R_t\times (\hy,\bs))$ in a similar manner. We will set up the Fredholm theory of moduli spaces in a different way in Section \ref{Sec19}. For now, let us stick to these loosely defined spaces. 

For any $\gamma\in \SC_{k,loc}$ and $I\subset \R_t$, define the analytic energy $\E_{an}(\gamma; I)$ over the interval $I$ to be the integral of \eqref{E10.8} over $I\times \hy$ and 
\[
\E_{an}(\gamma)\colonequals \E_{an}(\gamma, \R_t). 
\]
One standard assumption below is the finiteness of the total energy $\E_{an}$. Since $\E_{an}(\gamma; I)$ is alway non-negative, it implies that 
\[
\E_{an}(\gamma; I)<\E_{an}(\gamma; \R_t)<\infty \text{ for any } I\subset \R_t.
\]

The primary result of this section is the compactness theorem.

\begin{theorem}\label{T11.1} Suppose $\{\gamma_n=(A_n,\Phi_n)\}\subset \SC_{k,loc}$ is a sequence of solutions to the Seiberg-Witten equations \eqref{4DSWEQ} on $\R_t\times\hy$ and their analytic energy
	\[
	\E_{an}(\gamma_n)\colonequals 	\E_{an}(\gamma_n, \R_t)<C
	\]
	is uniformly bounded by a positive constant $C>0$. Then we can find a sequence of gauge transformations $u_n\in \CG_{k+1,loc}(\R_t\times \hy)$ with the following properties. For a subsequence $\{\gamma_n'\}$ of $\{u_n(\gamma_n)\}$ and any finite interval $I\subset \R_t$, the restriction of each $\gamma_n'$ on $I\times \hy$
	\[
	\gamma_n'|_{I\times \hy}
	\]
	lies in $\SC_l(I\times (\hy,\bs))$. In addition, they converge in $L^2_{l}(I\times \hy)$-topology for any $l\geq 2$. 
\end{theorem}

The main difficulty is to deal with the cylindrical end of $\hy$ and the proof relies on the exponential decay of $L^2_l$-norms. To state the result, recall that $\Omega_{n,S}\ (n\in \Z,\ S\in \R_s)$ defined in \eqref{E2.4} is a bounded sub-domain of $\C$ with smooth boundary, which is centered at $(n, S)\in \R_t\times \R_s$. 

\begin{theorem}\label{11.5} For any $C>0$ and $l\in \Z_{\geq 1}$, there exists constants $\zeta(\hy,\bs), M_l(C, \hy,\bs)>0$ with the following significance. For any solution $\gamma=(A,\Phi)\in \SC_{k,loc}(\R_t\times \hy)$ to the Seiberg-Witten equations $(\ref{4DSWEQ})$ on $\R_t\times (\hy, \bs)$ with analytic energy $\E_{an}(A,\Phi)<C$, we can find a gauge transformation $u\in \CG_{k+1,loc}(\R_t\times \hy)$ such that 
	\begin{equation}\label{E11.1}
	\|u(\gamma)-\gamma_0\|_{L^2_{l,A_0}(\Omega_{n,S}\times \Sigma)}\leq M_le^{-\zeta S},
	\end{equation}
	for any $l\geq 1$, $n\in \Z$ and $S\geq 0$. Here $\gamma_0$ is the reference configuration in $\SC_{k,loc}(\R_t\times\hy)$.
\end{theorem}

Theorem \ref{T11.1} is an easy corollary of Theorem \ref{11.5}.
\begin{proof}[Proof of Theorem \ref{T11.1}] It suffices to prove the case when $I=[-2,2]$. The rest will follow by a patching argument (cf. \cite[Section 13.6]{Bible}). By Theorem \ref{11.5}, for any $\gamma_n$ in that sequence, we may assume the exponential decay \eqref{E11.1} holds for $\gamma_n-\gamma_0$ . Take $S\gg 1$ and let $Y_S=\{s\leq S\}$ be the truncated 3-manifold. 
	
	With the energy equation in Proposition \ref{Energy10.2}, the classical compactness theorem \cite[Theorem 5.2.1]{Bible} implies that a subsequence of $\{\gamma_n\}$ converges smoothly (up to gauge) in the interior of the compact manifold $I\times Y_S. $ Suppose $\{u_n: I\times Y_S\to S^1\}$ is the sequence of gauge transformations, then the restriction
 \[
 u_n: I\times [S-1,S]_s\times \Sigma\to S^1
 \]
 must lie in the same homotopy class when $n\gg1$ (by \eqref{E11.1}). We may correct $\{u_n\}$ so their restrictions lie in the trivial homotopy class. By a patching argument, we extend $u_n$ over the whole space $I\times \hy$ by setting $u_n\equiv 1$ when $s\geq S+1$. By Theorem \ref{11.5}, a subsequence of $\{u_n( \gamma_n)\}$ converges in fact in $L^2_l$-topology on $[-2+\epsilon, 2-\epsilon]\times \hy$ for some small $\epsilon>0$. This completes the proof of the theorem (some details are left to the readers).
\end{proof}

The proof of Theorem \ref{11.5} will dominate the rest of the section. 
\subsection{Decay of Local Energy Functional}\label{Subsec11.1} Recall from Definition \ref{D1.3} that the local energy functional of $\gamma=(A,\Phi)$ over $\Omega_{n,S}\subset \HH^+$ is defined as
	\begin{equation*}
\E_{an}(A,\Phi; \Omega)\colonequals\int_{\Omega}\int_{\Sigma} \frac{1}{4}|F_{A^t}|^2+|\nabla_A\Phi|^2+|(\Phi\Phi^*)_0+\rho_4(\omega^+)|^2. \qedhere
\end{equation*}
with $\omega=\mu+ds\wedge\lambda$. We wish to first get an estimate on $\E_{an}(A,\Phi; \Omega_{n,S})$ for a solution $(A,\Phi)$ to (\ref{4DSWEQ}) on $\R_t\times \hy$ when $S\gg 1$. The main results are as follows.
 \begin{theorem}\label{11.1} For any $C,\epsilon>0$, there exists a constant $R_0(\epsilon, C, \hy,\bs)>0$ with the following significance. For any solution $(A,\Phi)\in \SC_k(\R_t\times\hy)$ to the Seiberg-Witten equations $(\ref{4DSWEQ})$ on $\R_t\times (\hy, \bs)$ with analytic energy $\E_{an}(A,\Phi)<C$ and any $S>R_0$, we have
 	\[
 	\E_{an}(A,\Phi; \Omega_{n,S})<\epsilon.
 	\]
 \end{theorem}

 The uniform decay in Theorem \ref{11.1} can be improved into exponential decay using Theorem \ref{T2.5}: 
 \begin{theorem}\label{11.2} 
 	For any $C>0$, there exists constants $\zeta(\hy,\bs), M_0(C, \hy,\bs)>0$ with the following significance. For any solution $(A,\Phi)\in \SC_k(\R_t\times \hy)$ to the Seiberg-Witten equations $(\ref{4DSWEQ})$ on $\R_t\times (\hy, \bs)$ with analytic energy $\E_{an}(A,\Phi)<C$, any $n\in \Z$ and $S>0$,
	\[
	\E_{an}(A,\Phi; \Omega_{n,S})<M_0e^{-\zeta S}.
	\]
\end{theorem}

The proof of Theorem \ref{11.1} will dominate the rest of Subsection \ref{Subsec11.1} and it relies on Theorem \ref{T2.4} and \ref{T2.6} in an essential way. Let us first state a lemma in which we set $\Omega_S\colonequals \Omega_{0,S}$.

\begin{lemma} \label{11.3}Let $J=[-3,3]\supset I=[-2,2]$. For any $\epsilon>0$, there exists constants $R_0(\hy, \epsilon), \eta(\hy, \epsilon)>0$ with the following significance.  For any solution $(A,\Phi)$ to the Seiberg-Witten equations $(\ref{4DSWEQ})$  on $J\times (\hy,\bs)$ with $\E_{an}(A,\Phi;J)<\eta$ and any $S>R_0$, we must have
	\[
	\E_{an}(A,\Phi; \Omega_{S})<\epsilon.
	\]
\end{lemma}
\begin{proof} Suppose on the contrary that there exists a sequence $\{(A_n,\Phi_n)\}_{n\geq 1}$ of solutions to the Seiberg-Witten equations (\ref{4DSWEQ}) on $J\times (\hy, \bs)$, a sequence of numbers $\eta_n\to 0$ and $R_n\to\infty$ such that 
	\[
	\E_{an}(A,\Phi; J)<\eta_n \text{ and } \E_{an}(A_n, \Phi_n; \Omega_{R_n})\geq \epsilon.
	\]
By Proposition \ref{Energy10.2} and Lemma \ref{L11.3}, 
\[
\E_{an}(A_n,\Phi_n; J\times [0,\infty)_s)\leq C_2'
\]
for some uniform constant $C_2'>0$. Let $\beta_n=(A_n',\Phi_n')(t,s)=(A_n,\Phi_n)(t, s-R_n)$ be the translated configuration defined on $J\times [-R_n, R_n]\times \Sigma$. Since we have a uniform bound on 
\[
\E_{an}(\beta_n; J\times [-R_n,R_n]),
\]
the classical compactness theorem \cite[Theorem 5.2.1]{Bible} ensures that there is a subsequence of $\{\beta_n\}$ that converges in $\SC_{loc}^\infty$ topology to a solution $\beta_\infty=(A_\infty, \Phi_\infty)$ on $J\times \R_s\times \Sigma$. On the other hand, if we write $\beta_\infty$ as 
\[
(\cgamma(t),c(t))=(B(t),\Psi(t),c(t)),
\]
then Proposition \ref{Energy10.2} implies 
$$\pt \cgamma(t)+\bd_{\cgamma(t)}c(t)=-\grad \CL_{\omega}(\cgamma(t))=0,$$
since $\eta_n\to 0$ as $n\to \infty$.  By making $\beta_\infty$ into temporal gauge (i.e $c(t)\equiv 0$), we conclude that $\cgamma(t)$ is independent of $t\in I$ and solves the 3-dimensional Seiberg-Witten equations \eqref{3DSWEQ} or \eqref{3DDSWEQ}. 

This is the place where the property \ref{P7} is used. By Theorem \ref{T2.6}, up to gauge, $\gamma(t)$ has to be $\R_s$-translation invariant, so
\[
\E_{an}(\beta_\infty; I\times [-3,3])=0.
\]
This contradicts the assumption that $ \E_{an}(A_n, \Phi_n; \Omega_{R_n})\geq \epsilon$ for each $n$. 
\end{proof}

\begin{proof}[Proof of Theorem \ref{11.1}] 
	Suppose on the contrary that there exists a sequence 
	$$\{\beta_m=(A_m,\Phi_m)\}_{m\geq 1}\subset \SC_{k,loc}(\R_t\times \hy)$$ 
	of solutions to the Seiberg-Witten equations (\ref{4DSWEQ}) on $\R_t\times (\hy, \bs)$, a sequence of integers $n_m\geq 0$ and numbers $R_m\to\infty$ such that 
	\[
	\E_{an}(\beta_m)<C \text{ and } \E_{an}(A_m, \Phi_m, \Omega_{n_m,R_m})\geq \epsilon.
	\]
	
	Let $J_n=[n-3,n+3]$ for each $n\in \Z$. For each $m$, define the significant set of $\beta_m$ as 
	\[
	K_m=\{n\in \Z: \E_{an}(\beta_m, J_n)>\eta \},
	\]
	where $\eta=\eta(\epsilon, \hy,\bs)$ is the constant obtained in Lemma \ref{11.3}. Then $n_m\in K_m$. Since there is a uniform upper bound on $\E_{an}(\beta_m,\R_t)$, we know that 
	\[
	|K_m|<C_1\colonequals 6C/\eta. 
	\]
	By passing to a subsequence, we assume $|K_m|$ are the same for all $m$. Place elements of $K_m$ in the increasing order:
	\[
	a_1^{m}<a_2^m<\cdots <a_k^m,\ k=|K_m|.
	\] 
	By passing to a further subsequence, we require that  $\lim_{m\to\infty} |a^m_{i+1}-a^m_i|$ exists (either finite or infinite) for each $1\leq i\leq k$ and it is infinite precisely when $i$ is one of 
	\[
i_0\colonequals -1<i_1<i_2<\cdots <i_l<i_{l+1}\colonequals k.
	\]
	Let $N=\max_{0\leq j\leq l, m\geq 0} |a^m_{i_{j+1}}-a^m_{i_j+1}|$. Now consider the translated configuration 
	\[
	\beta_m'=(A_m',\Phi_m') \text{ with } (A_m',\Phi_m')(t,s)=(A_m,\Phi_m)(t-n_m, s-R_m)
	\]
	defined on $\R_t\times [-R_m, R_m]\times\Sigma$. What we have shown so far implies that 
	\begin{itemize}
\item $\E_{an}(\beta_m', [-N,N]_t\times[-R_m, R_m]_s)$ is bounded above by a constant $C_2$ independent of $\beta_m'$. This follows from energy equations and the assumption that $\E_{an}(\beta_m)\leq C$. 
\item For any $j\in \Z$ with $|j|\geq N$ and any $S\in\R_s$, $\E_{an}(\beta'_m, \Omega_{j,S})<\epsilon$ when $m\gg 1$. Indeed, by the choice of $N$, when $m\gg 1$, $n_m+j\not\in K_m$ and  $R_m\gg R_0-S$. Now apply lemma \ref{11.3}
	\end{itemize}

By the classical compactness theorem \cite[Theorem 5.2.1]{Bible}, up to gauge, a subsequence of $\{\beta_m'\}$ will converge in $\SC_{loc}^\infty$-topology to a solution $\beta_\infty=(A_\infty,\Phi_\infty)$ defined on $\R_t\times \R_s\times \Sigma$. Moreover, we have the following estimates on its analytic energy:
\begin{itemize}
\item For some large constant $M>0$, $\E_{an}(\beta_\infty, \Omega_{j,S})<\epsilon$ whenever $|j|>N$ or $|S|>M$;
\item $\E_{an}(\beta_\infty, [-N,N]_t\times [-M,M]_s)<\infty$;
\item $\E_{an}(\beta_\infty,\Omega_{0,0})\geq \epsilon$.
\end{itemize}

	Now we  draw a contradiction from Theorem \ref{T2.4} which rules out such solutions.  
\end{proof} 

\subsection{Decay of $L^2_k$-norm}\label{Subsec12.2} Having addressed the exponential decay of the local energy functional 
\[
	\E_{an}(A,\Phi; \Omega_{n,S})
\]
in Theorem \ref{11.2}, let us estimate the $L^2_k$-norm of $(A,\Phi)$ over the sub-domain $\Omega_{n,S}$ in terms of $\E_{an}(A,\Phi; \Omega_{n,S})$. Aside from Remark \ref{rmk-10.2}, this is the second reason why the local energy functional is useful. For the sake of simplicity, let us state the results for the compact domain
\[
\Omega_0\subset [-2, 2]_t\times \R_s \subset \C
\]
defined in \eqref{E2.3}. Let $M=\Omega_0\times\Sigma$. Recall that $\gamma_*=(A_*,\Phi_*)$ defined by \eqref{E9.7} is the standard configuration on $\C_z\times\Sigma$. For any smooth $\gamma=(A,\Phi)\in \SC(M)$, set $(a,\phi)=\gamma-\gamma_*$ and consider the gauge fixing condition 
	\begin{equation}\label{E11.5}
\left\{\begin{array}{rl}
\bd_{\gamma_*}^*(a,\phi)&\colonequals -d^*a+i\re\langle \phi, i\Phi_*\rangle=0\\
\langle a,\vn\rangle&=0 \text{ at } \partial M.
\end{array}
\right.
\end{equation}

The proof of Theorem \ref{11.5} requires three additional lemmas, summarized as follows:
\begin{itemize}
\item Lemma \ref{L11.5}: put $\gamma$ into the Coulomb-Neumann gauge slice of $\gamma_*$;
\item Lemma \ref{L11.6}: once $\gamma$ is in the slice, estimate the $L^2_{1,A_*}$-norm of $(a,\phi)=\gamma-\gamma_*$ in terms of $\E_{an}(\gamma; \Omega_0)$;
\item Lemma \ref{L11.8}:  once $\gamma$ is in the slice, estimate the $L^2_{l,A_*}$-norm of $(a,\phi)=\gamma-\gamma_*$ in terms of $\E_{an}(\gamma; \Omega_0)$ for any $l\geq 1$.
\end{itemize}

\begin{lemma}\label{L11.5} There exist constants $\epsilon_0, C_0>0$ with the following significance. For any configuration $\gamma\in \SC(\Omega_0\times\Sigma)$ with
	\begin{equation}\label{E11.6}
\|\gamma-\gamma_*\|_{L^2_{2, A_*}(M)}<\epsilon_0
	\end{equation}
then we can find a smooth function $f: M\to i\R$ such that $e^f\cdot \gamma$ satisfies the Coulomb-Neumann gauge fixing condition \eqref{E11.5}. Moreover, 
\[
\|e^f\cdot \gamma-\gamma_*\|_{L^2_{2, A_*}(M)}\leq C_0\|\gamma-\gamma_*\|_{L^2_{2, A_*}(M)}. 
\]
\end{lemma}
\begin{proof} Let $\K_2$ be the subspace of $\CT_{2,\gamma_*}\colonequals L^2_2(M, iT^*M\oplus S^+)$ subject to the gauge fixing condition \eqref{E11.5}. Consider the non-linear map:
\begin{align*}
U: L_3^2(M;i\R)\times \K_2&\to \CT_{2,\gamma_*} \\
(f, (a,\phi))&=(a-df, (e^{f}-1)\cdot \Phi_*+e^f\cdot\phi).
\end{align*}
The linearized operator $\D_0U$ of $U$ at $(0,(0,0))$ is invertible. Now our lemma  follows from the implicit function theorem.
\end{proof}

Suppose now that $\gamma$ already lies in the Coulomb-Neumann gauge slice of $\gamma_*$. The next step is to estimate $\|(a,\phi)\|_{L^2_{1,A_*}}$ in terms of the local energy functional $\E_{an}(A,\Phi; \Omega_0)$. 

\begin{lemma}\label{L11.6} There exist constants $\epsilon_1, C_1>0$ with the following significance. For any $\gamma$ subject to the gauge fixing condition \eqref{E11.5}, if  $\|(a,\phi)\|_{L^2_{1,A_*}}<\epsilon_1$, then
	\[
\|(a,\phi)\|_{L^2_{1,A_*}}^2\leq C_1\cdot  \E_{an}(\gamma,\Omega_0). 
	\]
\end{lemma}
\begin{proof} Consider the non-linear operator:
	\begin{align*}
	\CF(a,\phi)&=\CF_1+\CF_2 \text{ where}\\
	\CF_1(a,\phi)&=(da, \nabla_{A_*}\phi+a\otimes \Phi_*, (\Phi_*\phi^*+\phi\Phi_*^*)_0, \bd^*_{\gamma_*}(a,\phi)),\\
	\CF_2(a,\phi)&=(0, a\otimes \phi, (\phi\phi^*)_0, 0),
	\end{align*}
	so $\CF_1$ is the linear part of $\CF$ and $\|\CF(a,\phi)\|_{L^2(M)}^2=\E_{an}(\gamma,\Omega_0)$ by Definition \ref{D1.3}. Using the identity
	\[
	|(\Phi_*\phi^*+\phi\Phi_*^*)_0|^2+|\im \langle \phi,\Phi_*\rangle|^2=|\Phi_*|^2|\phi|^2,
	\]
	we calculate that 
	\begin{align*}
	\|\CF_1(a,\phi)\|^2_{L^2(M)}&=\|da\|_2^2+\|d^*a\|_2^2+\|\nabla_{A_*}\phi\|_2^2+\|a\otimes \Phi_*\|_2^2+\||\phi||\Phi_*|\|_2^2+K_3 \text{ where }\\
	K_3&=2\re \int_M \langle \nabla_{A_*}\phi, a\otimes \Phi_*\rangle-\langle\phi,  (d^*a)\Phi_*\rangle\\
	&=2\re\int_M d^*(\langle \phi,\Phi_*\rangle\cdot a )+\langle a\otimes\phi,  \nabla_{A_*}\Phi_*\rangle=0. 
	\end{align*}
	In the last step, we used the facts that $\Phi_*$ is $\nabla_{A_*}$-parallel and $\langle a, \vn\rangle=0 $ at $\partial M$. Hence, 
	\[
	\|\CF_1(a,\phi)\|_{L^2(M)}\geq c_1 \|(a,\phi)\|_{L^2_{1,A_*}},
	\]
	for some $c_1>0$. Finally, 
	\begin{align*}
	\|\CF\|_2&\geq \|\CF_1\|_2-\|\CF_2\|_2\geq  c_1 \|(a,\phi)\|_{L^2_{1,A_*}}-m_3\|(a,\phi)\|_{L^2_{1,A_*}}^2\geq \frac{c_1}{2}\|(a,\phi)\|_{L^2_{1,A_*}}
	\end{align*}
	if $\|(a,\phi)\|_{L^2_{1,A_*}}\leq c_1/2m_3$, where $m_3$ is the constant that appears in the Sobolev embedding $L^2_1\times L^2_1\to L^4$.  
\end{proof}

Now we come to estimate the $L^2_k$-norm of $(a,\phi)$. Consider a closed subset $\Omega_0'\subset \Omega_0$ with a smooth boundary such that 
\[
[-1,1]_t\times [1,3]\subset (\Omega_0')^\circ\subset  \Omega_0'\subset (\Omega_0)^\circ.
\]
\begin{lemma}\label{L11.8} There exist constants $\epsilon_k, C_k>0$ for each $k\geq 1$ with the following significance. For any smooth solution $\gamma\in \SC(M)$ to the Seiberg-Witten equations \eqref{4DSWEQ}, if $\gamma$ is subject to the gauge fixing condition \eqref{E11.5} and $\|(a,\phi)\|_{L^2_{1,A_*}(M)}<\epsilon_k$, then 
	\[
	\|(a,\phi)\|_{L^2_{k,A_*}(\Omega_0'\times \Sigma)}^2\leq C_k\cdot \E_{an}(\gamma, \Omega_0). 
	\]
\end{lemma}
\begin{proof} The case when $k=1$ is settled in Lemma \ref{L11.6}. For $k>1$, this follows from the standard bootstrapping argument \cite[P.107]{Bible}. To illustrate, consider the case when $1<k<2$. Take a cut-off function $\chi_4$ such that 
	\[
	\chi_4\equiv 1 \text{ on } \Omega_0';\ \supp\chi_4\subset (\Omega_0)^\circ. 
	\]
The section $v\colonequals (a,\phi)\in C^\infty(M, iT^*M\oplus S)$ is subject to a non-linear elliptic equation:
\[
Dv+v\# v=0
\]
where $\#$ stands for a certain bilinear form that involves only point-wise multiplication. By G\aa rding's inequality,  for any $0<\eta<1$, 
\begin{align*}
\|\chi_4v\|_{L^2_{1+\eta}(M)}&\leq \|D(\chi_4 v)\|_{L^2_\eta(M)}+\|v\|_2\leq  m_4\|v\|_{L^2_1}+\|(\chi_4v)\#v\|_{L^2_\eta}\\
&\leq m_4 \|v\|_{L^2_1}+m_5\|\chi_4v\|_{L^2_{1+\eta}}\|v\|_{L^2_1}
\end{align*}
If $\|v\|_{L^2_1}<1/(2m_5)$, then we use the rearrangement argument to show that 
\[
\|v\|_{L^2_{1+\eta}(\Omega_0'\times\Sigma)}\leq \|\chi_4v\|_{L^2_{1+\eta}(M)}\leq 2m_4\|v\|_{L^2_1}\leq 2m_4 \sqrt{C_1} \cdot \sqrt{ \E_{an}(\gamma, \Omega_0)},
\]
so we set $\epsilon_{1+\eta}=\min\{\epsilon_1, 1/(2m_5)\}$. In the last step, we used Lemma \ref{L11.6} to estimate $\|v\|_{L^2_{1,A_*}}$ in terms of $\E_{an}(\gamma,\Omega_0)$. When $k\geq 2$, we need more cut-off functions to separate $\Omega_0'$ from $\Omega_0$ and use inductions. In fact, we can take 
\[
\epsilon_k=\min\{\epsilon_1, 1/(2m_5)\}
\]
for any $k>1$. 
\end{proof}

\begin{proof}[Proof of Theorem \ref{11.5}] We divide the proof into three steps. Lemma \ref{L11.5} and \ref{L11.8} will be used only in the last step. In \Step 1 and \Step 2, we arrange so that the assumptions of these lemmas can be satisfied. 
	
	\medskip
	
	\Step 1. By the classical compactness theorem \cite[Theorem 5.2.1]{Bible}, for any $\epsilon>0$, we can find a constant $\eta(\epsilon)>0$ with the following property. Under the assumption of Theorem \ref{11.5}, if $\E_{an}(\gamma, \Omega_0)<\eta(\epsilon)$, then there exists a gauge transformation  $u': \Omega_0\to S^1$ such that 
	\[
\|	u'(\gamma)-\gamma_*\|_{L^2_2(\Omega_0'\times\Sigma)}<\epsilon.
	\]
	
	At this point, we have no controls of the function $\eta: \R_+\to\R_+$. 
	
	\medskip
	
	\Step 2.  We wish to find a gauge transformation $u_1\in \CG_{k+1,loc}(\R_t\times \hy)$ such that 
	\begin{equation}\label{E11.7}
\|u_1(\gamma)-\gamma_0\|_{L^2_{2,A_*} (\Omega_{n,S}\times\Sigma)}	< \min\{\epsilon_0, \frac{\epsilon_l}{C_0}\}. 
	\end{equation}
for any $n\in\Z$ and $S\gg 1$, where $\epsilon_0$ and $\epsilon_l$ are positive constants constructed in Lemma \ref{L11.5} and \ref{L11.8}. \eqref{E11.7} is provided by the uniform $L^\infty$ decay of the local energy functional. Let  $S=m\in \Z_{\geq 0}$ be an integer and apply \Step 1 to the domain
\[
\Omega_{n,m}, \forall n\in \Z, m>R_0(\eta(\epsilon), C),
\]
where $R_0$ is the constant obtained in Theorem \ref{11.1}. We find gauge transformations $u_{n,m}\in \CG^e(\Omega_{n,m}\times\Sigma)$ such that 
\[
\|u_{n,m}(\gamma)-\gamma_0\|_{L^2_2(\Omega_{n,m}'\times\Sigma)}<\epsilon. 
\]
Here $\Omega_{n,m}'$ is the translated domain of $\Omega_0'\subset \Omega_0$:
\[
\Omega_{n,m}'=\{(t,s):(t-n,s-m)\in \Omega_0'\}\subset \Omega_{n,m}.
\]

The collection of domains $\{(\Omega_{n,m}')^\circ \}$ still forms an open cover of $\R_t\times [R_0+1)_s\times \Sigma$. By a patching argument (cf. \cite[Section 13.6]{Bible}), we can find a global gauge transformation $u_1$ such that 
\[
\|u_1(\gamma)-\gamma_0\|_{L^2_1(\Omega_{n,m}\times\Sigma)}<N_1\epsilon. 
\]
for a constant $N_1>0$. Then one may achieve \eqref{E11.7} by starting with $\epsilon$ small enough. 

\medskip

\Step 3. Now apply Lemma \ref{L11.5} to $u_1(\gamma)$ on each $\Omega_{n,m}$ with $m>R_0$. We find some smooth functions $f_{n,m}: \Omega_{n,m}\times \Sigma\to i\R$ such that 
\begin{align*}
\|e^{f_{n,m}}\cdot u_1(\gamma)-\gamma_0\|_{L^2_{1,A_*}(\Omega_{n,m}\times \Sigma)}&\leq \|e^{f_{n,m}}\cdot u_1(\gamma)-\gamma_0\|_{L^2_{2,A_*}(\Omega_{n,m}\times \Sigma)}\\
&\leq C_0  \| u_1(\gamma)-\gamma_0\|_{L^2_{2,A_*}(\Omega_{n,m}\times \Sigma)}\leq \epsilon_l.
\end{align*}
and $e^{f_{n,m}}\cdot u_1(\gamma)$ lies in the Coulomb gauge slice \eqref{E11.5} of $\gamma_*$. Using Lemma \ref{L11.8} and Theorem \ref{11.2}, we estimate the $L^2_{l,A_*}$-norm of the resulting configuration:
\[
\|e^{f_{n,m}}\cdot u_1(\gamma)-\gamma_0\|^2_{L^2_{l,A_*}(\Omega_{n,m}'\times \Sigma)}\leq C_l\cdot \E_{an}(\gamma,\Omega_{n,m})\leq C_l M_0e^{-\zeta m}. 
\]

Finally, using the patching argument once again, we find a global gauge transformation $u\in \CG_{k+1,loc}(\R_t\times \hy)$ such that 
\[
\|u(\gamma)-\gamma_0\|^2_{L^2_{l,A_*}(\Omega_{n,m}\times \Sigma)}\leq N_2C_lM_0e^{-\zeta m}. 
\]
for a constant $N_2>0$. This completes the proof of Theorem \ref{11.5}.
\end{proof}

\part{Perturbations}\label{Part4}

In order to make the moduli spaces on $\R_t\times \hy$ smooth and define the Floer homology of the 3-manifold $(Y,\partial Y=\Sigma)$, a suitable perturbation $\CSd_\omega=\CL_\omega+f$ of the Chern-Simons-Dirac functional $\CL_\omega$ is needed. We follow the construction of tame perturbations in \cite[Section 10-11]{Bible}. However, there is one distinct feature of our situation, which requires some technical tricks to deal with: 
\begin{enumerate}[label=($\star$)]
\item\label{compactrequirement} We want the perturbation supported within \textbf{a compact region} of $\hy$ so that the Seiberg-Witten equations (\ref{4DSWEQ}) defined on $\R_t\times\hy$ remains unperturbed on the planar end $\HH^2_+\times \Sigma$, and Theorem \ref{T2.5} is applicable. 
\end{enumerate}

Hence, the error term $f$ must factorize through the restriction map to the truncated manifold $Y_n\colonequals\{s\leq n\}\subset \hy$ for some $n\geq 0$:
\[
\SC_{k-1/2}(\hy,\bs)\to \SC_{k-1/2}(Y_n,\bs). 
\]

As a result, the perturbation space is not large enough to separate all tangent vectors and points of $\SC_{k-1/2}(\hy,\bs)$ as in \cite[Proposition 11.2.1]{Bible}. Nevertheless, we can still achieve the transversality of moduli spaces on $\R_t\times \hy$, even with this smaller perturbation space. In fact, one may even require that $n=0$, so $Y_n=Y=\{s\leq 0\}$.

\smallskip

Part \ref{Part4} is organized as follows. In Section \ref{S14}, we introduce the so-called tame perturbations (Definition \ref{D14.2}) and state the formal mapping properties that they enjoy. 

In Section \ref{S15}, we take up the task to construct tame perturbations. The separation properties are examined carefully in Subsection \ref{Subsec15.2}. The Banach space $\Pa$ of tame perturbations is constructed in Subsection \ref{Subsec15.5}.

Section \ref{Sec15} is devoted to the compactness theorems for perturbed Seiberg-Witten equations. Since tame perturbations are made compactly supported, the proofs in Section \ref{Sec11} apply verbatim to this case. 


\section{Abstract Perturbations}\label{S14}

The perturbation that we deal with is a continuous section $(k>1)$ 
\[
\q: \SC_{k-\half}(\hy,\bs)\to \CT_0
\]
where $\CT_0$ is the $L^2$-completion of the tangent bundle $T\SC_{k-1/2}(\hy,\bs)$ introduced in Section \ref{Sec10}. The perturbation $\q$ is required to be the formal gradient of a $\CG_{k+1/2}(\hy)$-invariant continuous function $f: \SC_{k-1/2}(\hy,\bs)\to\R$, and we write $\q=\grad f$. This means that
\[
f(\cgamma(1))-f(\cgamma(0))=\int_{0}^{1} \langle \dot{\cgamma}, \q(\cgamma(t))\rangle_{L^2}dt 
\]
for any smooth path $\cgamma:[0,1]\to \SC_{k-1/2}(\hy,\bs)$. Take 
\[
\CSd_\omega=\CL_\omega+f
\]
to be the perturbed Chern-Simons-Dirac functional. Let $I=[t_1, t_2]$ and $\hz$ be the product \spinc manifold $I\times (\hy,\bs)$. The  down-ward gradient flowline equation of $\CSd_\omega$ becomes
\begin{align}\label{E15.3}
\dt\cgamma(t)&=-\grad \CSd_{\omega}(\cgamma(t))-\bd_{\cgamma(t)} c(t)\\
&=-\grad \CL_{\omega}(\cgamma(t))-\bd_{\cgamma(t)} c(t)-\q(\cgamma(t)),\nonumber
\end{align}
where $\cgamma(t)=(B(t),\Psi(t))$ is a underlying path in $\SC_{k-1/2}(\hy,\bs)$ and 
\begin{equation}\label{E14.1}
A=\dt+B(t)+c(t)dt\otimes\Id_S,\ \Phi|_{\{t\}\times Y}=\Psi(t)
\end{equation}
is the corresponding 4-dimensional configuration $\gamma=(A,\Phi)$ in $\SC(\hz)$. In this way, the continuous section $\q$ extends to a section of the trivial bundle $\V_0$ over $\SC(\hz)$:
\begin{equation}\label{E14.2}
\hq: \SC(\hz)\to \V_0=L^2(\hz, i\su(S^+)\oplus S^-)\times  \SC(\hz)
\end{equation}
by sending $\gamma=(A,\Phi)$ to $\q(\cgamma(t))$ at each time slice $t\in I$. Here we use the 3-dimensional Clifford multiplication $\rho_3$ to identify the bundle $iT^*\hy$ with $ i\su(S^+)$ over $\hz$. We wish that this section $\hq$ extends to a smooth section of $\V_k\to  \SC_k(\hz)$ for any $k\geq 2$, so \eqref{E15.3} is cast into the perturbed Seiberg-Witten equation $\F_{\hz, \q}=0$ where
\begin{align*}
\F_{\hz, \q}&\colonequals \F_{\hz}+\hq:  \SC_k(\hz)\to \V_{k-1},
\end{align*}
and $\F_{\hz}$ is defined as in \eqref{4DSWEQ}.

We do not have a canonical $L^2_j$ norm on the space $\Gamma(\hz, i\su(S^+)\oplus S^-)$. For each $\gamma=(A,\Phi)\in \SC_k(\hz)$, we define a norm at the fiber $\V_j|_\gamma$ using $A$ as the covariant derivative, i.e.
\[
\|v\|_{L^2_{j, A}}^2\colonequals \sum_{n=0}^j \|\nabla_A^n v\|^2
\]
for any $v\in \V_j|_\gamma$. This family of norms on $\V_j$ is equivariant under the gauge action of $\CG_{k+1}(\hz)$. Similarly, we define the $L^2_{j,A}$ norm on $\CT_j\to \SC_k(\hz)$. Then the $l$-th derivative of $\hq$ at $\gamma$ is a bounded multi-linear map:
\begin{align*}
\D_\gamma^l\hq&\in \Mult^l\big({\bigtimes}_l L^2_{k,A}(\hz,iT^*\hz\otimes S^+), L^2_{k,A}(i\su(S^+)\oplus S^-)\big)\\
&=\Mult^l({\bigtimes}_l \CT_k, \V_{k}). 
\end{align*}

The bundle map $\D_\gamma^l\q$ might not be a local operator: it does not necessarily send compactly supported sections on $\hy$ to another section with the same or smaller support. However, this is a property enjoyed by derivatives $\D^l_\gamma \F_{\hz}$ of the unperturbed Seiberg-Witten map $\F_{\hz}$, which motivates the next definition:

\begin{definition} \label{D14.1}For any closed subset $\Omega\subset \hy$, a perturbation $\q$ is said to be supported on $\Omega$ if $\supp\ \q(\cgamma)\subset \Omega$ for any $\cgamma\in \SC_{k-1/2}(\hy,\bs)$ and 
	\[
	\q(\cgamma_1)=\q(\cgamma_2)
	\]
	for any configurations $\cgamma_1, \cgamma_2\in \SC_{k-1/2}(\hy,\bs)$ such that $\cgamma_1=\cgamma_2$ on $\Omega$.
\end{definition}

We are primarily interested in the case when $\Omega=Y_n=\{s\leq n\}$ for some $n\geq 0$. It turns out that the choice of the integer $n$ is inconsequential for the Floer homology, so we may safely set $n=0$ and focus on the case when $\Omega=Y$. 

\begin{remark} One may even take $\Omega=[0,1]_s\times\Sigma\subset\hy$ and  the construction in Section \ref{S15} would be simplified if one uses the gauge fixing condition along each fiber $\{s\}\times\Sigma$. 
\end{remark}

For technical reasons, we also need completions of bundles and the configuration space with respect to other Sobolev norms $L^p_k$ with $p\neq 2$. Let 
\[
\SC_k^{(p)},\ \CT_k^{(p)},\ \V_k^{(p)}
\]
be the resulting space and bundles when $k\geq 1$ and $1\leq p\leq\infty$. Note that $\SC_k^{(2)}(\hz)=\SC_k(\hz)$ and so on. 

Let us state the constraints on the perturbation $\q=\grad f$.

\begin{definition}\label{D14.2} Let $Y'$ be a smooth co-dimension $0$ submanifold of $\hy$ with possibly non-empty boundary. We usually take $Y'$ to be either $Y=\{s\leq 0\}$ or $\hy$. For each integer $k\geq 2$, a perturbation $\q$ given as a section 
	\[
	\q: \SC(\hy,\bs)\to \CT_0.
	\]
	is called \textbf{$k$-tame in $Y'$} if it is the formal gradient of a continuous $\CG(\hy)$-invariant function $f$ on $\SC(\hy)$ such that
	\begin{enumerate}[label=(A\arabic*)]
\item\label{A1} the corresponding 4-dimensional perturbation $\hq$ defines an element:
\[
\hq\in C^\infty(\SC_j(\hz), \V_j)
\]
for any integer $j\in [2, k]$;

\item\label{A2} When $p>3$, $\hq$ also defines an element in
\[
 C^\infty(\SC_j^{(p)}(\hz),\V_j^{(p)})
\]
for any integer $j\in [1,k]$;

\item\label{A3} $\hq$ extends to a continuous map:
\[
\SC_1(\hz)\to \V^{(m)}_0
\]
for any $2\leq m<4$. 

\item\label{A4} for each integer $j\in [-k,k]$, the first derivative
\[
\D \hq\in C^\infty(\SC_k(\hz), \Hom(T\SC_k(\hz), \V_k))
\]
extends to a smooth map 
\[
\D \hq\in C^\infty(\SC_k(\hz), \Hom(\CT_j, \V_j));
\]
\item\label{AA5} for any $(B,\Psi)\in \SC_k(\hy)$, the $L^2_k$-section $\q(B,\Psi)$  is supported on $Y'$: 
\[
\supp\ \q(B,\Psi)\subset Y'.
\]
Moreover, there exists a constant $m_2>0$ such that 
\[
\|\q(B,\Psi)\|_{L^2(Y')}\leq m_2(\|\Psi\|_{L^2(Y')}+1),
\]
for any $(B,\Psi)\in \SC_k(\hy)$. 

\item\label{AA6} For any $0\leq\epsilon< \half$, $\hq$ extends to a continuous map 
\[
\SC_{1-\epsilon}(\hz)\to \V_0.
\]

\item\label{A7} the 3-dimensional perturbation $\q$ defines a $C^1$-section 
\[
\q:\SC_1(\hy)\to \CT_0. 
\]

	\end{enumerate} 
We simply say that $\q$ is tame in $Y'$ if $\q$ is $k$-tame in $Y'$ for any $k\geq 2$. We may not mention the support $Y'$ when $Y'=\hy$. 
\end{definition}

\begin{remark} When $Y'=\hy$, Definition \ref{D14.2} agrees with \cite[Definition 10.5.1]{Bible}, with some minor changes in properties \ref{A2}\ref{A3}\ref{AA5}\ref{AA6}. Our construction of tame perturbations in Section \ref{S15} ends up with weaker mapping properties, in exchange for having them compactly supported.
\end{remark}

\begin{remark} Let us briefly explain where these properties will be used:
	\begin{itemize}
\item \ref{A1}\ref{A2}\ref{A3}\ref{AA6} will be used in the compactness theorem for the perturbed Seiberg-Witten equations, i.e. Theorem \ref{T1.4}. They give intermediate steps in the bootstrapping arguments;
\item \ref{AA5} is used in the energy equation for the perturbed Seiberg-Witten equations, i.e. Proposition \ref{P1.1};
\item \ref{A4} is relevant with the linear theory in Part \ref{Part5};
\item \ref{A7} will be used in the proof of the exponential decay result in time direction, which we will not actually work out in this paper, cf. \cite[Section 13.4]{Bible}, in particular \cite[Lemma 13.4.3]{Bible}.\qedhere
	\end{itemize}
\end{remark}

\section{Constructing Tame Perturbations}\label{S15}

\subsection{Cylinder Functions}\label{S15.1} The construction of cylinder functions in the book \cite[Section 11]{Bible} involves a global gauge slice, which prevents perturbations being local. Instead, we adopt a variation that is reminiscent of the holonomy perturbations in instanton Floer homology to achieve our goal. 

First, we fix a smooth embedding of $S^1\times D^2$ into $\hy$, where $D^2=B(0,1)\subset \R^2$ is the unit disk:
\[
\iota: S^1\times D^2\to \hy.
\]
To find such an $\iota$, one may first embed the core $S^1\times\{0\}$ into $\hy$ and extend this map to a tubular neighborhood of the image. We pull back the metric and the spin bundle $S\to \hy$ via $\iota$. The induced Riemannian metric $g_1\colonequals \iota^*g_Y$ might not agree with the product metric 
\[
g_{std}\colonequals\iota^*g_{\hy}|_{S^1\times\{0\}}+  g_{D^2},
\]
on $S^1\times D^2$, where $g_{D^2}$ is the standard Euclidean metric of $D^2$. They are related by a smooth symmetric bundle map $K: T^*(S^1\times D^2)\to T^*(S^1\times D^2)$ (with respect to $g_{std}$) such that 
\[
\langle b_1, b_2\rangle_1=\langle K(b_1), b_2\rangle_{std}.
\]
for any co-vectors $b_1$ and $b_2$. The volume forms of $g_1$ and $g_{std}$ differ by a smooth positive function $\eta>0$:
\[
dvol_{1}=\eta\cdot dvol_{std}.
\]
It is only important to know that $K$ and $\eta$ are smooth; the Clifford multiplication $\rho_3$ is never needed for the purpose of perturbations.

Let $(B_0,\Psi_0)$ be the reference configuration in $\SC_k(Y)$. For any $(B,\Psi)\in \SC_k(\hy)$, take the difference
\[
(b,\psi)\colonequals (B,\Psi)-(B_0,\Psi_0)\in L^2_k(\hy, iT^*\hy\oplus S). 
\]

There are three classes of perturbations to be considered. The first two concern the imaginary valued 1-form $b$. The last one deals with the spin section $\Psi$. 
\begin{enumerate}[label=(B\arabic*)]
\item\label{B1} For any compactly supported 1-form $c\in \Omega^1_c(S^1\times D^2, i\R)$, define 
\begin{align*}
r_c: \SC_k(\hy)&\to \R \\
(b,\psi)&\mapsto \int_{S^1\times D^2} b\wedge d\bar{c}\\
&=\int_{S^1\times D^2} \langle b, *_1dc\rangle_{g_1} dvol_1=\int_{S^1\times D^2} \langle b, *_{std}dc\rangle_{g_{std}}dvol_{std}, 
\end{align*}
where $*_1$ and $*_{std}$ stand for the Hodge star operators of $g_1$ and $g_{std}$ respectively. The formal gradient of $r_c$ is 
\[
\grad r_c=*_1dc,
\]
while using $g_{std}$ we obtain
\[
\grad_{std} r_c\colonequals *_{std}dc=\eta K(\grad r_c).
\]

\item\label{B2} Fix a compactly supported 2-form $\nu\in \Omega_c^1(D^2, i\R)$ on the disk $D^2$ with
\[
\int_{D^2} \nu=i,
\]
and define 
\begin{align*}
r_\nu: \SC_k(\hy)&\to \R \\
(b,\psi)&\mapsto \int_{S^1\times D^2} b\wedge \pi^*\overline{\nu},
\end{align*}
where $\pi: S^1\times D^2\to D^2$ is the projection map. Unlike $r_c$, $r_\nu$ is not fully gauge-invariant. For any $u\in \CG_{k+1}(\hy)$, 
\[
r_\nu(u(b,\psi))-r_\nu(b,\psi)=-2\pi \deg (u\circ \iota: S^1\times \{0\}\to S^1)\in 2\pi \Z. 
\]

Hence, $r_\nu$ descends to a circle valued function
\[
[r_\nu]: \SC_k(\hy)\to \R/(2\pi\alpha\Z)
\]
where $\alpha\in \Z_{\geq 0}$ is the multiplicity of $\iota_*([S^1\times\{0\}])$ in $H_1(Y, \Sigma; \Z)$, i.e  $\iota_*([S^1\times\{0\}])$ is $\alpha$ times a primitive class in $H_1(Y, \Sigma; \Z)$. Using the Euclidean metric of $D^2$, one may conveniently set 
\[
\nu=i\chi_2(z)dvol_{D^2}
\]
where $\chi_2$ is a cut-off function on $D^2$ with $\chi_2(z)\equiv 1$ when $|z|\leq \half $.

\item\label{B3} Fix a gauge transformation $u_1: \hy \to S^1$ with the following properties:
\begin{itemize}
\item $u_1$ is smooth on $\hy$;
\item The composition $u_1\circ \iota: S^1\times\{0\}\to S^1$ is harmonic and has degree $\alpha$. 
\item $u_1\circ \iota: S^1\times D^2\to S^1$ is constant in $D^2$.  
\end{itemize}

Let the transformation $u_1$ act on the bundle $\R_x\times S\to \R_x\times (S^1\times D^2)$ by the formula:
\[
u^n_1(x, \Phi)\mapsto  (x-2\pi n\alpha, u^n_1\Phi).
\]

Passing to the quotient space, we obtain a bundle $\Sph$ over $(\R/2\pi\alpha\Z)\times (S^1\times D^2)$. If $\Upsilon$ is a compactly supported smooth section of $\Sph$, let $\tup$ denote its lift as a section of $\R_x\times S\to \R_x\times (S^1\times D^2)$. Then $\tup$ is an equivariant section, as 
\[
\tup(x-2\pi n\alpha, \theta,z)=u^{n}_1 \tup(x,\theta,z)
\]
for any $(\theta,z)\in S^1\times D^2$ and $x\in \R_x$. Let $b_z=b|_{S^1\times \{z\}}$ be the restriction of the 1-form $b$ over the $S^1$-fiber at $z\in D^2$. Using the product metric $g_{std}$, we write
\[
b_z=b^1_z+b^h_z
\]
in terms of the Hodge decomposition along each fiber $S^1\times\{z\}$ with 
\[
b^1_z \text{ exact and } b^h_z \text{ harmonic } (\text{the coexact part } b^2_z=0).
\]
 Let $d_{S^1}^*$ be the adjoint of the exterior differential $d_{S^1}$ over $S^1\times \{0\}$ and 
\[
G: C^\infty (S^1, i\R)\to C^\infty (S^1, i\R)
\]
be the Green operator. Then the exact part $b_z^1$ can be explicitly written as 
\[
b_z^1=d_{S^1} Gd_{S^1}^*b_z,
\]
and $b^h_z$ stands for the harmonic part of $b_z$. It is tempting to form the map:
\begin{align*}
\upd: \SC(\hy)&\to C^\infty (S^1\times D^2, S)\\
(b,\psi)&\mapsto e^{-Gd^*_{S^1}b_z} \tup(r_\nu(b), \theta,z) \text{ on } S^1\times\{z\},
\end{align*}
which is \textbf{equivariant} under the action of $u^n_1$. However, $\upd$ is \textbf{not} equivariant under the action of the full gauge group $\CG(\hy)$ (compare \cite[P.173]{Bible}). In fact, $\upd$ is invariant under $\Map(D^2,S^1)$, the space of gauge transformations that are constant along each fiber $S^1\times \{z\}$. 

To circumvent this problem, let $\Psi_z$ and $\upd_z$ be the restriction of $\Psi$ and $\upd$ along the fiber $S^1\times \{z\}$ for any $z\in D^2$. Fix an $S^1$-invariant function $h: \C_w\to \R$. For instance, set 
\[
h(w)=\chi_3(|w|^2), \forall w\in \C,
\]
for some cut-off function $\chi_3:\R\to \R_{\geq 0}$ such that
\[
\chi_3(t)\equiv 1 \text{ if } t\leq 1;\ \chi_3(t)\equiv 0 \text{ if } t\geq 2.
\]
 Then the composition $ h(\sigma(z)): \SC(\hy)\to \R$ is fully gauge invariant, where
\begin{align*}\sigma(z)\colonequals \int_{S^1\times \{z\}}\langle \Psi_z, \upd_z\rangle.
\end{align*}
Finally, define 
\[
q_{\Upsilon}(b,\psi)=\int_{D^2} h(\sigma(z))\chi_2(z) dvol_{D^2},
\]
where $\chi_2$ is the cut-off function on $D^2$ defined in \ref{B2}. 
\end{enumerate}

By choosing a finite collection of 1-forms $c_1,\cdots, c_n$ and smooth sections $\Upsilon_1, \cdots, \Upsilon_m$ of $\Sph$, we obtain a map 
\[
\Xi=(r_{c_1}, \cdots, r_{c_n}, [r_{\nu}], q_{\Upsilon_1},\cdots, q_{\Upsilon_m}): \SC(\hy)\to \R^n\times (\R/2\pi\alpha\Z)\times \R^m.
\]

\begin{definition}\label{D15.1}A function $f$ defined on $\SC(\hy)$ is called \textit{a cylinder function} if it arises as the composition $g\circ \Xi$ where
\begin{itemize}
\item the map $\Xi: \SC(\hy)\to \R^n\times (\R/2\pi\alpha\Z)\times \R^m$ is defined as above, using any compactly supported forms $c_i\ (1\leq i\leq n)$ defined on $S^1\times D^2$ and compactly supported sections $ \Upsilon_j\ (1\leq j\leq m)$ on $(\R/2\pi\alpha\Z)\times (S^1\times D^2)$, for any $n,m\geq 0$;
\item the function $$g: \R^n\times (\R/2\pi\alpha\Z)\times \R^m\to \R$$ is any smooth function with compact support. 
\end{itemize}
A cylindrical function is fully gauge invariant. 
\end{definition}

\begin{theorem}\label{T15.2} For any cylinder function $f: \SC(\hy)\to \R$, its formal gradient 
	\[
	\grad f: \SC(\hy)\to \CT_0
	\]
	is a perturbation tame in $Y'=\im \iota$, in the sense of Definition \ref{D14.2}, where $\iota: S^1\times D^2\embed \hy$ is the embedding used to define $f$.  
\end{theorem}

We will prove Theorem \ref{T15.2} in Subsection \ref{Sub15.4}. 

\subsection{Cylinder Functions and Embeddings}\label{Subsec15.2} In this subsection, we examine the separating property of cylinder functions. The main results are Proposition \ref{P15.4} and \ref{P15.6}.

\medskip

 Fix an embedding $\iota:S^1\times D^2\embed \hy$, and define
\[
\Cylin(\iota)\colonequals \{f: f\text{ is a cylinder function defined via }\iota\}.
\]
 It is reasonable to ask: to what extend elements of $\Cylin(\iota)$ separate points and tangent vectors of $\SC(\hy)$. Apparently, if $(B_1,\Psi_1)$ is identical to $(B_2,\Psi_2)$ over the image of $\iota$ up to gauge, then they can not be separated by any element of $\Cylin(\iota)$, because only local information is employed when defining cylinder functions. In addition, they can not be separated if $B_1=B_2$ and 
 \[
e^{i\theta(z)}\Psi_1= \Psi_2
 \]
 for some smooth function $\theta: D^2\to \R$ as the function $h(\sigma(z))$ defined in \ref{B3} is fully gauge invariant. In fact, this is the worst case that can happen:
 
\begin{proposition}\label{P15.3} Take $\gamma_i=(B_i,\Psi_i)\in \SC(\hy)\ (i=1,2)$. Suppose for any cylinder function $f\in \Cylin(\iota)$, we always have 
	\[
	f(\gamma_1)=f(\gamma_2),
	\]
	then there exists a gauge transformation $v\in  \CG(\hy)$ and some function $\theta: B(0,1/3)\to\R$ such that 
	\[
	v(B_1)=B_2,\ e^{i\theta(z)}v\cdot\Psi_1=\Psi_2
	\] 
	over the smaller solid torus $\iota(S^1\times B(0,1/3))$. The function $\theta$ might not be continuous because of the zero locus of $\Psi_1$.
\end{proposition}

\begin{proof} Take $(b_i,\psi_i)=(B_i, \Psi_i)-(B_0,\Psi_0)$ and set 
	\[
	\vb=b_2-b_1. 
	\]
By our assumptions, $\gamma_1$ and $\gamma_2$ can not be separated by any functions of classes \ref{B1}\ref{B2} and \ref{B3}. First, we claim that $\delta b$ is closed on $S^1\times D^2$, since
	\[
	0=r_c(b_2)-r_c(b_1)=r_c(\delta b)=\int_{S^1\times D^2} \delta b\wedge d\bar{c}=\int_{S^1\times D^2} d(\vb)\wedge \bar{c}
	\]
	for any compactly supported 1-form $c$. Moreover, 
	\[
	r_\nu(\delta b)=r_\nu(b_2)-r_\nu(b_1)=2\pi n \alpha\in\R
	\]
	 for some $n\in \Z$,	since $[r_\nu](b_1)=[r_\nu](b_2)$. Using the gauge transformation $u_1$ from \ref{B3}, we may place $\gamma_1$ by 
	 \[
	 u_1^{-n}(\gamma_1)
	 \]
	 to make $r_\nu(b_2)-r_\nu(b_1)$ zero. From now on, let us assume $r_\nu(\delta b)=0$. 
	 
	 This allows us to conclude that $\delta b$ is exact on $S^1\times D^2$, so $\delta b=d\xi $ for some function $\xi: S^1\times D^2\to i\R$. By cutting off $\xi$ outside $B(0,2/3)$, we extend $\xi$ to the whole manifold $\hy$ (by zero outside of $\im \iota$). Finally, replace $\gamma_1$ by $e^{-\xi}\cdot\gamma_1$. 
	 
	 It remains to show that $\Psi_1=\Psi_2$ along the core $S^1\times\{0\}$ up to an overall phase $e^{i\theta}\in S^1$ when $\delta b=0$ on $S^1\times B(0,1/2)$. Let 
	\[
	\Psi_{1,0},\Psi_{2,0}
	\]
	be their restriction along the core $S^1\times\{0\}$. If they do not generate the same complex plane in $\Gamma(S^1\times\{0\}, S)$, then we can always find a section $\Upsilon_0\in \Gamma(S^1\times\{0\}, S)$ such that 
	\[
		\Psi_{1,0}\perp \Upsilon_0 \text{ and } \Psi_{2,0}\not\perp \Upsilon_0
	\]
	or the other way around. Extending $\Upsilon_0$ to a section $\Upsilon$ of \[
	\Sph\to (\R/2\pi \alpha\Z)\times S^1\times D^2,
	\]
	supported near $\{r_\nu(b_1)\}\times S^1\times\{0\}$ will result in a function $q_{\Upsilon}$ of class \ref{B3} that separates $\gamma_1$ and $\gamma_2$. 
	
	When $\Psi_{1,0}$ and $\Psi_{2,0}$ do generate the same complex plane, but $\|\Psi_{1,0}\|_{L^2(S^1)}\neq \|\Psi_{2,0}\|_{L^2(S^1)}$, one can construct $\Upsilon$ in a similar way. 
	 
	We obtain the function $\theta: B(0,1/3)\to \R$, by applying the same argument to the fiber $S^1\times \{z\}$ for any $z\in B(0,1/3)$.
\end{proof}

Hence, it is necessary to take into account all possible embeddings of $S^1\times D^2$ into $\hy$ in order to obtain the desired separating property: 
\begin{proposition}\label{P15.4} Recall that $Y=\{s\leq 0\}\subset \hy$. Let
	\[
	\Cylin(Y)\colonequals \bigcup_{\im \iota\subset Y} \Cylin(\iota)
	\]	
	be the union of all possible cylinder functions with $\im \iota\subset Y$. If $\gamma_1$ and $\gamma_2\in \SC(\hy)$ can not be separated by any element in $\Cylin(Y)$, then there is a gauge transformation $v\in \CG(\hy)$  that identifies $\gamma_1$ with $\gamma_2$ over $Y$, i.e.
	$$v(\gamma_1)=\gamma_2 \text{ on } Y.$$ 
\end{proposition}

\begin{proof} Again, take $(b_i,\psi_i)=(B_i, \Psi_i)-(B_0,\Psi_0)$ and set 
	\[
	\vb=b_2-b_1. 
	\]
	By the proof of Proposition \ref{P15.3}, we deduce that $\vb$ is closed over $Y$, and there is a gauge transformation $v\in \CG(\hy)$ such that $v(B_1)=B_2$. The remaining step is to verify
	\[
	v\cdot\Phi_1=\Phi_2
	\]
	up to a global constant $e^{i\theta}\in S^1$. By Proposition \ref{P15.3}, the equality $|\Phi_1|=|\Phi_2|$ holds point-wise on $Y$, and 
	\[
	e^{i\theta(y)}v\cdot\Phi_1=\Phi_2
	\]
	for some function $\theta: Y^\circ \to \R$ defined in the interior of $Y$. Suppose for some $y_1,y_2\in Y^\circ $, $\Phi_1(y_1),\Phi_1(y_2)\neq 0$. Choose an embedding $S^1\times\{0\}\embed Y$ that passes $y_1,\ y_2$ and extend it into an embedding of the solid torus:
	\[
	\iota: S^1\times D^2\to Y\subset \hy. 
	\]
	By Proposition \ref{P15.3}, the function $e^{i\theta}$ has to be constant along the core $S^1\times\{0\}$, so $e^{i\theta(y_1)}=e^{i\theta(y_2)}$. This allows us to modify $\theta$ to be a constant function $\theta \equiv \theta_0$, so 
	\[
	e^{i\theta_0}v\cdot\Phi_1=\Phi_2.\qedhere
	\]
\end{proof}
 
 Now we state the infinitesimal version of Proposition \ref{P15.3} and \ref{P15.4} concerning the separating property of tangent vectors. They are essential for the proof of transversality in Section \ref{Sec21}. Proposition \ref{P15.6} is a direct consequence of Proposition \ref{P15.5}, so we focus on the proof of the latter.
 
 \begin{proposition}\label{P15.5} Take $\gamma=(B,\Psi)\in \SC(\hy)$ and $V=(\delta b,\delta\psi)\in T_{\gamma}\SC(\hy)$. For a fixed embedding $\iota: S^1\times D^2\embed \hy$ and any $f\in \Cylin(\iota)$, suppose we always have 
 	\[
df(V )=0,
 	\]
 	then either 
 	\begin{itemize}
\item 
there exists some $\xi\in  \Lie(\CG(\hy))$ and some function $\theta: B(0,1/3)\to\R$ such that 
\[
(\delta b, \delta\psi)=(-d\xi, (\xi+i\theta(z))\Psi)
\] 
over the smaller solid torus $\iota(S^1\times B(0,1/3))$; or 
\item  $\Psi\equiv 0$ on $S^1\times\{z\}$ for some $z\in B(0,1/3)$.
 	\end{itemize}

 \end{proposition}
\begin{proposition}\label{P15.6} Suppose for some $\gamma=(B,\Psi)\in \SC(\hy)$ and some tangent vector $V\in T_\gamma \SC(\hy)$, we always have 
	\[
	df(V)=0
	\]
	for any $f\in \Cylin(Y)$. Then either 
	\begin{itemize}
\item $\Psi\equiv 0$ on $Y$, or 
\item 	for some $\xi\in \Lie(\CG(\hy))$, $V$ is generated by the infinitesimal action of $\xi$ over $Y$, i.e.
$$V=(-d\xi, \xi\Psi) \text{ on } Y.$$
	\end{itemize}
\end{proposition}

\begin{proof}[Proof of Proposition \ref{P15.5}] Since $V=(\vb,\vpsi)$ can not be separated by any functions in classes \ref{B1}\ref{B2}, $\vb$ has to be an exact 1-form on $S^1\times D^2$, so $\vb=-d\xi$ for some $\xi: S^1\times D^2\to i\R$. Since this problem is linear and the vector $(-d\xi, \xi\Psi)$ can not be separated, it remains to deal with the case when $\vb=0$ and show
	\[
	\vpsi=i\theta(z)\Psi
	\]
on $S^1\times B(0,1/3)$	for some function $\theta: B(0,1/3)\to \R$. For a fixed section $\Upsilon$ of $\Sph$, consider functions  $\sigma, \sigma_1: D^2\to \C$:
 \begin{align*}
\sigma(z)\colonequals \int_{S^1\times \{z\}}\langle \Psi_z, \upd_z\rangle,\ \sigma_1(z)\colonequals \int_{S^1\times \{z\}}\langle \vpsi_z, \upd_z\rangle.
 \end{align*}
 Then the differential of $q_\Upsilon$ along $V=(0,\delta\psi)$ can be computed directly as 
\[
dq_\Upsilon(0,\delta\psi)=\int_{D^2} 2\chi_2(z)\chi_3'(|\sigma_1|^2) \re(\sigma(z)\overline{\sigma_1}(z)) dvol_{D^2},
\]
where $\chi_3$ is the cut-off function used to define the $S^1$-invariant function $h$ in \ref{B3}. For any $z\in B(0,1/3)$, if $\Psi_z$ and $\vpsi_z$ do not lie in the same complex direction in $\Gamma(S^1\times\{z\}, S)$, then for some section $\upd_z\in \Gamma(\{r_\nu(b)\}\times S^1\times \{z\}, S)$, $\re(\sigma(z)\overline{\sigma_1}(z))$ is non-zero (it suffices to verify this statement for two vectors in $\C^2$). By properly extending $\upd_z$ to a section $\Upsilon$ of $\Sph$, we can make $dq_\Upsilon(0,\delta\psi)\neq 0$. 

Finally, if $\Psi_z\not\equiv 0$ and $\vpsi_z=w\Psi_z$ for some $w\in \C$, then $w$ has to be imaginary for the same reason. This proves the existence of $\theta(z)\in \R$ when $\Psi_z\not\equiv 0$. 
\end{proof}

\subsection{Estimates of Perturbations on Cylinders} In this subsection, we take up the proof of Theorem \ref{T15.2}.  Unlike the case of closed 3-manifolds (cf. \cite[Section 11.3]{Bible}),  gradients and Hessians of $f$ can not be estimated in a straightforward way; the use of anisotropic Sobolev spaces is already necessary. We will only state the estimates for the 3-manifold $\hy$, whose proof will follow from their analogue on the 4-manifold $[t_1,t_2]\times\hy$:
\begin{proposition}[cf. Proposition 11.3.3 in \cite{Bible}]\label{P15.7} For any $k\geq 2$ and any cylinder function $f$ defined using an embedding $\iota: S^1\times D^2\to \hy$, $\q=\grad f$ determines a smooth vector field on $\SC_k(\hy)$, and for each $l\geq 0$, there is a constant $C$ with 
	\[
	\|\D^l_{(B,\Psi)}\q\|\leq C (1+\|b\|_{L^2_{k-1}(Y')})^{2k(l+1)}(1+\|\Psi\|_{L^2_{k,B}(Y')})^{l+1},
	\]
	where $\D^l_{(B,\Psi)}\q$ is viewed as an element of $\Mult_l(\bigtimes_l\CT_k, \CT_k)$ and $Y'=\im \iota$. 
	
	In addition, for any $j\in[-k,k]$, the first derivative $\D\q$ extends to a smooth map 
	\[
	\D\q:\SC_k(\hy)\to \Hom(\CT_j,\CT_j)
	\]
	whose $(l-1)$-th derivative viewed as an element of $\Mult_l(\bigtimes_{l-1}\CT_k\times \CT_j, \CT_j)$ satisfies the same bound. 
\end{proposition}

\begin{remark}
The author was unable to prove this proposition when $k=1$. We will come back to this point in Subsection \ref{Sub15.4}.
\end{remark}

Let $I=[t_1,t_2]\subset \R_t$ and $\hz=I\times \hy$. As described in the beginning of Section \ref{S14}, each smooth perturbation $\q$ gives arise to a section
\[
\hq: \SC_k(\hz)\to \V_0
\]
of the trivial bundle 
\[
\V_0=L^2(\hz, i\su(S^+)\oplus S^-)\times  \SC_k(\hz)\to   \SC_k(\hz),
\]
where the bundle $iT^*Y\oplus S^+$  is identified with $(i\su(S^+)\oplus S^-)$ using the bundle map 
\[
(\rho_3, \rho_4(dt)),
\]
over the 4-manifold $\hz$. For any $\gamma=(A,\Phi)\in \SC_k(\hz)$, write 
\[
(a,\phi)=(A,\Phi)-(A_0,\Phi_0)\in L^2_k(\hz,iT^*\hz\oplus S^+),
\]
where $\gamma_0=(A_0,\Phi_0)$ is the reference configuration of $ \SC_k(\hz)$. 

\begin{proposition}[cf. \cite{Bible} Proposition 11.4.1]\label{P15.8} For any $k\geq 2$ and any cylinder function $f$ defined via the embedding $\iota: S^1\times D^2\to \hy$, consider its induced perturbation on the 4-manifold $\hz$:
	\[
	\hq=\grad f: \SC_k(\hz)\to \V_0. 
	\]
	
	\begin{enumerate}[label=(C\arabic*)]
\item\label{C1} The map $\hq$ extends to a smooth map 
\[
\SC_k(\hz)\to \V_k,
\]
whose $l$-th derivative regarded as a multi-linear map
\[
\D^l_{(A,\Phi)}\hq\in \Mult^l({\bigtimes}_l \CT_k(\hz), \V_k),
\]
satisfies the estimate:
\[
\|\D^l_{(A,\Phi)}\hq\|\leq C(1+\|a\|_{L^2_k(\Omega)})^{2k(l+1)}(1+\|\Phi\|_{L^2_{k,A}(\Omega)})^{l+1},
\] 
where $\Omega=I\times \im\iota\subset \hz$. 

\item\label{C2} For any $j\in [-k,k]$, the first derivative $\D\hq$ extends to a smooth map
\[
\D\hq: \SC_k(\hz)\to\Hom(\CT_j(\hz),\V_j)
\]
whose $(l-1)$-the derivative regarded as a multi-linear map
\[
\D^l_{(A,\Phi)}\hq\in \Mult^l({\bigtimes}_{l-1} \CT_k(\hz)\times \CT_j(\hz), \V_j),
\]
satisfies the same bound as in $\ref{C1}$. 
\item\label{C5} When $p>3$ and $k\geq 1$, the map $\hq$ extends to a smooth map 
\[
\SC_k^{(p)}(\hz)\to \V_k^{(p)},
\]
whose $l$-th derivative regarded as a multi-linear map
\[
\D^l_{(A,\Phi)}\hq\in \Mult^l({\bigtimes}_l \CT_k^{(p)}(\hz), \V_k^{(p)}),
\]
satisfies the estimate:
\[
\|\D^l_{(A,\Phi)}\hq\|\leq C(1+\|a\|_{L^p_k(\Omega)})^{2k(l+1)}(1+\|\Phi\|_{L^p_{k,A}(\Omega)})^{l+1}.
\]  

\item\label{C3} For any $2\leq p<4$, the map $\hq$ satisfies the estimate
\[
\|\hq\|_{L^{\bn(p)}}\leq C(1+\|(a,\phi)\|_{L^{p}_{1,A}(\Omega)}) \text{ with } \bn(p)=4p/(4-p).
\]
\item \label{C4} When $2\leq p<4$, the map $\hq$ extends to a continuous map from
\[
\SC_1^{(p)}(\hz)\to \V_0^{(m)} \text{ for any } m<\bn(p).  
\]
\item \label{C7} For any $0\leq \epsilon<\half$,  the map $\hq$ extends to a continuous map from
\[
\SC_{1-\epsilon}(\hz)\to \V_0.
\]
	\end{enumerate}
\end{proposition}
\begin{remark} Properties \ref{C1}\ref{C5}\ref{C4}\ref{C7} are essential in the proof of compactness of perturbed Seiberg-Witten equations in Section \ref{Sec15}. Starting with $p=2$, we have $\bn(p)=4>3$. 
\end{remark}

Before we proceed to the proof, let us add a few remarks to simplify the situation. For a fixed cylinder function $f$, one can either compute its gradient using the pull-back metric $g_1$ on $S^1\times D^2$, or using the standard product metric $g_{std}$:
\[
\grad_{std} f \text{ or } \q\colonequals \grad f  .
\]
If we write $\grad f=(\grad^0 f, \grad ^1 f)$ as entries of $L^2_k(\hy, iT^*\hy\oplus S)$, then 
\[
 \grad_{std} f=(\eta K(\grad^0 f), \eta \grad^1 f),
\]
where the function $\eta$ and the bundle map $K$ were introduced in Section \ref{S15.1}.
Since they are related by a smooth bundle map of $iT^*\hy\oplus S|_{\im\iota}$, it suffices to prove estimates for $\grad_{std} f$. The change of metrics of $S^1\times D^2$ will also affect the $L^2_{j,A}$-norms on $\CT_j$ and $\V_j$, which is again inconsequential for our estimates.

From now on, we assume $g_1=g_{std}$,  and the length of the core $S^1\times \{0\}$ is $2\pi$.  

The second remark concerns the anisotropic Sobolev spaces, which involves different orders of differentiability in different directions. In what follows, let 
\begin{align*}
Y'&=S^1\times D^2=(\R/2\pi\Z) \times D^2\subset \hy,\\
\Omega&=I\times S^1\times D^2=[t_1,t_2]\times Y'\subset \hz ,\\
M&=I\times D^2.
\end{align*}

Within the product manifold $\Omega$ only the direction along $S^1$-fibers is special. Let $\theta$ be the coordinate function of the circle $\R/2\pi\Z$, and define the $L^2_{m,l}$ norm $(l\leq m)$ of functions on $\Omega$ to be 
\[
\|\xi\|_{L^2_{m,l}(\Omega)}^p=\sum_{\substack{i+j\leq m,\\ i\leq l}}\int_{\Omega} |(\frac{\partial}{\partial \theta})^j \nabla ^i_M \xi|^p
\]
and let $L^p_{m,l}(\Omega)$ be the completion of smooth functions (or sections) with respect to this norm. We are mostly interested in the case when $p=2$. There are two useful lemmas: 
\begin{lemma} Consider the Banach space $L^p_{k+1,k}$ with $k\geq 2$ if $p=2$ and $k\geq 1$  if $p>3$. Then  $L^p_{k+1,k}$ is an algebra under the point-wise multiplication and $L^p_{k+1,k}\subset C^0$; Moreover, for any $|r|\leq k+1$ and $|q|\leq k$, $L^p_{r,q}(\Omega)$ is a module of $L^p_{k+1,k}$. 
\end{lemma}\label{L15.9}
\begin{proof} Note that $L^2_{k+1,k}(\Omega)\embed L^2_1(S^1, L^2_k(M))\embed C^0(S^1, C^0(M))$ when $k\geq 2$, and 
	\[
	L^p_{k+1,k}(\Omega)\embed L^p_1(S^1, L^p_k(M))\embed C^0(S^1, C^0(M))
	\]
	when $k\geq 1$ and $p>3$.
\end{proof}

\begin{lemma}\label{L15.10} For  any $(m,l)$ and $p\in [1,\infty)$, the slicewise operator $d_{S^1}G$ and $Gd_{S^1}^*$ are bounded linear operators from $L^p_{m,l}(\Omega)\to L^p_{m+1,l}(\Omega)$, where
	\[
	G: C^\infty(S^1)\to C^\infty(S^1)
	\]
	is the Green operator associated to the Hodge Laplacian operator. 
\end{lemma}
\begin{proof} It follows from the fact that $G$ extends to a bounded linear operator 
	\[
	G: L^p_m(S^1, \R)\to L^p_{m+2}(S^1,\R)
	\]
	for any $p\in [1,\infty)$ and $m\geq 0$. 
\end{proof}

\begin{proof}[Proof of Proposition \ref{P15.8}] Suppose the cylinder function $f$ arises as the composition $g\circ \Xi$:
	\[
	\SC_k(\hy)\xrightarrow{\Xi} \R^n\times (\R/2\pi\alpha\Z)\times \R^m\xrightarrow{g} \R
	\]
where $\Xi=(r_{c_1},\cdots, r_{c_n},[r_\nu],q_{\Upsilon_1},\cdots, q_{\Upsilon_m})$ is induced from a collection of 1-forms $c_1,c_2,\cdots c_n$ and sections $\Upsilon_1,\cdots, \Upsilon_m$. Let $x_i\ (1\leq i\leq n)$, $x$ and $y_j$ be the coordinate functions on $\R^n$, $\R/2\pi\alpha\Z$ and $\R^m$ respectively. Then set
\begin{align*}
X_i&\colonequals \grad (x_i\circ \Xi)=(*_3dc_i,0),\\
X_\nu&\colonequals \grad (x\circ \Xi)=(*_3\pi^*\nu,0) \text{ and }\\
Y_j&\colonequals \grad (y_j\circ \Xi).
\end{align*}

The expression of $Y_j$ requires some further work. First, we compute the differential:
\begin{align*}
d(y_j\circ \Xi)(\vb,\vpsi)&=2\re\int_{D^2} \chi_2(z) \frac{\partial h}{\partial w}(\sigma(z))dvol_{D^2}\cdot  d(\sigma(z)) (\vb,\vpsi)
\end{align*}
and 
\begin{align*}
d(\sigma(z)) (\vb,\vpsi)&=\int_{S^1\times\{z\}}\langle \vpsi, \upd_{j,z}\rangle +\langle \Psi_z, (\partial_x \Upsilon_j)^\dagger_z\rangle \langle \vb, X_\nu\rangle_{Y'}+\langle \Psi_z, (-Gd^*_{S^1} \vb_z)\upd_{j,z}\rangle. 
\end{align*}
where $Y'=S^1\times D^2\subset \hy$. This allows us to write $Y_j=(Y_j^0, Y_j^1)=2(\im W_j^0,W_j^1)$ with
\begin{equation}\label{E15.2}
W_j=\chi_2(z) \frac{\partial h}{\partial w}(\sigma(z))((-d_{S^1}G)\langle \Psi ,\upd_j\rangle +\langle \Psi, (\partial_x\Upsilon_j)^\dagger\rangle_{Y'} X_\nu,\upd_j).
\end{equation}

As sections of $\Sph\to (\R/2\pi\alpha\Z)\times (S^1\times D^2)$, $\partial_x\Upsilon_j$ denotes the derivative of $\Upsilon_j$ along the first factor. 
Finally, we obtain that
\begin{equation}\label{E15.1}
\q=\grad f=\sum_{i=1}^n (\frac{\partial g}{\partial x_i}\circ \Xi ) X_i+ (\frac{\partial g}{\partial x}\circ \Xi ) X_\nu+\sum_{j=1}^m (\frac{\partial g}{\partial y_j}\circ \Xi ) Y_j.
\end{equation}

To study the mapping properties of $\q$, we first examine the map:
\begin{align*}
\upd: \SC_k(\hy)&\to L^2(S^1\times D^2, S)
\end{align*}
and its  extension in dimension 4:
\begin{align*}
\updd: \SC_k(\hz)&\to L^2(\Omega, S^-)\ \text{ where } \Omega=I\times S^1\times D^2,\\
(A,\Phi)&\mapsto \upd(\cA(t), \cPhi(t)),\ \forall t\in I= [t_1,t_2]. 
\end{align*}
for any compactly supported section $\Upsilon$ of $\Sph\to (\R/2\pi\alpha\Z)\times S^1\times D^2$. 

\begin{lemma}[cf. Lemma 11.4.4 in \cite{Bible}]\label{L15.11} For any $k\geq 2$ and any $j\in [-k,k]$,  $\updd$ extends to a smooth map
	\[
	\SC_k(\hz)\to L^2_{j+1,j,A}(\hz, S^-)
	\]
	with the following properties.
	\begin{enumerate}[label=(D\arabic*)]
	\item\label{D1} For each $l\geq 0$, there is a constant $C>0$ such that the differential \[
	\D_{(A,\Phi)}^l\updd \in \Mult^l({\bigtimes}_l \CT_k(\hz), L^2_{j+1,j,A}(\hz, S^-))
	\]
	satisfies the bound 
	\[
	\|\D^l_{(A,\Phi)}\updd\|\leq C(1+\|a\|_{L^2_j})^j (1+\|a\|_{L^2_k})^k,\ \forall (A,\Phi)\in \SC_k(\hz).  
	\]
	\item\label{D2} The $l$-th derivative extends to an element of 
	\[
 \Mult^l({\bigtimes}_{l-1} \CT_k(\hz)\times \CT_j(\hz), L^2_{j+1,j,A}(\hz, S^-))
	\]
whose norm	satisfies the bound 
	\[
	\|\D^l_{(A,\Phi)}\updd\|\leq C (1+\|a\|_{L^2_k})^{2k},\ \forall (A,\Phi)\in \SC_k(\hz).  
	\]
	
	\item\label{D2.5}  For any $k\geq 1$ and $p>3$, $\updd$ extends to a smooth map
		\[
	\SC_1^{(p)}(\hz)\to L^{(p)}_{j+1,j,A}(\hz, S^-).
	\]
	whose $l$-th derivative extends to an element of 
	\[
	\Mult^l({\bigtimes}_{l-1} \CT_k^{(p)}(\hz)\times \CT_j^{(p)}(\hz), L^p_{j+1,j,A}(\hz, S^-))
	\]
	with norm bounded by 
	\[
	\|\D^l_{(A,\Phi)}\updd\|\leq C (1+\|a\|_{L^p_k})^{2k},\ \forall (A,\Phi)\in \SC_k^{(p)}(\hz).  
	\]
	\item\label{D3} For $i=0,1$ and any $p\in [2,\infty]$, we have the bound
	\[
\|\updd\|_{L^p_{A,i}}\leq C (1+\|a\|_{L^p_i})^i,\ \forall (A,\Phi)\in \SC_k(\hz).  
	\]
	\item\label{D4} For any $1\leq m<p$, $\updd$ extends to a continuous map from
	\begin{align*}
	\SC_1^{(p)}(\hz)&\to L^{m}_1(\hz, S^-).
	\end{align*}

\item\label{D5} For any $1\leq p',p<\infty$, $\updd$ extends to a continuous map from
\begin{align*}
\SC^{(p)}(\hz)&\to L^{p'}(\hz, S^-).
\end{align*}
	\end{enumerate}
\end{lemma}

\begin{proof} The proof of \ref{D1}\ref{D2}\ref{D2.5} carries though with little changes as in \cite[Lemma 11.4.4]{Bible}, using Lemma \ref{L15.9} in place of \cite[Lemma 11.4.3]{Bible}. In what follows, we will focus on \ref{D3}\ref{D4}\ref{D5}.
	
As this point, it is convenient to have a lemma that is slightly stronger than \cite[Lemma 11.4.5]{Bible}:
\begin{lemma}\label{L15.12} Let $\SH_1, \SH_2$ be any separable Banach spaces and $\dim\SH_1<\infty$. Suppose $\chi: \SH_1\to \SH_2$ be a smooth map with bounded $C^1$-norm. Then the composition map $\chi^*: \xi\mapsto \chi\circ \xi$ is continuous from
	\[
	L^1(\Omega_*, \SH_1)\to 	L^p(\Omega_*, \SH_2)
	\]
	for any finite measure space $\Omega_*$ and any $1\leq p<\infty$. Moreover, $\|\chi\circ\xi\|_\infty\leq \|\chi\|_\infty$. 
\end{lemma}
\begin{proof}[Proof of Lemma] It is clear that $\chi\circ \xi$ lies in $L^\infty(\Omega_*,\C)$ with $\|\chi\circ\xi\|_\infty\leq \|\chi\|_\infty$. Since $\Omega_*$ has a finite measure, $\chi\circ \xi\in L^p$.  We prove that $\chi^*$ is H\"{o}lder continuous. For any $\xi_1,\xi_2\in L^1(\Omega_*,\SH_1)$, 
	\begin{align*}
	\|\chi\circ \xi_1-\chi\circ \xi_2\|_p^p&=\int_{\Omega_*} \|\chi\circ \xi_1-\chi\circ \xi_2\|^p_{\SH_2}\leq \|2\chi\|_\infty^{p-1}\int_{\Omega_*} |\chi\circ \xi_1-\chi\circ \xi_2|_{\SH_2} \\
	&=\|2\chi\|_\infty^{p-1}\|\nabla \chi\|_\infty \int_{\Omega_*} | \xi_1- \xi_2|_{\SH_1}=\|2\chi\|_\infty^{p-1}\|\nabla \chi\|_\infty \|\xi_1-\xi_2\|_{L^1(\Omega_*,\SH_1)}. \qedhere
	\end{align*}
\end{proof}
	
	Back to the proof of Lemma \ref{L15.11}. 
	Let $(a,\phi)=(A,\Phi)-(A_0,\Phi_0)\in L^p_1(\hz, iT^*\hz\oplus S^+)$, then $\updd(A,\Phi)$ is defined as 
	\begin{equation}\label{F15.3}
	e^{-Gd^*_{S^1}a} \tup(r_\nu(a))
	\end{equation}
	as a section supported on
	\[
	\Omega=I\times S^1\times D^2
	\]
	with $r_\nu(a)=r_\nu(a|_{\{t\}\times \hy})\in L^p(I, \R)$. 
	
	\medskip
	
	\Step 1. Proof of \ref{D5}. It follows from Lemma \ref{L15.12} directly: the exponential map 
	\[
	\xi\mapsto e^{\xi}
	\]
	is continuous from $L^p(\Omega, i\R)\to L^{2p'}(\Omega, \C)$ for any $1\leq p, p'<\infty$, so the map
	 $$\varphi:a\mapsto \exp(-Gd^*_{S^1}a)$$ 
	 is continuous from $L^p\to L^{2p'}$. On the other hand, we view the map $a\mapsto  \tup(r_\nu(a))$ as the composition
	\begin{align*}
 L^p(\hz)&\to L^p(I, \R)\to L^{2p'} (I, L^{2p'}(\hy))=L^{2p'}(\hz),\\
 a&\mapsto r_\nu(a)\mapsto \tup(r_\nu(a)),
	\end{align*}
	so Lemma \ref{L15.12} applies. Finally,  $L^{2p'}\times L^{2p'}\to L^{p'}$ is continuous. 
	
	\medskip
	
\Step 2. Proof of \ref{D4}. Now we deal with the first derivative of $\updd$. Write $\nabla_A \updd=K_1+K_2+K_3+K_4$ with
	\begin{align}
K_1&= (-d_{S^1}Gd_{S^1}^*a) \updd, & K_3&=(e^{-Gd_{S^1}^*a})\nabla_{A_0}\tup(r_\nu(a)),\label{F15.4}\\
K_2&= (-Gd_{S^1}^*d_Ma) \updd, & K_4&= a\otimes\updd,\nonumber
	\end{align}
where $M=I\times D^2$. To prove \ref{D4}, we verify that each $K_i$ is continuous from $L^p_1\to L^{m}$ for any $m<p$. It is clear that each of the following terms:
\[
-d_{S^1}Gd_{S^1}^*a,\ -Gd_{S^1}^*d_Ma,\ a
\]
is continuous from $L^p_1$ to $L^p$. To analyze $K_3$, we expand $\nabla_{A_0}\tup(r_\nu(a))$ as 
\[
(\nabla_{B_0}\tup)(r_\nu(a))+(\widetilde{\partial_x\Upsilon})(r_\nu(a)) \langle \frac{d}{dt} a, X_\nu\rangle_{Y'},
\]
which is continuous from $L^p_1\to L^{p'}$ for any $1\leq p'<p$. Now we use \Step 1 to complete the proof of \ref{D4}. 

\medskip

\Step 3. Proof of \ref{D3}. It follows directly from the expression of $\updd$ and $\nabla_A\updd$, \eqref{F15.3} and \eqref{F15.4}, using the fact that $\|\varphi(a)\|_\infty=1$. 
\end{proof}

Back to the proof of Proposition \ref{P15.8}. The proof of \ref{C1}$\sim$\ref{C5} follows from \ref{D1}$\sim $\ref{D2.5} in the same line as \cite[Proposition 11.4.1]{Bible}, using Lemma \ref{L15.10}. 

\medskip

In what follows, we will explain how \ref{C3}\ref{C4}\ref{C7} follow from \ref{D3} and \ref{D5}. In fact, \ref{D5} provides better bounds than \ref{D4}. To estimate $\hq$, we investigate the section
\[
W_j=\chi_2(z) \frac{\partial h}{\partial w}(\sigma(z))((-d_{S^1}G)\langle \Phi ,\updd_j\rangle +\langle \Phi, (\partial_x\Upsilon_j)^\ddagger\rangle_{Y'} X_\nu,\updd_j),
\]
in place of $Y_j$, so
\[
\hq=\sum_{i=1}^n (\frac{\partial g}{\partial x_i}\circ \Xi ) X_i+ (\frac{\partial g}{\partial x}\circ \Xi ) X_\nu+2\sum_{j=1}^m (\frac{\partial g}{\partial y_j}\circ \Xi ) (\im W_j^0,  W_j^1).
\]

We break $W$ into four simpler pieces: $W=\varpi(V_1+V_2+V_3)$ where
	\begin{align*}
\varpi&=\chi_2(z) \frac{\partial h}{\partial w}(\sigma(z)), & V_1&=(-d_{S^1}G)\langle \Phi ,\updd\rangle ,\\
V_2&=\langle \Phi, (\partial_x\Upsilon)^\ddagger\rangle_{Y'} X_\nu, & V_3&=\updd.\nonumber
\end{align*}

\Step 1. Proof of \ref{C4}. The map $V_1: \SC^{(p)}_1(\hz)\to L^{m}(\hz, iT^* \hz)$ is continuous for any $m<\bn(p)$ when $2\leq p< 4$. Indeed, $V_1$ can be viewed as the composition 
\[
(\Phi, \updd)\in L^p_1\times L^{p'}\to  L^{\bn(p)}\times L^{p'} \xrightarrow{\times} L^{m}\xrightarrow{-d_{S^1}G} L^{m},
\]
when $p'$ is sufficiently large.

For any $p'\gg 1$, the map $V_2: \SC_1^{(p)}\to  L^{p'}(\hz, iT^*\hz)$ is also continuous since the map $(\Phi, (\partial_x\Upsilon)^\ddagger)\to \langle \Phi, (\partial_x\Upsilon)^\ddagger\rangle_{Y'} $ can be viewed as the composition:
 \begin{align*}
L^p_1\times L^{p'}\to L^p_1(I, L^p(\hy))\times L^{p'}&\to C^0(I, L^p(\hy))\times L^{p'}(I, L^{p'}(\hy))\\
&\xrightarrow{\times}L^{p'}(I, L^1(\hy))\xrightarrow{\int}L^{p'}(I).
 \end{align*}
 
By Lemma \ref{L15.11} \ref{D5}, $V_3$ is a continuous map into $ L^{\bn(p)}$. It remains to deal with $\varpi$, which is viewed as the composition of $\frac{\partial h}{\partial w}$ with the map
	\begin{align*}
	\sigma: \SC_1^{(p)}&\to L^1(M,\C),\ M=I\times D^2,\\
	(A,\Phi)&\mapsto\bigg( (t,z)\mapsto\int_{\{t\}\times S^1\times \{z\}} \langle\Phi, \updd\rangle\bigg). 
	\end{align*}
	
	The map $\sigma$ is continuous, since it is the composition:
	\[
	(\Phi, \updd)\in L^p_1\times L^{4}\xrightarrow{\times} L^1= L^1(M, L^1(S^1))\xrightarrow{\int_{S^1}} L^1(M, \C).
	\]

Since $\frac{\partial h}{\partial w}:\C_w\to \C$ is a smooth function with compact support, it follows from Lemma \ref{L15.12} that $\varpi: C^{(p)}_1\to L^{p'}$ is continuous for any $1\leq p'<\infty$. 

The same argument shows that 
\[
\frac{\partial g}{\partial x_i}\circ \Xi, \frac{\partial g}{\partial x}\circ \Xi, \frac{\partial g}{\partial y_j}\circ \Xi 
\]
are continuous functions into $L^{p'}(I,\R)$ for any $1\leq p'<\infty$. This completes the proof of $\ref{C4}$.

\medskip

\Step 2. Proof of \ref{C3}. It follows by replacing $L^{p'}$ by $L^\infty$ through out \Step 1, using \ref{D3} from Lemma \ref{L15.11}.

\medskip

\Step 3. Proof of \ref{C7}. It follows by replacing $L^p_1$ by $L^2_{1-\epsilon}$ through out \Step 1 with $0\leq \epsilon<\half$.

\medskip

The proof of Proposition \ref{P15.8} is now completed.
\end{proof}

\subsection{Proof of Theorem \ref{T15.2}}\label{Sub15.4} In this subsection, we verify that a cylinder function $f$ satisfies conditions in Definition \ref{D14.2} and prove Theorem \ref{T15.2}.

\begin{itemize}
\item \ref{A1} and \ref{A2} follows from \ref{C1} and \ref{C5}.

\item \ref{A3} is satisfied on account of \ref{C4}, as $\bn(2)=4$.

\item  \ref{A4} is a consequence of \ref{C2}, while \ref{AA6} follows from \ref{C7}.

\item As for \ref{AA5}, the statement on the support of $\q=\grad f$ is clear from the construction. The estimate on $\|\q\|_2$ is a consequence of the explicit formulae \eqref{E15.2} and \eqref{E15.1}. 
\end{itemize}

Only \ref{A7} requires some further explanation, as Proposition \ref{P15.7} does not extend to the case when $k=1$. The proof of \cite[Proposition 11.4.1]{Bible} fails here, as $L^2_{2,1}(S^1\times D^2)$ fails to be an algebra: 
\[
L^2_{2,1}(S^1\times D^2)\embed L^2_1(S^1, L^2_1(D^2))\not\embed C^0,
\]
 
 Nevertheless, it is at the borderline. As we are merely interested in $\CT_0$, losing a tiny amount of regularity is affordable. In fact, one can still prove that 
\[
\q: \SC_1(\hy)\to \CT_0
\]
is smooth. This completes the proof of Theorem \ref{T15.2}. 

\subsection{Banach Spaces of Tame Perturbations}\label{Subsec15.5}

In this subsection, we construct a Banach space of tame perturbations as described in Section \ref{S14}. Since only minor changes are needed,  we will only state the theorem and refer to \cite[Section 11.6]{Bible} for the actual proof.

First, we introduce a broader class of functions defined on $\SC_{k-1/2}(\hy,\bs)$, called generalized cylinder functions. In the definition of cylinder functions (cf. Definition \ref{D15.1}), one may allow entries of $\Xi$ to come from different embeddings of $S^1\times D^2$ into $\hy$. This motivates the next definition.

\begin{definition} \label{D15.2} A function $f'$ defined on $\SC_{k-1/2}(\hy,\bs)$ is called \textit{a generalized cylinder function} if it arises as the composition $g'\circ \Xi'$ where
	\begin{itemize}
		\item the map $\Xi'$ is defined using a collection of cylinder functions $f_1,\cdots, f_l$:
		\[
		\Xi'=(f_1,\cdots, f_l):\SC_{k-\half}(\hy,\bs)\to \R^l.
		\]
		Their underlying embeddings
		$
		\iota_j: S^1\times D^2\to \hy, 1\leq j\leq l
		$
		might be different. 
		
		\item the function $$g': \R^l\to \R$$ is any smooth function with compact support. \qedhere
	\end{itemize}
\end{definition}

\begin{theorem}\label{T15.3} Let $Y'$ is a smooth co-dimension 0 submanifold of $\hy$. Suppose a generalized cylinder function $f'$ is defined using a collection of embeddings $\{\iota_k\}_{1\leq k\leq l}$ with $\im \iota_k\subset Y'$ for all $\iota_k$, then $\grad f'$ is a perturbation tame in $Y'$ in the sense of Definition \ref{D14.2}.
\end{theorem}

The proof of Theorem \ref{T15.3}  is not essentially different from that of Theorem \ref{T15.2}.

\begin{theorem}\label{T15.14} Fix an open submanifold $Y'\subset \hy$. Let $\q^i\ (i\in\N)$ be any countable collection of tame perturbations arising as gradients of generalized cylinder functions on $\SC_{k-1/2}(\hy,\bs)$ with support in $Y'$. Then there exists a separable Banach space $\Pa$ and a linear map:
	\begin{align*}
	\mathfrak{O}: \Pa&\to C^0(\SC_{k-1/2}(\hy,\bs), \CT_0)\\
	\lambda&\mapsto \q^\lambda
	\end{align*}
	with the following properties:
	\begin{enumerate}[label=(F\arabic*)]
		\item For each $\lambda\in \Pa$, the element $\q^\lambda$ is a tame perturbation in $Y'$ in the sense of Definition \ref{D14.2}. 
		\item The image of $\mathfrak{O}$ contains all the perturbations $\q^i$ from the given countable collections.
		\item\label{F3} If $\hz=[t_1,t_2]\times \hy$ is a cylinder, then for all $k\geq 2$, the map 
		\begin{align*}
		\Pa\times \SC_k(\hz)&\to \V_k\\
		(\lambda,\gamma)&\mapsto \hq^\lambda(\gamma)
		\end{align*}
		is a smooth map of Banach manifolds. 
			\item\label{F4} For all $k\geq 1$ and $p=7/2$, the map 
		\begin{align*}
		\Pa\times \SC_k^{(p)}(\hz)&\to \V_k^{(p)}\\
		(\lambda,\gamma)&\mapsto \hq^\lambda(\gamma)
		\end{align*}
		is a smooth map of Banach manifolds. 
		\item\label{F5} For $\epsilon=1/4$, the map 
		\begin{align*}
		\Pa\times \SC_{1-\epsilon}(Y)&\to \CT_0(Y)\\
		(\lambda, \beta)&\mapsto \q^\lambda(\beta). 
		\end{align*}
		is continuous and satisfies the bound:
		\[
		\|\q^\lambda(B,\Psi)\|_2\leq \|\lambda\|_{\Pa}\cdot m_2(\|\Psi\|_{L^2(Y')}+1).
		\]
		

	\end{enumerate}
\end{theorem}
\begin{proof} See \cite[Theorem 11.6.1]{Bible}. 
\end{proof}

We do not distinguish $\lambda\in \Pa$ with its image $\q^\lambda$ in $C^0(\SC_{k-1/2}(\hy,\bs), \CT_0)$. 

\begin{remark} In property \ref{F4}, any index $3<p<4$ will make the Compactness Theorem \ref{T1.4} work. In property \ref{F5}, one may take any $0<\epsilon<1/2$.
\end{remark}

\begin{corollary}\label{C14.18} Suppose $\{\q_n\}\subset \Pa$ and $\|\q_n\|_\Pa\to 0$ as $n\to\infty$. Then for any bounded region $\SO\subset \SC_k(\hy,\bs)$, the $C^l$-norm of $\q_n$ converges to zero, i.e.
	\[
	\|\q_n\|_{C^l(\SO, \SC_k\to \CT_k)}\to 0 \text{ as } n\to\infty. 
	\]
\end{corollary}

Our primary interest is in the case when $Y'=Y=\{s\leq 0\}$, and let us specify the countable collection of tame perturbations associated to $Y'$ in Theorem \ref{T15.14}. We make the following choices in order:
\begin{itemize}
	\item a positive integer $l$;
	\item a compact subset $K'$ of $\R^l$;
	\item a smooth function $g'$ on $\R^l$ with support in $K'$
\end{itemize}
and for each $j\in \{1,\cdots,l\}$, 
\begin{itemize}
	\item an embedding $\iota: S^1\times D^2\embed (Y')^\circ$ into the interior of $Y'$;
	\item a pair of positive integers $n$ and $m$;
	\item compactly supported 1-forms $c_1,\cdots, c_n$ and compactly supported sections $\Upsilon_1,\cdots, \Upsilon_m$ of $\Sph$;
	\item a compact subset $K$ of $\R^n\times (\R/2\pi\alpha\Z)\times \R^m$;
	\item a smooth function $g$ on $\R^n\times (\R/2\pi\alpha\Z)\times \R^m$ with support in $K$. 
\end{itemize}

We require the resulting collection $\{\q^i\}_{i\in \N}$ to be dense in the space of all possible choices, in $C^\infty$-topology; see \cite[P. 192]{Bible} for a complete description. For the rest of the paper, we presume that such a collection $\{\q^i\}_{i\in \N}$ is chosen, once and for all, for $Y'=Y$.  Let $\Pa$ be the resulting Banach spaces constructed by Theorem \ref{T15.14}.

Each configuration and gauge transformation on $\hy$ can be restricted to $Y$, giving rise to maps:
\begin{align*}
R_c&: \SC_{k-\half}(\hy,\bs)\to \SC_{k-\half}(Y, \bs)\\
R_g&: \CG_{k+\half}(\hy,\bs)\to \CG_{k+\half}(Y, \bs). 
\end{align*}

Let $\SC^*(Y,\bs)$ be the irreducible part of $ \SC(Y, \bs)$ and form the quotient configuration space:
\[
\CB^*(Y,\bs)=\SC^*(Y,\bs)/ \im (R_g: \CG(\hy,\bs)\to\CG(Y,\bs)).
\]

Let us now state the separating property enjoyed by $\Pa$:  it is a direct consequence of Proposition \ref{P15.4} and \ref{P15.6} and the proof is omitted here.
\begin{theorem}\label{T15.17} Given a compact subset $K$ of a finite dimensional $C^1$-submanifold $M\subset \CB^*(\hy,\bs)$, suppose the restriction map to the truncated manifold $Y$
	\[
	[R_c]: \CB(\hy,\bs)\to \CB(Y,\bs)
	\]
	gives an embedding of $K$ into $\CB^*(Y,\bs)$. Then we can find a open neighborhood $U$ of $K$ in $M$, a collections of embeddings 
	\[
	\iota_j: S^1\times D^2\embed Y, 1\leq j\leq l
	\]
	and cylinder functions $f_k$ defined using $\iota_k$ such that the product map
	\[
	\Xi'=(f_1,\cdots, f_l): \CB^*(\hy,\bs) \to \R^l
	\]
	gives an embedding of $U$ into $\R^l$. If in addition, a tangent vector $V\in T_{\beta}\CB^*(\hy,\bs)$ at some $\beta\in K$ is given $(V$ is not necessarily tangential to $M)$ and $[r_c]_*(V)\neq 0$, then we can arrange so that 
	\[
	\Xi'_*(V)\neq 0\in T\R^l. 
	\]
\end{theorem}

\section{Compactness for Perturbed Seiberg-Witten Equations}\label{Sec15}

With the Banach space $\Pa$ of tame perturbations defined as in Subsection \ref{Subsec15.5}, we start to analyze the moduli space of perturbed Seiberg-Witten equations. The primary goal of this section is to prove the compactness theorem for solutions on $\R_t\times \hy$. Before that, we have to generalize results from Section \ref{Sec12} and \ref{Sec11} for the perturbed equations.

\subsection{Energy Equations} Choose a tame perturbation $\q=\grad f\in \Pa$ with 
\begin{equation}\label{E1.1}
\|\q\|_{\Pa}<1. 
\end{equation}

For all estimates and theorems below, \eqref{E1.1} will be a standard assumption. Following the notations in Section \ref{S14}, let $I=[t_1, t_2]_t$ and $\hz=I\times (\hy, \bs)$. Consider a solution $\gamma\in \SC_k(\hz)$ to the perturbed Seiberg-Witten equations
\begin{equation}\label{E1.2}
0=\F_{\hz, \q}(\gamma)\colonequals \F_{\hz}(\gamma)+\hq(\gamma). 
\end{equation}
Write $\gamma$ as $(c(t), B(t),\Psi(t))$ where $\cgamma(t)=(B(t),\Psi(t))$ is the underlying path in $ \SC_{k-1/2}(\hy)$. Then the equation \eqref{E1.2} can be cast into the form
\begin{align}\label{E1.3}
\dt\cgamma(t)=-\grad \CL_{\omega}(\cgamma(t))-\bd_{\cgamma(t)} c(t)-\q(\cgamma(t)).
\end{align}

\begin{proposition}\label{P1.1} For any perturbation $\q=\grad f\in \Pa$ with $\|\q\|_\Pa<1$ and any configuration $\gamma=(A,\Phi)$ on $\hz=I\times (\hy, \bs)$, the $L^2$-norm of the perturbed Seiberg-Witten map $\F_{\hz,\q}(A,\Phi)$ can be expressed as 
	\[
	\int_{Z}| \F_{\hz,\q}(A,\Phi)|^2=\E_{an}^\q(A,\Phi)-\E_{top}^\q(A,\Phi)
	\]
	where 
	\begin{align*}
	\E_{top}^\q(A,\Phi)&\colonequals 2\CSd_{\omega}(\cgamma(t_1))-2\CSd_{\omega}(\cgamma(t_2)),\\
	\E_{an}^\q(A,\Phi)&\colonequals\int_I\|\dt \cgamma(t)+d_{\cgamma(t)}c(t)\|^2_{L^2(\hy)}+\|\grad \CSd_{\omega}(\cgamma(t))\|^2_{L^2(\hy)},
	\end{align*}
	and $\CSd_{\omega}=\CL_{\omega}+f$ is the perturbed Chern-Simons functional. Moreover, there exist constants $C_1,C_2>0$ such that 
	\[
	\E_{an}(A,\Phi)<C_1\cdot  \E_{an}^\q(A,\Phi)+C_2,
	\]
	where $\E_{an}$ is the analytic energy defined in Proposition \ref{Energy10.2}.
\end{proposition}
\begin{proof} Only the last clause requires some work. By the Cauchy-Schwartz inequality, we have
	\begin{equation}\label{E1.4}
	2\E_{an}^\q(\gamma)\geq \E_{an}(\gamma)-2\int_I \|\q(\cgamma(t))\|^2_{L^2(\hy)}. 
	\end{equation}
	since $\grad \CSd_{\omega}=\grad \CL_{\omega}+\q$. By the property \ref{F5} from Theorem \ref{T15.14}, 
	\begin{align}\label{E1.6}
	\int_I \|\q(\cgamma(t))\|^2_{L^2(\hy)}&\leq 2m_2^2(1+\|\Phi\|^2_{L^2(I\times Y)}).
	\end{align}
Hence, it remains to estimate $\|\Phi\|^2_{L^2(I\times Y)}$	in terms of $\E_{an}^\q(\gamma)$. Recall from Lemma \ref{L11.3} that
	\begin{align}\label{E1.5}
	\E_{an}(A,\Phi)+C_2'&\geq \int_{I\times \hy}  \frac{1}{8}|F_{A^t}|^2+|\nabla_A\Phi|^2+|(\Phi\Phi^*)_0+\rho_4(\omega^+)|^2+\frac{s}{4}|\Phi|^2,\\
&	\geq  \int_{I\times Y} |(\Phi\Phi^*)_0+\rho_4(\omega^+)|^2+\frac{s}{4}|\Phi|^2.\nonumber
	\end{align}
	for some $C_2'>0$. Combining \eqref{E1.4}\eqref{E1.6}\eqref{E1.5} together, we obtain that 
	\begin{align}\label{E1.7}
	2\E_{an}^\q (\gamma)+C_2''&\geq \int_{I\times Y} \frac{1}{4}|\Phi|^4-C_3 |\Phi|^2\geq \int_{I\times Y} |\Phi|^2-C_4.
	\end{align}
	for some $C_2'', C_3, C_4>0$. This completes the proof.
\end{proof}

Now the proof of Lemma \ref{11.3} and Theorem \ref{11.1} can proceed with no difficulty. Let us record the results for perturbed equations: 

\begin{theorem}\label{T1.1} For any $C,\epsilon>0$, there exists a constant $R_0(\epsilon, C, \hy,\bs)>0$ with the following significance. For any tame perturbation $\q\in \Pa$ with $\|\q\|_{\Pa}<1$,  let $\gamma=(A,\Phi)$ be a solution to the perturbed Seiberg-Witten equations \eqref{E1.3} on $\R_t\times (\hy, \bs)$ with analytic energy $\E_{an}^\q(A,\Phi)<C$. Then for any $n\in \Z$ and $S>R_0$, we have
	\[
	\E_{an}(A,\Phi; \Omega_{n,S})<\epsilon.
	\]
	Here $\Omega_{n,S}\subset \C_z$ is the translated region of $\Omega_0$ defined in \eqref{E2.4}.
\end{theorem}

\begin{theorem}\label{T1.2} For any $C>0$, there exist constants $M_0(C, \hy,\bs), \zeta(C, \hy,\bs)>0$ with the following significance. For any perturbation $\q\in \Pa$ with $\|\q\|_{\Pa}<1$, suppose $(A,\Phi)$ is a solution to the perturbed Seiberg-Witten equations $(\ref{E1.3})$ on $\R_t\times (\hy, \bs)$ with analytic energy $\E_{an}^\q(A,\Phi)<C$, then for any $n\in \Z$ and $S>0$
	\[
	\E_{an}(A,\Phi, \Omega_{n,S})<M_0e^{-\zeta S}.
	\]
\end{theorem}

\begin{remark}\label{R1.4} The analogous result for the exponential decay in the time direction follows from the standard argument as in \cite[Section 13]{Bible}, assuming the non-degeneracy of critical points (cf. Definition \ref{D18.2}). Indeed, once we obtain the exponential decay of $\CSd_{\omega}$, one starts to estimate the $L^2_1$-norm and $L^2_k$-norm of $(A,\Phi)$ as in Subsection \ref{Subsec12.2}. The proof is omitted here.  
\end{remark}
\subsection{Compactness} 

The next theorem is the analogue of Theorem \ref{T11.1} when $\q\neq 0$. 

\begin{theorem}\label{T1.4} For any perturbation $\q\in \Pa$ with $\|\q\|_\Pa<1$, suppose $\{\gamma_n=(A_n,\Phi_n)\}\subset \SC_{k,loc}(\R_t\times (\hy,\bs))$ is a sequence of solutions to the perturbed Seiberg-Witten equations \eqref{E1.3} on $\R_t\times\hy$ and their analytic energy
	\[
	\E_{an}^\q(\gamma_n)\colonequals 	\E_{an}^\q(\gamma_n, \R_t)<C
	\]
	is uniformly bounded. Then we can find a sequence of gauge transformations $u_n\in \CG_{k+1,loc}(\R_t\times \hy)$ with the following properties. For a subsequence $\{\gamma_n'\}$ of $\{u_n(\gamma_n)\}$ and any finite interval $I\subset \R_t$, the restriction of each $\gamma_n'$ on $I\times \hy$
	\[
	\gamma_n'|_{I\times \hy}
	\]
	lies in $\SC_l(I\times (\hy,\bs))$. Additionally, they converge in $L^2_{l}(I\times \hy)$-topology for any $l>1$. 
\end{theorem}
\begin{proof} It suffices to deal with the compact region $I\times Y_1$ where $Y_1=\{s\leq 1\}$ is the truncated 3-manifold. Fix a reducible configuration $\gamma_0'$ on $I\times Y_1$ as reference. The bootstrapping argument works as follows: by passing to a subsequence and applying appropriate gauge transformations, we obtain that
	\begin{align*}
	&\gamma_n-\gamma_0' \text{ bounded in } L^2_1 \Rightarrow \gamma_n\to \gamma_\infty \text{ weakly in } L^2_1 \text{ for some } \gamma_\infty\\
	\Rightarrow & \gamma_n\to \gamma_\infty \text{ in } L^2_{3/4}\Rightarrow \hq(\gamma_n)\to \hq(\gamma_\infty )\text{ in } L^2 \text{ by } \ref{AA6} \text{ with }\epsilon=1/4\\
\Rightarrow&\gamma_n\to \gamma_\infty \text{ in } L^2_1 \text{ on interior domains} \Rightarrow \hq(\gamma_n)\to \hq(\gamma_\infty) \text{ in } L^{7/2} \text{ by } \ref{A3} \\
\Rightarrow&\gamma_n\to \gamma_\infty \text{ in } L^{7/2}_1 \text{ on interior domains} \Rightarrow \hq(\gamma_\infty)\to \hq(\gamma_n) \text{ in } L^{7/2}_1 \text{ by } \ref{A2}\\
\Rightarrow&\gamma_n\to \gamma_\infty \text{ in } L^{7/2}_2 \text{ on interior domains} \Rightarrow \hq(\gamma_n)\to\hq(\gamma_\infty) \text{ in } L^{7/2}_2\embed L^2_2 \text{ by } \ref{A1}\\
\Rightarrow&\gamma_n\to \gamma_\infty \text{ in } L^2_3 \text{ on interior domains} \cdots 
	\end{align*}
Once we arrive at $L^2_3$, one may proceed as in \cite[Theorem 10.7.1]{Bible}. To conclude convergence of $\gamma_n$ on interior domains from  the convergence of $\hq(\gamma_n)$, we use the properness of the Seiberg-Witten map, cf. Theorem \cite[Theorem 5.2.1]{Bible}. 
\end{proof}
\begin{remark} It is not clear to the author whether the $L^2_1$-norm of $\hq(\gamma)$ can be estimated in terms of the $L^2_1$-norm of $\gamma-\gamma_0$, so we adopt a different approach to arrive at the $L^2$-convergence of $\hq(\gamma_n)$, cf. \cite[Theorem 10.7.1]{Bible}.
\end{remark}

\begin{proposition}\label{P1.5} Suppose $\{\q_i\}\subset \Pa$ is a convergent sequence in $\Pa$ with $\|\q_i\|_\Pa<1$ and let $\beta_i\in \SC_k(\hy,\bs)$ be solutions of the equation 
	\[
	(\grad \CL_{\omega}+\q_i)(\beta_i)=0.
	\]
	Then there is a sequence of gauge transformations $u_i\in \CG_{k+1}(\hy)$ such that the transformed solutions $u_i(\beta_i)$ have a convergent subsequence in $\SC_k(\hy,\bs)$. 
\end{proposition}

\begin{proof} The proof follows the same line of argument of Theorem \ref{T1.4}. To conclude the convergence of 
	\[
	\q_i(\beta_i)\to \q_\infty(\beta_\infty),
	\] 
	use \ref{F3}\ref{F4}\ref{F5} from Theorem \ref{T15.14}. 
\end{proof}
\part{Linear Analysis}\label{Part5}

Over the non-compact manifold $\hy$, the inclusion map
\[
L^2_{k+1}(\hy)\embed L^2_{k}(\hy)
\]
is no longer compact. As a result, the spectrum of the extended Hessian of the Chern-Simons-Dirac functional $\CSd_{\omega}$, as a unbounded self-adjoint operator, is not discrete.

The goal of this part to understand the essential spectrum of extended Hessians and show that it is disjoint from the origin, in which case one can still speak of the spectrum flow. Moreover, we will show the linearization of the Seiberg-Witten equations together with the linearized gauge fixing equation form a Fredholm operator on the complete Riemannian 4-manifolds $\R_t\times \hy$ and $\CX$; so we have a well-posed moduli problem. 

Part \ref{Part5} is organized as follows. In Section \ref{Sec9}, we review an abstract formalism of spectral flow following the work of Robbin-Salamon \cite{RS95}. In Section \ref{Sec17} we collect some criterion from functional analysis that computes the essential spectrum following the textbook \cite{HS96} by Hislop and Sigal. These results will be applied to the extended Hessian $\EHess$ of $\CSd_{\omega}$ in Section \ref{Sec18}.  The key observation here is that $\EHess$ can be cast into the form (up to a compact perturbation):
\[
\sigma(\ps+D_\Sigma): \Gamma(\R_s\times \Sigma, E)\to \Gamma(\R_s\times \Sigma, E)
\]
such that $\sigma^2=-\Id_E$ and $D_\Sigma: \Gamma(\Sigma, E)\to \Gamma(\Sigma, E)$ is a first order self-adjoint operator that anti-commutes with $\sigma$, i.e.
\[
\sigma D_\Sigma+D_\Sigma\sigma=0.
\]
This observation was due to Yoshida \cite{Yoshida91}. A short discussion in the context of the gauged Witten equations can be found in \cite[Subsection 4.2]{Wang202}.

Section \ref{Sec19} and \ref{Sec22} are devoted to the linearization of the Seiberg-Witten map on $\R_t\times \hy$ and $\CX$ respectively. We will study the Fredholm property and the Atiyah-Patodi-Singer boundary value problem following the book \cite[Section 17]{Bible}.

\section{Spectral Flow and Fredholm Index}\label{Sec9}

In the section, we summarize the axioms that characterize the spectral flow. Let us first introduce a few notations before we state the main result: Theorem \ref{T16.1}.

Let $H_0$ be a real separable Hilbert space and $\A_0:H_0\to H_0$ be a self-adjoint operator with domain $W_0\colonequals D(\A_0)$ dense in $H_0$. We assume that $0$ does not lie in the essential spectrum of $\A_0$: 
\begin{equation}\label{E16.1}
0\not\in \sigma_{ess}(\A_0).
\end{equation}

$W_0$ becomes a Hilbert space with respect to the graph norm
\[
\|x\|_{W_0}^2\colonequals \|\A_0 x\|_{H_0}^2+\|x\|_{H_0}^2,\ \forall x\in W_0.
\] 
The inclusion map $W_0\embed H_0$ is \textbf{not} assumed to be compact, so $\sigma_{ess}(\A_0)$ might be non-empty. A pair $(W,H)$ of Hilbert spaces is called admissible if one can find a finite dimensional space $V=\R^n$ such that 
\[
W=W_0\oplus V,\ H=H_0\oplus V. 
\]

A symmetric operator $\A: W\to H$ is called admissible if one can find a symmetric \textbf{compact} operator $K: W\to H$ such that 
\[
\A=\begin{pmatrix}
\A_0 & 0\\
0& 0
\end{pmatrix}+K. 
\] 

By the Kato-Rellich theorem, $\A$ is self-adjoint with domain $D(\A)=W$. Let $\SL_{sym}(W,H)$ be the affine space of all admissible operators between $(W,H)$. It is topologized using the operator norm on the compact perturbation $K$. Let $\SB(\R,W,H)$ be the space of continuous maps $\A:\R\to \SL_{sym}$ such that the limits 
\[
\A^\pm=\lim_{t\to\pm\infty} \A(t): W\to H
\]
exist. The $\SC^k$-distance between two paths $\A_1$ and $\A_2$ is defined as 
\[
d_k(\A_1,\A_2)\colonequals \sup_{t\in \R}\sum_{0\leq j\leq k}\|\frac{d^j}{dt^j}(\A_1(t)-\A_2(t))\|_{W\to H}.
\]
Denote by $\SB^k(\R,W,H)\subset \SB(\R,W,H)$ be the subspace consisting of paths having finite $\SC^k$-distance with a constant path, endowed with $\SC^k$-topology. Note that $\SB^0(\R,W,H)=\SB(\R,W,H)$. Finally, define an open subset 
\[
\SA=\SA(\R,W,H)\colonequals \{\A\in\SB(\R,W,H): \A^\pm\text{ invertible}\}
\]
and set $\SA^k=\SA\cap \SB^k$. Given paths $\A,\A_l,\A_r\in \SA(\R,W,H)$ such that  $\A_l(t)=\A(0)=\A_r(-t),\ t\geq 0$, $\A$ is said to be the catenation of $\A_l$ and $\A_r$ and write
\[
\A=\A_l\#\A_r
\]
 if 
\[
\A(t)=\left\{\begin{array}{ll}
\A_l(t) & \text{ if } t\leq 0\\
\A_r(t)& \text{ if } t\geq 0
\end{array}
\right.
\] 
Given any two reference operators $(\A_{01}, W_{01}, H_{01})$ and $(\A_{02}, W_{02}, H_{02})$ satisfying the condition $\eqref{E16.1}$ and any two paths $\A_i\in \SA(\R,W_i, H_i),\ i=1,2$, one can form the direct sum
\[
\A_1\oplus \A_2\in \SA(\R,W_1\oplus W_2, H_1\oplus H_2). 
\]

Let us now state the axioms that characterize the spectrum flow along a path $\A\in \SA(\R, W,H)$. 

\begin{theorem}[cf. {\cite[Theorem 4.23]{RS95}}]\label{T16.1} For any reference operator $(\A_0, W_0,H_0)$ satisfying the condition \eqref{E16.1} and any finite dimensional auxiliary space $V$, there exists a unique map
	\[
	\mu: \SA(\R,W,H)\to \Z
	\]
	satisfying the following axioms 
	\begin{itemize}
\item $($Homotopy$)$ $\mu$ is constant on the connected components of $\SA(\R,W,H)$;
\item $($Constant$)$ If $\A$ is a constant path, then $\mu(\A)=0$;
\item $($Direct Sum$)$ $\mu(\A_1\oplus \A_2)=\mu(\A_1)+\mu(\A_2)$;
\item$($Catenation$)$ If $\A=\A_r\#\A_r$, then $\mu(\A)=\mu(\A_l)+\mu(\A_r)$;
\item $($Normalization$)$ For $W=H=\R$ and $\A(t)=\arctan(t)$, $\mu(\A)=1$. 
	\end{itemize}

The integer $\mu(\A)$ is called the spectral flow of $\A\in \SA(\R, W, H)$. 
\end{theorem}

\begin{proof} The proof follows the same line of argument as \cite[Theorem 4.23]{RS95}. The idea for existence works as follows. Define
	\[
	\SL_k=\{\A\in \SL_{sym}(W,H): \dim\ker \A=k\},
	\]
	then $\SL_k$ is a smooth Banach submanifold of $\SL_{sym}$ of real co-dimension $k(k+1)/2$. For any path $\A\in \SA$, find a $\SC^1$-path $\A'\in\SA^1$ that is homotopic to $\A$ and intersects each $\SL_k,\ k\geq 1$ transversely. Then $\mu(\A)$ is defined as the algebraic intersection of $\A'$ with $\SL_1$. For details, see \cite{RS95}. 
\end{proof}

There is another way to think of the spectral flow. For any path $\A\in \SA^k$, define the differential operator:
\begin{align*}
D_{\A}: \W_k\colonequals L^2_k(\R, W)\cap L^2_{k+1}(\R, H)&\to L^2_k(\R, H)\\
\xi(t)&\mapsto \dt \xi(t)+\A(t)\xi(t),
\end{align*}
where the $\W_k$-norm is defined as 
\[
\|\xi\|_{\W_k}^2=\int_\R \bigg(\sum_{0\leq j\leq k+1}\|\frac{d^j}{dt^j} \xi\|^2_H+\sum_{0\leq j\leq k}\|\frac{d^j}{dt^j} \xi\|^2_W\bigg) dt \text{ for all } \xi\in C^\infty_0(\R, W). 
\]
\begin{theorem}[cf. \cite{RS95} Theorem 3.12]\label{T16.2}
	For any $k\geq 0$ and any $\A\in \SA^k$ such that 
	\[
	\A(t)\to \A^\pm \text{ in } \SC^k_{loc}\text{-topology as } t\to \pm \infty,
	\]
	then $D_{\A}: \W_k\to L^2_k(\R,H)$ is a Fredholm operator of the index $\mu(\A)$. 
\end{theorem}
\begin{proof} As our situation is slightly simpler than \cite[Theorem 3.12]{RS95}, we present a direct proof using parametrix patching argument. The theorem holds when $\A(t)\equiv \A^+$ is a constant path and $\A^+$ is invertible. Indeed, 
	\begin{align*}
	\|(\dt+\A^+)\xi\|_{L^2_k(\R, H)}^2&=\sum_{0\leq j\leq k}\int_\R \|\frac{d^j}{dt^j}(\dt+\A^+)\xi\|_H^2=\sum_{0\leq j\leq k}\int_\R \|\frac{d^{j+1}}{dt^{j+1}}\xi\|_H^2+\|\A^+(\frac{d^j}{dt^{j}}\xi)\|_H^2\\
	&=\sum_{0\leq j\leq k}\int_\R \|\frac{d^{j+1}}{dt^{j+1}}\xi\|_H^2+\|\frac{d^j}{dt^{j}}\xi\|_W^2\gtrsim \|\xi\|_{\W_k}^2. 
	\end{align*}
	
	In general, let $
	\A^\pm=\lim_{t\to\pm\infty} \A(t)$ be the limiting operators of $\A$ and $Q^\pm: L^2_k(\R, H)\to \W_k$ be the inverse of $D_{\A^\pm}$. Choose cut-off functions $\beta_\pm$ on $\R_t$ such that 
	\begin{itemize}
\item $\beta_-+\beta_+=1$;
\item $\beta_+(t)\equiv 1$ when $t>1$;  $\beta_+(t)\equiv 0$ when $t<-1$.
	\end{itemize}

Take $Q_L=Q^-\beta_-+Q^+\beta_+$ and  $K^\pm=D_{\A}-D_{\A^\pm}=\A-\A^\pm$. We compute:
\begin{align*}
Q_LD_A&=Q^-D_A\beta_-+Q^-[\beta_-, D_A]+ Q^+D_A\beta_++Q^+[\beta_+, D_A]\\
&=\Id_{\W_k}+Q^-(K^-\beta_-)+Q^+(K^+\beta_+)+(Q^+-Q^-)\pt\beta_-\\
&=\Id_{\W_k}+Q^-(K^-\beta_-)+Q^+(K^+\beta_+)-Q^+((\A^+-\A^-)\pt\beta_-)Q^-. 
\end{align*}

(For the right parametrix, take $Q_R=\beta_- Q^-+\beta_+Q^+$).

To show that each error term gives arise to a compact operator, apply the next lemma to operators:
\[
K^-\beta_-, K^+\beta_+ \text{ and }(\A^+-\A^-)\pt\beta_-. 
\]
\begin{lemma}[{\cite[Lemma 3.18]{RS95}}] For any $k\geq 0$, suppose $K(t): W\to H$ is a $\SC^k$-family of compact operators that converges to zero in $\SC^k_{loc}$-topology as $t\to\pm\infty$, i.e. 
\[
\lim_{t\to\pm\infty} \|K(t+\cdot)\|_{\SC^k([-1,1])}=0. 
\]
Then the multiplication operator $K_*: \xi(t)\mapsto K(t)\xi(t) $ is compact from $\W_k$ to $L^2_k(\R, H)$. 
\end{lemma}
\begin{proof} [Proof of the Lemma] We follow the argument of \cite[Lemma 3.18]{RS95}. It suffices to show the operator $\xi(t)\mapsto \frac{d^j}{dt^j} (K(t)\xi(t))$ is compact from $\W_k$ to $L^2(\R, H)$ for any $0\leq j\leq k$. This reduces the problem to the case when $k=0$. 
	
	Let $\Comp(W,H)$ be the space of compact operators from $W$ to $H$. The function $K: \R\to \Comp(W,H)$ can be approximated in $\SC^0$-topology by linear combinations of characteristic functions. Each approximation $K_n$ is a finite sum 
	\[
	\sum_{j=0}^n\chi_{I_j} K_n^{(j)}
	\]
	where $\chi_{I_j}$ is the characteristic function  of a finite interval $I_j\subset \R$ and $K_n^{(j)}\in \Comp(W,H)$ is a compact operator. As $(K_n)_*\to K_*$ in the norm topology, it suffices to prove each $(K_n)_*$ is compact. We reduce to the case when $K=\chi_{I_1} K^{(1)}$ consists of a single term. 
	
	The final step is to approximate $ K^{(1)}$ by a sequence of finite rank operators. When $ K^{(1)}$ is a finite rank operator, $K_*$ is the composition of three operators:
	\[
	\W_0\xrightarrow{K_*} L^2_1(I_1, U)\to L^2(I_1, U)\to L^2(\R, H),
	\]
	where $U=\im  K^{(1)}$ is a finite dimensional real vector subspace of $H$, so the middle map is compact. This completes the proof of the lemma. 
\end{proof}

Back to the proof of Theorem \ref{T16.2}. To prove $\ind(D_\A)=\mu(\A)$, it remains to verify the assignment $\A\mapsto \ind(D_\A)$ satisfies all axioms of spectral flow in Theorem \ref{T16.1} when $k=0$. Only the catenation axiom is not obvious. However, by \cite[Proposition 4.26]{RS95}, the catenation axiom follows from the homotopy, direct sum and constant axioms. This completes the proof of Theorem \ref{T16.2}
\end{proof}

\section{Essential Spectrum}\label{Sec17}

To apply the general theory from the previous section, it is important to verify the condition \eqref{E16.1} for operators of interest. In this section, we discuss a class of model operators following the setup of \cite{Yoshida91}. The main result is Proposition \ref{P17.2}. This general formalism will be applied to the extended Hessians $\EHess$ of $\CSd_\omega$ in the next section. 

Recall that $\hy=Y\cup [-1,\infty)_s\times \Sigma$ is a 3-manifold with cylindrical ends. Suppose $E\to \hy$ is a real vector bundle over $\hy$ such that 
\[
E|_{ [-1,\infty)_s\times \Sigma}=\pi^* E_0
\]
and $E_0\to \Sigma$ is a vector bundle over $\Sigma$. Here $\pi: [-1,\infty)_s\times \Sigma\to \Sigma$ is the projection map. Bundles $E$ and $E_0$ are endowed with Riemannian metrics. We investigate a special class of first order differential operators 
\[
D_Y: C^\infty_0(\hy, E)\to C^\infty_0(\hy, E);
\]
satisfying the following constraints on $D_Y$: 
\begin{itemize}
\item $D_Y$ is elliptic and symmetric with respect to the $L^2$-inner product;
\item $D_Y=\sigma(\ds+D_\Sigma)$ on the cylindrical end $[-1,\infty)_s\times \Sigma$, where
\item $\sigma: E_0\to E_0$ is skew-symmetric bundle map of $E_0\to \Sigma$, i.e. $\sigma+\sigma^*=0$; moreover, $\sigma^2=-\Id_{E_0}$;
\item $D_\Sigma: C^\infty(\Sigma, E_0)\to C^\infty(\Sigma, E_0)$ is a first order self adjoint elliptic differential operator; moreover, $D_\Sigma$ anti-commutes with $\sigma$, i.e. $\sigma D_\Sigma+D_\Sigma \sigma=0$.
\end{itemize}

\begin{example} The simplest example of $D_Y$ is the Dirac operator. Let $E=S$ be the spin bundle and $D_Y=\sum_{1\leq i\leq 3}\rho_3(e_i)\nabla^B_{e_i}$ for some \spinc connection $B$. On the cylindrical end $[-1,\infty)_s\times \Sigma$, we require $B$ to take the form
	\[
	B=\ds+\cB
	\]
	for some \spinc connection $\cB$ on $\Sigma$. Set $\sigma=\rho_3(ds)$ on $[-1,\infty)_s\times \Sigma$.
\end{example}
\begin{proposition}\label{P17.2} Under above assumptions, $D_Y$ is a unbounded self-adjoint operator on $L^2(\hy, E)$ with domain $L^2_1(\hy, E)$. Moreover, the essential spectrum $\sigma_{ess}$ of $D_Y$ is 
	\[
	(-\infty, -\lambda_1]\cup [\lambda_1, \infty)
	\]
	where $\lambda_1$ is the first non-negative eigenvalue of $D_\Sigma$. In particular, if $D_\Sigma$ is invertible, then $0\not\in \sigma_{ess}(D_Y)$. 
\end{proposition}
\begin{remark} Since $D_\Sigma$ anti-commutes with $\sigma$, $-\lambda_1$ is also the first non-positive eigenvalue of $D_\Sigma$. The spectrum of $D_\Sigma$ is symmetric with respect to the origin.
\end{remark}

The proof of Proposition \ref{P17.2} will dominate the rest of this section. To compute the essential spectrum of $D_Y$, we need two additional results from functional analysis: Weyl's criterion and Zhislin's criterion.

\begin{definition}\label{D17.3}Suppose $\A: H\to H$ is a self-adjoint operator with domain $W\colonequals D( \A)\subset H$. For any $\lambda \in \C$,  a sequence $\{u_n\}$ is called a Weyl sequence for $(\A,\lambda)$ if $\{u_n\}\subset W$, $\|u_n\|_H=1$, $u_n\xrightarrow{w} 0$ weakly in $H$ and $(\A-\lambda)u_n\xrightarrow{s} 0$ strongly in $H$. 
\end{definition}
\begin{theorem}[Weyl's Criterion, \cite{HS96} Theorem 7.2] Under the assumption of Definition \ref{D17.3}, $\lambda\in\sigma_{ess}(\A)$ if and only if there exists a Weyl sequence for $(\A,\lambda)$. 
\end{theorem}

When $H=L^2(\hy, E)$, Weyl's criterion can be refined into Zhislin's criterion for locally compact operators.

\begin{definition}Suppose $H=L^2(\hy, E)$ and $\chi_B$ is the characteristic function for a subset $B\subset \hy$. A self-adjoint operator $\A$ on $H$ is called \textbf{locally compact} if the operator $\chi_B(\A-i)^{-1}: H\to H$ is compact for any compact subset $B\subset \hy$. 
\end{definition} 
\begin{definition} Let $Y_n=\{s\leq n\}, n\in \Z_{\geq 0}$ be the truncated 3-manifold. For any $\lambda\in \C$,  a sequence $\{u_n\}\subset W$ is called \textbf{a Zhislin sequence} for $(\A,\lambda)$ if $\|u_n\|_H=1$, $\supp( u_n)\subset Y_n^c$ and $(\A-\lambda)u_n\xrightarrow{s} 0$ in $H$. 
\end{definition} 

As $u_n$ is supported on the complement of $Y_n$, $u_n\xrightarrow{w} 0$. As a result, a Zhislin sequence is always a Weyl sequence. 

\begin{theorem}[Zhislin's Criterion, \cite{HS96} Theorem 10.6]\label{T17.7} Suppose $H=L^2(\hy, E)$ and $\A: H\to H$ is self-adjoint and locally compact. If $\A$ satisfies the commutator estimate:
	\begin{equation}\label{E17.1}
	\|[\A,\varphi_n](A-i)^{-1}\|_{H\to H}\to 0 \text{ as }n\to\infty,
	\end{equation}
	where $\varphi_n=\varphi(s(\cdot)/n)$ and $\varphi: \R\to \R$ is some cut-off function such that $\varphi(r)\equiv 1$ when $r\leq 1$ and $\varphi(r)\equiv  0$ when $r\geq 2$, then
	$\lambda\in\sigma_{ess}(\A)$ if and only if there exists a Zhislin sequence for $(\A,\lambda)$. 
\end{theorem}
\begin{proof}[Idea of the Proof] The ``if" part follows from Weyl's Criterion. Suppose $\lambda\in \sigma_{ess}(\A)$ and $\{u_m\}$ is a Weyl sequence for $(\A,\lambda)$. We wish to construct a Zhislin sequence for $(\A,\lambda)$ out of $\{u_m\}$. For any $n\in \Z_{\geq 0}$, choose a large number $m(n)$ and define
	\[
	v_n=(1-\varphi_n)u_{m(n)}. 
	\]
	First of all, $(\A-i)u_m=(\A-\lambda)u_m+(\lambda-i)u_m\xrightarrow{w} 0$ as $m\to\infty$. Because $\varphi_n(\A-i)^{-1}$ is compact, $\varphi_nu_m=\phi_n(A-i)^{-1}\circ (A-i)u_m\xrightarrow{s} 0$ as $m\to \infty$ for any fixed $n$. By taking $m(n)\gg n$, we ensure that $\|v_n\|_H\geq \half$.  
	
	The second step is to use the commutator estimate \eqref{E17.1} to prove $(\A-\lambda)v_n\xrightarrow{s} 0$ as $n\to\infty$. Finally, $\{v_n/\|v_n\|_H\}$ is the desired Zhislin sequence. For details, see \cite[Theorem 10.6]{HS96}
\end{proof}

\begin{remark} Zhislin's Criterion shows that the essential spectrum of $\A$ is determined completely by its behavior along the cylindrical end $[0,\infty)_s\times \Sigma$.
\end{remark}
\begin{proof}[Proof of Proposition \ref{P17.2}] $D_Y$ is a locally compact operator as $\chi_B(D_Y-i)^{-1}: L^2(\hy)\to L^2(\hy)$ factorizes through $L^2_1(B)$ when $B=Y_n$. The commutator estimate is also satisfied as 
	\[
	[D_Y, \varphi_n]=\frac{1}{n}\cdot \frac{d\varphi}{dr}(\frac{s}{n})\rho(ds)
	\]
	and its $L^\infty$-norm decays to zero. Applying Zhislin's criterion, we reduce to the case when $\hy=\R_s\times \Sigma$ is a cylinder and 
	\[
	D_Y=\sigma(\ds+D_\Sigma). 
	\]

To study the spectrum of $D_Y$ in this case, apply Fourier transformation in $\R_s$-direction. Our goal is to find eigenvalues of 
\[
\widehat{D_Y}(\xi)=\sigma(i\xi+D_\Sigma): \Gamma(\Sigma, E_0)\to \Gamma(\Sigma, E_0)
\]
for any fixed $\xi\in \R_\xi$. Let $\phi_\lambda$ be an eigenvector of $D_\Sigma$ with eigenvalue $\lambda>0$. As $D_\Sigma$ anti-commutes with $\sigma$, $-\lambda$ is also an eigenvalue; indeed,
\[
D_\Sigma(\sigma(\phi_\lambda))=-\lambda\sigma(\phi_\lambda). 
\]
As a result, $\spn_\C\{\phi_\lambda, \sigma(\phi_\lambda)\}$ is an invariant subspace of $\widehat{D_Y}(\xi)$:
\[
\widehat{D_Y}(\xi)=\begin{pmatrix}
0 & -1\\
1 & 0
\end{pmatrix}
\begin{pmatrix}
\lambda+i\xi & 0\\
0& -\lambda+i\xi
\end{pmatrix}=
\begin{pmatrix}
0 & \lambda-i\xi\\
\lambda+i\xi & 0
\end{pmatrix}
\]
whose eigenvalues are $\pm\sqrt{\xi^2+\lambda^2}$. Let $\phi_\lambda^\pm(\xi)$ be their associated eigenvectors respectively and set 
\[
\phi_n(s)\colonequals(\varphi(s-2n)-\varphi(s-n))\phi_\lambda^\pm(\xi)\exp(i\xi s). 
\] 
where $\varphi: \R\to \R$ is the cut-off function defined in Theorem \ref{T17.7}. Then $\{\phi_n/\|\phi_n\|_2\}$ is a Zhislin sequence for $(D_Y, \pm\sqrt{\xi^2+\lambda^2})$, and $\pm\sqrt{\xi^2+\lambda^2}\in \sigma_{ess}(D_Y)$ by Theorem \ref{T17.7}. 

When $\lambda'\in (-\lambda_1,\lambda_1)$, $(\widehat{D_Y}(\xi)-\lambda')$ is invertible for each $\xi\in \R_\xi$; their inverses are uniformly bounded. As a result, the operator
\[
D_Y-\lambda'
\]
is invertible, so $\lambda'\not\in \sigma_{ess}(D_Y)$. This completes the proof of Proposition \ref{P17.2}.
\end{proof}

\section{Extended Hessians}\label{Sec18}

In this section, we apply the abstract formalisms in Section \ref{Sec17} to the extended Hessians of $\CSd_{\omega}$ and compute its essential spectrum. The main result is Proposition \ref{P18.1}. The proof relies on the key observation from the first paper \cite[Proposition 7.4]{Wang202}: the Seiberg-Witten equations on $\C\times \Sigma$ is secretly the gauged Witten equations on $\C$. The structural results from \cite[Subsection 4.2]{Wang202} then becomes essential here. The formalism from Section \ref{Sec17} in fact applies to any gauged Witten equations. 

\medskip

Recall from Section \ref{Sec10} that the quotient configuration space 
\[
\CB_k(\hy,\bs)=\SC_k(\hy,\bs)/\CG_{k+1}(\hy)
\]
is a Hilbert manifold when $k>\half$. For any $\gamma\in \SC_k(\hy,\bs)$, denote by $[\gamma]$ its gauge equivalent class in $\CB_k(\hy, \bs)$. By Lemma \ref{L10.4}  the tangent space of $\SC_k(\hy,\bs)$ at $\gamma$ admits a decomposition: $$\CT_{k,\gamma}\colonequals T_\gamma \SC_k(\hy,\bs)=\J_{k,\gamma}\oplus \K_{k, \gamma}$$ 
 where 
\begin{align*}
\J_{k,\gamma}&=\im(\dg: L^2_{k+1}(\hy,i\R)\to \CT_{k,\gamma}) \text{ and }\\
 \K_{k, \gamma}&=\ker (\dg^*: \CT_{k,\gamma}\to L^2_{k-1}(\hy,i\R))
\end{align*}
form $L^2$-complementary sub-bundles of $\CT_k\to \SC_k(\hy,\bs)$. Moreover, 
$$
T_{[\gamma]}\CB_k(\hy,\bs)=\K_{k, \gamma}.
$$

Take a tame perturbation $\q=\grad f\in \Pa$. As the perturbed Chern-Simons-Dirac functional $\CSd_\omega=\CL_{\omega}+f$ is invariant under the identity component of $\CG_{k+1}(\hy)$, its gradient 
\[
\grad \CSd_{\omega}=\grad\CL_{\omega}+\q
\]
defines a smooth section of $\K_{k-1}\to \SC_k(\hy,\bs)$ and its Hessian is a symmetric bundle map:
\[
\D \grad \CSd_{\omega}: \CT_k\to \CT_{k-1}
\]
which is equivariant under the action of $\CG_{k+1}(\hy)$. As $\SC_k(\hy,\bs)$ is an affine space, the tangent bundle $\CT_{k}\to \SC_k(\hy,\bs)$ is endowed with the trivial flat connection, but the decomposition $\CT_k=\J_k\oplus \K_k$ is not parallel. Consider the composition of maps:
\[
\Hess_{\q}\colonequals \Pi_{\K_{k-1}}\circ \D\grad \CSd_{\omega}:  \K_k\to \K_{k-1},
\]
and write $ \D\grad \CSd_{\omega}$ into a block form: 
\begin{equation}\label{E18.1}
\D \grad \CSd_{\omega}=\begin{pmatrix}
y & x\\
x^* & \Hess_{\q}
\end{pmatrix}: \J_k\oplus \K_k\to \J_{k-1}\oplus \K_{k-1},
\end{equation}
where $x=\Pi_{\J_{k-1}}\circ \D\grad \CSd_{\omega}|_{\K_k}$ and $y=\Pi_{\J_{k-1}}\circ \D\grad \CSd_{\omega}|_{\J_k}$. Note that 
\[
x=0, y=0
\]
when $\gamma\in \Crit(\CSd_{\omega})$ is a critical point. Here is the another way to think of $\Hess_{\q}$. $\CSd_{\omega}$ descends to a circle valued functional $\overline{\CSd_{\omega}}$ on the quotient configuration space $\CB_k(\hy,\bs)$. The Hessian of $\overline{\CSd_{\omega}}$ at $[\gamma]\in \CB_k(\hy,\bs)$ regarded as a map
\[
\K_{k, \gamma}=T_{[\gamma]}\CB_k(\hy,\bs)\to \K_{k-1, \gamma}
\]
is precisely given by $\Hess_{\q}$. However, $\Hess_{\q}$ is not the convenient notion to work with from the gauge theoretic point of view. One looks instead at \textbf{the extended Hessian} $\widehat{\Hess}_\q$ of $\CSd_{\omega}$ whose expression at $\gamma\in \SC_k(\hy,\bs)$ is defined by 
\[
\widehat{\Hess}_{\q,\gamma}\colonequals \begin{pmatrix}
0& \dg^*\\
\dg & \D_\gamma \grad\CSd_{\omega} 
\end{pmatrix}:  L^2_{k}(\hy, i\R)\oplus \CT_{k,\gamma}\to  L^2_{k-1}(\hy,i\R)\oplus \CT_{k-1,\gamma}. 
\]
\begin{proposition}[cf. \cite{Bible} Proposition 12.3.1]\label{P18.1} The operator $\Hess_{\q,\gamma}: \K_k\to \K_{k-1}$ is symmetric. If $\gamma$  is a critical point of $\CSd_{\omega}$, then it is invertible if and only if the extended Hessian $\widehat{\Hess}_{\q,\gamma}$ at $\gamma$ is invertible. Moreover, the spectrum of $\EHess_{\q,\gamma}$ is real and 
	\[
	\sigma_{ess}(\widehat{\Hess}_{\q,\gamma})=(-\infty, -\lambda_1]\cup [\lambda_1,\infty)
	\]
	where $\lambda_1>0$ is a positive number depending only on the boundary data $(g_\Sigma,\lambda,\mu)$ of $\y\in \Cob_s$. In particular, $\widehat{\Hess}_{\q,\gamma}$ is a Fredholm operator of index $0$ for any $k\geq 1$. 
\end{proposition}
\begin{definition}\label{D18.2} A critical point $\gamma\in \SC_k(\hy, \bs)$ of the perturbed Chern-Simons-Dirac functional $\CSd_{\omega}=\CL_{\omega}+f$ is called \textbf{non-degenerate} if the extended Hessian $\widehat{\Hess}_{\q,\gamma}$ at $\gamma$ is invertible. 
\end{definition}

The proof of Proposition \ref{P18.1} will dominate the rest of this section. 

\begin{proof}[Proof of Proposition \ref{P18.1}] We focus on the essential spectrum of $\EHess_{\q,\gamma}$; the rest of statements follows from the same line of argument of \cite[Proposition 12.3.1]{Bible}. 
	
Let $\gamma_0=(B_0,\Psi_0)$ be the reference configuration of $\SC_k(\hy,\bs)$. Then $\gamma-\gamma_0=(b,\psi)\in L^2_k(\hy,iT^*\hy\oplus S)$ and 
\[
\EHess_{\q,\gamma}=\EHess_{0,\gamma_0}+h(b,\psi)+\begin{pmatrix}
0 & 0\\
0 & \D_\gamma\q
\end{pmatrix}.
\]
where $h(b,\psi)$ is an operator that involves only point-wise multiplication of $(b,\psi)$. When $g\in L^2_k(\hy)$ is fixed, the Sobolev multiplication 
\begin{align*}
L^2_k(\hy)\times L^2_k(\hy)&\to L^2_{k-1}\\
(f,g)&\mapsto fg
\end{align*}
is a compact operator in the first argument when $k\geq 1$ (see \cite[Theorem 13.2.2]{Bible}), so the error $h(b,\psi)$ is compact. As $\q$ is tame, by property \ref{A4}, $\D_\gamma\q: L^2_k\to L^2_k$ is bounded linear. In addition, since its image is supported on $Y\subset \hy$, the operator $\D_\gamma\q: L^2_k\to L^2_{k-1}$ is also compact. 

By the Kato-Rellich theorem, the essential spectrum is invariant under compact perturbations. It suffices to compute the essential spectrum of $\EHess_{0,\gamma_0}$. The general theory from Section \ref{Sec17} applies here, so we may concentrate on the special case when $\hy=\R_s\times \Sigma$ is a cylinder and $\gamma_0=(B_*,\Psi_*)$ is the $\R_s$-translation invariant solution defined by \eqref{E2.6}. 

\medskip

At this point, we have to recall some results \cite[Subsection 4.2]{Wang202}. The extended Hessian $\EHess_{\gamma_0}$ can be cast into the form $\sigma(\ps+\hat{D}_{\kappa})$ as an operator
\[
L^2_1(\hy, i\R\oplus (i\R\otimes ds)\oplus iT^*\Sigma\oplus S)\to L^2(\hy, i\R\oplus (i\R\otimes ds)\oplus iT^*\Sigma\oplus S)
\]
with 
\begin{equation}\label{E12.2}
\sigma=\begin{pmatrix}
0 & 1 & 0 &0\\
-1 & 0 & 0 & 0\\
0 & 0 & *_\Sigma & 0\\
0 & 0 & 0 & \rho_3(ds)
\end{pmatrix} 
\end{equation}
and $\hat{D}_\kappa$ defined as in \cite[P.36]{Wang202}. It is shown in \cite[Proposition 7.10]{Wang202} that $\hat{D}_\kappa$ is an invertible operator. Now we use Proposition \ref{P17.2} to conclude.
\end{proof}

\section{Linearized Operators on Cylinders}\label{Sec19}

In this section, we study the Seiberg-Witten moduli space on the cylinder $\R_t\times \hy$ and prove the Fredholm property of the linearized operator using the formalism of Section \ref{Sec9}. In Subsection \ref{Subsec19.2}, we will prove a separating property of the cokernel of the linearized operator, which will be crucial in the proof of transversality in Theorem \ref{T21.1}.

We have to justify that the proof of gluing theorem in \cite[Section 18, 19]{Bible} continue to work in our case, in the presence of essential spectra. This is done in Subsection \ref{Subsec19.3} and \ref{Subsec19.4}, where the relevant Atiyah-Patodi-Singer boundary value theory is also developed. 

\subsection{Linearized Operators}\label{Subsec19.1}
Here is the second reason why the extended Hessian is a natural object: it is more consistent with the 4-manifold theory. Suppose 
\[
\fa,\fb\in \Crit(\CSd_{\omega})\subset \SC(\hy,\bs)
\]
are non-degenerate critical points of the perturbed Chern-Simons functional $\CSd_{\omega}$ in the sense of Definition of \ref{D18.2}. To describe the moduli space of flowlines from $\fa$ to $\fb$, we fix a smooth configuration $\gamma$ on $\hz\colonequals \R_t\times \hy$ such that $\gamma$ is in the temporal gauge and 
\begin{align*}
\cgamma(t)&=\fa \text{ if } t<-1,\\
\cgamma(t)&=\fb \text{ if } t>1.  
\end{align*}

Consider the configuration space 
\[
\SC_k(\fa,\fb)=\{(A,\Phi)=\gamma_0+(a,\phi): (a,\phi)\in L^2_k(\hz, iT^*\hz\oplus S^+) \}.
\]
and the gauge group
\[
\CG_{k+1}(\hz)=\{u:\hz\to S^1: u-1\in L^2_{k+1}(\hz, \C) \}. 
\]

We are interested in solutions of the perturbed Seiberg-Witten equations on $\hz$:
\begin{equation}\label{E19.1}
0=\F_{\hz,\q}(\gamma)\colonequals\F_{\hz} (\gamma)+\hq(\gamma),
\end{equation}
where $\F_{\hz}$ is defined by \eqref{4DSWEQ} and $\hq$ is defined as in \eqref{E14.2}. We form the moduli space 
\[
\M_k(\fa,\fb)\colonequals \{\gamma\in \SC_k(\fa,\fb):\F_{\hz,\q}(\gamma)=0\}/ \CG_{k+1}(\hz). 
\]

We focus on the linearized theory of the moduli space in this section. Take a configuration $\gamma=(A,\Phi)\in \SC_k(\fa,\fb)$, then a tangent vector $V$ at $\gamma$ is a section 
\[
(\delta c(t), \delta b(t), \delta\psi(t))\in L^2_k(\R_t\times \hy, i\R\oplus iT^*\hy\oplus S). 
\]
It lies in the kernel of the linearized operator $\D_\gamma \F_{\hz,\q}$ (i.e. the tangent map) of $\F_{\hz,\q}$ if and only if it solves the equation 
\begin{equation}\label{E19.2}
\dt\begin{pmatrix}
\delta b(t)\\
\delta \psi(t)
\end{pmatrix}+\D_{\cgamma(t)}\grad \CSd_{\omega}\begin{pmatrix}
\delta b(t)\\
\delta \psi(t)
\end{pmatrix}+\bd_{\cgamma(t)} \delta c(t)=0,\ \forall t\in \R.
\end{equation}
\eqref{E19.2} is obtained by formally linearizing the equation \eqref{E15.3}. The convention of \eqref{E14.1} is also adopted here: $\cgamma(t)$ stands for the underlying path in $ \SC(\hy, \bs)$. 

On the other hand, the linearized action of $\CG(\hz)$ at $\gamma$ is given by:
\begin{align*}
\bd_\gamma: \Lie(\CG_{k+1}(\hz))=L^2_{k+1}(\hz,i\R)&\to T_\gamma \SC(\fa,\fb)\\
f(t)&\mapsto (-\dt f(t), \bd_{\cgamma(t)} f(t)),
\end{align*}
whose $L^2$-formal adjoint is 
\begin{align*}
\bd_\gamma^*: T_\gamma \SC(\fa,\fb)&\to L^2_{k-1}(\hz,i\R)\\
V(t)=(\delta c(t),\delta b(t), \delta\psi(t))&\mapsto \dt \delta c(t)+\bd_{\cgamma(t)}^*\begin{pmatrix}
\delta b(t)\\
\delta \psi(t)
\end{pmatrix}.  
\end{align*}

It follows that $\D_\gamma\F_{\hz,\q}$ together with the linearized gauge fixing operator $\bd_\gamma^*$ can be cast into the form:
\begin{equation}\label{E19.3}
V(t)\mapsto \dt V(t)+\EHess_{\q,\cgamma(t)} V(t),
\end{equation}
for $V(t)= (\delta c(t), \delta b(t), \delta\psi(t))$. By Theorem \ref{T16.2}, we have 
\begin{proposition}\label{P19.1}For any $\gamma\in \SC_k(\fa,\fb)$, the operator 
	\[
	(\bd_\gamma^*, \D_\gamma\F_{\hz,\q} ): L^2_k(\hz, i\R\oplus iT^*\hy\oplus S)\to L^2_{k-1}(\hz, i\R\oplus iT^*\hy\oplus S)
	\]
	is Fredholm. Its index is independent of $\gamma$ and equals the spectrum flow from $\EHess_{\q,\fa}$ to $\EHess_{\q,\fb}$. 
\end{proposition}
\begin{proof} The operator $	(\bd_\gamma^*, \D_\gamma\F_{\hz,\q} )$ differs from the $	(\bd_{\gamma_0}^*, \D_{\gamma_0}\F_{\hz,\q} )$ by a compact term. When $\gamma=\gamma_0$ is the reference configuration, apply Theorem \ref{T16.2}.
\end{proof}

\begin{definition}\label{D19.2} The moduli space $\M_k(\fa,\fb)$ is called regular if the linearized operator $(\bd_\gamma^*, \D_\gamma\F_{\hz,\q} )$ at $\gamma$ is surjective for any $[\gamma]\in \M_k(\fa,\fb)$. 
\end{definition}
\begin{definition}\label{D19.3} A tame perturbation $\q=\grad f\in \Pa$ is called admissible if 
	\begin{enumerate}[label=(E\arabic*)]
		\item\label{E1} all critical points of the perturbed Chern-Simons-Dirac functional $\CSd_{\omega}=\CL_{\omega}+f$ are non-degenerate in the sense of Definition \ref{D18.2};
		\item\label{E2} for any pair of critical points $\fa,\fb\in \Crit(\CSd_{\omega})$, the moduli space $\M_k(\fa,\fb)$ is regular in the sense of Definition \ref{D19.2}.\qedhere
	\end{enumerate}
\end{definition}

One may think of $\M_k(\fa,\fb)$ as the moduli space of down-ward gradient flowlines in the quotient space $\CB_{k-1/2}(\hy,\bs)$. The reference configuration $\gamma_0$ determines a homotopy class of paths connecting $[\fa]$ and $[\fb]$, so it is more appropriate to write
\begin{equation}\label{E19.11}
\M_{[\gamma]}([\fa],[\fb])\colonequals \M_k(\fa,\fb),\ [\gamma]\in \pi_1(\CB_{k-1/2}(\hy,\bs),[\fa],[\fb]). 
\end{equation}
By Theorem \ref{T1.4}, this space is independent of the Sobolev completion that we choose, so the subscript $k$ is dropped in our notation. 
\begin{remark} To identify a finite energy solution $\gamma$ in Theorem \ref{T1.4} with an element of $\M_k(\fa,\fb)$, we have to know the exponential decay of $\gamma$ in the time direction using the non-degeneracy of critical points, which is omitted in this paper; cf. Remark \ref{R1.4}. 
\end{remark}

Since the Seiberg-Witten equations on $\hz=\R_t\times \hy$ has an apparent $\R_t$-translation symmetry, $\M_{[\gamma]}([\fa],[\fb])$ is acted on freely by $\R_t$ if the topological energy $\E_{top}$ along the path 
\[
[\gamma]\in \pi_1(\CB_{k-1/2}(\hy,\bs),[\fa],[\fb])
\]
is positive. We form the unparameterized moduli space by taking the quotient space
\begin{equation}\label{E19.13}
\cM_{[\gamma]}([\fa],[\fb])\colonequals \M_{[\gamma]}([\fa],[\fb])/\R_t. 
\end{equation}
When $\q$ is admissible, $\cM_{[\gamma]}([\fa],[\fb])$ is a smooth manifold of dimension $\Ind 	(\bd_\gamma^*, \D_\gamma\F_{\hz,\q} )-1$.

\subsection{Sections in the Cokernel}\label{Subsec19.2} Our ultimate goal is to show that admissible perturbations, in the sense of Definition \ref{D19.3}, are generic, cf. Theorem \ref{T21.1}. To do this, we have to understand sections in the cokernel of $(\bd_\gamma^*, \D_\gamma\CF_{\hz, \q})$, when it is not surjective.

Suppose $U\in L^2(\hz, i\R\oplus iT^*\hy\oplus S)$ is $L^2$-orthogonal to the image of  $(\bd_\gamma^*, \D_\gamma\F_{\hz,\q} )$ at a solution $[\gamma]\in \M_k(\fa,\fb)$, then $U$ solves the equation 
\begin{equation}\label{E19.4}
-\dt U(t)+\EHess_{\q,\cgamma(t)} U(t)=0 \text{ by }\eqref{E19.3}.  
\end{equation}
By elliptic regularity, $U$ is smooth and $U\in L^2_1$. We write $U$ as 
\[
U(t)=(\delta c'(t),\delta b'(t), \delta \psi'(t)).
\]
The proof of Theorem \ref{T21.1} in Section \ref{Sec21} relies on a separating property of the section $U$:
\begin{lemma}\label{L19.4} Under above assumptions,  $\delta c'(t)\equiv 0$. Moreover, if $U\neq 0$ and $\cgamma(t)$ is never reducible on $\{t\}\times Y$, then there exists a time slice $t_0\in \R$ such that the tangent vector $(\delta b'(t_0), \delta\psi'(t_0))$ at $\cgamma(t_0)$ can be separated by a cylinder function $f$ tame in $Y$. Here, $Y=\{s\leq 0\}\subset \hy$ is the truncated 3-manifold. 
\end{lemma}
\begin{remark} By the unique continuation property, cf. Theorem \ref{T20.10} below, if  $\cgamma(t)$ is reducible at some slice $\{t\}\times Y$, then a solution $\gamma\in \SC_k(\fa,\fb)$ has to reducible globally, which is absurd. So the condition of Lemma \ref{L19.4} is fulfilled. 
\end{remark}

\begin{proof}[Proof of Lemma \ref{L19.4}] Consider a smooth function $\xi\in L^2_{k+1}(\hz, i\R)$ and the section 
	\[
	V_\xi=(0,	\bd_\gamma \xi)\in L^2_{k}(\hz, i\R\oplus iT^*\hy\oplus S).
	\]
	Since $e^{r\xi}\cdot \gamma$ also solves the equation $\F_{\hz, \q}=0$ for any $r\in \R$, taking the derivatives yields
	\[
	\D_\gamma\F_{\hz,\q} (\bd_\gamma \xi)=0,
	\]
	so the vector 
	\[
	(\bd_\gamma^*, \D_\gamma\F_{\hz,\q} )V_\xi=(\bd_\gamma^*\bd_\gamma \xi, 0,0)\in L^2(\hz, i\R\oplus iT^*\hy\oplus S)
	\]
	is $L^2$-orthogonal to $U$. Since the composition $\bd_\gamma^*\bd_\gamma: L^2_2(\hz, i\R)\to L^2(\hz,i\R)$ is an invertible operator and $L^2_{k+1}$ is dense in $L^2_2$, $\delta c'(t)=0$. Now \eqref{E19.4} is reduced to a pair of equations:
	\begin{align}
	0&=\bd_{\cgamma(t)}^*	(\delta b'(t), \delta \psi'(t)),\label{E19.5}\\
	\dt 	(\delta b'(t), \delta \psi'(t))&=\D_{\cgamma(t)} \grad \CSd_{\omega} 	(\delta b'(t), \delta \psi'(t)).\label{E19.6} 
	\end{align}
	
For the second clause of Lemma \ref{L19.4}, suppose on the contrary that $U(t)$ can not be separated for any $t\in \R_t$. By Proposition \ref{P15.6}, we can find a function $\xi(t)\in L^2_1(\hy, i\R)$ such that 
	\[
	(\delta b'(t), \delta \psi'(t))=\bd_{\cgamma(t)} \xi(t) =(-d_{\hy} \xi(t), \xi(t)\Psi(t)) \text{ on } \{t\}\times Y
	\]
	for each $t\in \R_t$. If we write $\grad \CSd_{\omega}$ as 
	\[
	(\grad \CSd_{\omega}^0, \grad^1 \CSd_{\omega})\in L^2(\hy, iT^*\hy\oplus S),
	\]
	then 
	\[
	\grad \CSd_{\omega}(u\cdot \cgamma)=(\grad^0 \CSd_{\omega}(\cgamma), u\cdot \grad^1 \CSd_{\omega}(\cgamma)). 
	\]
	In particular, 
	\[
	\D_{\cgamma(t)} \grad  \CSd_{\omega} (\bd_{\cgamma(t)} \xi(t))=(0, \xi(t)\cdot \grad^1\CSd_{\omega}(\cgamma(t)). 
	\]
	Even though $\bd_{\cgamma(t)} \xi(t)$ and $	(\delta b'(t), \delta \psi'(t))$ only agree over $\{t\}\times Y$, we still have 
	\[
	\D_{\cgamma(t)} \grad  \CSd_{\omega}	(\delta b'(t), \delta \psi'(t))=(0, \xi(t)\cdot \grad^1\CSd_{\omega}(\cgamma(t)) \text{ on } \{t\}\times Y,
	\]
	since the perturbation $\q$ is supported on $Y$ in the sense of Definition \ref{D14.1}. The equation \eqref{E19.6} then implies 
	\[
	\dt \delta b'\equiv 0 \text{ on } \R_t\times Y. 
	\] 
	As $U\in L^2$, $-d_{\hy}\xi(t)=\delta b'(t)\equiv 0$. Now the equation \eqref{E19.5} yields
	\[
	0=\Delta_{\hy} \xi(t)+|\Psi(t)|^2\xi(t)=|\Psi(t)|^2\xi(t) \text{ on } \{t\}\times Y,
	\]
	As a result, $U\equiv 0$ on $\R_t\times Y$. An elliptic operator of the form \eqref{E19.4} satisfies the unique continuation property, so $U\equiv 0$ on $\R_t\times \hy$. 
\end{proof}

\subsection{Spectral Projections}\label{Subsec19.3}  Having discussed the linearized operator on an infinite cylinder $\R_t\times \hy$, we start to look at a finite interval $I=[t_1,t_2]\subset \R_t$ and consider the Atiyah-Patodi-Singer boundary-value problems. As noted in the beginning of Section \ref{Sec19}, we have to justify that the proof of gluing theorem in \cite[Section 18,19]{Bible} remains valid in our case, in the presence of essential spectra. This subsection is devoted to an abstract formalism, while the application in gauge theory will be explained in Subsection \ref{Subsec19.4}. However, the results in these subsections will \textbf{not} be used elsewhere in this paper. 

\medskip

As we will work in a slightly abstract setting, define
\[
E_0\colonequals i\R\oplus iT^*\hy\oplus S\to \hy
\]
Take a reference operator $\A_0$ that acts on sections of $E_0$, extending to bounded linear operators
\[
\A_0: L^2_j(\hy, E_0)\to L^2_{j-1}(\hy, E_0). 
\]
for any $j\geq 1$. Moreover,  assume that $\A_0$ is a unbounded self-adjoint operators on $L^2$ and its spectrum is disjoint from the interval $(-\lambda_1/2, \lambda_1/2)$:
\begin{align}\label{E19.12}
\sigma(\A_0)&\subset (-\infty,-\lambda_1/2]\cap [\lambda_1/2,\infty) \text{ with }\\
\sigma_{ess}(\A_0)&=(-\infty,-\lambda_1]\cap [\lambda_1,\infty),\nonumber
\end{align}
for some $\lambda_1>0$ as in Proposition \ref{P18.1}. One may think of $\A_0$ as a first-order self-adjoint elliptic differential operator plus a compact perturbation. For convenience, suppose the $L^2_j$-norm on $C^\infty_c(\hy, E_0)$ is defined using $\A_0$:
\[
\|s\|_{L^2_j(E_0)}=\| (1+|\A_0|)^j s\|_{L^2(E_0)}, \forall s\in C^\infty(\hy, E_0).
\]

Let $K: C^\infty_c(\hy, E_0)\to C^\infty(\hy,E_0)$ be an operator acting on sections of $E_0$ extending to a compact operator:
\[
K: L^2_{j}(\hy,E_0)\to L^2_{j}(\hy, E_0)
\]
for any $j\geq 0$. Assume that $K$ is self-adjoint on $L^2(\hy, E_0)$. When the sum $\A\colonequals \A_0+K$ is invertible, $L^2(\hy, E_0)$ is the direct sum of the positive and negative spectral spaces of $\A$:
\[
L^2(\hy, E_0)=H^+_{\A}\oplus H^-_{\A}, 
\]
and for any $j\geq 0$, 
\begin{equation}\label{E19.7}
L^2_j(\hy,E_0)=(H^+_{\A}\cap L^2_j)\oplus (H^-_{\A}\cap L^2_j). 
\end{equation}

Let $E\to \hz\colonequals (-\infty, 0]\times \hy$ be the pull-back bundle of $E_0$ over the half cylinder and consider the operator: 
\[
D_\A=\dt+\A: C^\infty(\hz, E)\to C^\infty(\hz, E). 
\]

The next result is a direct consequence of Functional Calculus, cf. \cite[Theorem 17.1.4]{Bible}.
\begin{proposition}\label{P19.5} Let $\hz=(-\infty,0]\times \hy$ be the half cylinder. Suppose the operator
	\[
	\A=\A_0+K: L^2_1(\hy, E_0)\to L^2(\hy, E_0)
	\]
	is invertible, then the operator
\[
D_\A\oplus \Pi^{-}_{\A}\circ r: L^2_k( \hz, E)\to  L^2_{k-1}(\hz, E)\oplus (H^-_\A\cap L^2_{k-1/2}(\hy, E_0))
\]
is also invertible for any $k\geq 1$, where $r: L^2_k(\hz,E)\to L^2_{k-1/2}(\hy, E)$ is the restriction map at the boundary $\{0\}\times\hy$ and 
\[
\Pi_\A^-: L^2_{k-1/2}\to H^-_\A\cap L^2_{k-1/2}(\hy, E_0)
\]
is the spectral projection. The subspace $ H^-_\A\cap L^2_{k-1/2}$ is precisely the image of $\ker D_\A$ under $r$.
\end{proposition}

As $\A$ differs from $\A_0$ only by a compact operator, it is expected that $\Pi_\A^-$ forms a ``compact" family as $\A$ varies. We make this precise in the next proposition. 

\begin{proposition}\label{P19.6} Given an invertible operator $\A=\A_0+K$, the difference of their spectral projections 
	\[
	\Pi_{\A}^--\Pi_{\A_0}^-: L^2_{k-1/2}(\hy, E_0)\to L^2_{k-1/2}(\hy, E_0)
	\]
	is compact for any $k\geq 1$, i.e. $\Pi_{\A_0}$ and $\Pi_{\A}$ are $k$-commensurate in the sense of \cite[Definition 17.2.1]{Bible}.
\end{proposition}
\begin{proof} We follow the trick from \cite[Proposition 17.2.4]{Bible}. It suffices to show for any bounded sequence $\{w_i\}\subset L^2_{k-1/2}$, it image under $\Pi_{\A}^--\Pi_{\A_0}^-$ contains a converging subsequence.  In terms of the decomposition \eqref{E19.7}, we can deal with entries of $\{w_i\}$ separately. By the symmetry of $H^\pm_\A$, we focus on the case when $\{w_i\}\subset H^-_\A\cap L^2_{k-1/2}$. By Proposition \ref{P19.5}, there exists sections $\{v_i\}\subset L^2_k(\hz, E)$ such that 
	\[
	D_{\A} v_i=0 \text{ and } r(v_i)=w_i. 
		\]
		
	Apply Proposition \ref{P19.5} again for $\A_0$ to find solutions $\{u_i\}\subset L^2_k(\hz, E)$ with 
	\[
	D_{\A_0} u_i=-K(v_i) \text{ and } \Pi_{\A_0}^-\circ r(u_i)=0. 
	\]
Since $D_{\A_0}(u_i-v_i)=0$, $r(u_i-v_i)\in H_{\A_0}^-$. So
	\begin{align*}
		( \Pi_{\A}^-- \Pi_{\A_0}^-)(w_i)=(1-\Pi_{\A_0}^-)(w_i)=\Pi_{\A_0}^+\circ r(v_i)=\Pi_{\A_0}^+\circ r(u_i). 
	\end{align*}
	
	One may write the last term explicitly in terms of $v_i$ using formulae on \cite[P.299]{Bible}:
	\begin{equation}\label{E19.8}
v_i\mapsto	\Pi_{\A_0}^+\circ r(u_i)=y_i\colonequals \int^0_{-\infty} e^{t\A_0} (-K(v_i(t)))^+dt. 
	\end{equation}
	where $(\cdot)^+$ denotes the positive part in $H^+_{\A_0}$. As this point, approximate $K$ by finite rank operators. The operator $v\mapsto y $ defined by the expression \eqref{E19.8} is also approximated by finite rank operators in the norm topology, so \eqref{E19.8} is also compact.
	
	\medskip
	
	Here is the main difference of this proof from that of \cite[Proposition 17.2.4]{Bible}: the operator
	\[
	v\mapsto \int^0_{-\infty} e^{t\A_0} (v(t))^+dt,\ L^2_k(\hz, E)\to L^2_{k-1/2}(\hy, E_0), 
	\]
	is not compact as $\A_0$ has essential spectrum, so the compactness of $\Pi_{\A}^--\Pi_{\A_0}^-$ really arises from $K$. 
\end{proof}

With Proposition \ref{P19.6} in mind, we are ready to study the boundary value problem on a finite interval. 

\begin{proposition}\label{P19.7} Let $I=[t_1,t_2]_t$ be a finite interval and $\hz=I\times \hy$. Given invertible operators $\A_i=\A_0+K_i, i=1,2$ as compact perturbations of $\A_0$, consider the operator 
	\[
	D=\dt+\A_0+K(t): L^2_k(\hz, E)\to L^2_{k-1}(\hz, E)
	\]
	on $\hz$ and spectral projections
	\begin{align*}
	\Pi^+_{\A_1}\circ r_1&:  L^2_k(\hz, E)\to H^+_{\A_1}\cap L^2_{k-1/2}(\{t_1\}\times\hy, E_0),\\
		\Pi^-_{\A_2}\circ r_2&:  L^2_k(\hz, E)\to H^-_{\A_2} \cap L^2_{k-1/2}(\{t_2\}\times\hy, E_0).
	\end{align*}
	where $K: I\to \Hom(L^2_j, L^2_j), j\geq 0$ is a smooth family of self-adjoint compact operators. Then the operator 
	\[
	P\colonequals D\oplus (\Pi^+_{\A_1}, \Pi^-_{\A_2})\circ (r_1, r_2)
	\]
	 is Fredholm, whose index is equal to the spectrum flow from $\A_1$ to $\A_2$. In particular, the restriction map on the kernel of $D$:
	\[
		(\Pi^+_{\A_1}, \Pi^-_{\A_2})\circ (r_1, r_2): \ker D\to H^+_{\A_1}\cap L^2_{k-1/2}(\hy, E_0)\oplus H^-_{\A_2}\cap L^2_{k-1/2}(\hy, E_0)
	\]
	is Fredholm of the same index. 
\end{proposition} 

In the sequel, we will abbreviate $H^+_{\A}\cap L^2_{k-1/2}(\hy, E_0)$ into $H^+_{\A}$ when the regularity of sections is clear from the context. 

\begin{proof} We start with the model case when $K_1=K_2=K(t)\equiv  0$. The operator 
	\[
	P_0=D_{\A_0}\oplus(\Pi^+_{\A_0}\circ r_1\oplus 	\Pi^-_{\A_0}\circ r_2)
	\]
	is then invertible by direct computation using Functional Calculus. For the general case, note that $D-D_{\A_0}$ is a compact operator. As for the boundary projections, Proposition \ref{P19.6} implies that 
	\begin{align*}
	\Pi_{\A_0}^+: H^+_{\A_1}&\to H^+_{\A_0},\\
	\Pi_{\A_1}^+: H^+_{\A_0}&\to H^+_{\A_1},
	\end{align*}
	are mutual inverses modulo compact operators, which also holds for the negative projections $\{\Pi_{\A_0}^-,\Pi_{\A_1}^-\}$. To compute the index, we use the concatenation trick and compare $P$ with the operator on the infinite cylinder:
	\[
	\dt+\A_0+K'(t): L^2_k(\R_t\times \hy, E)\to L^2_{k-1}(\R_t\times \hy, E). 
	\]
	where $K'$ is a smooth path of compact operators connecting $K_1$ and $K_2$: 
	\[
	K‘(t)\equiv K_1 \text{ if } t\leq t_1;\ K’(t)\equiv K_2 \text{ if } t\geq t_2.
	\]
	Now apply Proposition \ref{P19.1} or Theorem \ref{T16.2}. If we write $L^2_k(\hz, E)$ as a direct sum 
	\[
	C\oplus \ker D
	\]
	where $C$ is the $L^2_k$-orthogonal complement of $\ker D$, then $P$ is cast into a lower triangular metric
	\begin{equation}\label{E19.9}
	\begin{pmatrix}
D & 0\\
* & 	(\Pi^+_{\A_1}, \Pi^-_{\A_2})\circ (r_1, r_2)
\end{pmatrix}.
	\end{equation}
	As $D|_{C}$ is a bijection by \cite[Proposition 17.1.5]{Bible} and the unique continuation property, the other diagonal entry has to be Fredholm of the same index as that of $P$.  
\end{proof}

\begin{remark}\label{R19.10}
Here is a major difference of our case from \cite[Proposition 17.2.5]{Bible}: the projection map onto the complementary spectral subspaces:
\[
	(\Pi^-_{\A_1}, \Pi^+_{\A_2})\circ (r_1, r_2): \ker D\to H^-_{\A_1}\oplus H^+_{\A_2}
\]
is not compact. To see this, consider the model case when $\A_1=\A_2=\A_0$ and $K(t)\equiv 0$, so $\ker D$ is parametrized by the image of $(\Pi^+_{\A_0}, \Pi^-_{\A_0})\circ (r_1, r_2)$. Sticking to the positive part, the composition map
\begin{align*}
H_{\A_0}^+\cap L^2_{k-1/2}(\{t_1\}\times \hy, E_0)&\to H_{\A_0}^+\cap L^2_{k-1/2}(\{t_2\}\times \hy, E_0)\\
w&\mapsto v\colonequals P^{-1} (0,w,0)\in \ker D\\
&\mapsto \Pi^+_{\A_0}\circ r_2(v). 
\end{align*}
is simply $e^{-\A_0(t_2-t_1)}$ acting on $H^+_{\A_1}$ which has essential spectrum $[0,e^{-\lambda_1(t_2-t_1)}]$. As a result, it is never compact. 
\end{remark}

To circumvent this problem, we have to refine the estimates when the 3-manifold $\hy$ is not compact. Recall that a Fredholm operator $P$ is invertible modulo compact operators. A right (left) parametrix $Q$ is a right (left) inverse of $P$ modulo compact operators, i.e.
\[
PQ=\Id+\text{ a compact term}.
\]
Such a $Q$ is unique up to a compact term and is also a (two-sided) parametrix.

The difference up to a compact term is always insignificant. This motivates the next definition and lemma:
\begin{lemma}\label{L19.9} Let $H_i,\ i=1,2$ be Hilbert spaces. For any operator $Q: H_2\to H_1$, define its essential norm as 
	\[
	\|Q\|_{ess}\colonequals \inf_{K \text{ compact }} \|Q+K\|_{H_2\to H_1}. 
	\]
	For any Fredholm operator $P: H_1\to H_2$ with a parametrix $Q$, the perturbed operator $P+F$ is Fredholm if $\|FQ\|_{ess} <1$. 
\end{lemma}
\begin{proof} As $(P+F)Q$ and $Q$ are Fredholm, $P+F$ is Fredholm as well. 
\end{proof}

Now let us recast Proposition \ref{P19.7} into a more convenient form for applications. Recall that the essential spectrum of $\A_0$ is away from the origin:
\[
\sigma_{ess}(\A_0)=(-\infty,\lambda_1]\cup [\lambda_1,\infty),
\]
for some $\lambda_1>0$.

\begin{proposition}\label{P19.10} Under the assumption of Proposition \ref{P19.7}, the operator $P$ is Fredholm. The essential norm of its parametrix $\widetilde{Q}$ is bounded by a constant $C_1$ that depends only on $\lambda_1$. The same conclusion applies to the projection map 
	\[
		(\Pi^+_{\A_1}, \Pi^-_{\A_2})\circ (r_1, r_2): \ker D\to H^+_{\A_1}\cap L^2_{k-1/2}(\{t_1\}\times \hy, E_0)\oplus H^-_{\A_2}\cap L^2_{k-1/2}(\{t_2\}\times\hy\, E_0).
	\]
	and its parametrix $Q$. Moreover, the essential norm of the complementary projection pre-composed with $Q$:  
		\[
	(\Pi^-_{\A_0}, \Pi^+_{\A_0})\circ (r_1, r_2)\circ Q: H^+_{\A_1}\oplus H^-_{\A_2}\xrightarrow{Q}\ker D\to H^-_{\A_0}\oplus H^+_{\A_0}
	\]
	is bounded above by $e^{-\lambda_1 |I|}$, where $|I|=|t_2-t_1|$ is the length of $I$.
\end{proposition}
\begin{proof} We divide the proof into four steps:
	
	\medskip
	
	\Step 1. Estimate $\widetilde{Q}$. When $K_1=K_2=K(t)\equiv 0$, we obtain the model operator  
\[
	P_0=D_{\A_0}\oplus(\Pi^+_{\A_0}\circ r_1\oplus 	\Pi^-_{\A_0}\circ r_2): L^2_k(\hz, E)\to L^2_{k-1}(\hz, E)\oplus (H^+_{\A_0}\oplus H^-_{\A_0}).
\]
Let $\wQ_0=(R,Q_0)$ be the inverse of $P_0$ with 
\begin{align*}
	Q_0: H^+_{\A_0}\oplus H^-_{\A_0}&\to L^2_k(\hz, E),\\
	R:  L^2_{k-1}(\hz, E)&\to L^2_k(\hz, E).
\end{align*}
 The norm $\|\wQ_0\|$ is bounded by a constant $C_1$ independent of the length $|I|$. In the general case, set $\wQ\colonequals (R, Q_0\circ ( \Pi^+_{\A_0}, \Pi^-_{\A_0}))$ with 
\[
 ( \Pi^+_{\A_0}, \Pi^-_{\A_0}): H^+_{\A_1}\oplus H^-_{\A_2}\to H^+_{\A_0}\oplus H^-_{\A_0}.
\]
Then $	\|\wQ\|\leq \|\wQ_0\|
	$, since we have used $\A_0$ to define the $L^2_j$-norm on $C_c^\infty(\hy, E_0)$. By Proposition \ref{P19.6}, projection maps:
	\[
	\Pi_{\A_i}^\pm: H^\pm_{\A_0}\to H^\pm_{\A_i},\  \Pi_{\A_0}^\pm: H^\pm_{\A_i}\to H^\pm_{\A_0}, i=1,2
	\]
	are mutual inverses modulo compact operators; so $\wQ$ is a parametrix of $P$.
	
	\Step 2. Estimate $Q$.  Using the block form \eqref{E19.9}, we write $\wQ$ as a 2 by 2 matrix:
	\[
	\begin{pmatrix}
Q_{11} & Q_{12}\\
Q_{21} & Q_{22}
	\end{pmatrix}.
	\]
	Take $Q\colonequals Q_{22}$ to be the bottom right entry, then
	\[
	Q: H^-_{\A_1}\oplus H^+_{\A_2}\to \ker D
	\]
	is a left parametrix of $(\Pi^+_{\A_1}, \Pi^-_{\A_2})\circ (r_1, r_2)$ and 
	\[
	\|Q\|\leq \|\wQ\|,
	\]
	since $C$ is $L^2_k$-orthogonal to $\ker D$ in \eqref{E19.9}.
	
	\medskip

	\Step 3. Estimate the complementary projection.
	It suffices to estimate the norm of 
	\[
	M\colonequals (\Pi^-_{\A_0}, \Pi^+_{\A_0})\circ (r_1, r_2)\circ Q.
	\]
	First of all, the estimate holds for the model case when $\A_1=\A_2=\A_0$ and $K(t)\equiv 0$, by Remark \ref{R19.10}. Define
		\[
	M_0\colonequals (\Pi^-_{\A_0}, \Pi^+_{\A_0})\circ (r_1, r_2)\circ Q_0.
	\]
	 Now we allow $K(t)\neq 0$, but $\A_1=\A_2=\A_0$. Write $Q'$ for the parametrix constructed in \Step 2. We have to compare 
	 \[
	 	M'\colonequals (\Pi^-_{\A_0}, \Pi^+_{\A_0})\circ (r_1, r_2)\circ Q'.
	 \]
	 with the model operator $M_0$, and show the difference $M-M_0$ is compact.
	 
	   For any $(w_1,w_2)\in H^+_{\A_0}\oplus H^-_{\A_0}$, sections $u\colonequals Q'(w)$ and $v\colonequals Q_0(w)$ obey the following equations respectively:
	\[
	\left\{\begin{array}{rl}
	D_{\A_0}(u)&=-K(t)u\\
	\Pi_{\A_0}^+\circ r_2 (u)&=w_1-k_1(w),\\
	\Pi_{\A_0}^-\circ r_2 (u)&=w_2-k_2(w),
	\end{array}
	\right.
	\qquad
	\left\{\begin{array}{rl}
	D_{\A_0}(v)&=0,\\
	\Pi_{\A_0}^+\circ r_2 (v)&=w_1,\\
	\Pi_{\A_0}^-\circ r_2 (v)&=w_2,
	\end{array}
	\right.
	\]
	where $(k_1,k_2)$ is a compact operator acting on $H^+_{\A_0}\oplus H^-_{\A_0}$. It follows that 
	\[
	w\mapsto 
	 (Q'-Q_0)(w)= u-v=P_0^{-1}(-K(t)Q'(w), -k_1(w),-k_2(w))
	\] 
	is a compact operator. 
	
	\medskip
	
	\Step 4. In the most general case, we allow $K(t)\neq 0$ and $\A_1, \A_2\neq \A_0$. Recall that $Q=Q'\circ (\Pi_{\A_0}^+, \Pi_{\A_0}^-)$, so $M=M'\circ (\Pi_{\A_0}^+, \Pi_{\A_0}^-)$ and 
	\begin{equation*}
\|M\|_{ess}\leq \|M'\|_{ess}= \|M_0\|_{ess}\leq e^{-\lambda_1 |I|}.\qedhere
	\end{equation*}
\end{proof}

Spectral projections are not the most relevant boundary conditions for the main applications in gauge theory, although they serve important intermediate steps.

\begin{proposition}\label{P19.11} Under the assumption of Proposition \ref{P19.5} with $\hz=(-\infty,0]\times \hy$, suppose $\Pi_1$ is any linear projection on $L^2_{k-1/2}(\hy, E_0)$ whose kernel is a complement of $H^-_{\A}$: 
	\begin{equation}\label{E19.10}
	\ker(\Pi_1)\oplus (H^-_{\A}\cap L^2_{k-1/2}(\hy, E_0))=L^2_{k-1/2}(\hy, E_0). 
	\end{equation}
	and let $H_1^-$ be the image of $\Pi_1$. Then the operator 
	\[
	D_\A\oplus \Pi_1\circ r: L^2_k(\hz, E)\to L^2_{k-1}(\hz, E)\oplus H_1^-
	\]
	is an isomorphism. 
\end{proposition}
\begin{proof} See \cite[Proposition 17.2.6]{Bible} or \cite[P.340-341]{Bible}.
\end{proof}
\begin{proposition}[cf. \cite{Bible} Proposition 17.2.6]\label{P19.12} Under the assumption of Proposition \ref{P19.10} with $\hz=I\times \hy$ and $I=[t_1,t_2]$, suppose $\Pi_1^+$ and $\Pi_2^-$ are any linear projections on $L^2_{k-1}(\hy, E_0)$ whose kernels are complements of $H^+_{\A_1}$ and $H^-_{\A_2}$ respectively, i.e. \eqref{E19.10} holds for $(\Pi_1^+, H_{\A_1}^+)$ and $(\Pi_2^-, H_{\A_2}^-)$. Let $H_1^-$ and $H_2^+$ be images of $\Pi_1^+$ and $\Pi_2^-$ respectively. Then there exists a constant $T_0(\Pi_1^+, \Pi_2^-)>0$ such that the operator
	\[
	D\oplus (\Pi_1^+, \Pi_2^-)\circ (r_1, r_2): L^2_k(\hz, E)\to L^2_{k-1}(\hz, E)\oplus H_1^+\oplus H_2^-, 
	\] 
	is Fredholm when $|I|>T_0$.
\end{proposition}

\begin{proof} There are two ways to proceed. In the first approach, one may use Proposition \ref{P19.11} to construct a parametrix of $	D\oplus (\Pi_1^+, \Pi_2^-)$; see Proposition \ref{P22.1} below.  In the second approach, we use the estimate on essential operator norms from Proposition \ref{P19.10}. It suffices to show the restriction map 
\[
(\Pi_1^+, \Pi_2^-)\circ (r_1, r_2):\ker D\to  H_1^+\oplus H_2^-
\]
is Fredholm. We focus on $H_2^+$ and pretend the other boundary does not exist. Write
\[
\Pi^-_2=\Pi^-_2\circ \Pi^-_{\A_2}+\Pi^-_2\circ (\Pi^+_{\A_2}-\Pi_{\A_0}^+)+\Pi_2^-\circ \Pi_{\A_0}^+.
\]
The middle term is compact. Since $\Pi_2^-: H_{\A_2}^-\to H_2^-$ is an isomorphism of Hilbert spaces, by Proposition \ref{P19.10}, 
\[
\Pi_2^-\circ \Pi_{\A_2}^-\circ  r_2:\ker D\to H_{\A_2}^-\xrightarrow{\Pi_2^-} H_2^-
\]
is Fredholm with parametrix $Q\circ (\Pi_2^-)^{-1}$. To apply Lemma \ref{L19.9}, we have to estimate the essential norm of 
\[
(\Pi_2^-\circ\Pi_{\A_0}^+)\circ(Q\circ (\Pi_2^-)^{-1})=\Pi_2^-\circ (\Pi_{\A_0}^+\circ Q)\circ (\Pi_2^-)^{-1},
\]
which is bounded above by $C(\Pi_2^-)\cdot e^{-\lambda_1 |I|}<1$ if $|I|\gg 1$. The constant $C(\Pi_2^-)$ depends only on the operator norms of 
\[
\Pi_2^-: H^-_{\A_2}\to H_2^- \text{ and } (\Pi_2^-)^{-1}: H_2^-\to H^-_{\A_2}.\qedhere
\]
\end{proof}

\smallskip

\subsection{Applications in Gauge Theory}\label{Subsec19.4} Having developed the abstract theory in Subsection \ref{Subsec19.3}, let us explain now how various operators are defined in gauge theory. For each tame perturbation $\q\in \Pa$ and a configuration  $\fa\in \SC_{k-1/2}(\hy, \bs)$, consider the extended Hessian
\[
\widetilde{\A}\colonequals \EHess_{\q,\fa},
\]
The reference operator $\A_0$ is taken to be a compact perturbation of $\tilde{\A}$ such that the conditon \eqref{E19.12} holds.

Recall that the space $L^2_{k-1/2}(\hy, E_0)$ admits a decomposition for each $\fa\in \SC_{k-1/2}(\hy,\bs)$:
\begin{align*}
 L^2_{k-1/2}(\hy, E_0)&=L^2_{k-1/2}(\hy, i\R)\oplus \CT_{k-1/2, \fa},\\
 &=L^2_{k-1/2}(\hy, i\R)\oplus \J_{k-1/2,\fa}\oplus \K_{k-1/2,\fa},
\end{align*}
on  which $\EHess_{\q,\fa}$ takes a block form:
\[
\begin{pmatrix}
0 & \bd_\fa^* & 0\\
\bd_\fa & 0 & 0\\
0 & 0 &\Hess_{\q,\fa}
\end{pmatrix}+ 
\begin{pmatrix}
0 & 0 & 0\\
0 & y & x\\
0 & x^* &0
\end{pmatrix}
\]
The operators $x,y$ are defined as in \eqref{E18.1} and they are compact. Denote the first matrix by $\A$ and consider its spectral decomposition:
\[
\Pi_{\A}^\pm: L^2_{k-1/2}(\hy, E_0)\to H^\pm_{\A}.
\]

As $\Hess_{\q,\fa}$ acts on $\K_{k-1/2, \fa}$, we also have the spectral decomposition of $\Hess_{\q,\fa}$:
\[
\K_{k-1/2, \fa}=\K_\fa^+ \oplus \K_\fa^-.
\]
Define subspaces:
\[
H^\pm_{\fa}\colonequals L^2_{k-1/2}(\hy, i\R)\oplus \{0\}\oplus \K_\fa^\pm\subset L^2_{k-1/2}(\hy, E_0),
\]
and the projection maps
\[
\Pi_\fa^\pm: L^2_{k-1/2}(\hy, E_0)\to H^\pm_\fa,
\]
whose kernels are
\[
\{0\}\oplus \J_{k-1/2,\fa}\oplus \K_\fa^\mp.
\]
The pairs $(\Pi_\fa^\pm, \Pi_{\A}^\pm)$ that satisfy the condition \eqref{E19.10}, cf. \cite[P.316]{Bible}. By Proposition \ref{P19.12}, the first statement of \cite[Theorem 17.3.2]{Bible} continues to hold in our case, and the proof the gluing theorem from \cite[Section 17-19]{Bible} remains valid. Proposition \ref{P19.10} is the replacement of \cite[Proposition 17.2.5]{Bible} in the presence of essential spectra.

\begin{remark} In practice, we will take $\q$ to be an admissible perturbation and $\fa$ to be a non-degenerate critical point of $\CSd_{\omega}$, in which case $\widetilde{\A}=\A$. Moreover, $\CSd_{\omega}$ has only finitely many critical points by the compactness theorem. Since only finitely many configurations are involved in the gluing theorem, we have a uniform upper bound on the constant $T_0$ in Proposition \ref{P19.12}, so it does not cause a problem.
\end{remark} 

Finally, let us compute the spectrum flow from $\Hess_{\q,\fa}$ to $\Hess_{\q,u\cdot\fa}$ as an application of Proposition \ref{P19.7}. 
\begin{lemma}[cf. \cite{Bible} Lemma 14.4.6]\label{L19.14} Consider the cylinder $\hz=\R_t\times(\hy,\bs)$ and the operator $(\bd_\gamma^*, \D_\gamma\F_{\hz,\q} )$ defined in Proposition \ref{P19.1} with $\fb=u\cdot \fa$ and $u\in \CG_{k+1}(\hy)$, then 
	\[
	\Ind (\bd_\gamma^*, \D_\gamma\F_{\hz,\q} )=([u]\cup c_1(\bs))[Y,\partial Y]\in 2\Z,\ \forall \gamma\in \SC_k(\fa, u\cdot \fa). 
	\] 
\end{lemma}
\begin{proof} We may use Proposition \ref{P19.7} and \cite[Proposition 14.2.2]{Bible} to identify this index to the index of an operator on $S^1\times \hy$. The spin bundle $S^+\to S^1\times \hy$ is constructed as 
	\[
	[0,1]\times S/ (0, v)\sim (1, u\cdot v). 
	\]
	Using the Atiyah-Patodi-Singer index theorem \cite[Theorem 3.10]{APS} instead, the proof of \cite[Lemma 14.4.6]{Bible} can now proceed with no difficulty. Indeed, over the cylindrical end of $S^1\times \hy$, the operator is cast into the form (up to a compact term)
	\[
	\pt+\sigma(\ps+D_\Sigma)=\sigma(\ps-\sigma\cdot \pt+D_\Sigma) \text{ on }S^1\times [0,+\infty)_s\times\Sigma.  
	\]
	Following the proof of Proposition \ref{P17.2}, the spectrum of $(-\sigma\cdot \pt+D_\Sigma)$ on $S^1\times\Sigma$ is discrete and symmetric with respect to the origin, so its $\eta$-invariant is zero. Moreover, $(-\sigma\cdot \pt+D_\Sigma)$ is invertible, so its kernel is trivial. 
\end{proof}

\section{Linearized Operators on Cobordisms}\label{Sec22}

Having addressed the linearized operators on the product manifold $\R_t\times \hy$, in this section, we explore the case for a morphism $\x: (\y_1,\bs_1)\to (\y_2,\bs_2)$ in the strict cobordism category $\SCob_s$. In this case, we have a relative \spinc cobordism
\[
(\hx, \bs_X): (\hy_1, \bs_1)\to (\hy_1,\bs_2).
\]
By attaching cylindrical ends, we obtain a complete Riemannian manifold 
\[
\CX\colonequals\bigg( (-\infty, -1]_t\times  \hy_1\bigg)\cup \hx\cup\bigg( 
[1,\infty)_t\times \hy_2\bigg)
\]
together with a closed 2-form $\omega_X$ on $\CX$ defined as in \eqref{E9.11}. There are two main tasks for this section:
\begin{itemize}
\item define the perturbation space of the Seiberg-Witten equaions on $\CX$. This is crucial for the transversality result in Section \ref{Sec21}, cf. Theorem \ref{T21.5};
\item prove that the linearized operator on $\CX$ is Fredholm.
\end{itemize}
They are addressed in Subsection \ref{Subsec22.1} and \ref{Subsec22.2} respectively.

\subsection{Perturbations}\label{Subsec22.1} Given a morphism $\x:(\y_1,\bs_1)\to (\y_2,\bs_2)$ in the strict \spinc cobordism category $\SCob_s$, the perturbation $\q_i\in \Pa(Y_i)$ encoded in the definition of $(\y_i,\bs_i)$ is admissible by \ref{P8}. Take a critial point
\[
\fa_i\in \Crit(\CSd_{\omega_i,\hy_i})\subset \SC_k(\hy_i,\bs_i), 
\]
for each $i=1,2$.  Pick a smooth configuration $\gamma$ on $\CX$ such that 
\begin{equation}\label{E22.3}
\left\{\begin{array}{rl}
\cgamma(t)&\equiv \fa_1 \text{ if } t<-1/2;\\ \cgamma(t)&\equiv \fa_2 \text{ if } t>1/2\\
\gamma(t) &\text{is in the temporal gauge when } |t|>1/2, \\
\gamma\big|_{\hx}&\in \SC_{k}(\hx,\bs). 
\end{array}
\right.
\end{equation}
 Now consider the configuration space on $\CX$:
\[
\SC_k(\fa_1,\CX, \fa_2)\colonequals \{(A,\Phi)=\gamma_0+(a,\phi): (a,\phi)\in L^2_k(\CX, iT^*\CX\oplus S^+) \}.
\]
and the gauge group
\[
\CG_{k+1}(\CX)=\{u:\CX\to S^1: u-1\in L^2_{k+1}(\CX, \C) \}. 
\]

The linearized action of $\CG_{k+1}(\CX)$ at $\gamma=(A,\Phi)\in \SC_{k}(\fa_1,\CX, \fa_2)$ is given by:
\begin{align*}
\bd_\gamma: L^2_{k+1}(\CX,i\R)&\to T_\gamma \SC(\fa_1,\CX, \fa_2)\\
f(t)&\mapsto (-df, f\Phi)
\end{align*}
whose $L^2$-formal adjoint is 
\begin{align*}
\bd_\gamma^*: T_\gamma \SC(\fa_1,\CX, \fa_2)&\to L^2_{k-1}(\CX,i\R)\\
(\delta a,\delta\phi)& \mapsto -d^*a+i\re\langle \delta\phi, i\Phi\rangle. 
\end{align*}

Let us now specify the class of perturbations involved in the Seiberg-Witten equations. Choose a cut-off function $\beta: \R_t\to \R$ with $\beta(t)\equiv 1$ if $|t|>3$ and $\beta(t)\equiv 0$ if $|t|<2$. Pick another cut-off function $\beta_0: \R_t\to \R$ supported on $[1,2]_t\subset \R_t$, equal to $1$ when $t\in [5/4, 7/4]$. Now consider the perturbed Seiberg-Witten equation: 
\begin{align}\label{E22.1}
\F_{\CX,\p}(\gamma)&=0,\ \gamma\in \SC_{k}(\fa_1,\CX,\fa_2),\\
\F_{\CX,\p}(\gamma)&\colonequals \F_{\CX}(\gamma)+\beta(t)[\hq_1(\gamma)+\hq_2(\gamma)]+\beta_0(t)(\hq_3(\gamma))+(\rho_4(\omega_3^+),0),\nonumber
\end{align}
where $\F_{\CX}$ is the unperturbed Seiberg-Witten map defined by the formula \eqref{4DSWEQ}. Here $\p$ denotes the quadruple 
\[
\p\colonequals (\q_1, \q_2,\q_3, \omega_3)\in \Pa(Y_1)\times \Pa(Y_2)\times \Pa(Y_2)\times \Omega^2_c([1,2]\times Y_2, i\R).
\]
where $\q_3\in \Pa(Y_2)$ is a tame perturbation supported on $Y_2$ and $\omega_3$ is an imaginary-valued exact 2-form compactly supported on $[1,2]\times Y_2$. The effect of $\omega_3$ is to deform $\omega_X$ into $\omega_X-\omega_3$, so the first equation of \eqref{4DSWEQ} is changed into 
\[
\half\rho_4(F_{A^t}^+-2\omega^+_X)-(\Phi\Phi^*)_0=-\rho_4(\omega_3^+),
\]
modulo perturbations from $\q_i$'s. In practice, it suffices to consider $\omega_3$ in the special form: 
\begin{equation}\label{E22.2}
\omega_3=d_{\CX} (\beta_0(t)f_3dt )=-\beta_0(t) dt\wedge d_{Y_2} f_3.
\end{equation}
for a compactly supported smooth function $f_3: [1,2]_t\times Y_2\to i\R$.

Within the space of all compactly supported smooth functions on $[1,2]_t\times Y_2$, we choose a countable subset that is dense in $C^\infty$-topology and form a Banach space as in Theorem \ref{T15.14}:
\[
\Pa_{\form}.
\]
The space $\Pa_{\form}$ is dense in $C^\infty_c([1,2]_t\times Y_2, i\R)$, and we define $\omega_3$ by the formula \eqref{E22.2} with $f_3\in \Pa_{\form}$. In all, the quadruple $\p$ takes value in a Banach space
\[
\p=(\q_1, \q_2,\q_3, \omega_3)\in \Pa(Y_1)\times \Pa(Y_2)\times \Pa(Y_2)\times \Pa_{\form}.
\]

Here $\q_1$ and $\q_2$ are encoded in the cylindrical ends of $\CX$; only the last two terms
\[
(\q_3,\omega_3)
\]
give rise to the actual perturbation in \eqref{E22.1}, allowing us to achieve transversality in Section \ref{Sec21}. Note that 
\[
\beta_0(t)\hq_3(\gamma) \text{ and } (\rho_4(\omega_3^+), 0)
\]
are both supported in the compact region $[1,2]_t\times Y_2$. Finally, we form the moduli space $\M_k(\fa_1, \CX, \fa_2)$ by taking the quotient space:
\begin{equation}\label{E22.4}
\M_k(\fa_1, \CX, \fa_2)\colonequals \{\F_{\CX,\p}(\gamma)=0: \gamma\in \SC_k(\fa_1, \CX, \fa_2) \}/ \CG_{k+1}(\CX),
\end{equation}
which is in fact independent of the subscript $k$, due to the exponential decay of the local energy functional, cf. Theorem \ref{T1.2}.

\subsection{Linearized Operators}\label{Subsec22.2} Similar to the case for $\hz=\R_t\times \hy$, the linearization of $\F_{\CX,\p}$ together with $\bd_\gamma^*$ forms a Fredholm operator. In particular, the cokernel is finite dimensional. 

\begin{proposition}\label{P22.1} For any $i=1,2$, let $\fa_i$ be a smooth non-generate critical point of $\CSd_{\omega_i}$ in $\SC_k(\hy_i, \bs_i)$. Then for any $\gamma\in \SC_k(\fa_1,\CX, \fa_2)$, the operator 
	\[
	(\bd_\gamma^*, \D_\gamma\F_{\CX,\p}): L^2_k(\CX,  iT^*\CX\oplus S^+)\to L^2_{k-1}(\CX, i\R\oplus i\Lambda^+ \CX\oplus S^-)
	\]
	is Fredholm. 
\end{proposition}

\begin{definition}\label{D22.2} The moduli space $\M_k(\fa_1, \CX, \fa_2)$ is called regular, if the operator $(\bd_\gamma^*, \D_\gamma\F_{\CX,\p})$ is surjective at any solution $[\gamma]\in \M_k(\fa_1, \CX, \fa_2)$, 
\end{definition}

\begin{proof}[Proof of Proposition \ref{P22.1}] It suffices to deal with the case for the reference configuration $\gamma=\gamma_0$ and when $(\q_3, \omega_3)=0$. As $\fa_i$ is non-degenerate, the operator on the infinite cylinder 
	\[
	D_i\colonequals \dt+\EHess_{\q_i,\fa_i}: L^2_k(\R_t\times\hy_i, i\R\oplus iT^*\hy_i\oplus S)\to L^2_{k-1}(\R_t\times \hy_i, i\R\oplus iT^*\hy_i\oplus S)
	\]
	is invertible for $i=1,2$. Denote the inverse by $Q_i$. Unlike Theorem \ref{T16.2}, the cut-off functions involved in the parametrix patching argument are more sophisticated, as we explain now. There are three of them:
	\[
	\beta_1,\ \beta_2 \text{ and } \beta_X \text{ with }\beta_1+\beta_2+\beta_X\equiv 1 \text{ and } \beta_X \text{ compactly supported}
	\]
 Over the region $\{s\leq 2\}\subset \X$, choose a partition of unity $\{\beta_1',\beta_2',\beta_X\}$ subordinate to the open cover $U_1\cup U_2\cup U_X$: 
	\begin{figure}[H]
		\centering
		\begin{overpic}[scale=.17]{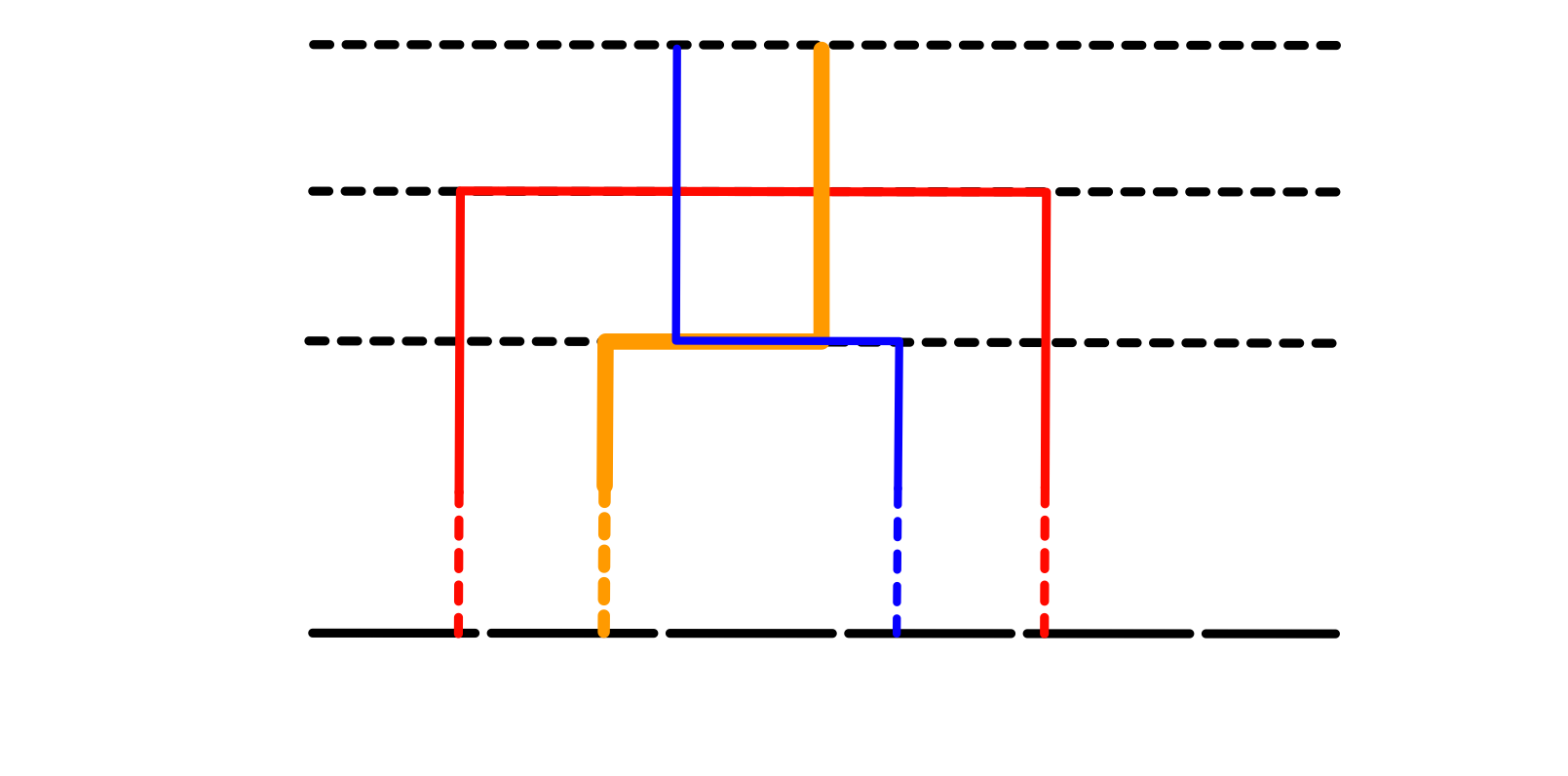}
			\put(46,15){$X$}
			\put(-2,15){$(-\infty, -1]\times \hy_1$}
			\put(72,15){$[1,\infty)\times \hy_2$}
			\put(16,5){$t=$}
			\put(26,5){\small$-2$}
			\put(36,5){\small$-1$}
			\put(47,5){\small$0$}
			\put(57,5){\small$1$}
			\put(66,5){\small$2$}
			\put(8,46){\small$s=2$}
			\put(15,36){\small$0$}
			\put(12,26){\small$-2$}
			\put(60,23){\color{red}$U_X$}
			\put(21,39){\color{orange}$U_1$}
			\put(71,39){\color{blue}$U_2$}
		\end{overpic}	
		\caption{An open cover of $\{s\leq 2\}$}
	\end{figure}
	
	Over the region $\{s\geq 2\}$, $\beta_X\equiv 0$ and $\beta_i(s,t)=\beta_i^T(t), i=1,2$ where $\{\beta_1^T,\beta_2^T\}$ is a partition of unity on the real line $\R_t$ subordinate to the cover 
	\[
	\R_t=(-\infty, T]\cup [-T,\infty),
	\]
	such that $|d\beta_i^T|\leq 4/T$. The value of $\beta_i$ in the transition area $\{1\leq s\leq 2\}$ is filled in by interpolation. To be more precise, pick a partition of unity $\{\alpha^L,\alpha^U\}$ on $\R_s$ such that $\alpha^U(s)\equiv 1$ when $s\geq 2$ and $\alpha^U(s)\equiv 0$ when $s\leq 1$. Set 
	\[
	\beta_i=\alpha^L(s)\beta_i'+\alpha^U(s)\beta_i^T(t),\ i=1,2.
	\]
	Finally, we take
	\[
	Q=\tilde{\beta}_1Q_1\beta_1+\tilde{\beta}_2Q_2\beta_2+\tilde{\beta}_XQ_X\beta_X,
	\]
	with $\tilde{\beta}_i$ constructed in a similar manner. Here we require that $\tilde{\beta}_i\equiv 1$ on $\supp\ \beta_i$ so that $\tilde{\beta}_i\beta_i=\beta_i$. The same holds for $(\tilde{\beta}_X,\beta_X)$; and also $\supp\ \tilde{\beta}_X$ is compact.
	
	The parametrix $Q_X$ is given by a local patching argument as usual. By taking $T\gg 0$, one verifies that $Q$ is indeed a parametrix for the operator $	(\bd_\gamma^*, \D_\gamma\F_{\CX,\p})$. 
\end{proof}

\part{Transversality}\label{Part6}

The primary goal of this part is to prove the key transversality result: Theorem \ref{T21.1}, which states that admissible perturbations on $(\hy,\bs)$, in the sense of Definition \ref{D19.3}, exist and are in fact generic. Because the perturbation space $\Pa(Y)$ that we consider are supported on the truncated 3-manifold $Y=\{s\leq 0\}$, only a weak separating property is satisfied, cf. Theorem \ref{T15.17}. As a result, a stronger unique continuation property is required in order to achieve transversality. 

 \medskip

Section \ref{Sec20} is devoted to the proof of unique continuation properties, which uses the Carleman estimates from \cite{K95}. In Section \ref{Sec21}, we prove Theorem \ref{T21.1} as well as its analogue for a general morphism $\x:(\y_1,\bs_1)\to (\y_2,\bs_2)$ in the $\SCob_s$, cf. Theorem \ref{T21.5}. 

\section{Unique Continuation}\label{Sec20}

\subsection{Statements} In this section, we prove the unique continuation properties of the perturbed Seiberg-Witten equations \eqref{E19.1}, which are crucial for the proof of Theorem \ref{T21.1}. The main results are listed as follows:
\begin{itemize}
\item the non-linear version: Theorem \ref{T20.9};
\item the linearized version: Theorem \ref{T20.20}; and 
\item the irreducibility of spinors: Theorem \ref{T20.10}. 
\end{itemize}

These theorems are summarized in the first subsection, while the rest of section is devoted to their proofs. Let us start with the non-linear version of unique continuation:
\begin{theorem}\label{T20.9} Let $I=(t_1,t_2)_t$ be an open finite interval. Consider a tame perturbation $\q\in \Pa$ supported on the truncated 3-manifold $Y=\{s\leq 0\}\subset \hy$ and the perturbed Seiberg-Witten equations on $\hz\colonequals I\times \hy$:
	\begin{equation}\label{E20.15}
	0=\F_{\hz,\q}(\gamma)\colonequals\F_{\hz}(\gamma)+\hq(\gamma).
	\end{equation}
	If two solutions $\gamma_1, \gamma_2$ are gauge equivalent on the slice $\{t_0\}\times Y$ at some $t_0\in I$, i.e there exists a gauge transformation $u\in \CG(\hy)$ such that 
	\[
	u(\gamma_1|_{\{t_0\}\times \hy})=\gamma_2|_{\{t_0\}\times \hy} \text{ on } Y,
	\]
	then $\gamma_1$ and $\gamma_2$ are gauge equivalent over the whole manifold $\hz$. 
\end{theorem}

The analogous result for closed 3-manifolds is \cite[Proposition 7.2.1]{Bible}. The main difference here is that $\gamma_1$ and $\gamma_2$ are \textbf{not} assumed to be gauge equivalent on the whole time slice $\{t_0\}\times \hy$; thus, the proof of \cite[Proposition 7.2.1]{Bible} does not apply directly here. 

\medskip

 Theorem \ref{T20.9} will follow from the strong unique continuation of the Seiberg-Witten equations if $\q=0$. The problem arises from the tame perturbation $\q$, which gives rise to non-local operators. We will provide a toy model in the next subsection to clarify this point, cf. Remark \ref{R20.2}. It is essential here that the region $\{t_0\}\times Y$ over which $\gamma_1$ and $\gamma_2$ agree contains the support of $\q$. 
 
 Before we proceed any further, let us state the linearized version of Theorem \ref{T20.9} and the version that concerns the irreducibility of spinors.
 
 \begin{theorem}[The Linearized Version]\label{T20.20} Let $I=(t_1,t_2)_t\subset \R_t$ be an open interval. Consider a tame perturbation $\q\in \Pa$ supported on the truncated 3-manifold $Y=\{s\leq 0\}\subset \hy$ and a smooth solution $\gamma$ to the perturbed Seiberg-Witten equation \eqref{E20.15} on the 4-manifold $\hz=I\times (\hy,\bs)$. Suppose a smooth tangent vector at $\gamma$
 	\[
 	V(t)=(\delta c(t),\delta b(t), \delta\psi(t))\in L^2_k(\hz, iT^*\hz\oplus S)
 	\]
 	lies in the kernel of  the linearized Seiberg-Witten map: 
 	\begin{equation}\label{E20.11}
 	0=\D_\gamma\F_{\hz,\q}(V),
 	\end{equation}
 	or equivalently, it solves the equation \eqref{E19.2}. If $V$ is generated by the linearized gauge action on $\{t_0\}\times Y$ at some $t_0\in I$, i.e. there exists a smooth function $\xi\in L^2_{k+1/2}(\hy, i\R)$ such that 
 	\[
 	(\delta b(0), \delta\psi(0))=\bd_{\cgamma(t)} \xi \text{ on } \{t_0\}\times Y.
 	\]
 	then $V$ is generated by the linearized gauge action on the whole manifold $\hz$, i.e. there exists a smooth function $\xi'\in L^2_{k+1}(\hz,i\R)$ such that 
 	\[
 	V=\bd_\gamma \xi' \text{ on } \hz.
 	\]
 \end{theorem}

\begin{theorem}[Irreducibiliy of Spinors]\label{T20.10} Let $I=(t_1,t_2)_t\subset \R_t$ be an open interval. For any tame perturbation $\q\in \Pa$ and a solution $\gamma=(A,\Phi)$ to the perturbed Seiberg-Witten equations \eqref{E20.15} on the 4-manifold $\hz=I\times\hy$, if the spinor 
	\[
	\Phi\equiv 0 \text{ on } \{t_0\}\times Y,
	\]
	for some $t_0\in I$, then $\Phi\equiv 0$ on $\hz$. 
\end{theorem}

The proofs of Theorem \ref{T20.9}-\ref{T20.10} will not be used elsewhere in this paper. They will dominate the rest of the section.

\subsection{A Motivating Problem} To better explain the ideas and point out the difference from the standard theory \cite[Section 7]{Bible}, let us first discuss a motivating problem that concerns the $\bpartial$-operator on the complex plane. Let 
\[
f:\C_z\to \C
\]
be a holomorphic function and $z=t+is$ be the complex coordinate of the domain. It is well-known that if $f$ vanishes along the interval $\{0\}\times [0,1]_s$, then $f\equiv 0$ over $\C_z$. 

We investigate a class of perturbations of the $\bpartial-$operator. The equation $\bpartial f=0$ can be formally cast into an evolution equation:
\[
\pt f=-D(f)
\]
where $D(f)=i\partial_s f$ is a self-adjoint operator on $L^2(\R_s, \C)$ (although we do not assume $f(t)\in L^2(\R_s,\C)$ for any time slice $\{t\}\times \R_s$). Consider a smooth function $K_1:\R_s\times \R_s\to \C$ with
\[
\supp\ K_1\subset [0,1]_s\times [0,1]_s 
\]
and form the convolution operator 
\begin{align*}
K:C^\infty(\R_s, \C)&\to C^\infty(\R_s,\C)\\
f&\mapsto K(f)(s)=\int_\R K_1(s,s') f(s')ds'. 
\end{align*}

Then $D_K\colonequals D+K$ is a compact perturbation of $D$, not necessarily self-adjoint anymore. More generally, let $V:\C_z\to \C$ be any smooth function and consider the equation 
\begin{equation}\label{E20.1}
	\pt f=-D_K(f)-V\cdot f \text{ on } \C=\R_t\times \R_s. 
\end{equation}

The potential $V$ can be viewed as a time-dependent perturbation of $D_K$. 

\begin{proposition}\label{P20.1} Suppose $f\in C^\infty(\C_z,\C)$ is a solution to the perturbed $\bpartial$-equation \eqref{E20.1} and $f(z)= 0$ for any $z\in \{0\}\times [0,1]_s$, then $f\equiv 0$ on $\C_z$. 
\end{proposition}

\begin{remark}\label{R20.2} If we only assume $f\equiv 0$  on $ \{0\}\times [\epsilon,1]_s$ for some small $\epsilon>0$, then for some kernel $K_1$ and potential $V$, the conclusion fails. Indeed, set $f(t,s)\equiv g(s)$ and $V\equiv 0$. Let $g$ be a cut-off function such that
	\[
	g(s)\equiv 0, \forall s\geq \epsilon \text{ and } g(s)\equiv 1, \forall s<\epsilon/2. 
	\]
	Then one can find $K_1$ with $K_1*g=-D(g)=-i\ps g$, so $g\in \ker D_K$. 
 \end{remark}

The problem here is that the convolution operator $K$ is not local: even if a function $g: [0,1]_s\to \C$ is supported on a small interval $[0,\epsilon]\subset [0,1]_s$, $K(g)=K_1*g$ might be non-vanishing on a much larger region. This is the analogue of the tame perturbation $\q$ in the Seiberg-Witten equations.

The proofs of Theorem \ref{T20.9}-\ref{T20.10} are modeled on that of Proposition \ref{P20.1}, which involves Carleman estimates, as we discuss in the next subsection. 

\subsection{Carleman Estimates} There are two classical ways to prove a strong unique continuation property like Proposition \ref{P20.1}. The first follows Agmon and Nirenberg \cite{AN67} and relies on a differential inequality. This is the approach adopted in the book \cite[Section 7]{Bible}. In this paper, we follow the second strategy and base our works on Carleman estimates \cite{C39}. The primary result that we consult is \cite[Theorem 1]{K95}. 

Let us first state a result in an abstract Hilbert space.
\begin{proposition}\label{P20.3}
Let $H$ be a Hilbert space and $L_i: H\to H,\ i=1,2$ be (unbounded) self-adjoint operators on $H$ satisfying the relation 
\begin{equation}\label{E20.2}
(L_1+rL_2+\alpha)^2-rL_2 \geq 0
\end{equation}
for any $r>0$ and $\alpha>\alpha_0(H, L_1, L_2)$; or equivalently, 
\begin{equation}\label{E20.5}
\|(L_1+rL_2+\alpha)v\|_H^2-\re\langle v,  (rL_2) v\rangle\geq 0\ \forall v\in D(L_1)\cap D(L_2).
\end{equation}
Here, $\alpha_0>0$ is a fixed large number depending only on $H$, $L_1$ and $L_2$.

\medskip

 Suppose $w: [0,r_0]_r\to D(L_1)\cap D(L_2)$ is a smooth function such that
 \begin{itemize}
\item for a constant $C_0>0$, the following estimate holds for any $r\in (0,r_0]$:
\begin{equation}\label{E20.3}
\| (\pr+\frac{1}{r}\cdot L_1+ L_2)w(r)\|_H\leq C_0 \|w(r)\|_H;
\end{equation}
\item $w(r)$ vanishes at the origin to the infinite order, i.e 
$
(\pr^n w)(0)=0
$
for any $n\geq 0$. In practice, we will only need the property that 
\begin{equation}\label{E20.18}
\|w(r)\|_H, \|\partial_r w(r)\|_H=\SO(r^n) \text{ as }r\to 0,
\end{equation}
for any $n\geq 1$. 
 \end{itemize}

 Then $w\equiv 0$. 
\end{proposition}


With loss of generality, we assume $r_0=1$ and let $x\colonequals \ln r\in (-\infty,0]$. Then the inequality \eqref{E20.3} becomes 
\begin{align}\label{E20.16}
g(x)&\colonequals (\px+L_1+e^x L_2) w(x),\\ \|g(x)\|_H&\leq C_0e^x\|w(x)\|_H,\ \forall x\in (-\infty,0].\nonumber
\end{align}

The key ingredient is the Carleman estimate. We follow the idea from \cite{AB80}. For any $\epsilon\in(0,1)$, consider the weight function $\varphi: (-\infty,0]\to \R_+$ implicitly determined by the relation $-\varphi(x)+\exp(-\epsilon\varphi(x))=x$, so $\varphi(x)\sim -x$ and 
\begin{align}
\px \varphi(x)&=-\frac{1}{1+\epsilon e^{-\epsilon\varphi(x)}}\in (-1,-\half),\label{E20.17}\\
\px^2 \varphi(x)&=\frac{\epsilon^2 e^{-\epsilon\varphi(x)}}{(1+\epsilon e^{-\epsilon\varphi(x)})^3}\geq C_1\epsilon^2\cdot e^{2\epsilon x}, \label{E20.4}
\end{align}
for a constant $C_1>0$. In what follows, we will always treat $\epsilon\in (0,1)$ as a fixed constant. 
\begin{proposition}[Carleman Estimates, \cite{K95} Theorem 1]\label{P20.4} Under the assumptions of Proposition \ref{P20.3}, for any $\epsilon\in (0,1)$, there is a constant $C(\epsilon)>0$ such that for any $\tau>2\alpha_0$ and $u\in C_c^\infty((-\infty, 0), D(L_1)\cap D(L_2))$, we have 
	\[
\tau\int_{(-\infty, 0)} \| e^{\tau\varphi(x)+\epsilon x} u(x)\|_H^2dx\leq C(\epsilon) \int_{(-\infty, 0)} \|e^{\tau \varphi(x)} (\px+L_1+e^x L_2) u(x)\|_H^2dx.
	\]
This estimate is uniform in $\tau$. 
\end{proposition}
	\begin{proof}[Proof of Proposition \ref{P20.3}] Fix some $x_0<0$. To apply Carleman estimates, choose a cut-off function $\chi: (-\infty, 0]_x\to [0,1]$ such that $\chi(x)\equiv 1$ when $x<x_0$ and $\chi(0)=0$.  Set $u(x)=\chi(x)w(x)$. The function $u(x)$ is not compactly supported on $(-\infty, 0)$, but its decay is faster than any exponential function as $x\to -\infty$, by \eqref{E20.18}. In this case, Proposition \ref{P20.4} still applies, cf. Remark \ref{R20.8}; so
\begin{align*}
\frac{\tau}{2C(\epsilon)}\int_{(-\infty, x_0]} &\| e^{\tau\varphi(x)+\epsilon x} w(x)\|_H^2dx\leq  \frac{\tau}{2C(\epsilon)}\int_{(-\infty, 0]} \| e^{\tau\varphi(x)+\epsilon x} u(x)\|_H^2dx,\\
&\leq \half \int_{(-\infty, 0]} \| e^{\tau\varphi(x)} (\px+L_1+e^xL_2) u(x)\|_H^2dx,\\
&\leq \int_{(-\infty, 0]} \|e^{\tau\varphi(x)} g(x)\|_H^2dx+ \int_{[x_0, 0]} \|e^{\tau\varphi(x)}[\px, \chi(x)] w(x)\|_H^2dx,\\
(\text{by }\eqref{E20.16})&\leq C_0 \int_{(-\infty, 0]} \|e^{\tau\varphi(x)+ x} w(x)\|_H^2dx+C_2e^{\tau \varphi(x_0)}\int_{[x_0,0]}\|w\|_H^2dx,
\end{align*}
where $C_2=\|\px\chi\|_\infty^2$. The upshot is that this inequality holds for any $\tau\gg 1$, so when $\tau>4C_0C(\epsilon)$, we use the rearrangement argument to derive that
\[
\frac{\tau}{2}\int_{(-\infty, x_0]} \| e^{\tau\varphi(x)+\epsilon x} w(x)\|_H^2dx\leq 2C(\epsilon)C_2 e^{\tau\varphi( x_0)}\int_{[x_0,0]}\|w\|_H^2dx.
\]

Let $\tau\to\infty$. We conclude that $w(x)\equiv 0$ when $x<x_0$. Since $x_0<0$ is arbitrary, $w\equiv 0$ on $(-\infty, 0]$. 
	\end{proof}

To complete the proof of Proposition \ref{P20.3}, it remains to prove Carleman estimates.

\begin{proof}[Proof of Proposition \ref{P20.4}] It is essentially the same argument as \cite[Theorem 1]{K95}. We record the proof here because a slight modification will be made in our actual applications. Set $v(x)\colonequals e^{\tau\varphi(x)}u(x)$, then 
	\[
	e^{\tau\varphi(x)} (\px+L_1+e^xL_2) e^{-\tau\varphi(x)} v(x)= (\px+L_1+e^xL_2+\tau(-\px\varphi(x)))v(x). 
	\]
Define $L(x)\colonequals L_1+e^xL_2+\tau(-\px\varphi(x))$ and compute 
	\begin{align*}
\int_{(-\infty, 0]} \|(\px+L(x)) v(x)\|_H^2dx&=\int_{(-\infty, 0]} \|\px v(x)\|_H^2+\|L(x)v(x)\|_H^2\\
&\qquad +\int_{(-\infty, 0]}2\re\langle \px v(x), L(x)v(x)\rangle_H dx.
	\end{align*}
	
	Using the fact that $L(x):H\to H$ is a self-adjoint operator, we integrate by parts:
\begin{align}\label{E20.7}
&\int_{(-\infty, 0]}2\re\langle \px v(x), L(x)v(x)\rangle_H=-\int_{(-\infty, 0]} \re\langle v(x), (\px L(x))v(x)\rangle\\
=& \int_{(-\infty, 0]}\re\langle v(x), e^x(-L_2) v(x)\rangle+\tau\int_{(-\infty, 0]}\langle v(x), (\px^2 \varphi(x) )v(x)\rangle\nonumber\\
(\text{by } \eqref{E20.4})\geq&  \int_{(-\infty, 0]}\re\langle v(x), e^x(-L_2) v(x)\rangle+C_1\epsilon^2\tau\int_{(-\infty, 0]}\|e^{\epsilon x}v(x)\|_H^2. \nonumber
\end{align}	

Set $\alpha=\tau(-\px\varphi)$. If $\tau>2\alpha_0$, then by \eqref{E20.17}, $
\alpha=\tau(-\px\varphi)>\tau/2>\alpha_0.$ Now we use the relation \eqref{E20.2} to conclude that 
\[
\int_{(-\infty, 0]} \|(\px+L(x)) v(x)\|_H^2dx\geq C_1 \epsilon^2\tau\int_{(-\infty, 0]}\|e^{\epsilon x}v(x)\|_H^2
\]
for any $\tau>2\alpha_0$ and $\epsilon\in (0,1)$. 
\end{proof}
\begin{remark}\label{R20.8} When $u(x): (-\infty,0)\to D(L_1)\cap D(L_2)$ is not compactly supported and yet $u(0)=0$, we have to verify that the boundary term in \eqref{E20.7} vanishes:
	\[
	\lim_{x\to -\infty}\re \langle v(x), L(x) v(x)\rangle. 
	\]
	Then one may assume that $\|u(x)\|_H,\ \|\px u(x)\|_H$ and $\|(L_1+e^x L_2)u(x)\|_H$ decay faster than any exponential functions as $x\to -\infty$. In Proposition \ref{P20.3}, this is guaranteed by \eqref{E20.16} and \eqref{E20.18}.
\end{remark}

\subsection{Applications} In this subsection, we give a few examples of $(H, L_1, L_2)$ for which the assumption \eqref{E20.5} is fulfilled and derive Proposition \ref{P20.1} from the abstract Proposition \ref{P20.3}. We will work out the Seiberg-Witten equations in the next subsection.

\begin{lemma}\label{L20.5} If self-adjoint operators $L_1, L_2: H\to H$ anti-commute, i.e.
\[
\{L_1, L_2\}\colonequals L_1L_2+L_2L_1=0,
\]
then the condition \eqref{E20.5} holds. 
\end{lemma}
\begin{proof} We rewrite the left hand side of \eqref{E20.5} as 
\[
\|(L_1+(1-\frac{1}{2\alpha})rL_2+\alpha)v\|_H^2+(1-(1-\frac{1}{2\alpha})^2)\| (rL_2) v\|_H^2 +\frac{\re\langle L_1v, (rL_2)v\rangle}{\alpha}\geq 0 \text{ if }\alpha>\frac{1}{4}.
\]
The last term vanishes because $\{L_1, L_2\}=0.$
\end{proof}
\begin{example}\label{EX20.7} The first example is the Dirac operator on $\C_z\times \Sigma$ where $\Sigma=\partial Y$ is a union of 2-tori endowed with a flat metric. We choose a \spinc connection $A$ on $\C_z$ such that 
	\[
	A=\dt+\ds+\cB
	\]
	for a fixed \spinc connection $\cB$ on the surface $\Sigma$. 
	
	Using the polar coordinate $(r,\theta)$ on the complex plane, the Dirac operator $D_A^+$ can be written as 
	\[
	D_A^+=\rho_4(dr)(\pr+\rho_3(rd\theta)\cdot( \frac{1}{r} \partial_\theta+ D_{\cB}^{\Sigma}))
	\]
	where $D_{\cB}^{\Sigma}$ is the Dirac operator associated to $\cB$ on the surface. Unlike $\rho_4(rd\theta)$,
	\[
	\rho_3(rd\theta)=\rho_4(dr)^{-1}\cdot \rho_4(rd\theta)=-\rho_4(dr\wedge rd\theta)
	\]
	is a constant bundle map. Proposition \ref{P20.3} applies to the operator $\rho_4^{-1}(dr)\cdot D_A^+$ with
	\[
	L_1^\BD=\rho_3(rd\theta)\cdot \partial_\theta,\ L_2^\BD= \rho_3(rd\theta)D_{\cB}^{\Sigma}. 
	\]
	and $H=L^2(S^1\times\Sigma, S^+)$. Indeed, by Lemma \ref{L20.5}, $\{L_1^\BD, L_2^\BD\}=0$.
\end{example}

\begin{example}\label{EX20.8} The second example concerns the self-dual operator 
	\begin{align*}
\Omega^1(X,i\R)&\to \Omega^+(X,i \R),\\
b&\mapsto d^+b,
	\end{align*}
 on the 4-manifold $X=\C_z\times \Sigma$. Using the polar coordinate at the origin $0\in \C_z$, we regard $b$ as an 1-form on 
	\[
	X'=[0,r_0)_r\times S^1\times \Sigma. 
	\]
Suppose $b$ does not contain the $dr$-component and write 
\[
b(r)=b_1(r) (rd\theta)+b_2(r)
\] 
with $b_1(r)\in H_1\colonequals L^2(S^1\times \Sigma, \R)$ and  $b_2(r)\in H_2\colonequals L^2(S^1\times \Sigma, T^*\Sigma)$. As the metric on $X'$ is given by 
\[
dr^2+(rd\theta)^2+g_\Sigma,
\]
the equation $d^+b=0$ is equivalent to that
\[
\pr \begin{pmatrix}
b_1(r)\\
b_2(r)
\end{pmatrix}+\bigg[
\frac{1}{r}\begin{pmatrix}
1 & 0\\
0 & *_\Sigma \partial_\theta
\end{pmatrix}+
\begin{pmatrix}
0 & *_\Sigma d_\Sigma\\
-*_\Sigma d_\Sigma &0
\end{pmatrix}
\bigg]
\begin{pmatrix}
b_1(r)\\
b_2(r)
\end{pmatrix}=0.
\]
To apply Proposition \ref{P20.3}, set $H=H_1\oplus H_2$ and \[
L_1^\Sph=\begin{pmatrix}
 1 & 0\\
 0& L_3
\end{pmatrix},\ L_2^\Sph=\begin{pmatrix}
0 & L_4\\
L_4^* & 0
\end{pmatrix},
\]
with  $L_3\colonequals *_\Sigma\partial_\theta: H_2\to H_2$ and $L_4\colonequals *_\Sigma d_\Sigma: H_2\to H_1$. To verify the condition \eqref{E20.5}, we calculate for each $v=(b_1, b_2)\in H$ that
\begin{align*}
\|(L_1^\Sph+rL_2^\Sph+\alpha)v\|_H^2-\langle v, rL_2^\Sph v\rangle&=\|\alpha b_1+rL_4b_2\|_{H_1}^2+\|rL_4^*b_1+(L_3+\alpha)b_2\|_{H_2}^2\\
&\qquad+ (2\alpha+1)\|b_1\|_{H_1}^2\geq 0 \text{ if }\alpha>-\half. 
\end{align*}
In this case, Lemma \ref{L20.5} is not applicable because the anti-commutator $\{L_1^\Sph, L_2^\Sph\}\neq 0$. 
\end{example}

In the proof of Proposition \ref{P20.1} below, we will work with operator $L_1, L_2$ that are not self-adjoint on $H$. Nevertheless, the abstract Proposition \ref{P20.3} still applies, since we can verify the first step of \eqref{E20.7} directly: this is the only place the self-adjointness was used. 

\begin{proof}[Proof of Proposition \ref{P20.1}]
Let $I=[0,1]_s$. For any $r\geq 0$, consider the contour $\Gamma_r=\Gamma_r^{(1)}+\Gamma_r^{(2)}+\Gamma_r^{(3)}+\Gamma_r^{(4)}$ with 
\begin{align*}
\Gamma_r^{(1)}&=\{r\}\times I, & \Gamma_r^{(2)}=& \{i+re^{i\theta}: 0\leq \theta\leq \pi\},\\
\Gamma_r^{(3)}&=\{-r\}\times I, & \Gamma_r^{(4)}=& \{re^{i\theta}: \pi\leq \theta\leq 2\pi\},
\end{align*}
and define
\begin{align*}
v_1(r)&=f|_{\Gamma_r^{(1)}\coprod \Gamma_r^{(3)}}\in H_1\colonequals L^2(I\coprod (-I), \C),\\	v_2(r)&=f|_{\Gamma_r^{(2)}\coprod \Gamma_r^{(4)}}\in H_2\colonequals L^2([0,\pi]_\theta\coprod [\pi, 2\pi]_\theta,\C).
\end{align*}
where $(-I)$ stands for the orientation reversal of $I$. Finally, set
\[
w(r)=(w_1(r),w_2(r))\colonequals (v_1(r),\sqrt{r} v_2(r))\in H\colonequals H_1\oplus H_2.
\]

Our assumptions imply that the function $w: [0, 1)\to H$ vanishes to the infinite order at the origin. To apply Proposition \ref{P20.3}, we look for the differential equation that governs $w(r)$. As the function $f$ solves the perturbed $\bpartial$-equation, we have  
	\begin{equation}\label{E20.6}
\pr w(r)+\big(\frac{1}{r}L_1+ L_2\big)w(r)=h(r)
\end{equation}
with 
\[
L_1=(0, i\partial_\theta-\half)\text{ and }L_2=(i\partial_s,0) \text{ on } H=H_1\oplus H_2. 
\]
	The error term $h(r)$ in \eqref{E20.6} is determined by the convolution operator $K$ and the smooth potential $V$, so the assumption \eqref{E20.3} is satisfied in our case. 
	
	Neither $L_1$ nor $L_2$ is a self-adjoint operator on $H$, but we still have 
	\begin{equation}\label{E20.9}
	\re\langle (\frac{1}{r} L_1+L_2) w(r), \pr w(r)\rangle =	\re\langle  w(r), (\frac{1}{r} L_1+L_2)  \pr w(r)\rangle
	\end{equation}
	 which justifies the equality \eqref{E20.7} in the proof of Proposition \ref{P20.4}.  Indeed, 
	 \begin{align*}
&\langle (\frac{1}{r} L_1+L_2) w(r), \pr w(r)\rangle-\langle  w(r), (\frac{1}{r} L_1+L_2)  \pr w(r)\rangle=\frac{i}{2r} \int_{[0,\pi]\coprod[\pi, 2\pi]}  \partial_\theta|v_2(r,\theta)|^2 d\theta
	\end{align*}
	is purely imaginary. As the relation 
		\[
	2\re\langle L_1 v, L_2 v\rangle=0, \forall v\in D(L_1)\cap D(L_2)
	\]
	still holds in our case, the proof of Lemma \ref{L20.5} remains valid. Now we use Lemma \ref{L20.5} and Proposition \ref{P20.3} to complete the proof. 
\end{proof}

\subsection{The Seiberg-Witten Equations}

Having discussed some toy problems, we are now ready to prove the strong unique continuation property for the perturbed Seiberg-Witten equations, by combining Example \ref{EX20.7} and \ref{EX20.8}.

\begin{proof}[Proof of Theorem \ref{T20.9}] With loss of generality, assume $I=[-1,1]$ and $t_0=0$. It suffices to show that $\gamma_1$ and $\gamma_2$ are gauge equivalent in an open neighborhood of $\{0\}\times Y$, then one may use induction to extend this neighborhood to the whole space $\hz=I\times \hy$. 
	
	To imitate the proof of Proposition \ref{P20.1}, consider the closed 3-manifold $\CY_r=\CY_r^{(1)}\cup \CY_r^{(2)}\cup \CY_r^{(3)}$ where
	\begin{align*}
	\CY_r^{(1)}&\colonequals \{r\}\times Y, & \CY_r^{(2)}&=\{re^{i\theta}: 0\leq\theta\leq\pi\}\times\Sigma,\\
	 \CY_r^{(3)}&\colonequals (-\{-r\}\times Y), & \forall r&\in [0, 1].
	\end{align*}

	Here $\CY_r^{(3)}$ is the orientation reversal of $\{-r\}\times Y$. Let $B_0$ be the reference \spinc connection on $\hy$, so $B_0$ agrees with the $\R_s$-invariant connection 
	\[
	\ds+\cB
	\]
on the cylindrical end $[-1,\infty)_s\times \Sigma$. Set $\gamma_0=(B_0, 0)$.  

\medskip
	
	Extend the gauge transformation $u$ constantly in the time direction and replace $\gamma_1$ by $u(\gamma_1)$. Construct gauge transformations $u_i,\ i=1,2$ such that $u_i\equiv \Id$ on $\{0\}\times \hy$ and $\gamma_i'\colonequals u_i(\gamma_i)$ is in the temporal gauge (the $dt$-component vanishes). Consider the difference 
	\[
	\delta_i(t)\colonequals  \gamma_i'|_{\{t\}\times \hy}-\gamma_0\in C^\infty(\hy, iT^*\hy\oplus S)
	\]
	Formally, $\delta_i$ is subject to an evolution equation:
	\[
	\pt \delta_i(t)+L_2^Y \delta_i(t)+\delta_i(t)\#\delta_i(t)+\q(\delta_i(t)+\gamma_0)=c. 
	\]
where $\#$ is a symmetric bilinear form that involves only point-wise multiplications. Here $c$ is a constant error term determined by $\gamma_0$ and 
\[
L_2^Y=\begin{pmatrix}
*_3d_{\hy} & 0\\
0 & D_{B_0}
\end{pmatrix}.
\]

Now take the difference $\delta(t)\colonequals \delta_2(t)-\delta_1(t)$. Over the space $[-1,1]_t\times Y$, we have
\begin{equation}\label{E20.8}
\pt \delta(t)+L^Y_2(\delta(t))=h_1(t)\in  C^\infty(Y, iT^*Y\oplus S)
\end{equation}
and $\|h_1(t)\|_{L^2(Y)}\leq C\|\delta(t)\|_{L^2(Y)}$ for a uniform constant $C>0$. Moreover, 
\[
 \pt^n \delta(0)\equiv 0 \text{ on } Y \text{ for any }n\geq 0.
\]
When $n=0$, this follows from the assumption that $\gamma_1=\gamma_2$ on $\{0\}\times Y$. When $n\geq 1$, this is a consequence of the equation \eqref{E20.8} and its higher time derivatives. As a result, all derivatives of $\delta$ vanish on $\{0\}\times Y$.

Set $H_1=L^2(Y, iT^*Y\oplus S)$ and define 
\[
v_1(r)=(\delta(r)|_{Y}, \delta(-r)|_{Y})\in H_1 \oplus H_1. 
\]

Then $\pr^n v_1(0)=0$ for any $n\geq 0$.

 To deal with the middle part $\CY_r^{(2)}$, consider the polar coordinate at $0\in \C_z$ and restrict $\delta$ to a section of 
\[
iT^*X'\oplus S\to X'\colonequals [0, 1]_r\times [0,\pi]_\theta\times \Sigma\subset \R_t\times \{s\geq 0\}\times \Sigma. 
\]
The section $\delta$ is not necessarily in the radial temporal gauge: the $dr$-component of $\delta$ only vanishes when $\theta=0,\pi$. One has to construct gauge transformations $u'_i: X'\to S^1$ on $X'$ such that $u_i'|_{\{0\}\times [0,\pi]_\theta\times \Sigma}\equiv \Id$ and $u_i'(\gamma_i')$ is the radial temporal gauge. Then we define 
\[
v_2(r)=u_2'(\gamma_2')(r)-u_1'(\gamma_1')(r)\in H_2\colonequals L^2([0,\pi]_\theta\times\Sigma, i\R\oplus iT^*\Sigma\oplus S).
\]

Then the path $v_2(r)$ is subject to the equation 
\[
\pr v_2(r)+\bigg(\frac{1}{r}\begin{pmatrix}
L_1^\Sph & 0\\
0 & L_1^\BD
\end{pmatrix}+
\begin{pmatrix}
L_2^\Sph & 0\\
0 & L_2^\BD
\end{pmatrix}\bigg) v_2(r)=h_2(r)\in H_2. 
\]
and $\|h_2(r)\|_{H_2}\leq C\|v_2(r)\|_{H_2}$ for a constant $C>0$. The Seiberg-Witten equations are not perturbed on $X'$, so the error term $h_2(r)$ involves only point-wise multiplications with $v_2(r)$. Operators $L_i^\Sph$ and $L_i^\BD, i=1,2$ are defined as in Example \ref{EX20.7} and \ref{EX20.8}.

As all derivatives of $\delta$ vanish on $(0,0)\times \Sigma$, $\partial_r^n v_2(0)=0$ for any $n\geq 0$. 
	
Finally, let $H=(H_1\oplus H_1)\oplus H_2$ and define
\[
w(r)=(w_1(r), w_2(r))\colonequals (v_1(r), \sqrt{r} v_2(r))\in H. 
\]

Now the path $w:[0, 1)_r\to H$ is subject to the equation
\begin{equation}\label{E20.10}
\pr w(r)+(\frac{1}{r} L_1+L_2)w(r)=(h_1(r),-h_1(-r),\sqrt{r} h_2(r)). 
\end{equation}
 with 
\[
L_1=(0,0, \begin{pmatrix}
L_1^\Sph & 0\\
0 & L_1^\BD
\end{pmatrix}-\half), L_2=(L_2^Y, -L_2^Y, \begin{pmatrix}
L_2^\Sph & 0\\
0 & L_2^\BD
\end{pmatrix}). 
\]

 To apply Proposition \ref{P20.3}, we have to verify:
 \begin{itemize}
\item the positivity condition  \eqref{E20.5};
\item the symmetry condition \eqref{E20.9}; note that neither $L_1$ nor $L_2$ is self-adjoint.
 \end{itemize}

At this point, we have reduced the problem to some formal properties of $L_1$, $L_2$ and $w(r)$. We will treat the form component and the spinor component of \eqref{E20.10} separately. The verification of \eqref{E20.5} and \eqref{E20.9} will dominate the rest of the proof. 

\medskip

\Step 1. The Form Component and the Self-Dual operators. In this case, the positivity condition \eqref{E20.5} follows from the same argument as in Example \ref{EX20.8} and Proposition \ref{P20.1}. It can be checked separately on each of $\CY_r^{(i)},\ 1\leq i\leq 3$. As for \eqref{E20.9}, we focus on the common boundary of $\CY_r^{(1)}$ and $\CY_r^{(2)}$. Suppose the form components of $v_1(r)$ and $v_2(r)$ are given respectively by 
\begin{align*}
v_1|_{\{r\}\times (-1,0]_s\times \Sigma} &\rightsquigarrow a_1ds+a_2, & & a_1(r)\in C^\infty((-1,0]_s\times \Sigma, i\R),\\
&&&a_2(r)\in C^\infty((-1,0]_s\times \Sigma, iT^*\Sigma), \\
 v_2 &\rightsquigarrow b_1(rd\theta)+b_2, & & b_1(r)\in C^\infty([0,\pi]_\theta\times\Sigma, i\R),\\ &&&b_2(r)\in C^\infty([0,\pi]_\theta\times\Sigma, iT^*\Sigma).
\end{align*}

Near the boundary of $\CY_r^{(1)}$, we have 
\[
(*_3 d_Y)\begin{pmatrix}
a_1(r)\\
a_2(r)
\end{pmatrix}=\begin{pmatrix}
0 & *_\Sigma d_\Sigma\\
-*_\Sigma d_\Sigma & *_\Sigma\ps
\end{pmatrix}
\begin{pmatrix}
a_1(r)\\
a_2(r)
\end{pmatrix}.
\]

Then we calculate (the operator $L_2^\Sph$ is ignored here as it is  always self-adjoint):
\begin{align*}
\langle (*_3 d_Y)v_1,(\pr v_1)\rangle_{\{r\}\times Y}-\langle  v_1, (*_3 d_Y)(\pr v_1)\rangle_{\{r\}\times Y}&=\langle *_\Sigma a_2(r,0), (\pr a_2)(r,0)\rangle_{(r,0)\times \Sigma}. \\
\langle\frac{1}{r} L_1^\Sph w_2,(\pr w_2)\rangle_{H_2}-\langle  w_2,\frac{1}{r} L_1^\Sph(\pr w_2)\rangle_{H_2}&=-\langle *_\Sigma b_2(r,0), (\pr b_2)(r,0)\rangle_{(r,0)\times \Sigma}\\
&\qquad+\frac{1}{r}\underbrace{\int_{[0,\pi]_\theta} \partial_\theta \langle *_\Sigma b_2, b_2\rangle_\Sigma}_{=0}+\cdots. 
\end{align*}

It remains to verify that $a_2(r,0)=b_2(r,0)$ on $\CY_r^{(1)}\cap \CY_r^{(2)}$. Suppose the restriction of the form component of $\delta$ on $X'= [0, 1]_r\times [0,\pi]_\theta\times \Sigma$ is $fdr+c_1(rd\theta)+c_2$ with
\begin{align*}
f(r), c_1(r)\in C^\infty([0,\pi]_\theta\times\Sigma, i\R),\ c_2(r)\in C^\infty([0,\pi]_\theta\times\Sigma, iT^*\Sigma).
\end{align*}

It is clear that $a_2$ and $c_2$ agree along the common boundary of $\CY_r^{(1)}$ and $\CY_r^{(2)}$. Moreover, 
\[
f(r, \theta)\equiv 0 \text{ if } \theta=0 \text{ or }\pi. 
\]
To put $\delta$ into radial temporal gauge, we applied further gauge transformations, so $(b_1,b_2)$ is related to $\delta$ by the formulae:
\begin{align*}
b_1(r)&=c_1(r)-\frac{1}{r}\int_0^r (\partial_\theta f)(r')dr',& b_2(r)&=c_2(r)-\int_0^r (d_\Sigma f)(r')dr'.
\end{align*}

As a result, $a_2(r,s)|_{s=0}=b_2(r,\theta)|_{\theta=0}$. This equality does not a priori hold for $a_1$ and $b_1$, but it is not needed in the proof. 

\medskip

\Step 2. The Spinor Component and the Dirac operators. The proof of \eqref{E20.9} proceeds in the same way as in \Step 1. We focus on the positivity condition \eqref{E20.5}. Suppose  the spinor components of $v_i(r),\ 1\leq i\leq 3$ are given respectively by 
\begin{align*}
v_1|_{\{r\}\times Y} &\rightsquigarrow \Phi_1(r)\in C^\infty(Y, S),& v_2 &\rightsquigarrow \Phi_2(r)\in C^\infty([0,\pi]_\theta\times\Sigma, S),\\
v_3|_{\{r\}\times Y} &\rightsquigarrow \Phi_3(r)\in C^\infty(Y, S).&&
\end{align*}

We focus on sections 
\[
(\Phi_1(r), \Phi_3(r), \sqrt{r}\Phi_2(r))\in L^2(Y,S)\oplus L^2(Y,S)\oplus L^2([0,\pi]_\theta\times \Sigma, S)
\]
and operators:
\[
\frac{1}{r} \begin{pmatrix}
0 & 0 & 0\\
0& 0 &0\\
0 & 0 & L_1^\BD-\half
\end{pmatrix}+\begin{pmatrix}
D_{B_0} & 0 &0\\
0& -D_{B_0} & 0\\
0 & 0 & L_2^\BD
\end{pmatrix}.
\]

Unlike Example \ref{EX20.7}, $L^\BD_1$ is not self-adjoint in this case. In general, 
\[
2\re\langle L^\BD_1 v, L^\BD_2 v\rangle_{L^2([0,\pi]_\theta\times\Sigma)}=\int_{\{\theta\}\times\Sigma}\langle v, D_{\cB}^\Sigma v\rangle \ \bigg|_{\theta=0}^{\theta=\pi} \neq 0,\ v\in L^2([0,\pi]_\theta, S).
\]

Let $v=\sqrt{r}\Phi_2(r)$ and follow the proof of Lemma \ref{L20.5}:
\begin{align*}
&\|(L_1^\BD+rL_2^\BD+(\alpha-\half))\sqrt{r}\Phi_2\|^2_{L^2([0,\pi]_\theta\times\Sigma)}-\re\langle \sqrt{r}\Phi_2, (rL_2^\BD) \sqrt{r}\Phi_2\rangle\\
\geq&\frac{2r^2}{2\alpha-1}\re\langle L_1^\BD\Phi_2, L_2^\BD \Phi_2\rangle_{L^2([0,\pi]_\theta\times\Sigma)}\\
=&\frac{r^2}{2\alpha-1}\bigg(\int_{\{\pi\}\times\Sigma}\langle \Phi_2, D_{\cB}^\Sigma \Phi_2\rangle-\int_{\{0\}\times\Sigma}\langle \Phi_2, D_{\cB}^\Sigma \Phi_2\rangle\bigg).
\end{align*}

Just as in \Step 1, sections $\Phi_1$ and $\Phi_2$ have the same boundary value along $\CY_r^{(1)}\cap \CY_r^{(2)}$:
\[
\Phi_1(r,s)\big|_{s=0}=\Phi_2(r,\theta)\big|_{\theta=0}.
\]
Therefore, it remains to verify the inequality:
\[
\|(rD_{B_0}+\alpha)\Phi_1\|^2_{L^2(Y)}-\re\langle \Phi_1, (rD_{B_0})\Phi_1\rangle\geq \frac{r^2}{2\alpha-1}\int_{\{0\}\times\Sigma}\langle \Phi_1, D_{\cB}^\Sigma \Phi_1\rangle.
\]
The left hand side can be rewritten as 
\[
(1-\frac{1}{2\alpha-1})\|(rD_{B_0}+\frac{(2\alpha-1)^2}{4\alpha-4})\Phi_1\|^2_{L^2(Y)}+\frac{r^2}{2\alpha-1} \|D_{B_0}\Phi_1\|^2_{L^2(Y)}+\frac{4\alpha^2-6\alpha+1}{8\alpha-8}\|\Phi_1\|^2_{L^2(Y)}.
\]
Using the Weitzenb\"{o}ck formula \cite[(4,15)]{Bible}, the last two terms are bounded below by 
\begin{align*}
 &\frac{r^2}{2\alpha-1}\bigg(\|\nabla_{B_0}\Phi_1\|^2_{L^2(Y)}+\int_{\{0\}\times\Sigma}\langle \Phi_1, D_{\cB}^\Sigma \Phi_1\rangle+\int_Y\frac{s}{4}|\Phi_1|^2+\langle \Phi_1, \half \rho_3(F_{B_0^t})\Phi_1\rangle\bigg)\\
 &\qquad+\frac{\alpha-1}{2}\|\Phi_1\|^2_{L^2(Y)}\geq \frac{r^2}{2\alpha-1} \int_{\{0\}\times\Sigma}\langle \Phi_1, D_{\cB}^\Sigma \Phi_1\rangle+\frac{\alpha-\alpha_0}{2}\|\Phi_1\|^2_{L^2(Y)}.
\end{align*}
Then we take $\alpha>\alpha_0\colonequals e^{100}\max\{ \|s\|_\infty, \|F_{B_0^t}\|_\infty,1\}$. 

The common boundary $\CY_r^{(2)}\cap \CY_r^{(3)}$ is dealt with similarly. Hence, the positivity condition \eqref{E20.5} holds when $\alpha>\alpha_0$. Now we use Proposition \ref{P20.3} and \ref{P20.4} to complete the proof. 
\end{proof}

\subsection{Irreducibility of Spinors} We accomplish the proof of Theorem \ref{T20.10} in this subsection, following the idea above. The spinor part of the equation \eqref{E20.15} is cast into the form
	\[
	\dt \Psi(t)+D_{B(t)}\Psi(t)+\q^1(B(t),\Psi(t))=0. 
	\]
	where $\q^1$ is the spinor part of the perturbation $\q=(\q^0,\q^1)$. As $\q^1(B(t),0)\equiv 0$, we have 
	\begin{align*}
\|\q^1(B(t),\Psi(t))\|_2&=\|\q^1(B(t),\Psi(t))-\q^1(B(t),0)\|_2\\
&=\int_0^1 \|\D_{(B(t),r\Psi(t))} \q^1(\Psi(t))\|_2dr\leq C\|\Psi(t)\|_{L^2(Y)},
	\end{align*}
for a constant $C>0$ and any $t\in [t_0-\epsilon, t_0+\epsilon]$. Now the proof of Theorem \ref{T20.9} can proceed with no difficulty.

\subsection{The Linearized Version} In this subsection, we accomplish the proof of Theorem \ref{T20.20}. To some extent, it suffices to ``linearize'' each step of the proof of Theorem \ref{T20.9}. Again, assume $I=[-1,1]_t$ and $t_0=0$. 
	\[
	\xi^{(1)}(t)=\xi-\int_{0}^t \delta c(t')dt'\in C^\infty(\hz, i\R), 
	\]
	and set $V_1=V-\bd_\gamma \xi^{(1)}$. This new section $V_1$ is smooth, and  
	\begin{align*}
	V_1(t)&=(0, \delta b_1(t), \delta \psi_1(t) )\in L^2_k(Z,iT^*Z\oplus S),\\
	 V_1(0)&=0 \text{ on } \{0\}\times Y. 
	\end{align*}
	As $\gamma$ solves the non-linear equation \eqref{E20.15}, $\dg f^{(1)}$ is a solution to the linear equation \eqref{E20.11}, and so is $V_1$. The equation \eqref{E19.2} is formally an evolutionary equation on $I\times Y$:
\begin{equation}\label{E20.13}
\dt\begin{pmatrix}
\delta b_1(t)\\
\delta \psi_1(t)
\end{pmatrix}+\begin{pmatrix}
*_3d_{Y} & 0\\
0 & D_{B_0}
\end{pmatrix}\begin{pmatrix}
\delta b(t)\\
\delta \psi(t)
\end{pmatrix}=\eta(t)\begin{pmatrix}
\delta b(t)\\
\delta \psi(t)
\end{pmatrix},t\in \R.  
\end{equation}
where $\eta(t): L^2(Y)\to L^2(Y)$ is a family of bounded linear operators determined by $\cgamma(t)$. 

 To borrow the proof of Theorem \ref{T20.20}, we focus on $\CY_r^{(2)}$. Using polar coordinates, we write 
\[
V_1(r)=(\delta c_1'(r), \delta b_1'(r), \delta\psi_1'(r))\in C^\infty(X', i\R\oplus iT^*([0,\pi]_\theta\times\Sigma)\oplus S),
\]
on $X'=[0,1]_r\times [0,\pi]_\theta\times \Sigma\subset \HH^2_+\times \Sigma$. To put $V_1(r)$ into radial temporal gauge, consider the function
\[
f^{(2)}(r)=-\int_0^r \delta c_1'(r')dr' \text{ on } X'. 
\]
Then $f^{(2)}(r,\theta)\equiv 0$ when $\theta=0, \pi$, and the section $V_1-\dg f^{(2)} $ solves the linear equation \eqref{E20.11} on $X$. The proof of Theorem \ref{T20.9} is now applicable. We conclude that 
\begin{align}\label{E20.14}
V_1(t)&\equiv 0 \text{ on } I\times Y\\
V_1-\dg f^{(2)}&\equiv 0 \text{ on } X'. \nonumber
\end{align}

We extend $f^{(2)}$ by zero over the product $I\times Y$. One might worry that $f^{(2)}$ does not form a smooth function on the union
\[
(I\times Y)\bigcup X',
\]
as we pointed out in \Step 1 in the proof of Theorem \ref{T20.20}. However, once the unique continuation property is established, the smoothness of $f^{(2)}$ follows from \eqref{E20.14} and the smoothness of $V_1$. As a result, 
\[
V_1=\dg f^{(2)} \text{ on } (I\times Y)\bigcup X'. 
\]

By induction, we can extend the region where this equality holds. This completes the proof of Theorem \ref{T20.20}. 

\section{Transversality}\label{Sec21}

With all machinery developed so far, we are ready to prove the transversality result on the cylinder $\R_t\times \hy$ in this section. Here is the main result:
\begin{theorem}\label{T21.1} For any relative \spinc manifold $(\hy,\bs)$ satisfying constraints in the strict cobordism category $\Cob_s$, one can find an admissible perturbation $\q\in \Pa(\hy,\bs)$, in the sense of Definition \ref{D19.3}. Here $\Pa(\hy,\bs)$ is the Banach space of tame perturbations constructed Subsection \ref{Subsec15.5}.
\end{theorem}

Pick an admissible perturbation $\q(\bs)$ for each relative \spinc structure $\bs$ on $Y$. By putting them altogether, we obtain an object $\y=(Y,\psi, g_Y, \omega,\q)$ in the category $\Cob_s$: the property \ref{P8} is fulfilled. In this case, the moduli spaces $\M_{[\gamma]}(\fa,\fb)$ defined in Section \ref{Sec19} will become a smooth manifold, and the Floer homology of $(\y,\bs)$ will be defined in Part \ref{Part7}. 

Theorem \ref{T21.1} is a formal consequence of the unique continuation properties, Theorem \ref{T20.9}-\ref{T20.10} and the separating properties of cylinder functions, Theorem \ref{T15.17}. The transversality result for a general morphism $\x: (\y,\bs_1)\to (\y_2,\bs_2)$ in the category $\SCob_s$ is proved in Subsection \ref{Subsec21.3}, cf. Theorem \ref{T21.5}.

\subsection{Transversality for the 3-Dimensional Equations} Consider the Banach space of perturbations $\Pa$ and a tame perturbation $\q=\grad f\in \Pa$. We start with the first condition \ref{E1} in Definition \ref{D19.3} which concerns the 3-dimensional equation 
\[
\grad \CSd_{\omega}(\fa)=0,
\] 
Recall from Definition \ref{D18.2} that a critical point $\fa\in \SC_k(\hy,\bs)$ of $\CSd_{\omega}=\CL_{\omega}+f$ is called non-degenerate if the extended Hessian at $\fa$
\[
\EHess_{\q,\fa}
\]
is invertible. In fact, this is a generic condition for a perturbation $\q\in \Pa$.
\begin{theorem}[cf. \cite{Bible} Theorem 12.1.12]\label{T21.2}
	 There is a residue (and in particular non-empty) subset of $\Pa$ such that for every $\q$ in this subset, any critical point $\fa\in \Crit(\CSd_{\omega})$ is non-degenerate. For such a perturbation, $\Crit(\CSd_{\omega})$ comprises a finite collection of gauge orbits. 
\end{theorem}
\begin{proof} The proof follows the same argument as in \cite[Section 12.5]{Bible} with one slight modification, as we explain now. Suppose for some $\q\in \Pa$ and $\fa\in \Crit{\CSd_{\omega}}$, the tangent vector $v=(0, \delta b, \delta \psi)\neq 0$ lies in the kernel of $\EHess_{\q,\fa}$:
	\begin{equation}\label{E21.2}
	(0, \delta b, \delta \psi)\in \ker \EHess_{\q,\fa}. 
	\end{equation}
	We have to show that $v$ is separated by a cylinder function. To apply Proposition \ref{P15.6}. we verify that $v$ is not generated by the infinitesimal gauge action on $Y$. Suppose on the contrary that
	\begin{equation}\label{E21.1}
	(\delta b, \delta \psi)=\bd_{\fa} \xi\text{ on }  Y
	\end{equation}
	for some $\xi\in L^2_{k+1}(\hy, i\R)$, then by the unique continuation property of tangent vectors, Theorem \ref{T20.20}, for a possibly different function $\xi'\in L^2_{k+1}(\hy,\R)$, the equation \eqref{E21.1} holds on $\hy$:
	\[
		(\delta b, \delta \psi)=\bd_{\fa} \xi'. 
	\]
	By \eqref{E21.2}, $\bd_\fa^*(\delta b, \delta \psi)=0$ , so $(\delta b, \delta \psi)$ is $L^2$-orthogonal to the subspace $\J_{k,\fa}\subset \CT_{k,\fa}$. This implies that $v=0$, which a contradiction. Alternatively, we may apply the linearized version of \cite[Theorem 7.2.1]{Bible} on the 4-manifold
	\[
	S^1\times \hy,
	\]
	which possesses a cylindrical end $S^1\times [0,\infty)_s\times \Sigma$. Now we use Proposition \ref{P15.6} to find a cylinder function $f\in \Cylin(Y)$ supported on $Y\subset \hy$ such that 
	\[
	df(v)\neq 0. 
	\]
	The rest of the proof then follows \cite[Section 12.5]{Bible}.
\end{proof}

\subsection{Transversality on Cylinders} Suppose a tame perturbation $\q_1=\grad f_1$ in the residue subset of Theorem \ref{T21.2} has been chosen.  Then the critical set of $\CSd_{\omega}^1\colonequals \CSd_{\omega}+f_1$ consists of a finite collection of gauge orbits; let their representatives be 
\[
\fa_i,\ 1\leq i\leq r. 
\]

We wish to find a closed Banach subspace $\Pa'$  of $\Pa$ such that for any generic $\q_2\in \Pa'$ with $\|\Pa\|\ll 1$, the sum 
\[
\q=\q_1+\q_2
\]
is an admissible perturbation. The Banach subspace $\Pa'$ that we consider is 
\begin{equation}\label{E21.7}
\Pa'\colonequals \{\q_2\in \Pa:  \q_2(\fa_i)=0, \D_{\fa_i}^1\q_2=0, \forall i=1,\cdots, r \},
\end{equation}
so the perturbation $\q_2$ vanishes to the first order at each representative $\fa_i$. The subspace $\Pa'$ is clearly closed inside $\Pa$. Let us first verify the property \ref{E1} for $\q=\q_1+\q_2$. 
\begin{lemma}[\cite{Bible} Lemma 15.1.2]\label{L21.3} There exists some $\eta>0$ such that for any $\q_2=\grad_2 f_2\in \Pa'$ with $\|\q_2\|_{\Pa}<\eta$, the critical set of $\CSd_{\omega}\colonequals \CL_{\omega}+(f_1+f_2)$ agrees with that of $\CSd_\omega^1=\CL_{\omega}+f_1$. As a result, the first condition \ref{E1} of Definition \ref{D19.3} continues to hold for the sum $\q=\q_1+\q_2$.  
\end{lemma}

In particular, for any $\q_2\in \Pa'$, the critical points of $\CSd_{\omega}$ in the quotient configuration space $\CB_{k-1/2}(\hy,\bs)$ are still given by $[\fa_i], 1\leq i\leq r$ and
\[
\D_{\fa_i} \grad(\CL_{\omega}+f_1)=\D_{\fa_i} \grad(\CL_{\omega}+f_1+f_2),\ 1\leq i\leq r.
\]
So each $\fa_i$ is still non-degenerate in the sense of Definition \ref{D18.2}. Here $[\fa_i]$ is the image of $\fa_i$ in $\CB_{k-1/2}(\hy,\bs)$. 

\begin{proof}[Proof of Lemma] Suppose on the contrary that there is a sequence of tame perturbations $\q_2^{(j)}\in \Pa'$ and a sequence of configurations $\beta_i\in \SC_{k-1/2}(\hy,\bs)$ such that 
	\[
	\|\q_j\|_\Pa\to 0,\ (\grad\CSd_{\omega}^1+\q_2^{(j)})(\beta_j)=0
	\]
	and each $\beta_j$ is not gauge equivalent to any of $\fa_i, 1\leq i\leq r$. By Proposition \ref{P1.5}, a subsequence of $\{\beta_j\}$ converges to some $\fa_i$ up to gauge. Fix $0<\epsilon\ll 1$ and let $\SO_i(\epsilon)$ be the $\epsilon$-neighborhood of $\fa_i$ in $\SC_{k+1/2}(\hy,\bs)$. When $j\gg 1$, each $\beta_j\in \SO_i(\epsilon)$, and one may use gauge transformations to put $\beta_j$ into the Coulomb gauge slice at $\fa_i$, i.e.
	\[
	\bd_{\fa_i}^*(\beta-\fa_i)=0.
	\]
 Then
	\begin{equation}\label{E21.8}
	\grad\CSd_\omega^1(\beta_j)-\grad\CSd_\omega^1(\fa_i)=-(\q_2^{(j)}(\beta_j)-\q_2^{(j)}(\fa_j)).
	\end{equation}
	As $\fa_i$ is non-degenerate as a critical point of $\CSd_{\omega}^1$, the $L^2_{k-1/2}$-norm of the left hand side is bounded below by 
	\[
	c\|\beta_j-\fa_i\|_{L^2_{k+1/2,\fa_i}}
	\]
	for some $c>0$. On the other hand, as $\q_2^{(j)}\to 0$ in $\Pa$, the $C^2$-norm of $\q$ over the bounded neighborhood $\SO_i(\epsilon)$ converges to zero, by Corollary \ref{C14.18}:
	\[
	\sup_{\gamma\in \SO_i(\epsilon)}\|\D^2_\gamma\q_2^{(j)}\|\to 0 \text{ as } j\to\infty.
	\]
	As a result, the $L^2_{k-1/2}$-norm of the right hand side of \eqref{E21.8} is bounded above by 
	\[
	\|\beta_j-\fa_i\|_{L^2_{k-1/2,\fa_i}}^2\leq 	\|\beta_j-\fa_i\|_{L^2_{k+1/2,\fa_i}}^2\leq \epsilon \|\beta_j-\fa_i\|_{L^2_{k+1/2,\fa_i}},
	\]
	when $j\gg1$, which yields a contradiction if $\epsilon<c$. 
\end{proof}

Theorem \ref{T21.1} now follows from the strong unique continuation property Theorem \ref{T20.9}-\ref{T20.10} together with Lemma \ref{L19.4}. The proof is modeled on \cite[Section 15]{Bible}. In what follows, we will only point out the necessary changes to be made. 

\begin{proof}[Proof of Theorem \ref{T21.1}] Let $\fa,\fb\in \Crit(\CSd_{\omega})$ be critical points of $\CSd_{\omega}$ and $\hz=\R_t\times \hy$ be the infinite cylinder. Following the scheme of \cite[Section 15]{Bible} and notations from Subsection \ref{Subsec19.1}, it suffice to show for any $\q_2\in \Pa'$ and any solution $\gamma\in \SC_k(\fa,\fb)$ to the perturbed equation
	\[
	0=\F_{\hz, \q}\colonequals \F_{\hz}+\hq,
	\]
the operator 
\begin{align}\label{E21.3}
\Pa'\times L^2_k(\hz, i\R\oplus iT^*\hy\oplus S)\to L^2_{k-1}(\hz, i\R\oplus iT^*\hy\oplus S)\\
(\delta \q, V)\mapsto \delta\hq (\gamma)+(\bd_\gamma^*, \D_\gamma\F_{\hz,\q} )(V)\nonumber
\end{align}
is surjective. The section $\delta\hq(\gamma)$ lies in $L^2_{k-1}$ as the underlying path $\cgamma: \R\to \SC_{k-1/2}(\hy,\bs)$ decay exponentially to either $\fa$ or $\fb$ as $t\to\pm\infty$ and $\delta \q$ vanishes at $\fa$ and $\fb$ to the first order. 

Suppose first that $\delta \hq=0$ in \eqref{E21.3}, then \eqref{E21.3} becomes a Fredholm operator by Proposition \ref{P19.1}, and its cokernel is finite dimensional. It remains to show that for any section
\[
U=(\delta c'(t), \delta b'(t), \delta\psi'(t))\in L^2(\hz, i\R\oplus iT^*\hy\oplus S)
\]
that is $L^2$-orthogonal to the image of $(\bd_\gamma^*, \D_\gamma\F_{\hz,\q} )$, there exist some $\delta \hq\in \Pa'$ such that 
\begin{equation}\label{E21.4}
\langle \delta \hq(\gamma(t)), U\rangle_{L^2(\R\times \hy)}\neq 0. 
\end{equation}

We first explain how to achieve \eqref{E21.4} for a generalized cylinder function $f: \SC_{k-1/2} (\hy,\bs)\to \R$:
\begin{equation}\label{E21.6}
\int_{t\in\R_t} \langle \grad f(\cgamma), U(t)\rangle_{L^2(\hy)}\neq 0. 
\end{equation}

By the unique continuation properties, Theorem \ref{T20.9}, \ref{T20.20} and \ref{T20.10},  the underlying path $\cgamma: \R\to \SC_{k-1/2}(\hy,\bs)$ satisfies the following properties
\begin{itemize}
\item for any $t_1\neq t_2\in \R_t$, $\cgamma(t_1)$ and $\cgamma(t_2)$ are not gauge equivalent over $Y$;
\item for any $t\in \R_t$, $\cgamma(t)$ is not gauge equivalent to $\fa_i$ on $Y$ for any $1\leq i\leq r$; moreover, $\cgamma(t)$ is irreducible on $Y$;
\item for any $t\in \R_t$, its derivative $\pt \cgamma(t)$ is not generated by the infinitesimal gauge action over $Y$.
\end{itemize}

As for the section $U$ in the cokernel, by  Lemma \ref{L19.4}, we have 
\begin{itemize}
\item $\delta c'(t)\equiv 0$;
\item for some $t_0\in \R_t$, $U(t_0)=(0,\delta b'(t_0),\delta\psi'(t_0))$ are not generated by the infinitesimal gauge action over $Y$. 
\end{itemize}

 Take a large constant $T>0$ such that $t_0\in [-T,T]$. To apply Theorem \ref{T15.17}, let the compact subset $K$ be the image of
\[
\{\fa_i:1\leq i\leq r \}\bigcup \{\cgamma(t): t\in [-T,T]\}
\]
in the quotient configuration space $\CB^*_{k-1/2}(\hy,\bs)$. Then we can find a finite collection of cylinder functions $\{f_j,\ 1\leq j\leq l\}$ defined using embeddings $\iota_j: S^1\times D^2\embed Y$ such that the map
\[
\Xi'_*=(f_1,\cdots, f_l): \CB_{k-1/2}(\hy,\bs)\to\R^l
\]
gives an embedding of $K$ and $\Xi'(U(t_0))\neq 0$. Choose a smooth function 
\[
g': \R^l\to \R
\]
supported in a small neighborhood $\overline{\Omega}$ of $\Xi'([\cgamma(t_0)])$ with the following additional properties
\begin{itemize}
\item $ \Xi'([\fa_i])\not\in \Omega,\ \forall 1\leq i\leq r; $
\item $(\Xi'\circ \cgamma)^{-1}(\overline{\Omega})$ is a small connected interval $[t_0-\epsilon_1,t_0+\epsilon_2]$ around $t_0$; to achieve this, we take $T\gg 1$;
\item lastly, the integral
\begin{equation}\label{E21.10}
\int_{\R_t} dg'(\Xi'_*(U(t)))dt\neq 0.
\end{equation}
\end{itemize}

The last property would be impossible if for some constant $\alpha\in \R$, $\Xi_*'(U(t))=\alpha \Xi_*'(\pt\cgamma_t)$ for any $t\in [t_0-\epsilon,t_0+\epsilon]$. However, this cannot hold for the whole real line; otherwise one may draw a contradiction from equations \eqref{E19.2} and \eqref{E19.6}. Then we can achieve \eqref{E21.10} by taking a different time slice $t_0\in \R_t$ and possibly a different $\Xi'$.
\medskip

As a result, the inequality \eqref{E21.6} is achieved for the composition:
\[
f\colonequals g'\circ \Xi': \CB_{k-1/2}(\hy,\bs)\to \R,
\]
Note that $f\equiv 0$ in some $L^2_{k-1/2}$-neighborhood of $\{[\fa_i]:1\leq i\leq r\}$, so $\grad f$ satisfies the constraints in \eqref{E21.7}. By the density of the Banach space $\Pa$, we can approximate $\grad f$ by an element $\delta\hq$ in $\Pa'$ and the inequality \eqref{E21.4} holds for this approximation. 

The rest of the proof follows the same line of argument as in \cite[Proposition 15.1.3]{Bible}.
\end{proof}
\subsection{Transversality on 4-Manifolds in General}\label{Subsec21.3}

Recall the set up from Section \ref{Sec22}. For a morphism $\x: (\y_1,\bs_1)\to (\y_2,\bs_2)$, the Seiberg-Witten equations $\F_{\CX,\p}=0$ on the complete Riemannian 4-manifold $\CX$ is perturbed by a quadruple 
\[
\p=(\q_1,\q_2,\q_3,\omega_3).
\]
While $(\q_1,\q_2)$ are encoded in the objects $\y_1$ and $ \y_2$, the pair 
\[
(\q_3,\omega_3)\in \Pa(Y_3)\times \Pa_{\form}
\]
is the actual perturbation that allows us to achieve transversality.

\begin{definition}\label{D21.4} The quadruple $\p$ is said to be admissible if 
	\begin{itemize}
	\item each $\q_i\in \Pa(Y_i), i=1,2$ is admissible in the sense of Definition \ref{D19.3};
	\item  for any \spinc cobordism $(\hx, \bs_X): (\hy_1,\bs_1)\to (\hy_2,\bs_2)$ (with a prescribed planar metric $g_X$), the moduli space $\M_k(\fa_1,\CX,\fa_2)$ is regular in the sense of Definition \ref{D22.2}. Here $\fa_i\in \Crit(\CSd_{\omega_i,\hy_i})$ is a critical point of the perturbed Chern-Simons-Dirac functional $\CSd_{\omega_i,\hy_i}$ on $\hy_i$, $i=1,2$. \qedhere
	\end{itemize}
\end{definition}

\begin{theorem}\label{T21.5} Under above assumptions, for any fixed admissible perturbations $(\q_1,\q_2)$ on $\hy_1$ and $\hy_2$ respectively, there is a residue subset of $\Pa(Y_2)\times \Pa_{\form}$ such that for every pair $(\q_3, \omega_3)$ in this subset, the quadruple $\p$ is admissible. 
\end{theorem}
\begin{proof} Following the proof of Theorem \ref{T21.1}, it suffices to verify that the operator 
\begin{align}\label{E21.9}
\Pa(Y_2)\times \Pa_{\form}\times L^2_k(\CX, iT^*\CX\oplus S^+)&\to  L^2_{k-1}(\CX, i\R\oplus i\su(S^+)\oplus S-)\\
(\delta\q_3, \delta\omega_3, V)&\mapsto (\bd_\gamma^*, \D_\gamma\F_{\CX,\p})V+\beta_0(t)\delta\hq_3(\gamma)+\rho_4(\delta\omega_3^+),\nonumber
\end{align}
is surjective, for any solution $\gamma\in \SC_k(\fa_1,\CX,\fa_2)$ to the perturbed equation $\F_{\CX,\p}=0$. We begin with $(\delta \q_3, \delta \omega_3)=0$, then \eqref{E21.9} becomes a Fredholm operator by Proposition \ref{P22.1}. Suppose $U\in L^2(\CX, i\R\oplus i\Lambda^+ \CX\oplus S^-)$ is $L^2$-orthogonal to the image of $(\bd_\gamma^*, \D_\gamma\F_{\CX,\p})$. it remains to find $(\delta \hq_3,\delta\omega_3)$ such that 
\begin{equation}\label{E22.10}
\langle U, \beta_0(t)\delta\hq_3(\gamma)+\rho_4(\delta\omega_3^+)\rangle_{L^2}\neq 0. 
\end{equation}

Let $I=[1,2]_t$ and write 
\[
U=(\delta\xi, \delta\omega,\delta\phi) \text{ with }\delta\xi\in L^2(\CX, i\R). 
\]
The same argument as in the proof of Lemma \ref{L19.4} implies that $\delta\xi\equiv 0$. The inner product \eqref{E22.10} is supported on the compact submanifold
\[
\hz\colonequals I\times\hy_2,
\]
over which the formal adjoint of $(\bd_\gamma^*, \D_\gamma\F_{\CX,\p})$ is cast into the form \eqref{E19.4}. If instead we write
\[
U(t)=(0,\delta b(t),\delta\psi(t))\in L^2_1(\hz, i\R\oplus iT^*\hy_2\oplus S) \text{ on } I\times \hy_2,
\]
then we are back to the cylindrical case. Here we have used the bundle map
\[
(\rho_3, \rho_4(dt))
\]
to identify $iT^*\hy_2\oplus S$ with $i\su(S^+)\oplus S^-$ over $\hz$. 

However, Lemma \ref{L19.4} does not apply directly here, so we argue as follows. If there exists some $t_0\in \supp \beta_0\subset [1,2]$ such that $U(t_0)$ is separated by some cylinder function $f$, then we set $\delta\omega_3=0$ and proceed as in the proof of Theorem \ref{T20.9}. 

If not, then by the proof of Lemma \ref{L19.4}, for any $t\in [5/4, 7/4]$, there exists some function  $\xi(t)\in L^2_1(\hy, i\R)$ such that
\[
(\delta b(t),\delta\psi(t))= \bd_{\cgamma(t)}\xi(t) \text{ on } \{t\}\times Y_2.
\] 
Moreover, 
\begin{equation}\label{E22.11}
\dt d_{Y_2} \xi(t)\equiv 0 \text{ and } \Delta_{Y_2}\xi(t)+\xi(t)|\Psi(t)|^2=0 \text{ on } [5/4, 7/4]\times Y_2.
\end{equation}
Recall that $\delta\omega_3=-\beta_0(t)dt\wedge d_{Y_2}f_3$ for a compactly supported function $f_3: I\times Y_2\to i\R$, so
\[
\rho_4(\delta\omega_3^+)=\rho_3(d_{Y_2}(\beta_0(t) f_3)). 
\]

If $U$ is orthogonal to $\rho_4(\delta\omega_3^+)$ for any $\delta\omega_3\in \Pa_{\form}$, then $\Delta_{Y_2}\xi(t)\equiv 0$. By \eqref{E22.11}, $U(t)\equiv 0$ on $[5/4, 7/4]\times Y_2$. By unique continuation, $U\equiv 0$ on the whole manifold $\CX$. 
\end{proof}

\part{Floer Homology}\label{Part7}

Let $(\y,\bs)\in \SCob_s$ be an object in the strict \spinc cobordism category, as defined in Section \ref{Sec2}.  The underlying 3-manifold $Y$ of $\y$ is compact, connected and oriented, whose boundary is identified with a disjoint union of 2-tori $\Sigma$ by the diffeomorphism $\psi:\partial Y\to \Sigma$. The quintuple $\y=(Y,\psi,g_Y,\omega,\{\q\})$ also dictates a cylindrical metric $g_Y$ and a closed 2-form $\omega\in \Omega^2(Y,i\R)$.  $\bs\in \Spincr(Y)$ is a relative \spinc structure of the 3-manifold $Y$. 

\smallskip

 The primary goal of this part is to define the functor
\[
\HM_*: \SCob_s\to \NR\Mod
\]
which assigns the monopole Floer homology $\HM_*(\y,\bs)$ for each object $(\y,\bs)\in \SCob_s$, generalizing the construction of Kronheimer-Mrowka for closed 3-manifolds. 

\smallskip

So far we have addressed two fundamental problems in order to define the functor $\HM_*$:
\begin{itemize}
\item the compactness issue; see Theorem \ref{T11.1}  for the unperturbed equations and Theorem \ref{T1.4} for the perturbed ones;
\item the transversality issue; see Theorem \ref{T21.1} for the case of cylinders and Theorem \ref{T21.5} for morphisms in $\SCob_s$. 
\end{itemize}

Although the proof of the gluing theorem is omitted in this paper, it follows from the standard procedure in \cite[Section 17-19]{Bible}, as noted in Subsection \ref{Subsec19.4}.  

\smallskip

Now the construction of monopole Floer homology becomes straightforward by following the standard argument. Part \ref{Part7} is organized as follows. In Section \ref{Sec27}, we explain the basic construction using $\BF_2$-coefficient. Section \ref{Sec28} is devoted to the canonical grading as well as the canonical mod $2$ grading of $\HM_*(\y,\bs)$.

 In Section \ref{Sec29}, we address the orientation issue, which allows us to define the monopole Floer homology $\HM_*(\y,\bs)$ using $\Z$-coefficient. The key ingredient is the notion of relative orientations, which compare the orientations of two Fredholm operators using the excision principle, cf. Theorem \ref{T24.2} and Definition \ref{DD.2}. The proof is postponed to Appendix \ref{AppD}. 

\section{The Basic Construction: $\BF_2$-coefficient}\label{Sec27}

In this section, we define the monopole Floer homology $\HM_*(\y,\bs)$ for each object $(\y,\bs)\in \SCob_s$ using $\BF_2$-coefficient. For the most general case, we have to use a Novikov ring $\NR_2$. To work with the field $\BF_2$ of two elements, we will pass to a subcategory of $\SCob_s$ in which case a monotonicity condition is required.
\subsection{Novikov Rings} Let us first explain the construction of $\HM_*(\y,\bs)$ using a Novikov ring 
\[
\NR_2=\{ \sum_{n_i} a_iq^{n_i}:\ a_i\in \BF_2,\ n_i\in \R,\ \lim_{i}n_i=-\infty\}, 
\]
which is a complete topological group. Each element of $\NR_2$ is a Laurent series in a formal variable $q$ with possibly infinitely many terms in negative degrees. For any object $(\y,\bs)\in \SCob_s$, the perturbation $\q=\grad f$ encoded in the quintuple $\y$ is admissible in the sense of Definition \ref{D19.3}. Let $\FC(\y,\bs)$ be the set of critical points of $\CSd_{\omega}=\CL_\omega+f$ in the quotient configuration space $\CB_k(\hy,\bs)$, then $\FC(\y,\bs)$ is a finite set by Theorem \ref{T21.2}. Then the chain group $C_*(\y,\bs)$ is freely generated by $\FC(\y,\bs)$ over $\NR_2$:
\[
C_*(\y,\bs)=\bigoplus_{[\fa]\in \FC(\hy,\bs)} \NR_2\cdot [\fa]. 
\]
with differential $\partial$ defined as
\begin{equation}\label{E27.1}
\partial [\fa]=\sum_{\substack{z\in \pi_1(\CB_k(\hy,\bs); [\fa],[\fb])\\\dim \cM_{z}([\fa],[\fb])=0}} [\fb]\cdot \#\cM_{z}([\fa],[\fb])\cdot q^{-\E_{top}^{\q}([\fa],[\fb];z)}.
\end{equation}
The unparameterized moduli space $\cM_z([\fa],[\fb])\colonequals \M_z([\fa],[\fb])/\R_t$ is defined as in \eqref{E19.13}. The topological energy $\E_{top}^\q([\fa],[\fb];z)$ for a homotopy class of paths $z\in \pi_1(\CB_k(\hy,\bs);[\fa],[\fb])$ equals twice the drop of $\CSd_{\omega}$ along $\gamma$
\[
2(\CSd_{\omega}(\fa)-\CSd_{\omega}(\fb))
\]
if $\gamma: [0,1]\to \SC_k(\hy,\bs)$ is a lift of $z$ with $\gamma(0)=\fa$ and $\gamma(1)=\fb$. This expression is suggested by Proposition \ref{P1.1}. To ensure the sum in \eqref{E27.1} is convergent in $\NR_2$, we need a finiteness result:
\begin{lemma}\label{L27.1} For any $C>0$, there are only finitely many homotopy classes of paths $z\in \pi_1(\CB_k(\hy,\bs); [\fa],[\fb])$ such that $\E_{top}^{\q}([\fa],[\fb],z)<C$ and $\cM_z([\fa],[\fb])$ is non-empty. Moreover, each $\cM_z([\fa],[\fb])$ is compact if its dimension equals zero. 
\end{lemma}

To show $\partial^2=0$, we follow the standard argument and look at the compactification of moduli spaces $\cM_z([\fa],[\fb])$ when $\dim=1$. Readers are referred to \cite[Section 22]{Bible} for the details. The monopole Floer homology of $(\y,\bs)$ is then defined as the homology of the chain complex $(C_*(\hy,\bs),\partial)$:
\[
\HM_*(\y,\bs)\colonequals H_*((C_*(\hy,\bs),\partial)). 
\]

To make $\HM_*$ into a functor:
\[
\HM_*: \SCob_s\to \NR_2\text{-}\mathrm{Mod},
\]
we assign for each morphism $\x: \y_1\to \y_2$ a chain map:
\[
m(\x; g_X,\p): (C_*(\y_1,\bs_1),\partial_1)\to (C_*(\y_2,\bs_2),\partial_2)
\]
which relies on a planar metric $g_X$ of the strict cobordism $X:Y_1\to Y_2$ and a quadruple
\[
\p=(\q_1,\q_2,\q_3,\omega_3)\in \Pa(Y_1)\times \Pa(Y_2)\times \Pa(Y_2)\times\Pa_{\form}.
\]
Here $\p$ is required to be admissible in the sense of Definition \ref{D21.4}. While $(\q_1,\q_2)$ are encoded in the objects $(\y_1,\y_2)$, $(\q_3,\omega_3)$ are the actual perturbations to the Seiberg-Witten equations on the complete Riemannian 4-manifold $\CX$. Now define 
\begin{equation}\label{E27.2}
m(\x; g_X,\q)[\fa_1]=\sum_{\substack{\bs_X\in \Spinc(X;\bs_1,\bs_2)\\\dim \M(\fa_1,\bs_X,\fa_2)=0}} [\fa_2]\cdot \# \M(\fa_1, \bs_X,\fa_2)\cdot q^{-\E_{top}^{\p}(\fa_1, \bs_X, \fa_2)},
\end{equation}
where $\fa_i$ is a lift of $[\fa_i]\in \FC(\y_i)$ in $\SC_k(\hy,\bs)$ for $i=1,2$. The moduli space $\M(\fa_1, \bs_X,\fa_2)$ is defined as in \eqref{E22.4} with the admissible quadruple $\q$ as perturbations. The topological energy is given by the formula
\begin{equation}\label{E27.4}
\E_{top}^{\p}(\fa_1, \bs_X, \fa_2)\colonequals 2\CSd_{\omega_1}(\fa_1)-2\CSd_{\omega_2}(\fa_2)+C(A_0,\omega_X)
\end{equation}
where $A_0$ is a background \spinc connection on $\hx$ such that the restriction $A_0|_{\hy_i}$ is the reference connection on $\hy_i$ that defines the Chern-Simons-Dirac functional $\CSd_{\omega_i}$ for $i=1,2$. The constant $C(A_0,\omega_X)$ is given concretely by 
\begin{equation}\label{E27.5}
C(A_0,\omega_X)=\frac{1}{4}\int_{\hx} F_{A^t_0}\wedge F_{A^t_0}-\int_{\hx}F_{A^t_0}\wedge \omega_X,
\end{equation}
as suggested by \eqref{top}. To make sense of the expression \eqref{E27.2}, we need another finiteness result:
\begin{lemma}\label{L27.2} For any $C>0$, any pair of critical points $([\fa_1],[\fa_2])\in \FC(\y_1)\times\FC(\y_2)$ and any admissible quadruple $\p$, there are only finitely many relative \spinc cobordisms  $\bs_X\in \Spincr(X;\bs_1,\bs_2)$ such that $\E_{top}^{\p}(\fa_1, \bs_X, \fa_2)<C$ and $\M(\fa_1, \bs_X,\fa_2)$ is non-empty. Moreover, each moduli space $\M(\fa_1, \bs_X,\fa_2)$ is compact if its dimension equals zero.
\end{lemma}

Lemma \ref{L27.1} and Lemma \ref{L27.2} follow from the Compactness Theorem \ref{T1.4} and its analogue for a general cobordism. Readers are referred to \cite[Corollary 31.2.5]{Bible} for more details; their proofs are omitted here. By analyzing the moduli space $\M(\fa_1, \bs_X,\fa_2)$ with $\dim=1$, we conclude that $m(\x;g_X,\q)$ is a chain map by the standard argument. The chain maps induced from different auxiliary data $(g_X,\q)$ are all chain homotopic to each other, so the resulting maps on the homology are independent of $(g_X,\q)$
\[
\HM(\x)\colonequals [m(\x;g_X,\p)]: \HM_*(\y_1,\bs_1)\to\HM_*(\y_2,\bs_2),
\]

To show that $\HM$ defined this way is a functor and satisfies the composition law in Theorem \ref{1T2}, we follow \cite[Section 26]{Bible}.  

\subsection{Monotonicity} To define the monopole Floer homology using $\BF_2$-coefficient, it is necessary to pass to a subcategory of $\SCob_s$, as we explain in this subsection. 
\begin{definition}\label{D27.3} An object $(\y,\bs)=(Y,\psi, g_Y, \omega,\q,\bs)\in \SCob_s$ is called monotone if the period class $[\omega]\in H^2(Y; i\R)$ is proportional to the image of $c_1(\bs)$ in $\im(H^2(Y,\partial Y;\Z)\to H^2(Y;\R)$: 
	\[
	[\frac{\omega}{\pi i}]=\alpha\cdot c_1(\s) \in H^2(Y; \R) \text{ for some }\alpha\in \R.  
	\] 
	In addition, $(\hy,\bs)$ is called
	\begin{itemize}
\item positively monotone if $\alpha<1$;
\item balanced if $\alpha=1$;
\item negatively monotone if $\alpha>1$.\qedhere
	\end{itemize}
\end{definition}

In light of Lemma \ref{D9.4}, under the monotonicity assumption, we have 
\[
\CSd_{\omega}(u\cdot \gamma)-\CSd_{\omega}(\gamma)=2(1-\alpha)\pi^2[u]\cup c_1(\s),
\]
for any $\gamma\in \SC_k(\hy,\bs)$ and $u\in \CG_{k+1}(\hy)$. In particular, $\CSd_{\omega}$ becomes a real valued functional if $(\y,\bs)$ is balanced. One necessary condition of monotonicity is that  $\mu=0$. The construction described below will work in general for any monotone objects, but let us focus on the special case when the period class $[\omega]=0\in H^2(Y; i\R)$ and the form $\bomega$ defined in \ref{P5} vanishes, for the sake of simplicity; so 
\[
\omega=\omega_\lambda=\chi_1(s)ds\wedge\lambda.
\]
In this case, $(\hy,\bs)$ is always positively monotone, since $\alpha=0$.

Under this assumption, the chain group $C_*(\y,\bs; \BF_2)$ is a finite dimensional $\BF_2$-vector space:
\[
C_*(\y,\bs; \BF_2)\colonequals \bigoplus_{[\fa]\in \FC(\hy,\bs)} \BF_2\cdot [\fa]. 
\]
with differential defined by 
\begin{equation}\label{E27.3}
\partial [\fa]=\sum_{\substack{z\in \pi_1(\CB_k(\hy,\bs); [\fa],[\fb])\\\dim \cM_{z}([\fa],[\fb])=0}} [\fb]\cdot \#\cM_{z}([\fa],[\fb])
\end{equation}

In light of Lemma \ref{L27.1}, to make sense of  this expression, we need an upper bound on the topological energy $\E_{top}^\q([\fa],[\fb];z)$: 
\begin{lemma}\label{L27.4} For any $[\fa],[\fb]\in \FC(\y,\bs)$, there exists a constant $C>0$ such that $$\E_{top}^\q([\fa],[\fb];z)<C,$$
for any homotopy classes of paths $z\in \pi_1(\CB_k(\hy,\bs), [\fa],[\fb])$ with $\dim \cM_{z}([\fa],[\fb])=0$. 
\end{lemma}

As for a morphism $\x: (\y_1,\bs_1)\to (\y_2,\bs_2)$ with $\omega_1=\omega_2=\omega_\lambda$, $\bomega_X$ is a compactly supported 2-form  (see \ref{Q7}) on $X$. We require that the class defined in \ref{Q8} vanishes: $[\omega_X]_{cpt}=0\in H^2(X,\partial X;\Z)$. This time the chain map $m(\x;g_X, \q)$ is defined as 
\begin{align*}
m(\x;g_X,\q): C_*(\y_1,\bs_1;\BF_2)&\to C_*(\y_2,\bs_2;\BF_2)\\
[\fa_1]&\mapsto \sum_{\substack{\bs_X\in \Spinc(X;\bs_1,\bs_2)\\\dim \M(\fa_1,\bs_X,\fa_2)=0}} [\fa_2]\cdot \# \M(\fa_1, \bs_X,\fa_2).
\end{align*}

Again, we need a upper bound on $\E_{top}(\fa_1,\bs_X,\fa_2)$ to ensure the sum in the expression above is finite: 
\begin{lemma}\label{L27.5} Under above assumptions, for any pair of critical points $([\fa_1],[\fa_2])\in \FC(\y_1,\bs_1)\times\FC(\y_2,\bs_2)$, any planar metric $g_X$ and any admissible quadruple $\p$, there is a constant $C>0$ such that $$\E_{top}^\p(\fa_1, \bs_X, \fa_2)<C$$
which holds for any $\bs_X\in\Spincr(X;\bs_1,\bs_2)$ with $\dim\M(\fa_1, \bs_X,\fa_2)=0$. 
\end{lemma} 

Lemma \ref{L27.4} and \ref{L27.5} follow directly from a general statement relating the dimension with the topological energy $\E_{top}$. In Proposition \ref{P27.6} below, we will think of a homotopy class of paths as a relative \spinc cobordism, following the ideas in Subsection \ref{Subsec2.4}. 
\begin{proposition}\label{P27.6} Under above assumptions, for any relative \spinc cobordism $\bs_{X},\bs_{X}'\in\Spincr(X;\bs_1,\bs_2)$, we have 
\begin{align*}
\E_{top}(\fa_1, \bs_X', \fa_2)-\E_{top}(\fa_1, \bs_X, \fa_2)= -4\pi^2\big(\dim \M(\fa_1, \bs_X',\fa_2)-\dim\M(\fa_1, \bs_X,\fa_2)\big)
\end{align*}
\end{proposition}

In particular, the topological energy $\E_{top}^\q(\fa_1,\bs_X,\fa_2)$ is independent of the choice of $\bs_X\in \Spincr(X;\bs_1,\bs_2)$ if $\dim \M(\fa_1, \bs_X,\fa_2)=0$.
\begin{proof} Suppose $\bs_X'=\bs_X\otimes L$ for a relative complex line bundle in the class $[L]\in H^2(X,\partial X; \Z)$. In terms of \eqref{E27.4} and \eqref{E27.5}, we compute the difference of the topological energy
	\begin{align*}
	\E_{top}(\fa_1, \bs_X', \fa_2)-\E_{top}(\fa_1, \bs_X, \fa_2)&=C(A_0(\bs_X'),\omega_X)-C(A_0(\bs_X),\omega_X)\\
	&=-2\pi^2[L]\cup (c_1(\bs_X)+c_1(\bs_X'))[X,\partial X]\\
	&=-4\pi^2 [L]\cup (c_1(\bs_X)+[L])[X,\partial X]. 
	\end{align*}
	where $c_1(\bs_X)$ and $\ c_1(\bs_X')$ are understood as elements in $H^2(X, [-1,1]\times \Sigma;\Z)$. On the other hand, pick an arbitrary non-vanishing section $\Phi_0$ of 
	\[
	S^+|_{\partial X}\to \partial X. 
	\]
	Any relative \spinc structure $\bs_X\in \Spincr(X;\bs_1,\bs_2)$ dictates an identification of $\bs_X|_{\partial X}$ with a standard \spinc structure on the boundary $\partial X$, so it makes sense to define the relative Euler number $
	e(\bs_X; \Phi_0)[X,\partial X]$ for any non-vanishing section $\Phi_0$ of the spin bundle $S^+\to \partial X$. In particular, 
	\[
	\big(e(\bs_X'; \Phi_0)-e(\bs_X; \Phi_0)\big)[X,\partial X]=[L]\cup (c_1(\bs_X)+[L])[X,\partial X]. 
	\]
	In Proposition \ref{P23.5} below, we will associate a homotopy class of non-vanishing sections $[\Phi_0(\fa_1,\fa_2)]$ to any pair $(\fa_1,\fa_2)$ such that 
	\begin{equation}\label{E27.6}
	e(\bs_X; \Phi_0(\fa_1,\fa_2))[X,\partial X]=\dim \M(\fa_1,\bs_X,\fa_2)
	\end{equation}
	for any $\bs_X\in \Spincr(X;\bs_1,\bs_2)$. In fact, \eqref{E27.6} follows from the Index Axiom \ref{Axiom1} of the canonical grading of $\HM_*(\y,\bs)$. Another approach is to show
	\[
		\big(e(\bs_X'; \Phi_0)-e(\bs_X; \Phi_0)\big)[X,\partial X]=\dim \M(\fa_1, \bs_X',\fa_2)-\dim\M(\fa_1, \bs_X,\fa_2)
	\]
	for any non-vanishing section $\Phi_0$ directly using the excision principle. This completes the proof of Proposition \ref{P27.6}
\end{proof}

Finally, one has to verify that $m(\x;g_X,\q)$ is a chain map and a generic homotopy of auxiliary data $(g_X,\q)$ gives rise to a chain homotopy of $m(\x;g_X,\q)$. The argument is not different from that of \cite[Section 25]{Bible}. 
\section{Canonical Gradings}\label{Sec28}

In this section, we introduce the canonical grading of the monopole Floer homology $\HM_*(\y,\bs)$. It is more natural to think of the grading set of $\HM_*(\y,\bs)$
\[
\Xi^{\pi}(\y,\bs)
\]
as the space of unit-length relative spinors on $\hy$ modulo gauge transformations, identified also as a subset of homotopy classes of oriented relative 2-plane fields on $Y$. In particular, 
\[
\Xi^{\pi}(\y,\bs_1)=\Xi^{\pi}(\y,\bs_2)
\]
if $\bs_1$ and $\bs_2$ come down to the same \spinc structure on $Y$. 

The main result of this section is Proposition \ref{P23.5}, which characterizes the canonical grading in terms of the Index Axiom \ref{Axiom1} and the Normalization Axiom \ref{Axiom2}. They are inspired by the following index computation for a closed Riemannian 4-manifold $X$: 
\[
\dim \M(X,\s_X)=e(\s_X)[X]
\]
where $\M(X,\s_X)$ is the Seiberg-Witten moduli space and $e(\s_X)$ is the Euler class of the spin bundle $S_X^+\to X$. The canonical mod 2 grading will be discussed in Subsection \ref{Subsec23.2}.

\subsection{Homotopy Classes of Oriented Relative 2-Plane Fields} For a closed $3$-manifold $Y$, recall that the three flavors of monopoles Floer homology:
\[
\widecheck{\HM}_\bullet(Y),\ \widehat{\HM}_\bullet(Y),\ \overline{\HM}_\bullet(Y)
\]
defined in the book \cite{Bible} are graded by the homotopy classes of oriented 2-plane fields over $Y$. The analogous statement continues to hold in our case, using \textbf{relative} oriented 2-plane fields instead, as we explain now. The following lemma from \cite{Bible} explains the relationship between 2-plane fields and \spinc structures:
\begin{lemma}[\cite{Bible} Lemma 28.1.1]\label{L18.1} On an oriented Riemannian 3-manifold $Y$, there is a bijection between
	\begin{enumerate}[label=(\roman*)]
\item oriented $2$-plane fields $\xi$;
\item 1-forms $\theta$ of length $1$; and 
\item isomorphism classes of pairs $(\s, \Psi)$ comprising a \spinc structure and a unit-length spinor $\Psi$. 
	\end{enumerate}
\end{lemma}

Over the infinite cylinder $\R_s\times \Sigma$, we defined in \eqref{E2.6} a preferred $\R_s$-translation invariant solution 
\[
\gamma_*=(B_*,\Psi_*)
\]
to the perturbed Seiberg-Witten equations $\eqref{3DSWEQ}$. The perturbation is provided by a covariantly constant 2-form 
\[
\omega_*\colonequals \mu+ds\wedge\lambda
\]
The correspondence in Lemma \ref{L18.1} then identifies 
\begin{equation}\label{E23.1}
\text{the unit length 1-form }\theta_*\colonequals  i*_3 \frac{\omega_*}{|\omega_*|} \leftrightarrow \text{the unit length spinor }\frac{\Psi_*}{|\Psi_*|},
\end{equation}
Indeed, as $\gamma_*$ solves the equations \eqref{3DSWEQ}, $(\Psi_*\Psi_*^*)_0=\rho_3(*_3\omega_*)$, so 
\[
\C\Psi_* \text{ and } \C(\Psi_*)^\perp
\]
are $i$ and $-i$ eigenspaces of $\rho_3(\theta_*)$ respectively. In particular, \eqref{E23.1} determines a preferred oriented 2-plane fields $\xi_*$ on $\R_s\times\Sigma$ by Lemma \ref{L18.1}. Now we return to a 3-manifold $\hy$ with cylindrical ends and state a relative version of Lemma \ref{L18.1}.
\begin{definition}\label{D18.3}
 An oriented 2-plane field $\xi$ on $\hy$ is called \textbf{relative} if $\xi$ agrees with $\xi_*$ over the cylindrical end $[0,\infty)_s\times\Sigma$. Similarly, we define 
\begin{itemize}
	\item \textbf{relative} 1-forms and
	\item  \textbf{relative} spinors
\end{itemize}
using $\theta_*$ and $\Psi_*/|\Psi_*|$ as the models along the end $[0,\infty)_s\times\Sigma$.
\end{definition}

\begin{lemma}\label{L18.2} For any object $\y\in \Cob_s$, let $\hy$ be the extended 3-manifold with cylindrical ends. Then there is a bijection between:
	\begin{enumerate}[label=(\roman*)]
		\item oriented relative $2$-plane fields $\xi$;
		\item 1-forms relative $\theta$ of length $1$; and 
		\item isomorphism classes of pairs $(\s, \Psi)$ consisting of a \spinc structure $\s$ with $c_1(\s)|_\Sigma=0\in H^2(\Sigma, \Z)$ and a unit-length spinor $\Psi$ that is gauge equivalent to a relative spinor. 
	\end{enumerate}
\end{lemma}
\begin{remark} In the last description, the identification of $\bs|_\Sigma$ is not specified and a gauge transformation does not necessarily lie in the identity component when restricted to $\Sigma$.  
\end{remark}

For each relative \spinc structure $\bs\in \Spincr(Y)$, let $\Xi(\hy,\bs)$ be the space of unit-length relative spinors on $\hy$. The index set for the monopole Floer homology $\HM_*(Y,\bs)$ will be 
\begin{equation}\label{E18.2}
\Xi^\pi(Y,\bs)\colonequals \pi_0(\Xi(\hy,\bs))/H^1(Y, \partial Y; \Z)
\end{equation}
where $H^1(Y, \partial Y; \Z)=\pi_0(\CG(\hy,\bs))$ acts on $\pi_0(\Xi(\hy,\bs))$ by gauge transformations. The last description in Lemma \ref{L18.2} suggests that 
\[
\Xi^\pi(Y,\bs_1)\cong \Xi^\pi(Y,\bs_2)
\]
if $\bs_1$ and $\bs_2$ come down to the same \spinc structure on $Y$. In this way, $\Xi^\pi(Y,\bs)$ is identified with a subset of homotopy classes of oriented relative 2-plane fields. 

\medskip

Now let us introduce the axioms that characterize the canonical grading of $\HM_*(\y,\bs)$.
\begin{definition}\label{D23.3} For any configuration $\fa\in \SC_k(\hy, \bs)$ and any tame perturbation $\q\in \Pa(Y)$, the pair $\fc=(\fa,\q)$ is called \textbf{non-degenerate} if the extended Hessian $\EHess_{\fa,\q}$ is invertible.
\end{definition}

For any non-degenerate pair $\fc=(\fa,\q)$, we will assign an element
\[
\gr(\fc)\in \pi_0(\Xi(\hy,\bs)).
\]
which descends to a map
\begin{equation}\label{E23.3}
\gr^\pi: (\SC_k(\hy,\bs)\times \Pa)/\CG_{k+1}(\hy)\dashrightarrow \Xi^{\pi}(\hy,\bs), [\fc]\mapsto [\gr(\fc)],
\end{equation}
on the ``non-degenerate locus" of the quotient space. To state the axioms that characterize the grading function $\gr$, consider a relative \spinc cobordism 
\[
(\hx, \bs_X): (\hy_1, \bs_1)\to (\hy_1,\bs_2).
\]
We defined the moduli space $\M_k(\fa_1,\CX,\fa_2)$ in Section \ref{Sec22}, when $\fa_i$ is a critical point of $\CSd_{\omega_i, Y_i}$ for $i=1,2$. However, if we are interested only in the linear theory, one may take $\fa_1$ and $\fa_2$ to be any configurations. Pick a reference configuration $\gamma$ on $\CX$ satisfying conditions \eqref{E22.3}. Then the linearized operator:
\begin{align}\label{E23.8}
\CQ(\fc_1, \bs_X, \fc_2)&\colonequals	(\bd_\gamma^*, \D_\gamma\F_{\CX,\p}): L^2_1(\CX,  iT^*\CX\oplus S^+)\to L^2(\CX, i\R\oplus i\Lambda^+ \CX\oplus S^-) \\
\text{ with } \p&=(\q_1,\q_2,0,0)\nonumber
\end{align}
is Fredholm, by Proposition \ref{P22.1}, provided that $\fc_i=(\fa_i,\q_i)$ is non-degenerate for $i=1,2$. Any such choices of $\gamma$ will provide the same operator $\CQ(\fc_1, \bs_X, \fc_2)$ up to compact terms, so the underlying path $\gamma$ is omitted from our notations. 

Now we are ready to state the axioms that characterize the grading function $\gr$. 

\begin{enumerate}[label=(A-I)]
\item\label{Axiom1} (Index Axiom) The Fredholm index of $\CQ(\fc_1, \bs_X, \fc_2)$ equals the relative Euler number:
\[
e(S^+; \Psi_1,\Psi_*/|\Psi_*|, \Psi_2)[X,\partial X]\in \Z. 
\]
where $\Psi_i$ is a unit-length relative spinor  on $\hy_i$ representing $\gr(\fc_i)$. Since $\Psi_1, \Psi_*/|\Psi_*|$ and $\Psi_2$ form a unit-length spinor of $S^+$ on the boundary 
\[
\partial X=(-Y_1)\cup [-1,1]_t\times \Sigma\cup Y_2, 
\]
the relative Euler class 
$e(S^+; \Psi_1,\Psi_*/|\Psi_*|, \Psi_2)\in H^4(X,\partial X; \Z)$ of this spinor is well-defined.
\end{enumerate}

\begin{enumerate}[label=(A-II)]
	\item\label{Axiom2} (Normalization Axiom) Suppose $\fa=(B,\Psi)\in \SC_k(\hy,\bs)$ is a configuration such that 
	\begin{enumerate}[label=(V\arabic*)]
\item\label{V1} $\Psi$ is nowhere vanishing;
\item\label{V2} $\Psi\equiv \Psi_*$ on $[0,+\infty)_s\times\Sigma$, where $\Psi_*$ is the standard spinor on $\R_s\times\Sigma$;
\item\label{V3}  for any $\tau\geq 1$, define the rescaled configuration $\fa(\tau)\colonequals (B,\tau\Psi)$; then the extended Hessian $\EHess_{\fa(\tau)}$ at $\fa(\tau)$ is always invertible for any $\tau\geq 1$.
	\end{enumerate}

We define that
	\[
	\gr(\fc)=[\Psi /|\Psi|]\in \pi_0(\Xi(\hy,\bs)) \text{ if } \fc=(\fa,0).
	\]
	Note that $\fa(\tau)$ lies in a different configuration space obtained by rescaling the boundary date $(\lambda,\mu)$. 
\end{enumerate}

\begin{enumerate}[label=(A-III)]
\item\label{Axiom3} (Equivariance Axiom) The grading function 
\[
\gr: \SC_k(\hy,\bs)\times \Pa\dashrightarrow \pi_0(\Xi(\hy,\bs))
\]
is equivariant under the action of $\CG_{k+1}(\hy)$ meaning that 
\[
\gr(u\cdot \fa, \q)=[u]\cdot \gr(\fa, \q)
\]
for any non-generate pair $(\fa,\q)$ and $u\in \CG_{k+1}(\hy)$. 
\end{enumerate}

The Index Axiom \ref{Axiom1} can not determine the grading function $\gr$ completely. On the other hand, the Equivariance Axiom \ref{Axiom3} is redundant, since it follows from \ref{Axiom1}\ref{Axiom2}. It is added to justify the quotient map $\gr^{\pi}$ in \eqref{E23.3}. Here is the main result of this section:
\begin{proposition}\label{P23.5} There exists a unique grading function \[
	\gr:\SC_k(\hy,\bs)\times\Pa\dashrightarrow \Xi(\hy,\bs)
	\]
	satisfying axioms \ref{Axiom1}\ref{Axiom2}\ref{Axiom3}.
\end{proposition}

The proof of Proposition \ref{P23.5} will dominate the rest of this subsection. It relies on two additional lemmas. On the one hand, we have to show the desired configurations in the Normalization Axiom \ref{Axiom2} exist at least for some special metrics on $Y$. 
\begin{lemma}\label{L23.4} For any $3$-manifold $Y$ with $\partial Y\cong \Sigma$, there exists some cylindrical metric $g_Y$ and a configuration $\fa\in \SC_k(\hy,\bs)$ that satisfies all constraints in Axiom \ref{Axiom2}.
\end{lemma}

On the other hand, we have to show that Axioms \ref{Axiom1} and \ref{Axiom2} are consistent.

\begin{lemma}\label{L23.6} For any relative \spinc cobordism $(\hx, \bs_X): (\hy_1, \bs_1)\to (\hy_1,\bs_2),$ suppose non-generate pairs $\fc_i=(\fa_i,0), i=1,2$ are given as in \ref{Axiom2}, then 
	\[
	\Ind \CQ(\fc_1, \bs_X, \fc_2)=e(S^+; \frac{\Psi_1}{|\Psi_1|}, \frac{\Psi_*}{|\Psi_*|}, \frac{\Psi_2}{|\Psi_2|})[X,\partial X],
	\]
	where $\Psi_i\in \Gamma(\hy_i,S)$ is the spinor component of $\fa_i\in \SC_k(\hy_i,\bs_i)$. 
\end{lemma}

\begin{proof}[Proof of Lemma \ref{L23.6}] This lemma is in the spirit of \cite[Theorem 3.3]{KM97} and we follow the argument therein. When $X_1$ is a closed Riemannian 4-manifold, the index formula:
	\[
	\dim \M(X_1,\bs_{X_1})=e(S^+)[X_1]
	\]
	is a consequence of the Atiyah-Singer Index Theorem and \cite[Lemma 28.2.3]{Bible}. Using the excision principle, this allows us to reduce Lemma \ref{L23.6} to the special case when 
	\[
	e(S^+; \Psi_1,\Psi_*,\Psi_2)[X,\partial X]=0.
	\]
	At this point, choose a reference configuration $\gamma=(A,\Phi)$ on $\CX$ such that the spinor $\Phi$ is non-vanishing everywhere, and 
	\[
	\gamma|_{\HH^2_+\times\Sigma}=(A_*,\Phi_*) 
	\]
	is the standard configuration on the planar end. By rescaling the spinor $\Phi$, we define
	\[
	\gamma(\tau)\colonequals (A,\tau\Phi). 
	\]
	which lies a different configuration space on $\CX$. As the pair $\fc_i(\tau)\colonequals (\fa_i(\tau),\q_i=0), i=1,2$ are non-degenerate for any $\tau\geq 1 $ by assumption \ref{V3}, the linearized operator at $\gamma(\tau)$ gives rise to a continuous family of Fredholm operators:
	\[
	\CQ(\tau)\colonequals \CQ(\fc_1(\tau), \bs_X, \fc_2(\tau)).
	\]
	The proof of \cite[Lemma 3.11 \& Corollary 3.12]{KM97} is valid here, as $\q_i=0, i=1,2$. As a result, $\CQ(\tau)$ is invertible when $\tau\gg 1$; so
	\[
	\Ind \CQ(1)=\lim_{\tau\to\infty}\Ind\CQ(\tau) =0. \qedhere
	\] 
\end{proof}

\begin{proof}[Proof of Lemma \ref{L23.4}] Following the proof of Lemma \ref{L23.6}, one can easily show the extended Hessian $\EHess_{\fa(\tau)}$ is invertible when $\tau\gg1$ for any fixed configuration $\fa=(B,\Psi)$ satisfying properties \ref{V1} and \ref{V2}, but we have to pick a good metric on $\hy$ so that this range is $[1,+\infty)$.
	
If $Y_1$ is a closed 3-manifold, one may instead rescale the metric:
\[
Y_1(\tau)=(Y_1,\tau^2 g_{Y_1}). 
\]
and regard $\fa$ as a configuration on the pull-back \spinc structure on $Y(\tau)$. The Seiberg-Witten theory does not tell the difference between:
\[
(Y(\tau), \fa) \text{ and } (Y, \fa(\tau)),
\]
so for $\tau_0\gg 1$, $(Y(\tau_0), \fa)$ satisfies constraints \ref{V1}\ref{V3} in Axiom \ref{Axiom2}. 

In our case, instead of rescaling the whole manifold 
\[
\hy=Y\cup [0,\infty)_s\times \Sigma,
\]
we rescale the compact region $Y$ and insert a long cylinder:
\[
\hy(\tau)\colonequals Y(\tau)\cup [0, R(\tau)]_s\times\Sigma\cup  [0,\infty)_s\times \Sigma. 
\]
The metric of $[0,R(\tau)]_s\times \Sigma$ interpolates the metrics $\tau^2g_\Sigma$ and $g_\Sigma$ at boundary. We make this interpolation mild enough by taking $R(\tau)\gg 1$. The extension of $\fa$ over the cylinder $[0,R(\tau)]_s\times \Sigma$:
\[
(B',\Psi')
\]
must interpolate $(B_*,\Psi_*)$ at boundary in a mild way. One may use the oriented relative 2-plane field $\xi_*$ and construct the spinor $\Psi'$ using Lemma \ref{L18.1}. Now \cite[Lemma 3.11]{KM97} applies, and all constraints in \ref{Axiom2} are satisfied by 
\[
(\hy(\tau_0),\tilde{\fa})
\]
when $\tau_0\gg1$, where $\tilde{\fa}$ is the extension of $\fa$ on $\hy(\tau_0)$.
\end{proof}

\begin{proof}[Proof of Proposition \ref{P23.5}] The proof is modeled on that of \cite[Subsection 28.2]{Bible} which can now proceed with no difficulties. We first deal with the existence of $\gr$ and divide the proof in six steps. 
	
\medskip
	
\Step 1. Construction. Fix a reference relative \spinc 3-manifold $(\hy_0, \bs_0)$. Let $\fc_0=(\fa_0,0)$ be a non-generate pair constructed by Lemma \ref{L23.4}, then the value $\gr(\fc)$ is determined by \ref{Axiom2}.  Take $\Psi_0$ as a unit-length relative spinor on $\hy_0$ that represents $\gr(\fc)$.

 By \cite[Proposition 28.1.2]{Bible}, any two relative \spinc manifolds $(\hy_0,\bs_0)$ and $(\hy_1, \bs_1)$ admit a relative \spinc cobordism $(\hx, \bs_X)$
 \begin{equation}\label{E18.3}
 (\hx, \bs_X):(\hy_0, \bs_0)\to (\hy_1, \bs_1)
 \end{equation}
 
  The Index Axiom \ref{Axiom1} then determines a unique homotopy class $[\Psi_1]$ of unit-length relative spinors on $\hy_1$ such that 
\[
\Ind \CQ(\fc_0, \bs_X,\fc_1)=e(S^+; \Psi_0,\Psi_*/|\Psi_*|, \Psi_1)[X,\partial X]. 
\]
As noted in Remark \ref{R9.2}, an isomorphism 
\[
\varphi_1: (\hx,\bs_X)|_{\hy_i}\cong (\hy_1,\bs_1)
\]
is always encoded in a relative \spinc cobordism. Define  $\gr(\fc_1)\colonequals (\varphi_1)_*[\Psi_1]\in \pi_0(\Xi(\hy_1,\bs_1))$.

\medskip

\Step 2. $\gr$ is well-defined. Suppose there is another relative \spinc cobordism
\begin{equation}\label{E23.2}
(\hx_1, \bs_{X_1}): (\hy_0,\bs_0)\to (\hy_1,\bs_1),
\end{equation} 
then we reverse the orientation of $(\hx_1, \bs_{X_1})$ and form the composition:
\[
(\hx, \bs_X)\#_{(\hy_1,\bs_1)}((-\hx_1), \bs_{-X_1}): (\hy_0,\bs_0)\to(\hy_0,\bs_0). 
\]
By Lemma \ref{L23.6} and the additivity of Fredholm indices and relative Euler classes, the values of $\gr(\fc_1)$ defined using either \eqref{E18.3} or \eqref{E23.2} are equal. 

\medskip

\Step 3. Axiom \ref{Axiom1} holds for $\gr$. The proof is similar to \Step 2. Instead of \eqref{E23.2}, given any \spinc cobordism $(\hx_2, \bs_{X_2}): (\hy_1, \bs_1)\to (\hy_1,\bs_2)$, we take the pre-composition with \eqref{E18.3}:
\[
(\hx, \bs_X)\#_{(\hy_1,\bs_1)}(\hx_2, \bs_{X_2}): (\hy_0,\bs_0)\to(\hy_2,\bs_2). 
\]
The rest of the argument is unchanged.

\Step 4. Axiom \ref{Axiom2} holds for $\gr$. This is by Lemma \ref{L23.6}. 

\Step 5. Uniqueness. This is clear from \Step 1. 

\Step 6. Axiom \ref{Axiom3}. There are two ways to proceed. In  \Step 1, one may change the isomorphism $\phi_1$ by an automorphism of $(\hy_1,\bs_1)$, i.e a gauge transformation $u\in \CG_{k+1}(\hy)$. As a result, the grading function $\gr$ is gauge equivariant. 

In the second approach, we verify the following fact: for the product manifold $X=[-1,1]_t\times Y$ and $\CX=\R_t\times \hy$, 
\begin{equation}\label{E23.9}
\Ind \CQ(\fc, \bs,u\cdot \fc)=e(S^+; \Psi,\Psi_*, u\cdot \Psi)[X,\partial X]. 
\end{equation}
for any non-generate pair $\fc$ and any gauge transformation $u\in \CG_{k+1}(\hy)$ such that $u\equiv 1$ on $[0,\infty)_s\times \Sigma$. Here $\Psi$ is a relative spinor on $\hy$ representing $\gr(\fc)$. The identity \eqref{E23.9} now follows  from Lemma \ref{L19.14}. 
\end{proof}

\subsection{Canonical Mod 2 Gradings}\label{Subsec23.2} Now we focus a single relative \spinc 3-manifold $(\hy,\bs)$. In order to define the Euler characteristic of the monopole Floer homology
\[
\chi(\HM_*(\hy,\bs))
\]
we need a mod 2 reduction of the canonical grading $\gr^{\pi}$. For each non-generate pair $\fc=(\fa,\q)$, in the sense of Definition \ref{D23.3}, we will assign a number 
\begin{equation}\label{E23.4}
\grt(\fc)\in \Z/2\Z,
\end{equation}
characterized by the following axioms:
\begin{enumerate}[label=(B-\Roman*)]
\item\label{B-I} (Reduction Axiom) Let $(\hx,\bs_X)=[-1,1]_t\times (\hy,\bs)$ be  the product \spinc manifold. For any $\fc_1,\fc_2$ non-generate, we have
\[
\grt(\fc_1)-\grt(\fc_2)=\ind \CQ(\fc_1, \bs_X, \fc_2)\mod 2,
\]
\item\label{B-II} (Invariance Axiom) The mod 2 grading function 
\[
\grt: \SC_k(\hy,\bs)\times \Pa\dashrightarrow \Z/2\Z
\]
is invariant under the action of $\CG_{k+1}(\hy)$.
\end{enumerate}

Again, the Invariance Axiom \ref{B-II} is redundant, as it follows from \ref{B-I}. One may fix the value $\grt(\fc_1)$ for one particular pair $\fc_1$ and decide the other value $\grt(\fc_2)$ using the Reduction Axiom \ref{B-I}, so such a mod 2 grading function $\grt$ clearly exists. It is not unique, as the value of $\grt(\fc_1)$ is arbitrary. 

\medskip

This ambiguity is fixed \textbf{simultaneously} for all relative \spinc structures $\bs\in \Spincr(Y)$, once \textbf{a homological orientation} of $(Y,\partial Y)$ is chosen, as explained in \cite{MT96}, which is also reminiscent of the case of 4-manifolds as treated in \cite[Subsection 24.8]{Bible}. Since this story has been standard nowadays, we only give a brief sketch here. 

One may alternatively think of $\grt(\fc)$ as an orientation of the extended Hessian 
\[
\EHess_{\fc}.
\]
As $\fc$ is non-generate, an orientation of this invertible operator $\EHess_{\fc}$ is equivalent to a choice of signs in $\{\pm 1\}$. However, this standpoint allows us to extend the domain of $\grt$ to the whose space $\SC_k(\hy,\bs)\times \Pa$. Indeed, $\{\EHess_\fc\}$ forms a continuous family of Fredholm operators, and as such gives rise to a determinant line bundle over the base:
\[
\begin{tikzcd}
\R\cong \det \EHess_\fc\arrow[r] & L\arrow[d]\\
& \SC_k(\hy,\bs)\times \Pa. 
\end{tikzcd}
\]
The real line bundle $L$ is trivial as $\SC_k(\hy,\bs)\times \Pa$ is contractible. To orient $L$, it suffices to orient one particular fiber $L_\fc$; we choose the one at $\fc=(\fa,0)$ such that $\fa$ agrees with the standard configuration:
\[
(B_*,\Psi_*)
\]
on the cylindrical end $[0,\infty)_s\times \Sigma$. As explained in the proof of Proposition \ref{P18.1}, the extended Hessian $\EHess_\fa$ in this case is cast into the form
\[
\sigma(\ps+\widehat{D}_{\kappa_*})
\]
on the cylindrical end $[0,\infty)_s\times\Sigma$, where 
\begin{equation}\label{E23.7}
\widehat{D}_{\kappa_*}: L^2_1(\Sigma, i\R\oplus i\R\oplus T^*\Sigma\oplus S)\to L^2(\Sigma, i\R\oplus i\R\oplus T^*\Sigma\oplus S)
\end{equation}
is an invertible self-adjoint elliptic operator. For the precise expression, see \cite[Subsection 7.4]{Wang202}. Let $H^\pm$ be the $(\pm)$-spectral subspaces of $\hatD$. Instead of $\EHess_\fa$, we consider the operator with a spectral boundary projection:
\begin{equation}\label{E23.5}
\EHess_\fa\oplus \Pi^{-}\circ r: L^2_k(Y, i\R\oplus iT^*Y\oplus S)\to L^2_{k-1}(Y, i\R\oplus iT^*Y\oplus S)\oplus (H^-\cap L^2_{k-1/2}). 
\end{equation}
on the truncated 3-manifold $Y=\{s\leq 0\}$. At this point, we can further deform $\fa$ so that $\Psi\equiv 0$, in which case
\[
\EHess_\fa=\begin{pmatrix}
0 & -d & 0\\
-d^* & *d & 0\\
0 & 0 & D_{B_0}
\end{pmatrix} \text{ on } Y
\]
for a reference \spinc connection $B_0$, and 
\[
\EHess_\fa=\sigma(\ps+\hat{D}_0) \text{ with }D_0=\begin{pmatrix}
D_{\form} & 0\\
0 & D_{\cB_*}^\Sigma
\end{pmatrix}
\]
in the collar $(-1,0]_s\times \Sigma$. Here 
\[
D_{\form}=\begin{pmatrix}
0 & 0 & -*_\Sigma d_\Sigma\\
0 & 0& -d_\Sigma^*\\
*_\Sigma d_\Sigma & -d_\Sigma & 0
\end{pmatrix}
: L^2_1(\Sigma, i\R\oplus i\R\oplus iT^*\Sigma)\to  L^2(\Sigma, i\R\oplus i\R\oplus iT^*\Sigma)
\]
is a self-adjoint operator with kernel $H^0(\Sigma, i\R)\oplus H^0(\Sigma,i\R)\oplus H^1(\Sigma, i\R)$ and 
\[
D_{\cB_*}^\Sigma: L^2_1(\Sigma, S)\to  L^2(\Sigma, S)
\]
is the Dirac operator on the surface, which is complex linear. Consider the projection map
\[
\Pi_{\form}=\Pi_1\oplus \Pi_{\form}^-: L^2(\Sigma)\to H^1(\Sigma, i\R)\oplus H^-_{\form}.
\]
where $\Pi_{\form}$ is the projection map onto the negative spectral subspace of $D_{\form}$ and $\Pi_1$ is the projection onto $H^1(\Sigma, i\R)\subset \ker D_{\form}$. 
\begin{lemma} The kernel and the cokernel of the operator:
	\begin{equation}\label{E23.6}
	\begin{pmatrix}
0 & -d \\
-d^* & *d 
\end{pmatrix}\oplus (\Pi_{\form}\circ r): L^2_1(\R\oplus iT^*Y)\to L^2(\R\oplus iT^*Y)\oplus H^1(\Sigma, i\R)\oplus H^-_{\form}.
	\end{equation}
	are isomorphic to $H^0(Y;i\R)\oplus H^1(Y,\partial Y; \R)$ and $H^0(Y,\partial Y; \R)\oplus H^1(Y ;i\R)$ respectively. In particular, an orientation of \eqref{E23.6} is equivalent to a homological orientation of $(Y,\partial Y)$. 
\end{lemma}

Finally, to relate the operator \eqref{E23.5} with \eqref{E23.6}, we have to deform the boundary projection $\Pi^-$ in \eqref{E23.5}. Notice that the operator $\hatD$ in \eqref{E23.7} relies on the standard spinor $\Psi_*$. The deformation is then made by taking 
\[
\Psi_*\mapsto \tau \Psi_*, \tau\to 0. 
\]
In the limit, $\hatD$ will recover $\hat{D}_0$, which is no longer invertible. At this point, one has to examine the deformation of spectral projections very carefully, which is independent of relative \spinc structures. In this way, an orientation of \eqref{E23.6} gives rise to an orientation of $L$.

\section{Floer Homology with $\Z$-coefficient}\label{Sec29} Let $\NR$ be the Novikov ring of Laurent series with integral coefficients
\[
\NR=\{ \sum_{n_i} a_iq^{n_i}:\ a_i\in \Z,\ n_i\in \R,\ \lim_{i}n_i=-\infty\}.
\]
To define the monopole Floer homology over $\NR$, we have to orient moduli spaces in a consistent way. Since the space $\SC_k(\hy,\bs)$ does not contain any reducible configurations, the strategy used in \cite[Section 20]{Bible} does not work directly here. Moreover, our cobordism maps are induced from oriented 4-manifold with corners. It is not crystal clear what is meant to be a homology orientation in this case.

We will address this problem using an analytic approach. The main result of this section is Theorem \ref{T24.2}, which leads to the replacement of homology orientations in Definition \ref{D24.3}. The proof of Theorem \ref{T24.2} relies on the notion of relative orientations that compares the determinant line bundles of two Fredholm operators in the excision principle. We will develop the relevant theory in Appendix \ref{AppD} and accomplish the proof of Theorem \ref{T24.2} in Subsection \ref{SubsecD.10}.  The construction of the functor $$\HM_*: \SCob_{s,b}\to \NR\Mod$$ is explained in Subsection \ref{Subsec24.4}.

\subsection{Determinant Line Bundles and Direct Sums} To start, let us recall the basic theory of determinant line bundles of Fredholm operators from \cite[Section 20.2]{Bible}. Given two real Hilbert spaces $E$ and $F$,  consider a continuous family of Fredholm operators 
\[
\A_z: E\to F,\ z\in \CZ,
\]
parametrized by a topological space $\CZ$. \textbf{The determinant line bundle} of this family is a real line bundle over $\CZ$
\[
\det \A\to \CZ
\]
such that the fiber $\det \A_z$ at each $z\in \CZ$ is identified with 
\[
\Lambda^{\max}\ker \A_z\otimes( \Lambda^{\max}\coker \A_z)^*.
\]
When the determinant line bundle $\det\A\to \CZ$ is orientable, denote the 2-element set of orientations by 
\[
\Lambda(\A) \text{ or }\Lambda(\det \A). 
\]
\begin{example}\label{EX24.1} Let $\A_*: E\to F$ be a reference Fredholm operator and $\CZ$ be the space of all compact operators:
	\[
	\CZ=\{z: E\to F: z \text{ compact}\}.
	\]
	Then the family $\{\A_z=\A_*+z: z\in \CZ\}$ is parametrized by a contractible space $\CZ$. An orientation of $\A_*$ is meant to be an orientation of this contractible family. Denote the 2-element set of orientations by
	\[
	\Lambda(\A_*) \text{ or } \Lambda(\det \A_*). \qedhere
	\]
\end{example}

Given two families of operators $\A'\to \CZ$ and $\A''\to \CZ$ parametrized by the same space, we form a new family by taking the point-wise direct sum of Fredholm operators
\[
\A_z=\A_z'\oplus\A''_z: E'\oplus E''\to F'\oplus F''. 
\]
 Then there is a natural isomorphism of real line bundles constructed in \cite[P.379]{Bible}:
\begin{equation}\label{E24.3}
q: \det \A'\otimes \det\A''\to \det \A. 
\end{equation}
Suppose $\alpha_z'$ and $\alpha_z''$ are elements in $\Lambda^{\max} \ker \A'_z$ and $\Lambda^{\max} \ker \A''_z$ respectively, while $\beta_z'$ and $\beta_z''$ are corresponding elements in $\Lambda^{\max} \coker \A'_z$ and $\Lambda^{\max} \coker \A''_z$. Then the bundle map $q$ is locally defined (up to a positive scalar) by the formula:
\begin{align*}
\big(\alpha'_z\otimes (\beta'_z)^*\big)\otimes \big(\alpha''_z\otimes (\beta''_z)^*\big)&\mapsto (-1)^r (\alpha'_z\wedge \alpha''_z)\otimes (\beta'_z\wedge\beta''_z)^* \text{ where }\\
r&=\dim\coker \A'_z \times \ind(\A''_z).
\end{align*}

The sign $(-1)^r$ is added here to ensure that the bundle map $q$ is continuous as the base point $z$ varies in $\CZ$. Moreover, the bundle map $q$ becomes associative when we consider the direct sum of three families of operators parametrized by the same space $\CZ$.

\medspace

For any 2-element set $\Lambda$, let $\Z/2\Z$ act on $\Lambda$ by involutions. For any $\Lambda_1$ and $\Lambda_2$ with $\Z/2\Z$ action, we form their product set 
\[
\Lambda_1\Lambda_2\colonequals \Lambda_1\times_{\Z/2\Z}\Lambda_2. 
\]

As a result, by passing to the 2-element sets of orientations, the bundle map $q$ descends to an associative multiplication, denoted also by $q$: 
\[
q: \Lambda(\A')\times \Lambda(\A'')\to \Lambda(\A'\oplus \A''),
\]
or an isomorphism preserving the $\Z/2\Z$-action:
\[
q: \Lambda(\A')\Lambda(\A'')\xrightarrow{\cong}\Lambda(\A'\oplus \A'').
\] 

\subsection{Homology Orientations} Having discussed the abstract properties of determinant line bundles, let us explain now the primary application in gauge theory. Given a morphism $\x: (\y_1,\bs_1)\to (\y_2,\bs_2)$ in the strict cobordism category $\SCob_s$, consider non-degenerate pairs (in the sense of Definition \ref{D23.3})
\[
\fc_i=(\fa_i,\q_i)\in \SC_k(\hy_i,\bs_i)\times\Pa(Y_i), i=1,2. 
\]
By looking at the linearized Seiberg-Witten map and the linearized gauge fixing equation on the complete Riemannian 4-manifold $\CX$, we obtained in $\eqref{E23.8}$ a Fredholm operator $\CQ(\fc_1, \bs_X,\fc_2)$ for any relative \spinc cobordism $(\hx,\bs_X): (\hy_1,\bs_1)\to (\hy_2,\bs_2)$. Define 
\[
\Lambda(\fc_1, \bs_X, \fc_2)\colonequals \Lambda(\CQ(\fc_1, \bs_X,\fc_2))
\]
for any non-degenerate pairs $\fc_1,\fc_2$ and any $\bs_X\in \Spincr(X;\bs_1,\bs_2)$. The 2-element set $\Lambda(\fc_1, \bs_X, \fc_2)$ is understood in the sense of Example \ref{EX24.1}. Since the different choices of the reference configuration $\gamma$ will give rise to the same operator $\CQ(\fc_1, \bs_X,\fc_2)$ up to compact terms, $\Lambda(\fc_1, \bs_X, \fc_2)$ is independent of the choice of $\gamma$.

Our goal is to identify these 2-element sets $\Lambda(\fc_1, \bs_X, \fc_2)$ in a canonical way for all relative \spinc cobordisms $\bs_X\in \Spincr(X;\bs_1,\bs_2)$. As a result, if the orientation is fixed for one particular $\bs_X$, then it automatically fixes the choice for any other relative \spinc cobordisms.

Recall that $\Spincr(X;\bs_1,\bs_2)$ is a torsor over $H^2(X, \partial X;\Z)$. 
\begin{theorem}\label{T24.2} For any isomorphism classe of relative line bundles $[L]\in H^2(X,\partial X; \Z)$, there exists a natural bijection
	\[
	e_L: \Lambda(\fc_1, \bs_X,\fc_2)\to \Lambda(\fc_1, \bs_X\otimes L,\fc_2),
	\]
	for any $\bs_X\in \Spincr(X;\bs_1,\bs_2)$	satisfying the following two properties:
	\begin{enumerate}[label=(U\arabic*)]
		\item\label{U1} $e_{L_1}\circ e_{L_2}=e_{L_1\otimes L_2}$; 
		\item\label{U2} the collection $\{e_L\}$ is compatible with the concatenation map $q$ meaning that the diagram 
		\begin{equation}\label{E24.5}
		\begin{tikzcd}
		\Lambda(\fc_1, \bs_{12},\fc_2) \Lambda(\fc_2, \bs_{23},\fc_3)\arrow[r,"q"]\arrow[d,"e_{L_{12}}\otimes e_{L_{23}}"] & \Lambda(\fc_1, \bs_{13},\fc_3)\arrow[d,"e_{L_{13}}"]\\
		\Lambda(\fc_1, \bs_{12}\otimes L_{12},\fc_2) \Lambda(\fc_2, \bs_{23}\otimes L_{23},\fc_3)\arrow[r,"q"] & \Lambda(\fc_1, \bs_{13}\otimes L_{13},\fc_3)
		\end{tikzcd}
		\end{equation}
		is commutative for any relative \spinc cobordisms:
		\begin{align*}
		(\hx_{12}, \bs_{12}): (Y_1,\bs_1)&\to (Y_2,\bs_2), \\
		(\hx_{23}, \bs_{23}): (Y_2,\bs_3)&\to (Y_3,\bs_3).
		\end{align*}
		Here $(\hx_{13},\bs_{13})=(\hx_{12}\#\hx_{23}, \bs_{12}\#\bs_{23})$ is the concatenation of relative cobordisms and $L_{13}=L_{12}\#L_{23}$ is the concatenation of relative line bundles. 
	\end{enumerate}
\end{theorem} 

\begin{remark}The proof of Theorem \ref{T24.2} is constructive: we will construct each $e_L$ explicitly and verify properties \ref{U1}\ref{U2} by hands. The key ingredient is the notion of relative orientations, which allows us to reduce the problem from a non-compact manifold $\CX$ to a closed 4-manifold. In the latter case, we know how to construct $e_L$, since the Dirac operator and the self-dual operator are now decoupled. The relevant theory is developed in Appendix \ref{AppD}. The proof of Theorem \ref{T24.2} will be accomplished in Subsection \ref{SubsecD.10}. 
\end{remark}

The horizontal maps $q$ in the diagram \eqref{E24.5} require some further explanations. Take non-degenerate pairs $\fc_i$ on $\hy_i$ for $1\leq i\leq 3$. Instead of $\CQ$, we look at operators on $\hx_{ij}$ with spectral projections:
\begin{equation}\label{E24.7}
\CQ'(\fc_i,\bs_{ij}, \fc_j)\colonequals D_{ij}\oplus (\Pi^+_{\A_i}, \Pi^-_{\A_j})\circ (r_i, r_j), 1\leq i<j\leq 3, 
\end{equation}
understood in the sense of Proposition \ref{P19.7} and Subsection \ref{Subsec19.4} adapted to the case of general cobordisms. In particular, $\Pi^\pm_{\A_i}$ are spectral projections of the extended Hessians at $\fc_i$: 
\[
\EHess_{\fc_i}: L^2_k(\hy_i, i\R\oplus iT^*Y_i\oplus S)\to L^2_k(\hy_i, i\R\oplus iT^*Y_i\oplus S),\ 1\leq i\leq 3.
\]
The 2-element set $\Lambda(\fc_1,\bs_{12},\fc_2)$ can be defined using $\CQ'(\fc_i,\bs_{ij}, \fc_j)$ instead.  As explained in \cite[P. 384]{Bible}, there is a canonical bundle isomorphism defined using the map $\eqref{E24.3}$,
\begin{equation}\label{E24.4}
q: \det \CQ'(\fc_1,\bs_{12},\fc_2)\otimes \det \CQ'(\fc_2,\bs_{23},\fc_3)\to \det \CQ'(\fc_1,\bs_{13},\fc_3).
\end{equation}
which descends to an associative multiplication:
\[
q: \Lambda(\fc_1, \bs_{12},\fc_2) \Lambda(\fc_2, \bs_{23},\fc_3)\to  \Lambda(\fc_1, \bs_{13},\fc_3).
\]

Our construction of homology orientations is based upon Theorem \ref{T24.2}.
\begin{definition}\label{D24.3} Following the notations in Theorem \ref{T24.2}, for any triple $(\fc_1,\x,\fc_2)$, define the 2-element set of \textbf{homology orientations} as the quotient space
	\[
	\Lambda(\fc_1, \x, \fc_2)\colonequals \coprod_{\bs_X\in\Spincr(X;\bs_1,\bs_2)} \Lambda(\fc_1, \bs_X,\fc_2)\big/\{ e_L \}_{[L]\in H^2(X,\partial X; \Z)},
	\]
	where $\x: \y_1\to \y_2$ is any morphism in $\Cob_s$ and for $i=1,2$, $\fc_i\in \SC_k(\hy_i,\bs_i)\times \Pa(Y_i)$ is a non-degenerate pair. By the property \ref{U2} in Theorem \ref{T24.2}, the concatenation map $q$ descends to an associative multiplication: 
	\[
	q:\Lambda(\fc_1, \x_{12},\fc_2) \Lambda(\fc_2, \x_{23},\fc_3)\to  \Lambda(\fc_1, \x_{13},\fc_3). \qedhere
	\]
\end{definition}

\begin{remark} If we replace $\CX$ by a closed Riemannian 4-manifold $X_1$, the construction above will recover the original definition of homology orientations of $X_1$, i.e. orientations of the real line
	\[
	\Lambda^{\max}H^2_+(X_1,\R)\otimes( \Lambda^{\max} H^1(X_1, \R))^*.
	\]
	Here $H^2_+(X_1,\R)$ is any maximal positive subspace of $H^2(X_1, \R)$ with respect to the intersection form. 
\end{remark}

Now let us specialize to the case when $X=[-1,1]\times Y$ is a product cobordism and $\bs_1=\bs_2=\bs$. This is relevant for orienting moduli spaces on the cylinder $\R_t\times \hy$. The non-degenerate pairs $\fc_1,\fc_2$ now lie in the same space:
\[
 \SC_{k}(\hy,\bs)\times \Pa(Y).
\]

\begin{definition} Let $I=[-1,1]$. Define the 2-element set $\Lambda([\fc_1],[\fc_2])$ to be the homology orientations of $(\fc_1,I\times \y ,\fc_2)$ in the sense of Definition \ref{D24.3}, where $[\fc_i]$ denotes the class in the quotient configuration space $\CB_k(\hy,\bs)\times\Pa(Y)$. More concretely, $\Lambda([\fc_1],[\fc_2])$ is realized as the quotient space
	\[
	\coprod_{[L]\in H^2(I\times Y, \partial (I\times Y); \Z)} \Lambda(\fc_1, \bs\otimes L,\fc_2)\big/\{ e_L \}.\qedhere
	\]
\end{definition}

	When $\fc_1=\fc_2\in \SC_{k}(\hy,\bs)\times \Pa$, there is a canonical element $v(\fc_1)$ in $\Lambda([\fc_1],[\fc_1])$ induced from 
\[
1\in \Lambda(\CQ(\fc_1, \R_t\times(\hy,\bs), \fc_1)).
\] 
In this case, we choose an $\R_t$-invariant configuration $\gamma$ on $\R_t\times \hy$ to define the operator $\CQ(\fc_1,I\times (\hy,\bs),\fc_1)$. Because $\fc_1$ is non-degenerate, $\CQ$ is invertible. The canonical element $1$ denotes the positive orientation of this invertible operator.  
\begin{remark} Here we have identified the homotopy classes of paths $\pi_1(\CB_{k}(\hy,\bs); [\fc_1],[\fc_2])$ with the space of relative \spinc cobordisms $\Spincr(I\times Y; \bs,\bs)$, following the ideas in Subsection \ref{Subsec2.4}. When $\fc=(\fa,\q)$ is a critical point of the perturbed Chern-Simons-Dirac functional $\CSd_{\omega}$, the canonical element $v(\fc)$ orients automatically the moduli space $\cM_{z}(\fc,\fc)$ in \eqref{E19.13} for any $z\in \pi_1(\CB_{k}(\hy,\bs); [\fa])$. Moreover, this orientation is compatible with concatenation of paths by the associativity of the concatenation map $q$.
\end{remark}

\subsection{Floer Homology with $\Z$-coefficient} \label{Subsec24.4} Having defined homology orientations on cylinders and general cobordisms, let us now explain the construction of $\HM_*(\y,\bs)$ using the integral coefficient. In the most general case, we have to use a Novikov ring defined over $\Z$:
\[
\NR=\{ \sum_{n_i} a_iq^{n_i}:\ a_i\in \Z,\ n_i\in \R,\ \lim_{i}n_i=-\infty\}.
\]
To work with $\Z$ directly, we have to assume the monotonicity condition in Definition \ref{D27.3} for the object $(\y,\bs)$ and pass to a sub-category of $\SCob_s$.

To better illustrate our construction below, we focus on the first case. Only formal adaptations are actually needed for the second case. At this point, we have to enlarge the strict cobordism category $\SCob_s$ slightly to incorporate a base point for each object. 
\begin{definition}\label{D24.8} An object of \textit{the based strict cobordism category} $\SCob_{s,b}$ is a triple $(\y,\bs, \fc_*)$ where  $(\y,\bs)$ is an object of $\SCob_s$ and $\fc_*=(\fa_*,\q)\in \SC(\hy,\bs)\times \Pa(Y)$ is a non-degenerate pair. We require that the tame perturbation $\q=\grad f$ is the one encoded in the object $\y\in\Cob_s$ for the relative \spinc structure $\bs$. A morphism of $\SCob_{s,b}$ is a pair
	\begin{equation}\label{E24.6}
(\x,o): (\y_1,\bs_1,\fc_{*,1})\to  (\y_1,\bs_1,\fc_{*,2})
	\end{equation}
	where $\x: \y_1\to \y_2$ is a morphism in $\Cob_s$ and $o\in \Lambda(\fc_{*,1},\x,\fc_{*,2})$ is a choice of homology orientations in the sense of Definition \ref{D24.3}.  
\end{definition} 

The based strict cobordism category $\SCob_{s,b}$ is only a formal enlargement of $\SCob_s$. The base point $\fc_*$ is included here to remove the ambiguity of orientations on the cylinder $\R_t\times\hy$. More precisely, for any object $(\y,\bs,\fc_*)\in \SCob_{s,b}$ and for any critical point $\fa\in \Crit(\CSd_{\omega})$ of  $\CSd_{\omega}=\CL_{\omega}+f$, define 
\[
\Lambda([\fa])\colonequals \Lambda([\fc_*],[(\fa,\q)]),
\]
 and form the chain group
\[
C_*(\y,\bs, \fc_*)=\bigoplus_{[\fa]\in \FC(\y,\bs)} \Z\Lambda([\fa])\otimes_\Z \NR
\]
where $\Z/2\Z$ acts  non-trivially on $\Z$ and we set $\Z\Lambda([\fa])\colonequals \Z\times_{\Z/2\Z} \Lambda([\fa])$.
\begin{remark} For closed 3-manifolds, the role of $\fc_*$ is played by a reducible configuration $\fc_{*}'$ in the blown-up configuration space; see \cite[Section 20.3]{Bible}. In that case, the choice of $\fc_{*}'$ does not matter, since there is a canonical element in 
	\[
	\Lambda([\fc_*'],[\fc_*''])
	\]
	when $\fc_{*}'$ and $\fc_{*}''$ are both reducible. However, this property does not hold in our case.
\end{remark}
In the formula of the differential $\partial$ below,  we take the sum over all possible triples 
\[
([\fa],[\fb],z)\in \FC(\y,\bs)\times\FC(\y,\bs)\times\pi_1(\CB_k(\y,\bs); [\fa],[\fb])
\]
 such that $\dim \cM_{z}([\fa],[\fb])=0$:
\begin{align}\label{E24.2}
\partial&=\sum_{[\fa]}\sum_{[\fb]}\sum_{z}\sum_{[\gamma]\in \cM_{z}([\fa],[\fb])}\Gamma[\gamma]:C_*(\y,\bs, \fc_*)\to C_*(\y,\bs, \fc_*).
\end{align}
 Since each unparameterized solution $[\gamma]\in \cM_{z}([\fa],[\fb])$ is a point, the positive orientation of $\gamma$ defines an element $v([\gamma])$ in $\Lambda([(\fa,\q)],[(\fb,\q)])$. Combining with the concatenation map $q$, this provides a homomorphism of abelian groups:
 \[
 \epsilon[\gamma]=\Id_\Z\ \otimes\ q(\cdot, v[\gamma]):\Z\Lambda([\fa])\to \Z\Lambda([\fb]).
 \]
The $\NR$-module homomorphism $\Gamma[\gamma]$ in \eqref{E24.2} is then defined by taking into account the topological energy $\E_{top}$:
 \begin{align*}
 \Gamma[\gamma]\colonequals  \epsilon[\gamma]\otimes q^{-\E_{top}^\q([\fa],[\fb];z)}&: \Z\Lambda([\fa])\otimes \NR\to \Z\Lambda([\fb])\otimes \NR.
 \end{align*}
 
 The differential $\partial$ on $C_*(\y,\bs, \fc_*)$ is formed by taking the sum of all $ \Gamma[\gamma]$. 

\medskip

 Now we come to define $\HM_*$ for the morphism sets of $\SCob_{s,b}$. For any morphism $(\x,o):(\y_1,\bs_1,\fc_{*,1})\to (\y_2,\bs_2,\fc_{*,2})$ of the based cobordism category $\SCob_{s,b}$, pick a planar metric $g_X$ and an admissible quadruple $\p$ as the perturbation. The chain map is now defined as 
\begin{align}\label{E24.1}
m(\x,o;g_X,\q)=\sum_{[\fa_1]}\sum_{[\fa_2]}\sum_{ \bs_X}\sum_{[\gamma]\in  \M(\fa_1,\bs_X,\fa_2)} \Gamma[o,\gamma]: C_*(\y_1,\bs_1,\fc_{*,1})\to
 C_*(\y_2,\bs_2,\fc_{*,2}),
\end{align}
 where the sum is over all possible triples \[
 ([\fa_1],[\fa_2], \bs_X)\in \FC(\y,\bs_1)\times \FC(\y,\bs_2)\times \Spincr(X,\bs_1,\bs_2),
 \]
  such that $\dim  \M(\fa_1,\bs_X,\fa_2)=0$. Each solution $[\gamma]$ in $\M(\fa_1,\bs_X,\fa_2)$ is a $0$-dimensional manifold, whose positive orientation determines a class $v([\gamma])$ in 
  \[
  \Lambda((\fa_1,\q_1),\x,(\fa_2,\q_2)).
  \]
  
  We obtain a morphism 
    \[
  \epsilon[o,\gamma]:   \Z\Lambda([\fa_1])\to   \Z\Lambda([\fa_2])
  \]
  by chasing around the diagram:
  \[
\begin{tikzcd}[column sep=3cm]
\Lambda(\fc_{*,1},(\fa_1,\q_1))\arrow[d,dashed,"{\epsilon[o,\gamma]}"]\arrow[r,"q(\cdot{,\ v([\gamma])})"]&\Lambda(\fc_{*,1},\x, (\fa_2,\q_2))\arrow[d,equal]\\
\Lambda(\fc_{*,2},(\fa_2,\q_2))\arrow[r,"q(o{,} \cdot)"] & \Lambda(\fc_{*,1},\x, (\fa_2,\q_2))
\end{tikzcd}
\]
Here $o\in \Lambda(\fc_{*,1}, \x,\fc_{*,2})$ is the reference homology orientation that we picked up in the morphism $(\x,o)$. 
  The $\NR$-module homomorphism $\Gamma[o,\gamma]$ in \eqref{E24.1} is defined by the formula
 \begin{align*}
\Gamma[o,\gamma]\colonequals  \epsilon[o,\gamma]\otimes q^{-\E_{top}^\p([\fa_1],[\fa_2];\ \bs_X)}&: \Z\Lambda([\fa_1])\otimes \NR\to \Z\Lambda([\fa_2])\otimes \NR.
\end{align*}

One can verify that each $(C_*(\y,\bs,\fc_*),\partial)$ is indeed a chain complex and $m(\x,o;g_X,\q)$ gives rise to a chain map by following the standard argument in \cite[Section 22]{Bible}. Then the covariant functor 
\[
\HM: \SCob_{s,b}\to \NR\Mod
\]
is obtained by taking the homology groups, and it satisfies the modified composition law in Theorem \ref{1T2}. 
  
\subsection{Invariance} Having constructed the monopole Floer homology $\HM_*(\y,\bs,\fc_*)$, our next step is to discuss the extend to which it is a topological invariant of $(Y,\partial Y)$. The definition of $\y$ involves an orientation preserving diffeomorphism $\varphi: \partial Y\to \Sigma$, a cylindrical metric $g_Y$, a closed 2-form $\omega$ and a collection of admissible perturbations $\{\q\}$, one for each relative \spinc structure $\bs\in \Spincr(Y)$. It turns out that only the boundary data $(g_\Sigma,\lambda,\mu)$, the isotopy class of $\varphi$ and the relative cohomology class $[\omega]_{cpt}\in H^2(Y,\partial Y;[\mu])$ (as defined in \ref{P5}) may potentially affect this group. We have two immediate corollaries of Theorem \ref{1T2}.

\begin{corollary}\label{C24.10} For any object $(\y,\bs,\fc_*)\in \SCob_{s,b}$, the monopole Floer homology group $\HM_*(\y,\bs)$ is independent of the choices of the base point $\fc_{*}$, the cylindrical metric $g_Y$ and the admissible perturbation $\q$ associated to $\bs$, up to canonical isomorphisms. In particular, the isomorphism class of $\HM_*(\y,\bs)$ is not affected if one replaces $\omega$ by $\omega+d c$ for a compactly supported 1-form $c\in \Omega^1_c(\hy,i\R)$. 
\end{corollary}
\begin{proof} The product cobordism $[-1,1]_t\times Y$ between $(Y, \psi,g_Y, \omega,\{\q\} )$ and $(Y, \psi,g_Y', \omega,\{\q'\})$ provides the canonical isomorphism between their Floer homology groups. For the second clause, one observes that the function
	\[
	(B,\Psi)\mapsto \half\int_{\hy} (B^t-B_0^t)\wedge dc. 
	\]
	defines a tame perturbation on the configuration space; so one may use the first clause to conclude. 
\end{proof}

Recall that $\lambda\in \Omega_h^1(\Sigma, i\R)$ is a harmonic 1-form on $\Sigma=\coprod \T^2_j$ such that $\lambda_j\colonequals \lambda|_{\T^2_j}\neq 0$. Let $[\lambda_j]\in H^1(\T^2_j,i\R)\embed H^1(\Sigma; i\R)$ be the cohomology class of $\lambda_j$.

\begin{corollary}\label{C24.11} For any object $(\y,\bs,\fc_*)\in \SCob_{s,b}$, the isomorphism class of $\HM_*(\y,\bs,\fc_*)$ is not affected if we apply an isotopy to the diffeomorphism $\varphi:\partial Y\to \Sigma$ or change the class $[\omega]_{cpt}\in H^2(Y,\partial Y;[\mu])$ by an element of the form 
	\[
	\delta_*(\sum_{j=1}^n a_j[\lambda_j]),
	\]
	where $a_j\in \R, 1\leq j\leq n$ and $\delta_*: H^1(\Sigma; i\R)\to H^2(Y,\partial Y; i\R)$ is the co-boundary map.
\end{corollary}
\begin{proof} The first clause follows from Example \ref{Ex9.5}. However, given two isotopic diffeomorphisms $\varphi_1,\varphi_2:\partial Y\to \Sigma$, there are different ways to connect them using isotopies; so the isomorphism constructed using Theorem \ref{1T2} is not canonical. This is due to the fact that the diffeomorphism group $\Diff_+(\T^2)$ of the 2-torus $\T^2$ is not simply connected. Indeed, by \cite[Theorem 1(b)]{EE67}, $\Diff_+(\T^2)$ has the same homotopy type of its linear subgroup $S^1\times S^1\times \SLL(2,\Z)$, so $\pi_1(\Diff_+(\T^2))\cong \Z\oplus \Z$. 
	
	The second clause follows from the fact that the class $[\omega]_{cpt}\in H^2(Y,\partial Y;[\mu])$ is not well-defined, unless a cut-off function $\chi_1$ in \ref{P5} is specified, as noted already in Remark \ref{R9.3}. We have studied the Seiberg-Witten equations on the 3-manifold $\hy$ with cylindrical ends, but they are different ways to write 
	\[
	\hy=Y\cup [0,\infty)_s\times\Sigma.
	\]
	Indeed, one may take $Y'=Y_1=\{s\leq1 \}$ and set $s'=s-1$; so 
		\[
	\hy=Y'\cup  [0,\infty)_{s'}\times\Sigma.
	\]
	However, the closed 2-form $\omega\in \Omega^2(\hy,i\R)$ is associated to different relative cohomology classes $[\omega]_{cpt}$ and $[\omega]_{cpt}'$ on $Y$ and $Y'$ respectively, according to \ref{P5}, which are related by 
	\[
	[\omega]'_{cpt}=[\omega]_{cpt}+\delta_*(a[\lambda])
	\]
	for some $a\neq 0\in \R$. Since $Y'$ and $Y$ are the same 3-manifold, while equipped with different cylindrical metrics, one may apply Corollary \ref{C24.10} to identify the Floer homology of $(Y',[\omega]_{cpt}')$ with that of $(Y, [\omega]_{cpt}')$. To deal with the general case, it suffices to choose different translation amounts for the coordinate function $s$ on different connected component of $[0,\infty)_s\times \Sigma$.  
\end{proof}

\part{Some Properties}\label{Part8}

Having defined the monopole Floer homology and the functor $\HM$ in Theorem \ref{1T2}, our next goal is to establish a finiteness result and provide a few calculations. The results obtained in this part concentrated on the 3-dimensional Seiberg-Witten equations \eqref{3DSWEQ}. Section \ref{Sec25} and \ref{Sec30} below are independent of each other and can be read separately.

\smallskip 

Section \ref{Sec25} is devoted to the proof of the finiteness result: Theorem \ref{1T4}. Given any object $\y\in\Cob_s$ satisfying the assumption of Theorem \ref{1T4}, we will show that only finitely many relative \spinc structures can support a solution to \eqref{3DSWEQ}. The key ingredient is the energy estimate in Proposition \ref{P25.1}, which leverages some identities observed first by Taubes \cite{Taubes96} and their general forms for any Riemannian 3-manifolds. Although they have been well-known for any experts working in this field, we record the statement of their general forms and a short proof in Appendix \ref{AppE} for the sake of completeness. 

\smallskip

Section \ref{Sec30} is devoted to the computation of the monopole Floer homology for the product manifold $\Sigma_{g,n}\times S^1$, where $\Sigma_{g,n}$ is a genus-$g$ surface with $n\geq 2$ cylindrical ends. To do this, we examine the dimensional reduction of \eqref{3DSWEQ} on the surface $\Sigma_{g,n}$ and make use of the results from \cite[Appendix C]{Wang202}. 

\section{Finiteness of Critical Points}\label{Sec25}

In this section, we present the proof of Theorem \ref{1T4}, which states that the monopole Floer homology $\HM_*(\y)$ is finitely generated if the harmonic 2-form $\mu$ is non-vanishing on $\Sigma=\partial Y$. 

Recall from Section \ref{Sec5} that we made the Assumption \ref{A1.2} for Theorem \ref{T2.6} to hold. It turns out that if the first alternative \ref{VV1} holds for any component of $\Sigma$, the properties of $\HM_*(\y)$ are much easier to understand. This is the situation in Theorem \ref{1T4}. Further results will be supplied in the third paper of this series \cite{Wang203}. Theorem \ref{1T4} follows immediately from an energy estimate:

\begin{proposition}\label{P25.1} For any object $\y\in \Cob_s$ such that the harmonic 2-form $\mu$ is non-vanishing on $\Sigma$, there exists a constant $C(g_Y,\omega)>0$ with the following property. For any relative \spinc structure $\bs\in \Spincr(Y)$, suppose the configuration $(B,\Psi)$ solves the 3-dimensional Seiberg-Witten equations \eqref{3DSWEQ}, i.e., it is a critical point of the Chern-Simons-Dirac functional $\CL_{\omega}$, then 
	\begin{equation}\label{E25.1}
\int_{\hy}  \frac{1}{8}|F_{B^t}|^2+|\nabla_B\Psi|^2+|(\Psi\Psi^*)_0+\rho_3(\omega)|^2+\frac{s}{4}|\Psi|^2<C.
	\end{equation}
\end{proposition}

\begin{remark} The estimate \eqref{E25.1} is automatic on a closed 3-manifold. One may apply Lemma \ref{L11.3} to obtain that
	\[
	\int_{\hy}  \frac{1}{8}|F_{B^t}|^2+|\nabla_B\Psi|^2+|(\Psi\Psi^*)_0+\rho_3(\omega)|^2+\frac{s}{4}|\Psi|^2<C'-\int_{\hy} F_{B_0^t}\wedge *\omega_h,
	\]
	where $\omega_h$ is the co-closed 2-form obtained in Lemma \ref{9.1}. However, the terms on the right hand side depend on the relative \spinc structure $\bs\in \Spincr(Y)$, which is not what we look for. 
\end{remark}

\begin{proof}[Proof of Theorem \ref{1T4}] It suffices to show that the group $\HM_*(\y,\bs)\neq \{0\}$ for only finitely many $\bs\in \Spincr(Y)$. For any such $\bs$, there is at least one critical point for the perturbed functional $\CSd_{\omega}$. Since $\HM_*(\y,\bs)$ is independent of the tame perturbation $\q$, we can work instead with a sequence of admissible tame perturbations $\q_n$ with $\|\q_n\|_{\Pa}\to 0$ and obtain a sequence of configurations $\gamma_n$ such that 
	\[
	\grad \CL_{\omega}(\gamma_n)=-\q_n(\gamma_n). 
	\]
	By Proposition \ref{P1.5}, a subsequence of $\{\gamma_n\}$ converges to a solution $(B,\Psi)$ of \eqref{3DSWEQ}. By \eqref{E25.1} and the Compactness Theorem \ref{11.5} adapted to the 3-manifold case, we have a point-wise estimate
	\begin{equation}\label{E25.2}
	|F_{B^t}|<C'e^{-\zeta s} \text{ for some }\zeta,C'>0,
	\end{equation}
	where $s$ is the coordinate function on the cylindrical end $[-2,\infty)_s\times\Sigma$ extended constantly over the interior of $Y$. Take a basis $\{\nu_j\}$ of $H_2(Y, \partial Y; \R)$ and suppose each $\nu_j$ is realized as a weighted sum of oriented surfaces with cylindrical ends in $\hy$. Then \eqref{E25.2} provides a uniform upper bound for the pairing $|\langle F_{B^t},\nu_j\rangle |$. As a result, $c_1(\bs)\in H^2(Y,\partial Y;\Z)$ can take only finitely many possible values. This completes the proof of Theorem \ref{1T4}.
\end{proof}

From now on, we focus on the Seiberg-Witten equations \eqref{3DSWEQ}. The proof of Proposition \ref{P25.1} relies on the maximum principle and some formulae from Taubes' paper \cite{Taubes96}, as we explain now.

\begin{proof}[Proof of Proposition \ref{P25.1}] Let $(B_0,\Psi_0)$ be the reference configuration and set $(b,\psi)=(B,\Psi)-(B_0,\Psi_0)$.  We divide the integral in \eqref{E25.1} into two parts:
	\[
	\int_{\hy}=\int_Y\ +\ \int_{[0,\infty)_s\times\Sigma},
	\]
	and estimate them separately. For the compact region $Y$, we make use of an a priori estimate:
	
	\begin{lemma}\label{L25.3} There exists a constant $C_1(g_Y,\omega)>0$ depending only on the Riemannian metric $g_Y$ and the 2-form $\omega$ such that the estimates
		\[
|\Psi|,	|F_{B^t}|, |\nabla_B\Psi|<C_1 
		\]
		hold for any solution $(B,\Psi)$ to \eqref{3DSWEQ}. 
	\end{lemma}

\begin{proof} To estimate the spinor $\Psi$, we use the Weitzenb\"{o}ck formula to derive that (cf. \cite[Section 4.6]{Bible})
\begin{equation}\label{E25.4}
\half \Delta |\Psi|^2+|\nabla_B\Psi|^2+\half |\Psi|^4=-\frac{s}{4}|\Psi|^2-\langle \Psi, \rho_3(\omega)\Psi\rangle\leq C_2|\Psi|^2,
\end{equation}
where $C_2=\|s\|_\infty+\|\omega\|_\infty+\|\Psi_*\|_\infty^2$ and $\Psi_*$ is the standard spinor on $\R_s\times\Sigma$. Since $\Psi\to \Psi_*$ as $s\to \infty$, 
\[
 |\Psi|^2-2C_2<0
\]
when $s\gg 0$. By the maximum principle and \eqref{E25.4}, $|\Psi|^2-2C_2<0$ on $\hy$. The estimate for $|F_{B^t}|$ now follows from the curvature equation of \eqref{3DSWEQ}. 

\medskip

To estimate $|\nabla_B\Psi|$, we borrow a formula from Taubes' paper \cite[Section 2(e)]{Taubes96}. 

\begin{lemma}[Proposition \ref{PE.6}]\label{L20.4} There exists a constant $C_3(g_Y,\omega)>0$ such that 
\begin{equation}\label{E25.5}
\half \Delta |\nabla_B\Psi|^2+\half |\Psi|^2|\nabla_B\Psi|^2\leq C_3(|\Psi||\nabla_B\Psi|+|\nabla_B\Psi|^2+|F_{B^t}||\nabla_B\Psi|^2).
\end{equation}
\end{lemma}

Although Lemma \ref{L20.4} is not stated explicitly in \cite{Taubes96}, it follows from the derivation of \cite[(2.38)(2.40)]{Taubes96}. For a family of generalized Seiberg-Witten equations, a similar estimate is obtained in \cite[Proposition 2.12]{WZ19}. The proof of Lemma \ref{L20.4} is deferred to Appendix \ref{AppE}.

\medskip

Given the bound on $|F_{B^t}|$, the right hand side of \eqref{E25.5} can be further controlled by 
\[
C_4|\Psi|^2+C_5|\nabla_B\Psi|^2. 
\]
Now consider the function $w\colonequals |\nabla_B\Psi|^2+C_5|\Psi|^2$. We combine \eqref{E25.4} and \eqref{E25.5} to derive:
\[
\half \Delta w+\half |\Psi|^2w\leq (C_2C_5+C_4)|\Psi|^2.
\]
The maximum principle then implies that $w\leq 2(C_2C_5+C_4)$. This completes the proof of Lemma \ref{L25.3}.
\end{proof}
	
	 Back to the proof of Proposition \ref{P25.1}. It remains to estimate the integral \eqref{E25.1} over the cylindrical end $[0,\infty)_s\times\Sigma$, where the metric is flat. We first exploit the energy equation to write 
	 \begin{align*}
\int_{[0,\infty)_s\times\Sigma}  \frac{1}{4}|F_{B^t}|^2+|\nabla_B\Psi|^2+|(\Psi\Psi^*)_0+\rho_3(\omega)|^2&=\int_{\{0\}\times\Sigma} \langle D_B^{\Sigma}\Psi,\Psi\rangle-2\langle b,\lambda\rangle\colonequals I_1+I_2
	 \end{align*}
	
	\begin{remark} Here $D_B^\Sigma$ denotes the Dirac operator on the surface $\Sigma$ associated to the connection $B|_{\{0\}\times \Sigma}$. The sum $I_1+I_2$ can be recognized as $2\re W_\lambda ( (B,\Psi)|_{\{0\}\times\Sigma})$ where $W_\lambda$ is the superpotential defined in \cite[Subsection 7.2]{Wang202}. This energy equation is derived from \cite[Section 4.5]{Bible}.
	\end{remark}

While $I_1$ can be estimated directly using Lemma \ref{L25.3}, the second term $I_2$ requires further work. We first extend $I_2$ to be a function on $[0,\infty)_s$:
\[
I_2(s)\colonequals -\int_{\{s\}\times\Sigma} 2\langle b,\lambda\rangle.
\]
The idea is to estimate the derivative $\big|\frac{d}{ds}I_2\big|$.

\medskip

Since our analysis below is purely local, we focus on a connected component of the half cylinder $[0,\infty)_s\times\Sigma$. To ease our notation, we pretend that $\Sigma$ is connected from now on.

Since $
\omega=\mu+ds\wedge\lambda
$ is parallel on $[0,\infty)_s\times\Sigma$, we write $\mu=\delta\cdot dvol_\Sigma$. By our assumption, $\delta\neq 0$. Following the notation from \cite[Section 10]{Wang202}, the spin bundle $S^+$ splits as
\begin{equation}\label{E25.6}
L^+_\omega\oplus L^-_\omega
\end{equation}  with $\rho_3(\omega)$ acting on by a diagonal matrix
\[
m \begin{pmatrix}
-1 & 0\\
0 & 1
\end{pmatrix}
\]
where $m=\sqrt{|\delta|^2+|\lambda|^2}$ is a constant. The splitting \eqref{E25.6} is also parallel. Write $\Psi(t)=\sqrt{2m}(\alpha(t),\beta(t))$ with respect to \eqref{E25.6}. As observed by Taubes \cite{Taubes96, Taubes94},  \eqref{E25.4} can be separated for $\Delta |\alpha|^2$ and $\Delta |\beta|^2$. In our case, we use  \cite[(2.4)]{Taubes96} to obtain that
\[
\half \Delta|\beta|^2+|\nabla_B\beta|^2+m(|\alpha|^2+|\beta|^2+1)|\beta|^2=0.
\]
By Lemma \ref{L25.3} and the maximum principle, this implies that 
\begin{equation}\label{E25.7}
|\beta(s)|\leq \|\beta(0)\|_{L^\infty(\Sigma)}e^{-\sqrt{2m}s}<C_1e^{-\sqrt{2m}s}.
\end{equation}

The curvature equation in \eqref{3DSWEQ} says that 
\[
\half \rho_3(F_{B^t})=m\begin{pmatrix}
|\alpha|^2-|\beta|^2-1 & 2\alpha\beta^*\\
2\alpha^*\beta & -|\alpha|^2+|\beta|^2+1
\end{pmatrix}=(1-|\alpha|^2)\rho_3(\omega)+\SO(|\beta|).
\]
Hence, by \eqref{E25.7} and Lemma \ref{L25.3},
\begin{equation}\label{E25.8}
|\half F_{B^t}-\omega (1-|\alpha|^2)|\leq C_6 e^{-\sqrt{2m}s}
\end{equation}
for some $C_6>0$. On the one hand, we integrate \eqref{E25.8} over each slice $\{s\}\times \Sigma$ to obtain that
\begin{align}\label{E25.9}
\bigg|\int_{\{s\}\times\Sigma} (1-|\alpha|^2)\bigg|= \frac{1}{|\delta|} \bigg|\int_{\{s\}\times\Sigma}\half F_{B^t}-\omega (1-|\alpha|^2)\bigg|\leq \frac{\Vol(\Sigma)\cdot C_6}{|\delta|}e^{-\sqrt{2m}s}.
\end{align}

Here we used the assumption that $|\delta|\neq 0$. On  the other hand, the component of \eqref{E25.8} involving $ds$ is precisely:
\[
ds\wedge (\frac{\partial b}{\partial s}-\lambda (1-|\alpha|^2)).
\]
We combine \eqref{E25.8} and \eqref{E25.9} to conclude that 
\begin{align*}
\bigg |\frac{d}{ds} I_2(s)\bigg|\leq \bigg|\int_{\{s\}\times\Sigma} \langle \frac{\partial b}{\partial s}-\lambda (1-|\alpha|^2),\lambda\rangle\bigg|+|\lambda|^2\bigg|\int_{\{s\}\times\Sigma} (1-|\alpha|^2)\bigg|\leq C_7e^{-\sqrt{2m}s},
\end{align*}
for some $C_7>0$. Since $I_2(\infty)=0$, it follows that $|I_2(0)|<C_7/\sqrt{2m}$. This completes the proof of Proposition \ref{P25.1}.
\end{proof}

\section{The Product Manifold $\Sigma_{g,n}\times S^1$} \label{Sec30}

In this section, we compute the monopole Floer homology of the product manifold $\Sigma_{g,n}\times S^1$, where $\Sigma_{g,n}$ is a genus-$g$ surface with $n$ cylindrical ends. Let us first recall the case for closed surfaces.

\medskip

Let $\Sigma_g$ be a closed oriented surface of genus $g\geq1$. Equip the 3-manifold $\Sigma_g\times S^1$ with a product metric. We are interested in the case when 
\[
c_1(\s)=2(d-g+1)\cdot k\in H^2(\Sigma_g\times S^1;\Z) \text{ for some } 0\leq d\leq 2(g-1),
\]
where $k$ is the Poincar\'{e} dual of $\{pt\}\times S^1$. Consider the 2-form
\[
\omega=\delta\cdot dvol_{\Sigma_g}+d\theta\wedge\lambda'. 
\]
for some $\delta\in i\R$ and harmonic 1-form $\lambda'\neq 0\in \Omega^1_h(\Sigma_g, i\R)$. If the holomorphic 1-form $(\lambda')^{1,0}$ on $\Sigma_g$ has only simple zeros, then the 3-dimensional Seiberg-Witten equations \eqref{3DSWEQ} associated to $(\Sigma_g\times S^1,\bs)$ can be solved explicitly (see \cite[Proposition C.2]{Wang202}): they have precisely 
\[
\binom{2g-2}{d}
\] 
solutions up to gauge, all of which are non-degenerate as critical points of $\CL_{\omega}$, concentrating on a single homology grading,  the one corresponding to $S^1$-invariant 2-plane fields with first Chern class $c_1(\s)$. As a result, the reduced monopole Floer homology can be computed as
\[
\HM_*^{red}(\Sigma_g\times S^1, [\omega];\bs)\cong \NR^{\binom{2g-2}{d}}.
\]
where $\NR$ is a Novikov ring. Since we have used a non-exact non-balanced perturbation, $\widehat{\HM}_*(\Sigma_g\times S^1, [\omega];\bs)$ and $\widecheck{\HM}_*(\Sigma_g\times S^1, [\omega];\bs)$ are both isomorphic to the reduced version. 

If one works instead with an exact perturbation and $0\leq d\leq g-2$, then $\HM_*^{red}$ computes the singular homology of the symmetric product $\Sym^{d} \Sigma_g$, whose rank is larger. In fact,
\[
\chi(\Sym^d\Sigma_g)=\binom{2g-2}{d} \text{ for any } 0\leq d\leq 2(g-1). 
\]

The goal of this section is to generalize this computation for surfaces with cylindrical ends, as we explain now.

\subsection{The Setup} Let $\Sigma_{g,n}=\Sigma_g\setminus\{p_1,\cdots,p_n\}$ be the punctured surfaces obtained from $\Sigma_g$ by removing $n$ distinct points. We require that 
\[
\chi(\Sigma_{g,n})=2-2g-n\leq 0;
\]
so the genus $g$ can be zero if $n\geq 2$. We identify a neighborhood $U_i$ of $p_i$ with a cylindrical end using the map:
\begin{align*}
\epsilon_j: [0,\infty)_s\times (\R/2\pi\alpha_j \Z)&\to B(0,1)\cong U_j\subset \Sigma,\\
(s,\theta_j)&\mapsto e^{-s-i\theta_j/\alpha_j}. 
\end{align*}
Pick a  metric of $\Sigma_{g,n}$ such that it restricts to the product metric on each end $[0,\infty)_s\times (\R/2\pi\alpha_j \Z)$; so the $j$-th boundary component $S^1_j\colonequals \{0\}\times \R/2\pi \alpha_j\Z$ has length $2\pi \alpha_j$ for some $\alpha_j>0$.

We will work with the product metric on $\Sigma_{g,n}\times S^1$. Let $\theta$ be the coordinate function on $S^1$ such that $d\theta=*_3dvol_{\Sigma_{g,n}}$. Define the closed 2-form $\omega$ to be
\[
\omega=\delta dvol_{\Sigma_{g,n}}+ d\theta\wedge \lambda'.
\]
such that $\delta\neq 0\in i\R$ and $\lambda'\in \Omega^1(\Sigma_{g,n},i\R)$ is closed. When restricted to each cylindrical end, we require that 
\begin{itemize}
\item $\lambda'$ is the constant 1-form $\delta_j d\theta_j+c_j ds$ for some $\delta_j,c_j\in i\R$;
\item $\delta_j\neq 0$ for any $1\leq j\leq n$, and $\sum \alpha_jc_j=0$.
\end{itemize}
As a result, 
\[
\omega=(-\delta_j) d\theta_j\wedge d\theta+ds\wedge (\delta d\theta_j-c_jd\theta) \text{ on }[0,\infty)_s\times S^1_j\times S^1, 
\]
and
\[
\begin{array}{ll}
\mu=((-\delta_j) d\theta_j\wedge d\theta)_{1\leq j\leq n}&\in \Omega^2_h(\partial \Sigma_{g,n}\times S^1, i\R),\\
\lambda=(\delta d\theta_j-c_jd\theta)_{1\leq j\leq n}&\in \Omega^1_h(\partial \Sigma_{g,n}\times S^1, i\R). 
\end{array}
\]
In particular, $*_2\lambda=\delta d\theta+c_jd\theta_j$ extends to a closed 1-form on $\Sigma_{g,n}\times S^1$ and the first alternative \ref{VV1} in Assumption \ref{A1.2} holds. Thus the monopole Floer homology of $(\Sigma_{g,n}\times S^1, \omega)$ is well defined.

\begin{proposition}\label{P27.1} Suppose the metric of $\Sigma_{g,n}$ and the 2-form $\omega$ are given as above. Consider the relative \spinc structure $\bs$ with 
	\[
	c_1(\bs)=(2d+\chi(\Sigma_{g,n}))\cdot k\in H^2(\Sigma_{g,n}\times S^1, \partial \Sigma_{g,n}\times S^1;\Z),
	\]
	where $k$ is the Poincar\'{e} dual of $\{pt\}\times S^1$ and
	\[
	0\leq d\leq -\chi(\Sigma_{g,n}).
	\]
	If in addition $\lambda'$ is harmonic and the holomorphic 1-form $(\lambda')^{1,0}$ has $(2g-2+n)$ simple zeros on $\Sigma_{g,n}$, then the 3-dimensional Seiberg-Witten equations \eqref{3DSWEQ} associated to $(\Sigma_{g,n}\times S^1,\omega; \bs)$ has precisely 
	\[
	\binom{2g-2+n}{d}
	\]
	solutions up to gauge. Moreover, they are non-degenerate as the critical points of $\CL_{\omega}$ and concentrate on a single homology grading in the sense of Section \ref{Sec28}. In particular,
	\[
	\HM_*(\Sigma_{g,n}\times S^1,\omega;\bs)\cong \NR^{	\binom{2g-2+n}{d}}. 
	\]
\end{proposition}

\begin{remark} By \cite[Lemma C.9\&C.10]{Wang202}, for any fixed $(\delta_i d\theta_j+c_jds)_{1\leq j\leq n}$, one can find a harmonic 1-form $\lambda'$ extending these forms, if the metric of $\Sigma_{g,n}$ is allowed to change. Thus the assumptions on $\lambda'$ can be always fulfilled for any boundary data $(\mu,\lambda)$.

On the other hand, one may ask if the cohomology class of $d\theta\wedge\lambda'$ can be fixed in Proposition \ref{P27.1}. This problem will be addressed in the third paper \cite{Wang203} using some formal arguments. 
\end{remark}

The proof of Proposition \ref{P27.1} relies on the computation from \cite[Proposition C.6]{Wang202}. The dimension reduction of \eqref{3DSWEQ} gives rise to a kind of vortex equations on $\Sigma_{g,n}$, which can be solved explicitly. Although the first paper \cite{Wang202} focused on the 2-torus, its main result, Proposition 1.5, generalizes to higher genus surfaces as well as surfaces with cylindrical ends. In particular, one can associate an infinite dimensional gauged Landau-Ginzburg model
\begin{equation*}
(M(\Sigma_{g,n}, \delta), W_{\lambda'}, \CG(\Sigma_{g,n}))
\end{equation*}
to $(\Sigma_{g,n},\lambda',\delta)$, whose gauged Witten equations on $\C$ recover the Seiberg-Witten equations on $\C\times \Sigma_{g,n}$. The downward gradient flowline equation of $\re W_{\lambda'}$ on $\R_s$ recovers the 3-dimensional equations \eqref{3DSWEQ} on $\R_s\times \Sigma_{g,n}$. 

We will work with $S^1$ instead of $\R_s$. But the situation is not very different. The structure of the extended Hessian can be analyzed as in \cite[Subsection 4.2]{Wang202}. In what follows, we will explain how this reduction works and refer the reader to the corresponding sections of the first paper \cite{Wang202} for the actual proofs.

\subsection{Proof of Proposition \ref{P27.1}} Our plan is to solve the Seiberg-Witten equations \eqref{3DSWEQ} explicitly. The first step is to show that under the assumption of Proposition \ref{P27.1}, any such solution $(B,\Psi)$ is $S^1$-invariant. 

To see this, identify $S^1$ with $\R/l\Z$, where $l>0$ is the length of $S^1$. We shall regard the relative \spinc structure $\bs$ as pulled back from $\Sigma_{g,n}$ with 
\[
S=L^+\oplus L^-,
\]
where $L^+\to \Sigma_{g,n}$ is a relative line bundle of degree $d$ and $L^-=L^+\otimes\Lambda^{0,1}\Sigma_{g,n}$. Moreover, $L^\pm$ is the $(\pm i)$-eigenspace of $\rho_3(d\theta)$. Write $\Psi=(\Psi_+,\Psi_-)$ under this decomposition and use $(\cdot )_\Pi$ to denote the off-diagonal part of an endomorphism of $S$. The 2-dimensional Clifford multiplication is now given by 
\[
\rho_2(e)\colonequals \rho_3(d\theta)^{-1}\rho_3(e): S\to S
\]
for any $e\in T^*\Sigma_{g,n}$, which allows us to define the Dirac operator associated to $\cB(\theta)\colonequals B|_{\{\theta\}\times\Sigma_{g,n}}$:
\[
D_{\cB}^{\Sigma_{g,n}}=\rho_2(e_i)\nabla^{\cB}_{e_i}: \Gamma(\Sigma_{g,n}, S)\to  \Gamma(\Sigma_{g,n}, S). 
\]
If $(B,\Psi)$ is put into the temporal gauge, then the 3-dimensional Seiberg-Witten equations \eqref{3DSWEQ} can be cast into the form:
\begin{align*}
\partial_\theta \cB(\theta)&=[-\rho_2^{-1}(\Psi\Psi^*)_\Pi+\lambda']\otimes\Id_S,\\
\partial_\theta \Psi(\theta)&=-D_{\cB(\theta)}^{\Sigma_{g,n}} \Psi(\theta),\\
0&=\half *_2 F_{\cB^t}+\frac{i}{2}(|\Psi_+|^2-|\Psi_-|^2)+\delta.
\end{align*}
When $\lambda'$ is a harmonic form on $\Sigma_{g,n}$, the first two equations give rise to a downward gradient flow for the functional $\re W_{\lambda'}$. Here $W_{\lambda'}$ is the superpotential associated to the gauged Landau-Ginzburg models, as defined in \cite[Appendix C]{Wang202}, which is invariant under the action of $\CG_e(\Sigma_{g,n})$, the identity component of the full gauge group $\CG(\Sigma_{g,n})$. Since $(\cB(0), \Psi(0))$ and $(\cB(l), \Psi(l))$ are related by a gauge transformation in $\CG_e(\Sigma_{g,n})$, the energy of this flowline is zero; so $(B,\Psi)$ must be $\theta$-invariant.

\medskip

The critical points of $\re W_{\lambda'}$ are computed in \cite[Proposition C.6]{Wang202}. In this case, the sections $\Psi_+$ and $\Psi_-^*$ are holomorphic with respect to some unitary connections on $L^+$ and $(L^-)^*$; moreover, 
\[
\Psi_+\otimes \Psi_-^*=-\sqrt{2}(\lambda')^{1,0}. 
\]
Thus the zero loci $Z(\Psi_+)$ and $Z(\Psi_-)$ give rise to a partition of $Z((\lambda')^{1,0})$. Conversely, any such partition produces a critical point; so there are $\binom{2g-2+n}{d}$ in total. 

\medskip

Let $\fa$ be an $S^1$-invariant solution to \eqref{3DSWEQ}. To see that $\fa$ is a non-degenerate critical point of $\CL_{\omega}$, we exploit the results from \cite[Subsection 4.2]{Wang202}: the extended Hessian at such a critical point $\fa$ can be cast into the form
\[
\EHess_{\fa}=\sigma(\partial_\theta+\hat{D}_{\kappa}). 
\]
where the bundle map $\sigma$ is defined as in  $\eqref{E12.2}$ with $\rho_3(ds)$ replaced by $\rho_3(d\theta)$. Moreover, 
\[
\hat{D}_{\kappa}: L^2_1(\hy, i\R\oplus (i\R\otimes d\theta)\oplus iT^*\Sigma_{g,n}\oplus S)\to  L^2(\hy, i\R\oplus (i\R\otimes d\theta)\oplus iT^*\Sigma_{g,n}\oplus S)
\]
is a self-adjoint differential operator anti-commuting with $\sigma$. It is shown in \cite[Proposition C.6\&7.10]{Wang202} that $\hat{D}_{\kappa}$ is invertible, and so is the extended Hessian $\EHess_{\fa}$. 

\medskip

Finally, we describe the canonical grading that $\fa$ belongs to using the Normalization Axiom \ref{Axiom2} from Section \ref{Sec28}. Let $\Psi\in \Gamma(\Sigma_{g,n}\times S^1, S)$ be the spinor of $\fa$. Although the assumption \ref{V2} does not hold strictly for $\Psi$, the canonical grading of $\fa$ is still given by 
the relative homotopy class of 
\[
\frac{\Psi}{|\Psi|}, 
\]
which is $S^1$-invariant. Since any non-vanishing relative sections of $S\to \Sigma_{g,n}$ are relatively homotopic to each other, the canonical grading of $\fa$ is determined by the $S^1$-invariance of $\Psi/|\Psi|$. This completes the proof Proposition \ref{P27.1}.

\appendix

\section{Relative Orientations}\label{AppD}

The primary goal of this appendix is to present the proof of Theorem \ref{T24.2}, which leads to the notion of homology orientations in Definition \ref{D24.3}. It allows us to orient the moduli spaces in consistently when the complete Riemannian 4-manifold $\CX$ possesses a planar end $\HH^2_+\times\Sigma$.  

To do this, we have to develop the theory of \textbf{relative orientations} in a systematic way. One possible  approach is to use the argument in \cite[Appendix]{KM97} in which case a Riemannian 4-manifold with a conic end is considered. The construction that we present here is slightly different. It is self-contained and combinatorial in nature, having the advantage of being very explicit and concrete. It relies on a simple proof of the excision principle of elliptic differential operators, which was due to Mrowka. 

The main results are Proposition \ref{PD.3} and \ref{PD.9}. As an application of this abstract theory, we will prove Theorem \ref{T24.2} in Subsection \ref{SubsecD.10}.

\subsection{Statements} The situation that we have here is similar to the excision principle of elliptic differential operators; we follow its setup. Given a oriented \textbf{compact} manifold $Y$, consider vectors bundles $E,F\to [-1,1]\times Y$ and a reference first-order elliptic differential operator:
\[
D: \Gamma([-1,1]\times Y,E)\to \Gamma([-1,1]\times Y, F). 
\]

We are interested in two classes of elliptic differential operators 
\[
\SL \text{ and } \SR. 
\]
Each element of $\SL$ consists of a pair $(X_i, L_i)$ satisfying the following properties:
\begin{enumerate}[label=(J\arabic*)]
\item\label{J1} $X_i$ is an oriented smooth manifold with boundary $Y$; moreover, there exists a collar neighborhood $W_i\subset X_i$ of $Y$ and a diffeomorphism 
\[
\phi_i: (W_i, Y)\to ([-1,1]\times Y, \{1\}\times Y)
\]
identifying $W_i$ with the standard cylinder; $X_i$ is not necessarily compact;
\item\label{J2} $L_i:\Gamma(X_i, E_i)\to \Gamma(X_i, F_i)$ is a first-order elliptic differential operator where $E_i, F_i\to X_i$ are vector bundles over $X_i$. The operator $L_i$ is cast into a standard form on the collar neighborhood $W_i$ in the following sense. There exist bundle isomorphisms
\[
\phi_i^E: E_i|_{W_i}\to E,\ \phi_i^F: F_i|_{W_i}\to F, 
\]
covering the diffeomorphism $\phi_i: W_i\to [-1,1]\times Y$ in \ref{J1} such that 
\[
L_i=(\phi_i^F)^{-1}\circ D\circ \phi_i^E \text{ on }W_i. 
\]
\end{enumerate}

Similar properties are required for an element $(X_j, R_j)$ of  $\SR$ with one distinction: the oriented boundary of $X_j$ is $(-Y)$, so under the diffeomorphism $\phi_j$, it is mapped to $\{-1\}\times (-Y)$:
\[
\phi_j: (W_j,  (-Y))\to ([-1,1]\times Y, \{-1\}\times (-Y))
\]

For any operators $(X_i, L_i)\in \SL$ and $(X_j, R_j)\in \SR$, we first glue up their underlying manifolds and obtain a manifold without boundary:
\[
X_i\#X_j: X_i\coprod X_j/\sim_{ij} \text{ where } \phi_i(x_i)\sim_{ij} \phi_j(x_j) \text{ if } x_i\in W_i, x_j\in W_j.
\]

Similarly we glue vector bundles and obtain $E_i\#E_j,\ F_i\#F_j\to X_i\#X_j$ using $(\phi^E_i, \phi^E_j)$ and $(\phi^F_i, \phi^F_j)$ instead. Finally, we glue operators and obtain 
\[
L_i\# R_j: \Gamma(E_i\#E_j)\to \Gamma(F_i\#F_j). 
\]
\begin{assumpt}\label{AD.1} The first-order elliptic differential operator
	\[
	L_i\# R_j: L^2_1(E_i\#E_j)\to L^2(F_i\#F_j)
	\]
	is assumed to be Fredholm for any elements $(X_i, L_i)\in \SL$ and $(X_j, R_j)\in \SR$.
\end{assumpt}

In terms of Example \ref{EX24.1}, define
\[
\Lambda(L_i\# R_j)
\]
to be the 2-element set of orientations of this Fredholm operator $L_i\# R_j$. 

From now on, we will omit the underlying manifolds when it is clear from the context. For any operators $L_1, L_2\in \SR$, we wish to define a 2-element set $\Lambda(L_1, L_2)$ such that any element $x\in \Lambda(L_1, L_2)$ defines a preferred $\Z/2\Z$-equivariant map
\[
\Lambda(L_1\# R_3)\to \Lambda( L_2\# R_3)
\]
for any $R_3\in \SR$. We will proceed in the opposite order and first define
\[
\Lambda(L_1, L_2; R_3)\colonequals \Hom_{\Z/2\Z}(\Lambda(L_1\#R_3), \Lambda(L_2\#R_3)).
\]

Then the goal is to construct natural bijections:
\begin{equation}\label{ED.8}
p(R_3, R_4): \Lambda(L_1, L_2; R_3)\to \Lambda(L_1, L_2; R_4)
\end{equation}
for any operators $L_1, L_2\in \SL$ and $L_3, L_4\in \SR$ such that  the following axioms are satisfied:
\begin{enumerate}[label=(C-\Roman*)]
\item\label{C-I} $p$ is associative meaning that for any three operators $R_j\in \SR, 3\leq j\leq 5$, we have 
\[
p(R_4, R_5)\circ p(R_3, R_4)=p(R_3, R_5): \Lambda(L_1, L_2; R_3)\to \Lambda(L_1, L_2; R_5);
\]
\item\label{C-II} $p$ is reflexive meaning that $p(R_3, R_3)=\Id$;
\item \label{C-III} When $L_1=L_2$, $p$ preserves the identity element: 
\[
p: 1\in \Lambda(L_1, L_1; R_3)\mapsto 1\in \Lambda(L_1, L_1; R_4);
\]
\item\label{C-IV} $p$ commutes with compositions of $\Hom$-sets, i.e. for any three operators $L_i\in\SL, 0\leq i\leq 2$, the following diagram is commutative:
\[
\begin{tikzcd}
 \Lambda(L_0, L_1; R_3)\times \Lambda(L_1, L_2; R_3)\arrow[r,"m"]\arrow[d,"{(p,p)}"]& \Lambda(L_0, L_2; R_3)\arrow[d,"p"]\\
  \Lambda(L_0, L_1; R_4) \times\Lambda(L_1, L_2; R_4)\arrow[r,"m"]& \Lambda(L_0, L_2; R_4),
\end{tikzcd}
\]
where horizontal arrows $m$ are given by compositions of maps.
\end{enumerate}

\begin{definition}\label{DD.2} For any classes $\SL$ and $\SR$, a collection of bijections $\{p\}$ satisfying axioms \ref{C-I}-\ref{C-IV} defines an equivalence relation on the disjoint union:
	\[
	\coprod_{R_j\in \SR} \Lambda(L_1, L_2; R_j).
	\]
	Let $\Lambda(L_1, L_2)$ be the quotient space, then the composition map $m$ descends to an associative multiplication:
	\[
	\bar{m}: \Lambda(L_0, L_1)\times\Lambda(L_1, L_2)\to \Lambda(L_0, L_2), 
	\]
	which admits a unit in each $\Lambda(L_i, L_i)$. An element of $\Lambda(L_1, L_2)$ is called \textbf{a relative orientation} of $L_1$ and $L_2$. 
\end{definition}


Here is the main result of this appendix. 

\begin{proposition}\label{PD.3} There exists a collection of bijections $\{p(R_3, R_4)\}$ satisfying $\ref{C-I}-\ref{C-IV}$ for any classes of operators $\SL$ and $\SR$ such that Assumption \ref{AD.1} holds. 
\end{proposition}

One can prove that the collection $\{p(R_3, R_4)\}$ is unique in a suitable sense:

\begin{proposition}\label{PD.4} Under the assumptions of Proposition \ref{PD.3}, suppose that there are two collections of bijections $\{p\}$ and $\{p'\}$ satisfying axioms $\ref{C-I}-\ref{C-IV}$, then one can find a function:
	\[
	\iota: \SL\times \SR\to \Z/2\Z
	\]
	such that 
	\[
	p(L_1, L_2; R_3, R_4)=(-1)^{\eta} p'(L_1, L_2; R_3, R_4):\Lambda(L_1, L_2; R_3)\to \Lambda(L_1, L_2; R_4)
	\]
	with $\eta=\iota(L_1, R_3)+\iota(L_1, R_4)+\iota(L_2, R_3)+\iota(L_2, R_4)$. In other words, $p'$ is obtained from $p$ by applying the automorphism 
	\[
	\iota(L_i, R_j): \Lambda(L_i\#R_j)\to \Lambda(L_i\# R_j)
	\]
	for each pair $(L_i, R_j)\in \SL\times \SR$. 
\end{proposition}

\begin{remark} The proof of Proposition \ref{PD.3} is constructive; see Proposition \ref{PD.9} below for a refined statement. In particular, we will pick up a preferred collection $\{p\}$ for our primary applications in gauge theory. Axioms \ref{C-II} and \ref{C-III} are redundant, since they follow from the other two axioms.
\end{remark}

\subsection{Compatibility with Direct Sums} Proposition \ref{PD.3} will guarantee the first property \ref{U1} in Theorem \ref{T24.2}, but \ref{U2} will require an additional property of the collection $\{p(R_3,R_4)\}$, as we explain now.

 The class $\SR$ can be extended slightly to incorporate more operators. Denote this new class by $\widehat{\SR}$. Each element of $\widehat{\SR}$ is a triple $\widehat{R}_j\colonequals (P_j, R_j, Q_j)$ where 
\begin{itemize}
\item $R_j\in \SR$;
\item $P_j: H^a_j\to H^b_j$ and $Q_j: H^c_j\to H^d_j$ are arbitrary Fredholm operators; here $H^a_j\sim H^d_j$ are arbitrary Hilbert spaces. 
\end{itemize}
Now instead of $L_i\#R_j$, we look at
\[
L_i\#\widehat{R}_j\colonequals P_j\oplus (L_i\# R_j) \oplus Q_j: H_j^a\oplus L^2_1(E_i\#E_j)\oplus H_j^c\to H_j^b\oplus L^2(F_i\#F_j)\oplus H_j^d. 
\]

To extend Proposition \ref{PD.3} for this new class of operators $\widehat{\SR}$, we impose a convenient condition on the first class $\SL$. 

\begin{definition}\label{DD.5} The class of operators $\SL$ is called even if for any $L_1, L_2\in \SL$, 
	\begin{equation}\label{ED.18}
	\Ind L_1\#R_3-\Ind L_2\# R_3\equiv 0\mod 2,\ \forall R_3\in \SR. \qedhere
	\end{equation}
\end{definition}

Also, we look for a normalization property on the map
\[
p(R_3, \hr_3):  \Lambda(L_1, L_2; R_3)\to \Lambda(L_1, L_2; \hr_3). 
\]
\begin{proposition}\label{PD.5} Suppose $\SL$ is an even family of operators and Assumption \ref{AD.1} holds for $(\SL, \SR)$, then there exists a collection of bijections $\{p(\hr_3,\hr_4)\}$ satisfying axioms \ref{C-I}-\ref{C-IV} for the class $\SL$ and $\widehat{\SR}$. This collection satisfies the following additional property: for any $\hr_3=(P_3, R_3,Q_3)\in \widehat{\SR}$, the following diagram fis commutative:
	 \begin{equation}\label{ED.19}
	\begin{tikzcd}
\Lambda(P_3) \Lambda(L_1\# R_3)\Lambda(Q_3)\arrow[r,"{\Id\otimes x\otimes \Id}"]\arrow[d,"q"] &  \Lambda(P_3) \Lambda(L_2\# R_3)\Lambda(Q_3)\arrow[d,"q"]\\
\Lambda(L_1\# \hr_3)\arrow[r,"{p(R_3,\hr_3)(x)}"] &  \Lambda(L_2\# \hr_3)
\end{tikzcd}
	 \end{equation}
 for any $x\in \Lambda(L_1, L_2; L_3)$. The vertical maps are induced from \eqref{E24.3}.
\end{proposition}

Proposition \ref{PD.5} will be proved in Subsection \ref{SubsecD.9}.

\subsection{Construction of Bijections} Our construction of bijections $\{p\}$ is motivated by a simple proof of the excision principle which states that 
\begin{equation}\label{ED.3}
\Ind(L_1\# R_3)+\Ind(L_2\# R_4)=\Ind(L_1\# R_4)+\Ind(L_2\# R_3)
\end{equation}
for any $L_1, L_2\in \SL$ and $R_3, R_4\in \SR$. The author learned this elegant proof of excision principle in a graduate course at MIT, taught by Prof. Mrowka, who has kindly agreed to present his proof here. 

Consider a cut-off function $\theta: [-1,1]\to \R$ such that 
\[
\theta(t)\equiv 0 \text{ if } t\in [-1,-\half];\  \theta(t)\equiv \frac{\pi}{2} \text{ if } t\in [\half, 1].
\] 
Over each manifold $X_i\# X_j$, $\theta$ extends to a global function by setting $\theta \equiv 0$ on $X_i\setminus W_i$ and $\theta \equiv 1$ on $X_j\setminus W_j$. Consider functions $\phi_L\colonequals\cos\theta \text{ and }\phi_R\colonequals \sin\theta$, then the matrix 
\[
U=\begin{pmatrix}
\phi_L & -\phi_R\\
\phi_R & \phi_L
\end{pmatrix} \text{ with inverse }
U^{-1}=\begin{pmatrix}
\phi_L & \phi_R\\
-\phi_R & \phi_L
\end{pmatrix}
\]
defines an invertible operator between Hilbert spaces:
\[
L^2_k(E_1\# E_3)\oplus L^2_k(E_2\# E_4)\to L^2_k(E_1\# E_4)\oplus L^2_k(E_2\# E_3)
\]
for any $k\in\{0,1\}$. The same statement holds if we use bundles $F_i$ instead. In what follows, we write $E_{ij}$ for $E_i\#E_j$,  $F_{ij}$ for $F_i\#F_j$ and $D_{ij}$ for $L_i\#R_j$.

\begin{lemma}\label{EL.1} The following diagram is commutative up to a compact operator:
	\begin{equation}\label{ED.2}
\begin{tikzcd}
L^2_1(E_{13})\oplus L^2_1(E_{24})\arrow[r,"U"] \arrow[d,"D_{13}\oplus D_{24}"]&L^2_1(E_{14})\oplus L^2_1(E_{13})\arrow[d,"D_{14}\oplus D_{23}"]\\
L^2(F_{13})\oplus L^2(F_{24})\arrow[r,"U"]& L^2(F_{14})\oplus L^2(F_{13})
\end{tikzcd}
	\end{equation}
\end{lemma}
\begin{proof} Note that the inclusion $L^2_1([-1,1]\times Y)\to L^2([-1,1]\times Y)$ is compact, since $Y$ is compact.
\end{proof}

Apparently, the excision principle \eqref{ED.3} is an immediate corollary of Lemma \ref{EL.1}. On the other hand, the digram \eqref{ED.2} also gives rise to an identification of orientations:
\begin{equation}\label{ED.6}
U_*: \Lambda(D_{13}\oplus D_{24})\to \Lambda(D_{14}\oplus D_{23})
\end{equation}
understood in the sense of Example \ref{EX24.1}. Let us make a more precise statement:
\begin{lemma}\label{LB.6} Suppose $\{\A_z: H_1\to H_2 \}_{z\in \CZ}$ is a family of Fredholme operators parametrized by a topological space $\CZ$. In addition, let  $\{U_z: H_0\to H_1\}_{z\in \CZ}$ and $\{V_z: H_2\to H_3\}_{z\in\CZ}$ be families of invertible operators parametrized by the same space. Form the new family $\{U_z\circ \A_z\circ V_z: H_0\to H_3\}_{z\in \CZ}$, then there is continuous bundle map:
	\begin{equation}\label{ED.4}
	(U,V)_*: \det \A\to \det (U\circ \A\circ V),
	\end{equation}
whose restriction at each fiber is given by 
	\[
	\alpha_z\otimes \beta_z^*\mapsto U_z^{-1}(\alpha_z)\otimes (V_z(\beta_z))^*
	\]
	if $\alpha_z$ and $\beta_z$ are elements in $\Lambda^{max}\ker \A_z$ and $\Lambda^{max}\coker \A_z$ respectively.
\end{lemma}
\begin{proof} One has to go back to the definition of determinant line bundles in \cite[Section 20.2]{Bible} to verify that $(U, V)_*$ is continuous, using the fact that $U$ and $V$ are families of invertible operators. 
\end{proof}
\begin{remark}\label{RD.7} It is clear that this construction is functorial with respect to compositions of families of invertible operators. 
\end{remark}

\begin{lemma}\label{LD.8} The bundle map $\eqref{ED.4}$ is functorial with respect to direct sums of operators in the following sense. Suppose $\{\A'_z: H'_1\to H'_2 \}_{z\in \CZ}$ and $\{\A''_z: H''_1\to H''_2 \}_{z\in \CZ}$ are two families of Fredholm operators, and similarly we have families of invertible operators:
	\[
	\{U_z'\}, \{U_z''\}, \{V_z'\}, \{V_z''\}. 
	\]
	as in Lemma \ref{LB.6} parametrized by the same topological space $\CZ$.Then we have a commutative diagram:
	\begin{equation}\label{ED.5}
	\begin{tikzcd}[column sep=3cm]
\det\A'\otimes \det\A''\arrow[r,"{(U',V')_*\otimes (U'',V'')_*}"]\arrow[d,"q"] & \det (U'\circ \A'\circ V')\otimes \det (U''\circ \A''\circ V'')\arrow[d,"q"]\\
\det(\A'\oplus\A'') \arrow[r,"{(U'\oplus U'',V'\oplus V'')_*}"] & \det (U'\oplus U'')\circ (\A'\oplus\A'')\circ  (V'\oplus V''),
\end{tikzcd}
	\end{equation}
	where vertical maps are induced from \eqref{E24.3}. 
\end{lemma}

In our primary applications, $\CZ$ is always a contractible space; see Example \ref{EX24.1}. In light of Lemma \ref{LB.6}, the identification in \eqref{ED.6} is in fact $(U^{-1}, U)_*$, but we will keep using the notation $U_*$ for convenience. Now consider the following diagram:
\begin{equation}\label{ED.7}
\begin{tikzcd}[column sep=huge]
\Lambda(D_{13})\Lambda(D_{24})\arrow[r,"{\bar{p}(R_3,R_4)}"]\arrow[d,"{q_{13;24}}"] &  \Lambda(D_{14})\Lambda(D_{23})\arrow[d,"q_{14;23}"]\\
\Lambda(D_{13}\oplus D_{24})\arrow[r,"(-1)^rU_*"]& \Lambda(D_{14}\oplus D_{23})\
\end{tikzcd}
\end{equation}
where vertical maps are induced from \eqref{E24.3}. The top horizontal arrow $\bar{p}(R_3,R_4)$ is equivalent to a map:
\[
p(R_3,R_4): \Hom_{\Z/2\Z}(\Lambda(D_{13}), \Lambda(D_{23}))\to\Hom_{\Z/2\Z}(\Lambda(D_{14}), \Lambda(D_{24}))
\]
for which we are aiming in \eqref{ED.8}. One may define $\bar{p}(D_3,D_4)$ by making the diagram \eqref{ED.7} commutative, but there is a choice of freedom for the sign $(-1)^r$. In fact, there is no reason to believe that the identification map $U_*$ in \eqref{ED.6} is just the natural one, as there are different ways to set up the excision picture.
\begin{proposition}\label{PD.9} Suppose Assumption \ref{AD.1} holds for the families of operators $(\SL,\SR)$. We construct the bijection in \eqref{ED.8} by declaring the diagram \eqref{ED.7} to be commutative with 
	\begin{equation}\label{ED.9}
	r(L_1, L_2; R_3, R_4)=\Ind D_{23}\cdot \Ind D_{24}+\Ind D_{24}. 
	\end{equation}
	Then the collection of bijections $\{p(R_3, R_4)\}$ satisfies Axioms \ref{C-I}-\ref{C-IV}. 
\end{proposition}

The proof of Proposition \ref{PD.9} will dominate Subsections \ref{SubsecB.3}-\ref{SubsecB.8}.

\subsection{A Toy Model}\label{SubsecB.3} To convince ourselves that the formula $\eqref{ED.9}$ indeed provides the correct convention, let us verify a degenerate case when $Y=\emptyset$.  In this case, we assume that every $L_i $ and $R_j$ are Fredholm operators themselves, so 
\[
D_{ij}=L_i\oplus R_j,
\]
and \eqref{ED.7} fits into a larger diagram:
\begin{equation}\label{ED.10}
\begin{tikzcd}[column sep=huge]
\Lambda(L_1)\Lambda(R_3)\Lambda(L_2)\Lambda(R_4)\arrow[r,"{\tilde{p}(R_3,R_4)=\Id}"]\arrow[d,"q_{1;3}\otimes q_{2;4}"] &\Lambda(L_1)\Lambda(R_4)\Lambda(L_2)\Lambda(R_2)\arrow[d,"q_{1;4}\otimes q_{2;3}"]\\
\Lambda(D_{13})\Lambda(D_{24})\arrow[r,"{\bar{p}(R_3,R_4)}"]\arrow[d,"{q_{13;24}}"] &  \Lambda(D_{14})\Lambda(D_{23})\arrow[d,"q_{14;23}"]\\
\Lambda(D_{13}\oplus D_{24})\arrow[r,"(-1)^rU_*"]& \Lambda(D_{14}\oplus D_{23})\
\end{tikzcd}.
\end{equation}

If we declare the top horizontal map $\tilde{p}(R_3, R_4)$ to be the identity map, then the resulting collection $\{p(R_3,R_4)\}$ will satisfy all axioms we want. Hence, we can determine the sign $(-1)^r$ on the bottom if the digram $\eqref{ED.10}$ is commutative. In this case, the matrix $U$ is a $4\times 4$ matrix:
\[
U=\begin{pmatrix}
1 & 0  & 0 &0\\
0 & 0 & 0 & -1\\
0 & 0 & 1 & 0\\
0 & 1 & 0 & 0
\end{pmatrix}: L^2_k(E_1\oplus E_3 \oplus E_2\oplus E_4)\to  L^2_k(E_1\oplus E_4 \oplus E_2\oplus E_3)
\]
for $k\in \{0,1\}$ (which is also true for $F_i$). To compute the sign induced from $U$, let us record two lemmas:
\begin{lemma}\label{LD10} Given $\{\A'_z: H'_1\to H'_2 \}_{z\in \CZ}$ and $\{\A''_z: H''_1\to H''_2 \}_{z\in \CZ}$ two families of Fredholm operators parametrized by the same topological space $\CZ$, consider the permutation operator:
	\[
	\tau=\begin{pmatrix}
0 & 1\\
1 & 0
	\end{pmatrix}: H'_1\oplus H'_2\to H'_2\oplus H'_1 \text{ and } H''_1\oplus H''_2\to H''_2\oplus H''_1.
	\]
	Then the following digram is commutative with $r_1=\Ind \A'\cdot \Ind\A''$: 
	\[
	\begin{tikzcd}[column sep=3cm]
\Lambda(\A')\Lambda(\A'')\arrow[r,"\Id"] \arrow[d,"{q(\A',\A'')}"]& \Lambda(\A'')\Lambda(\A')\arrow[d,"{q(\A'',\A')}"]\\
\Lambda(\A'\oplus \A'')\arrow[r,"{(-1)^{r_1}(\tau^{-1},\tau)_*}"]&\Lambda(\A''\oplus \A'). 
	\end{tikzcd}
	\]
	where vertical maps are induced from \eqref{E24.3}. 
\end{lemma}
\begin{lemma}\label{LD11} Given a family of Fredholm operators $\{\A_z: H_1\to H_2 \}_{z\in \CZ}$, consider the operator
	\[
\sigma=-\Id: H_1\to H_1 \text{ and } H_2\to H_2. 
	\]
	Then the map $(\sigma^{-1},\sigma)_*$ defined by Lemma \ref{LB.6} equals 
	\[
	(-1)^{\Ind\A}: \Lambda(\A)\to \Lambda(\A). 
	\]
\end{lemma}

By Remark \ref{RD.7}, we decompose $U$ into a composition of permutations and $\sigma$, so 
\begin{align*}
r&=\Ind L_2(\Ind L_3+\Ind L_4)+\Ind L_3\Ind L_4+\Ind L_4\\
&=\Ind D_{23}\cdot \Ind D_{24}+\Ind D_{24},
\end{align*}
by Lemma \ref{LD10} and \ref{LD11}. 

\subsection{Verification of Axiom \ref{C-III}} The toy model above can partially justify the choice of $r$ in \eqref{ED.9}. Let us give another reason by verifying Axiom \ref{C-III}, in which case $L_1=L_2$. Consider the family of operators parametrized by $\tau\in [0,1]$:
\begin{equation}\label{ED.12}
U_\tau=\begin{pmatrix}
\cos \theta_\tau & -\sin\theta_\tau\\
\sin \theta_\tau & \cos\theta_\tau
\end{pmatrix}: L^2_k(E_{23})\oplus L^2_k(E_{24})\to L^2_k(E_{24})\oplus L^2_k(E_{23}),\ k\in\{0,1\}
\end{equation}
with $\theta_\tau=\theta+\tau(\pi/2-\theta): X_{ij}\to \R$, so $U_0=U$ and 
\[
U_1=\begin{pmatrix}
0 & -1\\
1 & 0
\end{pmatrix}. 
\]
We have to verify the top horizontal map $\bar{p}(R_3, R_4)$ in \eqref{ED.7} is the identity map. The diagram \eqref{ED.7} remains commutative if we carry out the homotopy $\{U_\tau\}_{\tau\in [0,1]}$:
\begin{equation}
\begin{tikzcd}[column sep=huge]
\Lambda(D_{23})\Lambda(D_{24})\arrow[r,"{\bar{p}(R_3,R_4)}"]\arrow[d,"{q_{23;24}}"] &  \Lambda(D_{24})\Lambda(D_{23})\arrow[d,"q_{24;23}"]\\
\Lambda(D_{23}\oplus D_{24})\arrow[r,"(-1)^r(U_\tau)_*"]& \Lambda(D_{24}\oplus D_{23})\
\end{tikzcd}.
\end{equation}
When $\tau=1$, by Lemma \ref{LD10} and \ref{LD11}, $\bar{p}(R_3, R_4)=\Id$ if we define $r$  by \eqref{ED.9}.

\begin{remark} By Proposition \ref{PD.4}, there exist other choices of signs $(-1)^r$ in Proposition \ref{PD.9} that also fulfill Axioms \ref{C-I}-\ref{C-IV}, but \eqref{ED.9} seems to be the preferred one by what we have discussed so far. In fact, the toy model in Subsection \ref{SubsecB.3} may provide a normalization axiom that removes the ambiguity in Proposition \ref{PD.4}. 
\end{remark}

\subsection{Verification of Axiom \ref{C-II}} In this case, $R_3=R_4$. Analogous to \ref{C-III}, we consider the family of operators parametrized by $\tau\in [0,1]$:
\[
U_\tau=\begin{pmatrix}
\cos \theta_\tau' & -\sin\theta_\tau'\\
\sin \theta_\tau' & \cos\theta_\tau'
\end{pmatrix},
\] 
with $\theta_\tau'=(1-\tau) \theta$. Then $U_0=U$ and $U_1=\Id$. In this case, $r\equiv 0 \mod 2$. 

\subsection{Verification of Axiom \ref{C-I}} For operators $R_j\in \SR, j\in\{3,4,5\}$, we have to verify that 
\[
\bar{p}(R_3, R_4)\otimes \bar{p}(R_4,R_5) =\Id\otimes \bar{p}(R_3,R_5)
\]
as maps:
\[
\Lambda(D_{13})\Lambda(D_{24})\Lambda(D_{14})\Lambda(D_{25})\to \Lambda(D_{14})\Lambda(D_{23})\Lambda(D_{15})\Lambda(D_{24}).
\]

To do this, we introduce a huge diagram and explain the construction of each piece in 5 steps:
\begin{equation}\label{ED.13}
\begin{tikzcd}[column sep=3cm,row sep=0.9cm]
\Lambda(D_{13})\Lambda(D_{24})\Lambda(D_{14})\Lambda(D_{25})\arrow[r, "{\bar{p}(R_3, R_4)\otimes \bar{p}(R_4,R_5)}"]\arrow[d,"q_{13,24}\otimes q_{14;25}"] \arrow[rd, phantom, "\BW_1"]& 
\Lambda(D_{14})\Lambda(D_{23})\Lambda(D_{15})\Lambda(D_{24})\arrow[d,"q_{14,23}\otimes q_{15;24}"]\\
\Lambda(D_{13}\oplus D_{24})\Lambda(D_{14}\oplus D_{25})\arrow[r, "U_*\otimes U_*"] \arrow[d,"q_{1342;1425}"]\arrow[rd, phantom, "\BW_2"]& 
\Lambda(D_{14}\oplus D_{23})\Lambda(D_{15}\oplus D_{24})\arrow[d,"q_{1423;1524}"]\\
\Lambda(D_{13}\oplus D_{24}\oplus D_{14}\oplus D_{25})\arrow[r, "(V_1)_*"] \arrow[rd, phantom, "\BW_3"]& \Lambda(D_{14}\oplus D_{23}\oplus D_{13}\oplus D_{24})\\
\Lambda(D_{14}\oplus D_{24}\oplus D_{13}\oplus D_{25})\arrow[r, "(V_2)_*"] \arrow[u,"(V_3)_*"]\arrow[rd, phantom, "\BW_4"]
& \Lambda(D_{14}\oplus D_{24}\oplus D_{15}\oplus D_{23})\arrow[u,"(V_4)_*"]\\
\Lambda(D_{14}\oplus D_{24})\Lambda(D_{13}\oplus D_{25})\arrow[r, "U_*\otimes U_*"] \arrow[u,"q_{1424;1325}"]\arrow[rd, phantom, "\BW_5"]
& 
\Lambda(D_{14}\oplus D_{24})\Lambda(D_{15}\oplus D_{23})\arrow[u,"q_{1424;1523}"]\\
\Lambda(D_{14})\Lambda(D_{24})\Lambda(D_{13})\Lambda(D_{25})\arrow[r, "{\Id\otimes \bar{p}(R_3,R_5)}"]\arrow[u,"q_{14,24}\otimes q_{13;25}"] & 
\Lambda(D_{14})\Lambda(D_{24})\Lambda(D_{15})\Lambda(D_{23})\arrow[u,"q_{14,24}\otimes q_{15;23}"]
\end{tikzcd}.
\end{equation}

\Step 1. The first square $(\BW_1)$ is the tensor of two diagrams of the form \eqref{ED.7}, for operators $(L_1,L_2; R_3, R_4)$ and $(L_1,L_2; R_4, R_5)$. $(\BW_1)$ is commutative if we correct it by $(-1)^{a_1}$ where
\[
a_1\colonequals r_{12;34}+r_{12;45} \text{ with } r_{ij;kl}\colonequals r(L_i, L_j; R_k, R_l) \text{ defined by }\eqref{ED.9}.  
\]

\Step 2. Similarly, the last square $(\BW_5)$ is the tensor of two diagrams of the form \eqref{ED.7}, for operators $(L_1,L_2; R_4, R_4)$ and $(L_1,L_2; R_3, R_5)$. $(\BW_5)$ is commutative if we correct it by $(-1)^{a_5}$ with
\[
a_5\colonequals r_{12;35}.
\]

\Step 3. In the second square $(\BW_2)$, the bottom horizontal arrow is induced by the diagonal matrix 
\[
V_1=\begin{pmatrix}
U & 0\\
0 & U
\end{pmatrix}. 
\]
The square $(\BW_2)$ is constructed by Lemma \ref{LD.8}, and as such is commutative. 

\Step 4. In the fourth square $(\BW_4)$, the top horizontal arrow is induced by the same matrix 
\[
V_2=V_1=\begin{pmatrix}
U & 0\\
0 & U
\end{pmatrix}. 
\]
Similarly, the square $(\BW_2)$ is commutative by Lemma \ref{LD.8}.

\Step 5. In the third square $(\BW_3)$, the two vertical maps are induced respectively by 
\begin{equation}\label{ED.16}
V_3=\begin{pmatrix}
0 & 0 & -1 & 0\\
0 & 1 & 0 & 0\\
1 & 0 & 0 & 0\\
0 & 0 & 0 & 1
\end{pmatrix},
V_4=\begin{pmatrix}
1 & 0 & 0 & 0\\
0 & 0 & 0 & -1\\
0 & 0 & 1 & 0\\
0 & 1 & 0 & 0
\end{pmatrix}
\end{equation}
The commutativity of  $(\BW_3)$ follows from the next lemma:
\begin{lemma}\label{LB.16} The matrix $V_2$ is homotopic to the composition $V_4^{-1}\circ V_1\circ V_3$ by a path of invertible operators:
	\begin{equation}\label{ED.11}
V_4^{-1}\circ V_1\circ V_3=\begin{pmatrix}
0 & -\phi_R & -\phi_L & 0\\
\phi_R &  0 & 0 & \phi_L\\
\phi_L & 0 & 0 & -\phi_R\\
0 & -\phi_L & \phi_R & 0
\end{pmatrix}
\arraycolsep=1.4pt\def\arraystretch{1.5}
\begin{array}{l}
 :L^2_k(E_{14}\oplus E_{24}\oplus E_{14}\oplus E_{25})\to\\
 \qquad\qquad L^2_k(E_{14}\oplus E_{24}\oplus E_{15}\oplus E_{23}),
\end{array}
\end{equation}
for any $k\in\{0,1\}$. The same statement holds for bundles $F_{ij}$. 
\end{lemma}
\begin{proof}[Proof of Lemma \ref{LB.16}] We construct the homotopy in 2 steps. If we compare $V_2$ with \eqref{ED.11}, only positions of $\phi_L$ are different. It suffices to move them around by homotopy. 
	
	\Step 1. Take $\tau\in [0,1]$ and define:
	\[
V_2(\tau)=\begin{pmatrix}
\phi_L\cos \frac{\pi\tau}{2}& -\phi_R & -\phi_L\sin \frac{\pi\tau}{2} & 0\\
\phi_R &  \phi_L & 0 & 0\\
\phi_L\sin \frac{\pi\tau}{2}& 0 & \phi_L \cos \frac{\pi\tau}{2}& -\phi_R\\
0 & 0 & \phi_R & \phi_L
\end{pmatrix},\ \det V_2(\tau)=\phi_L^4+\phi_R^4+2\phi_L^2\phi_R^2\cos\frac{\pi\tau}{2}. 
	\]
	
		\Step 2. Take $\tau\in [1,2]$ and define:
	\[
	V_2(\tau)=\begin{pmatrix}
	0& -\phi_R & -\phi_L& 0\\
	\phi_R &  \phi_L\sin\frac{\pi\tau}{2}  & 0 & \phi_L\cos\frac{\pi\tau}{2} \\
	\phi_L& 0 & \phi_L & -\phi_R\\
	0 & -\phi_L\cos\frac{\pi\tau}{2}  & \phi_R & \phi_L\sin\frac{\pi\tau}{2}
	\end{pmatrix},\ \det V_2(\tau)=\phi_L^4+\phi_R^4-2\phi_L^2\phi_R^2\cos\frac{\pi\tau}{2}. 
	\]
	
	Then $V_2(0)=V_0$ and $V_2(2)=V_4^{-1}\circ V_1\circ V_3$. 
\end{proof}

Back to the proof of \ref{C-I}. To figure out the overall sign involved in the diagram \eqref{ED.13}, we have to compute the compositions of all left vertical maps and all right vertical maps using Lemma \ref{LD10} and \ref{LD11}. They are induced by $V_3$ and $V_4$ respectively, so the outcomes are
\begin{align*}
a_l&=\Ind D_{13}\Ind D_{14}+(\Ind D_{13}+\Ind D_{14})\Ind D_{24}+\Ind D_{13},\\
a_r&=\Ind D_{23}\Ind D_{24}+(\Ind D_{23}+\Ind D_{24})\Ind D_{15}+\Ind D_{23}.
\end{align*}

We have to verify that 
\begin{equation}\label{ED.14}
a_1+a_5+a_l+a_r\equiv 0\mod 2,
\end{equation}
which is the sum of 14 terms. In the computation below, we use the excision principle \eqref{ED.6} and set
\[
b=\Ind D_{1j}-\Ind D_{2j},\ 3\leq j\le 5,
\]
so
\begin{align*}
a_1+a_5+a_l+a_r&=2\Ind D_{23}\Ind D_{24}+2\Ind D_{25}+(\Ind D_{13}+\Ind D_{23})\\
&\qquad+(\Ind D_{23}+\Ind D_{24})(\Ind D_{15}+\Ind D_{25})\\
&\qquad+\Ind D_{24}(1+\Ind D_{14})+\Ind D_{13}(\Ind D_{14}+\Ind D_{24})\\
&\equiv b+(\Ind D_{23}+\Ind D_{24})\cdot b\\
&\qquad+\Ind D_{24}(1+\Ind D_{24})+\Ind D_{24}\cdot b+ \Ind D_{13}\cdot b\\
&\equiv b+b^2\equiv 0\mod 2. 
\end{align*}

This completes the proof of \ref{C-I}.
\begin{remark} The computation above is not enlightening at all. However, once we know the sum \eqref{ED.14} admits an expression that involves indices of $D_{ij}$ only, one may refer to the case when $Y=\emptyset$ in Subsection \ref{SubsecB.3}, as the computation does not see the difference. In that case, there is much easier to see why $\{q(R_3,R_4)\}$ are associative.
\end{remark}
\subsection{Verification of Axiom \ref{C-IV}}\label{SubsecB.8} We have formulated the problem in a way that is asymmetric in $\SL$ and $\SR$.  But Axiom \ref{C-IV} is identical to Axiom \ref{C-III} if one  interchanges the roles of $\SL$ and $\SR$. The proof \ref{C-IV} follows the same line of arguments as above. For any operators $L_i\in \SL, 0\leq i\leq 2$, we have to verify that 
\[
\bar{p}(L_0, L_1; R_3, R_4)\otimes \bar{p}(L_1, L_2; R_3,R_4) =\Id\otimes \bar{p}(L_0, L_2; R_3, R_4)
\]
as maps:
\[
\Lambda(D_{03})\Lambda(D_{14})\Lambda(D_{13})\Lambda(D_{24})\to \Lambda(D_{04})\Lambda(D_{13})\Lambda(D_{14})\Lambda(D_{23}),
\]
and the corresponding diagram is:
\begin{equation}\label{ED.15}
\begin{tikzcd}[column sep=3cm,row sep=0.9cm]
\Lambda(D_{03})\Lambda(D_{14})\Lambda(D_{13})\Lambda(D_{24})\arrow[r, "{\bar{p}(L_0, L_1)\otimes \bar{p}(L_1,L_2)}"]\arrow[d,"q_{03,14}\otimes q_{13;24}"] & 
\Lambda(D_{04})\Lambda(D_{13})\Lambda(D_{14})\Lambda(D_{23})\arrow[d,"q_{04,13}\otimes q_{14;23}"]\\
\Lambda(D_{03}\oplus D_{14})\Lambda(D_{13}\oplus D_{24})\arrow[r, "U_*\otimes U_*"] \arrow[d,"q_{0314;1324}"]& 
\Lambda(D_{04}\oplus D_{13})\Lambda(D_{14}\oplus D_{23})\arrow[d,"q_{0413;1423}"]\\
\Lambda(D_{03}\oplus D_{14}\oplus D_{13}\oplus D_{24})\arrow[r, "(V_1)_*"] & \Lambda(D_{04}\oplus D_{13}\oplus D_{14}\oplus D_{23})\\
\Lambda(D_{13}\oplus D_{14}\oplus D_{03}\oplus D_{24})\arrow[r, "(V_2)_*"] \arrow[u,"(V_3)_*"]
& \Lambda(D_{14}\oplus D_{13}\oplus D_{04}\oplus D_{23})\arrow[u,"(V_3)_*"]\\
\Lambda(D_{13}\oplus D_{14})\Lambda(D_{03}\oplus D_{24})\arrow[r, "U_*\otimes U_*"] \arrow[u,"q_{1413;0324}"]
& 
\Lambda(D_{13}\oplus D_{14})\Lambda(D_{04}\oplus D_{23})\arrow[u,"q_{1314;0423}"]\\
\Lambda(D_{14})\Lambda(D_{13})\Lambda(D_{03})\Lambda(D_{24})\arrow[r, "{\Id\otimes \bar{p}(R_3,R_5)}"]\arrow[u,"q_{13,14}\otimes q_{03;24}"] & 
\Lambda(D_{14})\Lambda(D_{13})\Lambda(D_{04})\Lambda(D_{23})\arrow[u,"q_{14,13}\otimes q_{04;23}"]
\end{tikzcd}.
\end{equation}
with $V_3$ defined as in \eqref{ED.16}. Again we have to verify the sum
\[
a_1'+a_5'+a_l'+a_r'\equiv 0\mod 2
\]
where 
\begin{align*}
a_1'&=r_{01;34}+r_{12;34}= r_{11;34}+r_{02;34}=a_5',\\
a_l'&=\Ind D_{13}\Ind D_{14}+(\Ind D_{13}+\Ind D_{14})\Ind D_{03}+\Ind D_{03},  &&\\
a_r'&=\Ind D_{13}\Ind D_{14}+(\Ind D_{13}+\Ind D_{14})\Ind D_{04}+\Ind D_{04}. &&
\end{align*}
If we set $c=\Ind D_{i3}-\Ind D_{i4},\ i\in \{0,1,2\}$, then 
\[
a_1'+a_5'+a_l'+a_r'\equiv c^2+c\equiv 0\mod 2. 
\]

The last step is to show that the matrix $V_2$ is homotopic to 
\[
V_3^{-1}\circ V_1\circ V_3=\begin{pmatrix}
\phi_L & \color{blue}{0} & 0& \color{blue}{-\phi_R}\\
\color{red}{0}& \phi_L & \color{red}{-\phi_R} & 0\\
0 & \color{blue}{\phi_R}& \phi_L & \color{blue}{0} \\
\color{red}{\phi_R}& 0 & \color{red}{0} & \phi_L
\end{pmatrix}
\arraycolsep=1.4pt\def\arraystretch{1.5}
\begin{array}{l}
:L^2_k(E_{13}\oplus E_{14}\oplus E_{03}\oplus E_{24})\to\\
\qquad\qquad L^2_k(E_{14}\oplus E_{13}\oplus E_{04}\oplus E_{23}).
\end{array}
\]

The homotopy is again constructed by ``rotating" the four entries colored red and the other four entries colored blue. The proofs of Proposition \ref{PD.3} and \ref{PD.9} are now completed.

\subsection{Proof of Proposition \ref{PD.5} }\label{SubsecD.9} We claim that the construction in Proposition \ref{PD.9} still works in this general setup, if $\SL$ is an even class of operators in the sense of Definition \ref{DD.5}. If we stick to operators $\hr=(P, R, Q)\in \widehat{\SR}$ with $P=\emptyset$, then the proof of Proposition \ref{PD.9} remains valid, since it does not see the difference. 

In the general case, let 
$
\hr_j=(P_j, R_j,Q_j)\in \widehat{\SR}, j=3,4. 
$
We wish to compare $p(\hr_3,\hr_4)$ with $p(R_3,R_4)$. To illustrate, we focus on the special case when $P_3=Q_3=\emptyset$ and verify the following digram is commutative:
\begin{equation}\label{ED.17}
\begin{tikzcd}[column sep=4cm]
\Lambda(D_{13})\Lambda(P_4)\Lambda(D_{24})\Lambda(Q_4)\arrow[r, "{\Id\otimes \bar{p}(R_3, R_4)}"]\arrow[d, "\Id\otimes q"]& \Lambda(P_4)\Lambda(D_{13})\Lambda(Q_4)\Lambda(D_{23})\arrow[d, "q\otimes\Id"]\\
\Lambda(D_{13})\Lambda(P_3\oplus D_{24}\oplus Q_4)\arrow[r,"{\bar{p}(R_3,\hr_4)}"]\arrow[d,"q_{13;24}"]& \Lambda(P_4\oplus D_{14}\oplus Q_4)\Lambda(D_{23})\arrow[d,"q_{14;23}"]\\
\Lambda(D_{13}\oplus P_4\oplus D_{24}\oplus Q_4)\arrow[r,"(-1)^r U_*"] &\Lambda(P_4\oplus D_{14}\oplus Q_4\oplus D_{23}). 
\end{tikzcd}
\end{equation} 

The second square comes from the digram \eqref{ED.7} with $R_4$ replaced by $\hr_4$; so
\[
U=\begin{pmatrix}
0 & -1 & 0 & 0\\
\color{red}{\phi_L}& 0 & \color{red}{-\phi_R}& 0\\
0 & 0& 0& -1\\
\color{red}{\phi_R}& 0 &\color{red}{\phi_L} & 0 
\end{pmatrix}
\]
and 
\[
r=(1+\Ind D_{23})\Ind (L_2\# \hr_4)=(1+\Ind D_{23})(\Ind P_4+\Ind D_{23}+\Ind Q_3). 
\]

One may verify that the first square of \eqref{ED.17} also is commutative, using diagrams like \eqref{ED.13} and \eqref{ED.15}. The computation boils down to 
\[
\Ind P_4\cdot (\Ind D_{13}+\Ind D_{23})\equiv 0 \mod 2,
\]
so the assumption that $\SL$ is even is crucially here.  In general, one has to verify that a digram like \eqref{ED.17} commutes for arbitrary $\hr_3, \hr_4\in \widehat{\SR}$. This reduces the problem from $\widehat{\SR}$ to the smaller family $\SR$: it suffices to verify axiom \ref{C-I}-\ref{C-IV} for $(\SL, \SR)$, but this is done in Proposition \ref{PD.9}. Details are left for the readers. 

Finally, to verify the additional property \eqref{ED.19}, we set
\[
\hr_3=(\emptyset, R_3, \emptyset),\ \hr_4=(P_3, R_3, Q_3),
\]
in the diagram \eqref{ED.17}. Then we use the fact that the top arrow  $\bar{p}(R_3,R_3)=\Id$ to conclude. 
\subsection{Proof of Theorem \ref{T24.2}}\label{SubsecD.10} Having developed the abstract theory of relative orientations, let us explain its application in gauge theory and prove Theorem \ref{T24.2}. Consider a strict cobordism $X: Y_1\to Y_2$, let 
\[
Y=\partial X=(-Y_1)\cup ([-1,1]\times\Sigma)\cup Y_2.
\]
We regard $Y$ as a compact oriented 3-manifold by smoothing the corners. 

For any relative \spinc cobordism $\bs_X\in \Spincr(X;\bs_1,\bs_2)$, let the operator
\[
L_{\bs_X}
\]
be the restriction of the Fredholm operator $\CQ(\fc_1, \bs_X, \fc_2)$ on $X$ and $R_*$ be the restriction on the complement $\CX\setminus X$, then 
\[
\CQ(\fc_1, \bs_X, \fc_2)=L_{\bs_X}\# R_*.
\]
Let $\SL=\{L_{\bs_X}: \bs_X\in \Spincr(X;\bs_1,\bs_2) \}$ be the space of all such operators. The underlying manifold of $L_{\bs_X}$ is always the compact 4-manifold $X$. As for the other class $\SR$, let $X_3$ be any smooth 4-manifold with boundary $(-Y)$
 such that  $X\#X_3$ is a closed oriented manifold and  $\bs_X|_{\partial X}$ extends to a \spinc structure $\bs_3$ on $X\#X_3$. Define $R_3$ to be the linearized Seiberg-Witten map together with the linearized gauge fixing equation on $X\#X_3$ restricted on $X_3$. As a result 
 \[
 L_{\bs_X}\# R_3
 \]
 is the linearized operator at some configuration for the closed \spinc manifold $(X\#X_3, \bs_X\# \bs_3)$. Set $\SR=\{R_*\}\cup \{\text{all possible }(X_3, R_3)\}$. Our goal is to construct the natural bijection 
\[
e_E: \Lambda( L_{\bs_X}\# R_*)\to \Lambda( L_{\bs_X\otimes E}\# R_*),
\]
for each relative line bundle $[E]\in H^2(X,\partial X; \Z)$. (Here we changed the notation for a line bundle to avoid confusion). Using the set of bijections $\{p(R_3,R_4)\}$ in Proposition \ref{PD.3} or \ref{PD.9}, we can define $e_E$ using any compact piece $(X_3, R_3)$ instead: 
\[
e_E: \Lambda( L_{\bs_X}\# R_3)\to \Lambda( L_{\bs_X\otimes E}\# R_3).
\]
It is constructed as follows. The linearized operator at a reducible configuration on $X_3\# R_3$ is 
\[
(d^*\oplus d^+)\oplus D_A^+
\]
The second operator is complex linear, while the first one is independent of the line bundle $[E]\in H^2(X,\partial X;\Z)$, so $e_E$ is defined by the commutative digram
\[
\begin{tikzcd}
\Lambda(d^*\oplus d^+)\Lambda(D_{A}^+)\arrow[r,"\Id\otimes h"]\arrow[d,"q"]&\Lambda(d^*\oplus d^+)\Lambda(D_{A'}^+)\arrow[d,"q"]\\
\Lambda( L_{\bs_X}\# R_3)\arrow[r,"e_E"]&\Lambda( L_{\bs_X\otimes L}\# R_3),
\end{tikzcd}
\]
where $h:\Lambda(D_{A}^+)\to \Lambda(D_{A'}^+)$ preserves the complex orientations. Notice that $\{e_E\}$ is independent of the compact piece $(X_3, R_3)$ by our construction of $\{p(R_3, R_4)\}$. 

Now the first property \ref{U1} of Theorem \ref{T24.2} follows from Axiom \ref{C-IV}. 

As for \ref{U2}, it suffices to address the special case when either $[E_{12}]=0$ or $[E_{23}]=0$. Technically, we have to work with the operators $\CQ'$ defined in \eqref{E24.7}, which involve manifolds with boundary and spectral projections. We can enlarge the family $\SR$ to incorporate such operators, so it is not a problem.

At this point, we conclude using the additional property \eqref{ED.19} in Proposition \ref{PD.5} by setting either $P_3=\emptyset$ or $Q_3=\emptyset$. The assumption is verified by the next lemma. 
\begin{lemma} The class of operators $\SL\colonequals\{L_{\bs_X}: \bs_X\in \Spincr(X; \bs_1,\bs_2) \}$ is even in the sense of Definition \ref{DD.5}. 
\end{lemma}
\begin{proof}[Proof of Lemma] By the excision principle, it suffices to verify the condition \eqref{ED.18} for a special operator $(X_3, R_3)\in \SR$. In particular, we take $(X_3, R_3)$ to be a compact piece. 
\end{proof}

\section{Some Formulae of Taubes}\label{AppE}

In this appendix, we summarize some formulae of the Seiberg-Witten equations from Taubes' paper \cite[Section 2]{Taubes96}. Although the primary applications of \cite{Taubes96} focus on symplectic 4-manifolds, it is well-known that some of them generalize to any Riemannian 4-manifolds. This observation forms the basis of the finiteness result in Section \ref{Sec25}. For the sake of completeness, we record their statements and prove Lemma \ref{L25.3}. 

\smallskip

Given an oriented Riemannian 4-manifold $X$, consider the Seiberg-Witten equations on $X$ perturbed by a self-dual 2-form $\omega^+\in \Gamma(X, i\Lambda^+ X)$:
\begin{align}
\half \rho_4(F_{A^t}^+)-(\Phi\Phi^*)_0&=\rho_4(\omega^+),\label{EE.1}\\
D_A^+\Phi&=0. \label{EE.2}
\end{align}
The 2-form $\omega^+$ is not assumed to be closed, and $X$ is not necessarily compact. Set 
\[
F\colonequals \half F_{A^t}\in \Omega^2(X, i\R).
\]
Then the curvature tensor $F_{A}|_{S^+}\in \Gamma(X, i\Lambda^2X\otimes\End(S^+))$ can be written as 
\begin{equation}\label{EE.3}
F_A|_{S^+}=F\otimes\Id_{S^+}+\SU,
\end{equation}
where $\SU$ is the traceless part of $F_A|_{S^+}$ and is independent of the \spinc connection $A$. By the Weitzenb\"{o}ck formula \cite[(4.14)]{Bible}, if $(A,\Phi)$ solves \eqref{EE.1}\eqref{EE.2}, then
\begin{equation}\label{EE.4}
\half \Delta_A\Phi+\half |\Phi|^2\Phi=-\rho_4(\omega^+)\Phi-\frac{s}{4}\Phi. 
\end{equation}

Our goal is to find explicit formulae for $
d^*F,\ \Delta |F|^2 \text{ and } \Delta |\nabla_A \Phi|^2. $
\begin{lemma}\label{LE.1} For any solution $(A,\Phi)$ to the perturbed Seiberg-Witten equations \eqref{EE.1}\eqref{EE.2}, we have \[
	d^*F=2d^*F^+=2d^*F^-=i\im \langle \Phi,\nabla_A\Phi\rangle+2d^*\omega^+. 
	\]
\end{lemma}
\begin{proof}
	Since $F$ is a closed 2-form, $
	dF^-=-dF^+$ and $d^*F^-=d^*F^+. 
	$
	To compute $d^*F^+$, we pick a local orthonormal framing $\{e_i\}_{1\leq i\leq 4}$ such that $\nabla_{e_i}e_k=0$ at $x\in X$. Moreover, we exploit the formula from \cite[Lemma 5.13]{Spin}:
	\[
	d^*F^+=-(\nabla_{e_i} F^+)(e_i,\cdot).
	\]
	By the curvature equation \eqref{EE.1}, we have 
	\begin{align*}
	(F^+-\omega^+)(e_i, e_k)&=-\frac{1}{4}\tr(\rho_4(F^+-\omega^+)\rho_4(e_i)\rho_4(e_k))=-\frac{1}{4}\tr((\Phi\Phi^*)_0\rho_4(e_i)\rho_4(e_k)).
	\end{align*}
	Now we use the Dirac equation \eqref{EE.2} to compute:
	\begin{align*}
	I\colonequals &-\nabla_{e_i}(F^+-\omega^+)(e_i, e_k)\\
	=&\frac{1}{4}\tr([(\nabla_{e_i}^A\Phi)\Phi^*+\Phi(\nabla_{e_i}^A\Phi)^*-\re\langle \Phi, \nabla_{e_i}^A\Phi\rangle \otimes\Id_{S^+}]\rho_4(e_i)\rho_4(e_k)))\\
	=&\frac{1}{4}\tr(\rho_4(e_i)\rho_4(e_k)(\nabla_{e_i}^A\Phi)\Phi^*)+\frac{1}{4}\tr(\Phi\underbrace{(\nabla_{e_i}^A\Phi)^*\rho_4(e_i)}_{=0}\rho_4(e_k))+\half \re\langle \Phi,\nabla_{e_k}^A\Phi\rangle. \\
	\end{align*}
	For the first term, we commute $\rho_4(e_i)$ and $\rho_4(e_k)$ to derive:
	\begin{align*}
	I=&-\frac{1}{4}\tr(\rho_4(e_k)\underbrace{\rho_4(e_k)(\nabla_{e_i}^A\Phi)}_{=0}\Phi^*)+\half\tr(\rho_4(e_k)\rho_4(e_k)(\nabla_{e_i}^A\Phi)\Phi^*)+ \half\re\langle \Phi,\nabla_{e_k}^A\Phi\rangle\\
=&-\half \langle \nabla_{e_k}^A\Phi, \Phi\rangle+\half \re\langle \Phi,\nabla_{e_k}^A\Phi\rangle=\half\cdot i\im \langle \Phi, \nabla_{e_k}^A\Phi\rangle. 
	\end{align*}
	We conclude that $2d^*F^+=2d^*\omega^++i\im \langle \Phi,\nabla_{e_k}^A\Phi\rangle\cdot \omega_k$, where  $\{\omega_i\}_{1\leq i\leq 4}$ are co-vectors dual to $\{e_i\}$. 
\end{proof}

Now we are read to compute the Hodge Laplacian of the curvature 2-form $F$. 
\begin{proposition}
	For any solution $(A,\Phi)$ to the perturbed Seiberg-Witten equations \eqref{EE.1}\eqref{EE.2}, we have
	\[
	(d+d^*)^2 F+|\Phi|^2 F=\langle \nabla_A\Phi\wedge \nabla_A\Phi\rangle+2dd^*\omega^++I(\Phi,\Phi),
	\]
	where $\langle \nabla_A\Phi\wedge \nabla_A\Phi\rangle$ denotes the imaginary valued 2-form 
	\[
	\sum_{i,j} \omega_i\wedge \omega_j\cdot \langle \nabla_{e_i}^A\Phi, \nabla_{e_j}^A\Phi\rangle=2i\sum_{i<j}\omega_i\wedge \omega_j \cdot \im  \langle \nabla_{e_i}^A\Phi, \nabla_{e_j}^A\Phi\rangle,
	\]
	and $I(\Phi,\Phi)=\sum_{i<k}\omega_i\wedge \omega_k\cdot i\im \langle \Phi, \SU(e_i, e_k)\Phi\rangle$ is a symmetric bilinear form.
\end{proposition}
\begin{proof}Since $dF=0$, it suffices to compute $dd^*F$. We exploit the formula from \cite[Lemma 5.13]{Spin}:
	\[
	d\nu=\omega_i\wedge \nabla_{e_i}\nu 
	\]
	for any $\nu\in \Omega^*(X,i\R)$. 
In particular, set $\nu=i\im\langle \Phi,\nabla_A\Phi\rangle$:
	\begin{align*}
	d\nu&=\omega_i\wedge \nabla_{e_i}(\omega_k\otimes i\im\langle \Phi,\nabla_{e_k}\Phi\rangle)=\omega_i\wedge \omega_k\cdot (i\im\langle \nabla_{e_i}^A\Phi,\nabla_{e_k}^A\Phi\rangle+i\im\langle \Phi,\nabla_{e_i}^A\nabla_{e_k}^A\Phi\rangle)\\
	&=\sum_{i<k}\omega_i\wedge \omega_k\cdot (2i\im\langle \nabla_{e_i}^A\Phi,\nabla_{e_k}^A\Phi\rangle+i\im\langle \Phi, F_A(e_i,e_k)\Phi\rangle)\\
	&=\langle \nabla_A\Phi\wedge \nabla_A\Phi\rangle-|\Phi|^2 F+I(\Phi, \Phi).
	\end{align*}
	At the last step, we used the decomposition 
	$
	F_A|_{S^+}=\half F_{A^t}\otimes \Id_{S^+}+\SU
$ from \eqref{EE.3}.
\end{proof}

Finally, we address the Laplacian of $|\nabla_A\Phi|^2$. Note that 
\[
\half \Delta|\nabla_A\Phi|^2+|\Hess_A\Phi|^2=\re\langle (\nabla_A^*\nabla_A)\nabla_A\Phi,\nabla_A\Phi\rangle. 
\]
We start with an explicit formula for the commutator $\Delta_A\nabla_A-\nabla_A\Delta_A$ where $\Delta_A\colonequals \nabla_A^*\nabla_A$. 
\begin{lemma}\label{LE.3} For any \spinc connection $A$ and any spinor $\Phi$, we have an identity:
	\begin{align*}
	\langle \Delta_A\nabla_A\Phi,\nabla_A\Phi\rangle&=\langle \nabla_A(\Delta_A\Phi),\nabla_A\Phi\rangle-\Ric(e_i, e_j)\re\langle \nabla_{e_i}\Phi,\nabla_{e_j}\Phi\rangle\\
	&\qquad+\re\langle (d^*_AF_A)\Phi, \nabla_A\Phi\rangle-2\re\langle F_A(e_i,e_j)\nabla_{e_i}\Phi,\nabla_{e_j}\Phi\rangle. 
	\end{align*}
\end{lemma}
\begin{remark}\label{RE.4} The last two terms can be recognized as follows:
	\begin{align*}
		d_A^*F_A&=d^*F\otimes\Id_{S^+}+d^*_{A_0}\SU,\\
	\re\langle F_A(e_i,e_j)\nabla_{e_i}\Phi,\nabla_{e_j}\Phi\rangle&=\re\langle \SU(e_i,e_j)\nabla_{e_i}\Phi,\nabla_{e_j}\Phi\rangle+2\sum_{i<j} F(e_i, e_j)\cdot i\im\langle \nabla_{e_i}\Phi,\nabla_{e_j}\Phi\rangle\\
	&=\re\langle \SU(e_i,e_j)\nabla_{e_i}\Phi,\nabla_{e_j}\Phi\rangle-\langle F, \langle\nabla_A\Phi\wedge\nabla_A\Phi\rangle\rangle.
	\end{align*}
Note that  $d^*_{A_0}\SU$ is independent of the reference \spinc connection $A_0$.
\end{remark}
\begin{proof}[Proof of Lemma \ref{LE.3}] Let $\{e_i=\partial_i\}$ in a normal coordinate at $x\in X$ and $\{\omega_i\}_{1\leq i\leq 4}$ be the co-vectors dual to $\{e_i\}$. Then
	\begin{align*}
	\nabla_A\Phi&=\omega_i\otimes\nabla_{e_i}\Phi,\\
	\Delta_A\nabla_A\Phi&=-\omega_i\otimes \nabla_{e_j}\nabla_{e_j}\nabla_{e_i}\Phi-\nabla_{e_j}\nabla_{e_j}\omega_i\otimes\nabla_{e_i}\Phi\\
	&=-\omega_i\otimes \nabla_{e_i}\nabla_{e_j}\nabla_{e_j}\Phi-2\omega_i\otimes F_A(e_j, e_i)\nabla_{e_j}\Phi\\
	&\qquad+\omega_i\otimes (d^*_AF_A)(e_i)\Phi-\nabla_{e_j}\nabla_{e_j}\omega_i\otimes\nabla_{e_i}\Phi,\\
	\Delta_A\Phi&=-\nabla_{e_j}\nabla_{e_j}\Phi-\langle \nabla_{e_j}\omega_i,\omega_j\rangle \nabla_{e_i}\Phi,\\
	\nabla_A	\Delta_A\Phi&=-\omega_i\otimes \nabla_{e_i}\nabla_{e_j}\nabla_{e_j}\Phi-\omega_k\otimes \langle\nabla_{e_k} \nabla_{e_j}\omega_i,\omega_j\rangle \nabla_{e_i}\Phi,
	\end{align*}
	Now take inner products with $\nabla_A\Phi$. To find the Ricci curvature, use relations:
	\begin{align*}
	\langle \nabla_{e_j}\nabla_{e_j}\omega_i, \omega_k\rangle&=-\langle\nabla_{e_j}\nabla_{e_j} e_k, e_i\rangle,& \langle \nabla_{e_k}\nabla_{e_j}\omega_i, \omega_j\rangle&=-\langle\nabla_{e_k}\nabla_{e_j} e_j, e_i\rangle,
	\end{align*}
	and $\nabla_{e_j}e_k=\nabla_{e_k}e_j$ in a normal neighborhood.
\end{proof}
\begin{proposition}\label{PB.5}	For any solution $(A,\Phi)$ to the perturbed Seiberg-Witten equations, we have
	\begin{align*}
	&\half \Delta |\nabla_A\Phi|^2+|\Hess_A\Phi|^2+\half |\Phi|^2|\nabla_A\Phi|^2+|\langle \nabla_A\Phi, \Phi\rangle|^2+\frac{s}{4}|\nabla_A\Phi|^2\\
	=&2\langle F,\langle\nabla_A\Phi\wedge \nabla_A\Phi\rangle\rangle+J_1+J_2 
	\end{align*}
	where
	\begin{align*}
	J_1&=-2\langle \SU(e_i, e_j)\nabla_{e_i}\Phi,\nabla_{e_j}\Phi\rangle-\Ric(e_i, e_j)\re\langle \nabla_{e_i}\Phi,\nabla_{e_j}\Phi\rangle-\re\langle \rho_4(\omega^+)\nabla_A\Phi,\nabla_A\Phi\rangle,\\
	J_2&=\re\langle (d_{A_0}^*\SU)\Phi,\nabla_A\Phi\rangle+\re\langle 2d^*\omega^+\otimes \Phi, \nabla_A\Phi\rangle\\
	&\hspace{1.527in}-\re\langle \rho_4(\nabla\omega^+)\Phi,\nabla_A\Phi\rangle-\frac{1}{4}\re\langle ds\otimes \Phi, \nabla_A\Phi\rangle. 
	\end{align*}
	In particular, $|J_1|\leq C|\nabla_A\Phi|^2$ and $|J_2|\leq C|\nabla_A\Phi||\Phi|$ for some function $C:X\to \R_{>0}$ depending only on $(g_X,\omega^+)$.
\end{proposition}
\begin{proof} By Lemma \ref{LE.3} and Remark \ref{RE.4}, it suffices to compute
	\[
	\re\langle (d^*F)\Phi,\nabla_A\Phi\rangle \text{ and }\re\langle \nabla_A\Delta_A\Phi, \nabla_A\Phi\rangle. 
	\]
	For the first term, we apply Lemma \ref{LE.1}:
	\begin{align*}
	\re\langle (d^*F)\Phi,\nabla_A\Phi\rangle&=-|\im \langle\Phi,\nabla_A\Phi\rangle|^2+\re\langle 2d^*\omega^+\otimes \Phi, \nabla_A\Phi\rangle. 
	\end{align*}
	
	For the second term, we apply $\Delta_A$ to \eqref{EE.4} to compute
	\begin{align*}
	\nabla_A\Delta_A\Phi+\half |\Phi|^2\nabla_A\Phi+\re\langle\Phi,\nabla_A\Phi\rangle \Phi=\rho_4(\omega^+)\nabla_A\Phi-\rho_4(\nabla\omega^+)\Phi-\frac{s}{4}\nabla_A\Phi-\frac{1}{4}ds\otimes\Phi.
	\end{align*}
To conclude, take the inner product with $\nabla_A\Phi$:
	\begin{align*}
	&\re\langle \nabla_A\Delta_A\Phi, \nabla_A\Phi\rangle+\half |\Phi|^2|\nabla_A\Phi|^2+|\re\langle\Phi,\nabla_A\Phi\rangle|^2\\
	=&-\re\langle \rho_4(\omega^+)\nabla_A\Phi,\nabla_A\Phi\rangle-\re\langle \rho_4(\nabla\omega^+)\Phi,\nabla_A\Phi\rangle-\frac{s}{4}|\nabla_A\Phi|^2-\frac{1}{4}\re\langle ds\otimes \Phi, \nabla_A\Phi\rangle.\qedhere
	\end{align*}
\end{proof}

Finally, let us state the corresponding results for 3-manifolds from which one can easily deduce Lemma \ref{L25.3}.
\begin{proposition}\label{PE.6} Let $(Y,g_Y)$ be any Riemannian 3-manifold, $\Rm_Y$ be the curvature tensor and $\omega\in \Omega^2(Y, i\R)$ be a closed 2-form. For any solution $(B,\Psi)$ to the 3-dimensional Seiberg-Witten equations \eqref{3DSWEQ}, write 
	\[
	F=\half F_{B^t}\in \Omega^2(Y, i\R). 
	\]
	Then we have 
	\[
	d^*F=i\im\langle \Psi,\nabla_B\Psi\rangle+d^*\omega,
	\]
	and 
\begin{align*}
\half \Delta |\nabla_B\Psi|^2+&|\Hess_B\Psi|^2+\half |\Phi|^2|\nabla_B\Psi|^2+|\langle \nabla_B\Psi, \Psi\rangle|^2\\
=&2\langle F,\langle\nabla_B\Psi\wedge \nabla_B\Psi\rangle\rangle+J_1(\nabla_B\Psi,\nabla_B\Psi)+J_2(\Psi,\nabla_B\Psi).
\end{align*}
where $J_1$ and $J_2$ are certain bilinear maps depending only on $R_Y$, $\omega$ and their first derivatives. In particular, if 
\[
\|\Rm_Y\|_{L^\infty_1}, \|\omega\|_{L^\infty_1}<\infty,
\]
then 
 $$|J_1(\nabla_B\Psi,\nabla_B\Psi)|\leq C|\nabla_B\Psi|^2\text{ and }|J_2(\nabla_B\Psi,\Psi)|\leq C|\nabla_B\Psi||\Psi|,$$
 for some constant $C>0$.
\end{proposition}

Proposition \ref{PE.6} follows from its 4-dimensional analogue: Lemma \ref{LE.1} and Proposition \ref{PB.5}.

\bibliographystyle{alpha}
\bibliography{sample}

\end{document}